\documentclass[reqno, 10pt, a4paper]{amsart}
\usepackage{amsmath, amssymb, enumerate, esint,tikz,mathabx,float}
\usepackage[top=3cm, bottom=3cm, left=3cm, right=3 cm] {geometry}
\usepackage{datetime}
\usepackage[mathscr]{eucal}
\usepackage{bbm}
\usepackage[colorlinks=true,linkcolor=blue, citecolor=red]{hyperref}
\usepackage[nameinlink,capitalize]{cleveref}

\crefname{equation}{}{} 
\numberwithin{equation}{section}
\newcommand{\secref}[1]{\ref{#1}. \nameref{#1}}
\allowdisplaybreaks

\usetikzlibrary{shapes.geometric, arrows}
\tikzstyle{startstop} = [rectangle, rounded corners, minimum width=3cm, minimum height=1cm,text centered, draw=black, fill=red!20]
\tikzstyle{io} = [trapezium, trapezium left angle=70, trapezium right angle=110, minimum width=3cm, minimum height=1cm, text centered, draw=black, fill=blue!30]
\tikzstyle{process} = [rectangle, minimum width=3cm, minimum height=1cm, text centered, draw=black,fill=orange!25]
\tikzstyle{decision} = [diamond, minimum width=3cm, minimum height=1cm, text centered, draw=black, fill=green!30]
\tikzstyle{arrow} = [thick,->,>=stealth]

\newtheorem{theoI}{Theorem}

\newtheorem{theo}{Theorem}[section]
\newtheorem{coro}[theo]{Corollary}
\newtheorem{lemm}[theo]{Lemma}
\newtheorem{prop}[theo]{Proposition}

\theoremstyle{definition}
\newtheorem{defi}[theo]{Definition}
\newtheorem{exam}[theo]{Example}
\newtheorem{rema}[theo]{Remark}

\newtheorem{assumption}[theo]{Assumption}
\newtheorem{setting}[theo]{Setting}
\newtheorem{conv}[theo]{Convention}

\newcommand{\bD}{{\bf D}}
 
\newcommand{\B}{\mathbb B}
\newcommand{\D}{\mathbb D}
\newcommand{\E}{\mathbb E}
\newcommand{\bF}{\mathbb F}
\newcommand{\bN}{\mathbb N}
\newcommand{\p}{\mathbb P}
\newcommand{\Q}{\mathbb Q}
\newcommand{\R}{\mathbb R}
\newcommand{\Z}{\mathbb Z}
\newcommand{\1}{\mathbbm{1}}

\newcommand{\od}{\mathrm{d}}
\newcommand{\im}{\mathrm{i}}

\newcommand{\cB}{\mathcal B}
\newcommand{\F}{\mathcal F}
\newcommand{\cG}{\mathcal G}
\newcommand{\cI}{\mathcal I}
\newcommand{\cL}{\mathcal L}
\newcommand{\cN}{\mathcal N}
\newcommand{\cR}{{\mathcal R}}
\newcommand{\cS}{\mathcal S}
\newcommand{\cT}{\mathcal T}

\newcommand{\cU}{\mathcal U}
\newcommand{\ep}{\varepsilon}

\newcommand{\bmo}{\mathrm{bmo}}
\newcommand{\BMO}{\mathrm{BMO}}
\newcommand{\bvlocal}{{{\mathrm{BV} _{\rm loc}(\R)}}}
\newcommand{\lbvlocal}[1]{{{\mathrm{LBV}^{{#1}}_{\rm loc}(\R)}}}
\newcommand{\CL}{\mathrm{CL}}
\newcommand{\defX}{{\mathcal D}_X}
\newcommand{\dom}{{\rm Dom}}
\newcommand{\hoel}[1]{{\textnormal{H\"ol}}_{#1}(\R)}
\newcommand{\hoell}[2]{{\textnormal{H\"ol}}_{#1}^{#2}(\R)}
\newcommand{\hoelO}[1]{{\textnormal{H\"ol}}_{#1}(\R)}
\newcommand{\RO}{{{\mathbb R}\!\setminus\!\{0\}}}
\newcommand{\defg}{{\mathcal R_Y}}
\newcommand{\cSM}{{\mathcal {SM}}}
\newcommand{\supp}{\operatorname{supp}}
\newcommand{\timed}{{\mathbb I}}

\newcommand{\ce}[2]{\E^{#1}\hspace{-.1cm}\left [ #2 \right ]}
\newcommand{\opD}{{\overline{D}}} 
\newcommand{\e}{\operatorname{e}}
\newcommand{\intsq}[2]{[#1;#2]}
\newcommand{\intsqw}[3]{[#1;#3]^{#2}}
\newcommand{\losc}{\underline{{\rm Osc}}}
\newcommand{\uosc}{\overline{{\rm Osc}}}
\newcommand{\sign}{{\rm sign}}
\newcommand{\wt}{\widetilde}
\newcommand{\pd}{\partial}

\newcommand{\sptext}[3]{\hspace{#1 em}\mbox{#2}\hspace{#3 em}}
\def\({\left(}
\def\){\right)}
\newcommand{\ov}{\overline}
\newcommand{\un}{\underline}
\newcommand{\equa}{\begin{eqnarray*}}
\newcommand{\tion}{\end{eqnarray*}}

\begin{document}

\title[Riemann-Liouville type operators]{On Riemann-Liouville type operators, BMO, gradient estimates 
       in the L\'evy-It\^o space, and approximation}
	
\author{Stefan Geiss}
\address{Department of Mathematics and Statistics, P.O.Box 35, FI-40014 University of Jyv\"askyl\"a, Finland}
\email{stefan.geiss@jyu.fi}
\author{Tran-Thuan Nguyen}
\address{Fachrichtung Mathematik, Saarland University, P.O.Box 151150, D-66041 Saarbr\"ucken, Germany}
\email{nguyen@math.uni-sb.de}

\thanks{The authors were supported by the Project 298641 \textit{'Stochastic Analysis and Nonlinear Partial Differential Equations,
        Interactions and Applications'} of the Academy of Finland.}
\begin{abstract}
We discuss in a stochastic framework the interplay between 
Riemann-Liouville type operators applied to stochastic processes, 
real interpolation, 
bounded mean oscillation, 
and an approximation problem for stochastic integrals.
We provide upper and lower bounds for gradient processes on the L\'evy-It\^o space,
which arise in the special case of the Wiener space from the Feynman-Kac theory for parabolic PDEs. 
The upper bounds are formulated in terms of BMO-conditions on the fractional integrated 
gradient, the lower bounds in terms of oscillatory quantities.
On the general L\'evy-It\^o space we are concerned with gradient processes with values in a Hilbert space, where the regularity 
depends on the direction in this Hilbert space.
We discuss two applications of our techniques: on the Wiener space an approximation problem for H\"older functionals and on the 
L\'evy-It\^o space an orthogonal decomposition of H\"older functionals into a sum of stochastic integrals with a control 
of the corresponding integrands.
\end{abstract}
\today
\maketitle
\setcounter{tocdepth}{1}
\tableofcontents


\section{Introduction}
\label{sec:introduction}

This article investigates the interplay between Riemann-Liouville type operators applied 
to c\`adl\`ag processes,
gradient estimates for functionals on the Wiener and the L\'evy-It\^o space, 
bounded mean oscillation (BMO), 
approximation theory, and the real interpolation method 
from Banach space theory. To explain this, we assume  
a stochastic basis $(\Omega,\F,\p,(\F_t)_{t\in [0,T]})$  
with finite time-horizon $T>0$. There are various applications in which 
stochastic processes $L=(L_t)_{t\in [0,T)}$ appear that have a singularity when
$t\uparrow T$, for example in $L_p$ for some $p\in [1,\infty]$. 
Examples are processes obtained from (semi-linear) parabolic backward PDEs within the 
Feynman-Kac theory, where these processes occur as integrands in stochastic integral representations 
(see \eqref{eqn:intro:def_varphi_Z} and \cref{sec:application_brownian_case})
or in backward stochastic differential equations as so-called $Z$-processes. The same type of processes also appears
as integrands in stochastic integral representations, based on generalised Galtchouk-Kunita-Watanabe projections,
on the L\'evy-It\^o space (see the integrand $(\psi_{t-}^j)_{t\in [0,T)}$ in \eqref{eqn:statement:intro:decomposition_hoelder_functional}
and \cref{sec:GKW_projection_general}).
Because these processes are often obtained by differentials that get singular when approaching the time of maturity, 
i.e. when $t\uparrow T$, we also call general processes with a singularity sometimes (with an abuse of notation)
{\it gradient processes}.

\smallskip

If one analyzes these examples, then one realizes the following:
\smallskip
\begin{enumerate}
\item [--]{\sc Self-similarity:}
      There is a Markovian structure behind that generates  a self-similarity in the sense that, 
      given $a\in (0,T)$ and $A\in \F_a$ of positive measure, then $(L_t)_{t\in [a,T)}$ restricted 
      to $A$ has similar properties as $(L_t)_{t\in [0,T)}$ has. 
      If one is interested in good distributional estimates for $(L_t)_{t\in [0,T)}$, then 
      the theory of weighted BMO-spaces is worthy to be considered because 
      conditional $L_2$-estimates (that are often accessible)
      yield exponential estimates (relative to a weight) by  John-Nirenberg type theorems. \medskip
\item [--]{\sc Polynomial blow-up:} In the problems mentioned above the size of the singularity
      of $L$ increases polynomially in time with a rate $(T-t)^{-\alpha}$ for some $\alpha>0$.
      In particular, this occurs in the presence of
H\"older functionals as terminal conditions in backward problems.
\end{enumerate}
\smallskip

The above observations will lead to an interplay between {\it Riemann-Liouville type operators}, $\BMO$, 
and the {\it real interpolation method}. As one starting point to investigate these connections was an 
{\it approximation problem} for stochastic integrals, we deal eventually with four objects that interact with each other.
To make the $\BMO$-context more transparent, we first investigate in \cref{theo:intro_L2} 
the connections between Riemann-Liouville type operators, interpolation, and approximation in the $L_2$-setting
(see also \cref{subsec:intro:L2} below). 
The consideration of the $\BMO$-setting follows a methodology known for singular integral operators or martingale transforms:
$L_p$-$L_p$ estimates for, say $p\in (1,\infty)$, yield to $\BMO$-$L_\infty$ endpoint estimates when $p\uparrow \infty$.
We will prove $\BMO$-H\"older estimates with the H\"older spaces as natural $L_\infty$-endpoint of the Besov spaces. 
\medskip

Before we explain the contents of the article in detail we provide definitions needed for this.
\medskip

{\sc Riemann-Liouville type operators.} 
For $\alpha>0$ and a c\`adl\`ag function $K:[0,T)\to \R$ we define 
the Riemann-Liouville type operator $\cI^\alpha K := (\cI_t^\alpha K)_{t\in [0,T)}$ by
\begin{equation}\label{eqn:intro:defin_RL}
   \cI_t^\alpha K:= \frac{\alpha}{T^\alpha} \int_0^T (T-u)^{\alpha-1} K_{u\wedge t} \od u
   \sptext{1}{and}{1}
   \cI_t^0 K:=K_t.
\end{equation}
\cref{sec:RL-operators} deals with general results 
for $\cI^\alpha K$ when $K$ is deterministic $K$, and for
$\cI^\alpha L$ when $L$ is a martingale.
The results are used  later in the article, but can be applied in other contexts 
as well.
\medskip

{\sc Weighted bounded mean oscillation.}
For $p\in(0,\infty)$ and adapted c\`adl\`ag processes $Y$ and $\Phi$, where $\Phi$ is non-negative, we let
$\|Y\|_{\BMO_p^{\Phi}([0,T))}:=\inf c$, where the infimum is taken over all $c\in [0,\infty)$ such that,
for all $t\in [0,T)$ and stopping times $\sigma:\Omega \to [0,t]$ one has
\[ \ce{\F_\sigma}{|Y_t-Y_{\sigma-}|^p} \leqslant c^p \Phi_{\sigma}^p \mbox{ a.s.} \]
If $\Phi\equiv 1$, then we let $\BMO_p([0,T)):=\BMO_p^{\Phi}([0,T))$.
\medskip

{\sc Maximal oscillation.} In \cref{sec:oscillation_general} for a stochastic process  $L=(L_t)_{t\in [0,T)}$ 
and $t\in (0,T)$ we introduce and investigate the oscillatory quantity
\[ \losc_t(L) := \inf_{s \in [0,t)} \| L_t  - L_s \|_{L_\infty} \]
and call $L$ of {\it maximal oscillation} with constant $c\geqslant 1$ if for all $t\in (0,T)$ one has
\[ \losc_t(L) \geqslant \frac{1}{c} \| L_t - L_0 \|_{L_\infty}.\]

{\sc H\"older spaces.}
To describe the regularity of terminal conditions (see Theorems  
\ref{statement:intro:gradient_Wiener_space},
\ref{statement:intro:gradient_LI_space_upper_boud}, 
\ref{statement:intro:decomposition_hoelder_functional})
we use a two-parametric scale of H\"older spaces: 
if $C_b^0(\R)$ consists  of the bounded continuous functions and $\hoell{1}{0}$ of the Lipschitz functions, both defined on $\R$ 
and vanishing at zero, then we define the H\"older spaces 
\[ \hoelO{\theta,q}:= (C_b^0(\R),\hoell{1}{0})_{\theta,q} \sptext{1}{for}{1}(\theta,q)\in (0,1)\times [1,\infty] \]
by real interpolation.
For example, if  we set 
\[ h_{\theta,a}(x):=0 
   \sptext{.75}{if}{.75} x<0
   \sptext{1.5}{and}{1.5}
   h_{\theta,a}(x):= \theta \int_0^{1\wedge x} y^{\theta -1} \left ( \frac{A}{A-\log y} \right )^a  \od y 
   \sptext{.75}{if}{.75} x\geqslant 0 \]
for $\theta\in (0,1)$ and $0 \leqslant a< (1-\theta)A$, then we get (see \cref{sec:proof_example_Hoelder})
\[ h_{\theta,0}(x) = ( \max\{0,x\})^\theta \wedge 1 \in \hoelO{\theta,\infty}
   \sptext{1.3}{and}{1.3} 
   h_{\theta,a} \in \hoelO{\theta,q} 
   \sptext{.5}{for}{.5} a>1/q
   \sptext{.5}{and}{.5} q\in [1,\infty). \]
\smallskip

{\sc Oscillation along a time-net.}
Let $\cT$ be the set of all deterministic time-nets $\tau=\{t_i\}_{i=0}^n$, $0=t_0<\cdots<t_n=T$, 
$n\geqslant 1$. For $\theta \in (0,1]$ and $\tau\in \cT$ we define the mesh-size
\[ \| \tau\|_\theta := \sup_{i=1,\ldots,n} \frac{t_i-t_{i-1}}{(T-t_{i-1})^{1-\theta}}
   \sptext{1}{for}{1}
   \tau=\{t_i\}_{i=0}^n\in \cT.
\]

The mesh-size $\|\cdot \|_\theta$ assigns more and more weight to grid-points close to $T$
when $\theta$ gets smaller. This is desirable to handle singularities when $t\uparrow T$.
The pro-type of nets $\tau_n$ of cardinality $n+1$ such that $\|\tau_n \|_\theta\sim 1/n$
are defined as $\tau_n^\theta:= \left \{T-T(1-(i/n))^\frac{1}{\theta}\right \}_{i=0}^n$ and satisfy
\[ \| \tau_n^\theta \|_\theta \le \frac{T^\theta}{\theta n}. \]
These time-nets go back (at least) to  \cite{Kusuoka:01} and \cite{Ge02}.
We define (in \cref{defi:R_process} and \cref{ass:random_measures_simplified})
for a c\`adl\`ag process $L=(L_t)_{t\in [0,T)}$, a positive c\`adl\`ag process $\sigma=(\sigma_t)_{t\in [0,T]}$,
$a \in [0,T]$, and  $\tau=\{t_i\}_{i=0}^n\in \cT$,
\[ \intsq{L}{\tau}_a^\sigma := \int_0^a \left| L_u - \sum_{i=1}^n L_{t_{i-1}} \1_{(t_{i-1},t_i]}(u) 
   \right|^2 \sigma_u^2 \od u
   \sptext{1}{and}{1} 
   \intsq{L}{\tau}_a := \intsq{L}{\tau}_a^1
   \sptext{.75}{if}{.75} \sigma\equiv 1. \]
Now we divide this introduction into the following parts:

\begin{itemize}
\item[--] \cref{subsec:intro:L2}
      An $L_2$-result behind.
\item[--] \cref{subsec:intro:from_L2_to_BMO}
      From $L_2$ to $\BMO$ - general results.
\item [--] \cref{subsec:intro:gradients_Wiener}
      Gradient estimates on the Wiener space.
\item [--] \cref{subsec:intro:gradients_Ito}
      Gradient estimates on the L\'evy-It\^o space.
\item [--] \cref{subsec:intro:approximation}
      An application to an approximation problem on the Wiener space.
\item [--] \cref{subsec:intro:holder_Ito}
      An application to representations of H\"older functionals on the  L\'evy-It\^o space.
\end{itemize}

\subsection{An $L_2$-result behind}
\label{subsec:intro:L2}
Given $\theta \in (0,1)$ and  a c\`adl\`ag martingale $L=(L_t)_{t\in [0,T)} \subseteq L_2$ we prove 
in \cref{theo:intro_L2} that 
\smallskip
\begin{align}
  \sup_{\tau\in \cT}\frac{\E [L;\tau]_T}{\|\tau\|_\theta}<\infty 
& \Longleftrightarrow (\cI_t^\frac{1-\theta}{2} L)_{t\in [0,T)} \mbox{  is a martingale closable in } L_2 \label{eqn:1:intro:RL_vs_intsqf} \\
& \Longleftrightarrow "L \in \Big (L_2([0,T);L_2(\Omega)),L_\infty([0,T);L_2(\Omega)) \Big )_{\theta,2}".\label{eqn:2:intro:RL_vs_intsqf}
\end{align}
The LHS of \eqref{eqn:1:intro:RL_vs_intsqf} is the central term to quantify Riemann approximations of stochastic integrals 
and to obtain optimal approximation rates. In the BMO-context we explain this in \cref{subsec:intro:approximation} on the Wiener space,
the term is also used on the L\'evy-It\^o space in \cref{sec:application_levy_case}.
The RHS of \eqref{eqn:1:intro:RL_vs_intsqf} is interpreted as follows: After {\it smoothing the process $L$} by applying the 
Riemann-Liouville type operator of order $\frac{1-\theta}{2}$ we get an object closable in $L_2$.
So giving to such an element the smoothness $1$, we obtain that $L$ has a fractional smoothness of 
order $1-\frac{1-\theta}{2} = \frac{1+\theta}{2}$ in $L_2$.
The term \eqref{eqn:2:intro:RL_vs_intsqf}, whose exact meaning is explained before \cref{theo:intro_L2}, says that $L$ belongs to a space 
resulting from real interpolation between two end-points: the first end-point consists of martingales $L$ with $\int_0^T\|L_t\|_{L_2}^2 \od t < \infty$, 
which is a typical condition for integrands of stochastic integrals, the other end-point consists of martingales 
$L$ with $\sup_{t \in [0, T)} \|L\|_{L_2} <\infty$, i.e. martingales closable in $L_2$.

The functional $[L;\tau]$ relates (under regularity assumptions) to 
$[ \cI^{\frac{1-\theta}{2}}  L ]$, the quadratic variation process of 
$\cI^{\frac{1-\theta}{2}}  L$, in terms of a scaling limit
\[
    [ \cI^{\frac{1-\theta}{2}}  L ]_b = L_1\mbox{-}\lim_{n\to \infty} \frac{2\theta n}{T}  \intsq{L}{\tau_n^{\theta}}_b   
 \sptext{1}{for}{1} b\in [0,T), \]
see \cref{statement:pointwise_convergence_sq} (and \cref{statement:pointwise_convergence_sq_random} using randomized time-nets).
\smallskip

\subsection{From $L_2$ to $\BMO$ - general results}
\label{subsec:intro:from_L2_to_BMO}
First we find the correct form of the equivalence \eqref{eqn:1:intro:RL_vs_intsqf} in the BMO-setting as follows: 

\begin{theoI}[{\bf upper bounds}]
\label{statement:intro_sec:var_vs_curvature}
For a c\`adl\`ag martingale $L=(L_t)_{t\in [0,T)} \subseteq L_2$ and 
$\theta \in (0,1]$ the following is equivalent:
\begin{enumerate}[{\rm (1)}]
\item $\sup_{\tau\in\cT}\|\tau\|_\theta^{-1} \|\intsqw {L}{}{\tau}\|_{\BMO_1([0,T))}
                             <\infty$
\item For some $c>0$ and $\bmo_2([0,T))$ defined as in \cref{definition:weighted_bmo} one has
\[ \cI^{\frac{1-\theta}{2}}L - L_0\in \bmo_2([0,T)) 
   \sptext{1.5}{and}{1.}
   |L_a -L_s| \leqslant c
               \frac{(T -s)^\frac{\theta}{2}}
                      {(T -a)^\frac{1}{2}} 
        \sptext{0.5}{for}{.5}
        0\leqslant   s < a < T \mbox{ a.s.}  \]
\end{enumerate}
\end{theoI}
\smallskip
\cref{statement:intro_sec:var_vs_curvature} is part of \cref{equivalence_martingale},
where we take $M_t:= L_t-L_0$, $\sigma \equiv 1$, and $\Phi\equiv 1$, and 
which is formulated (and needed)
for the more general weighted setting. Opposite to \cref{statement:intro_sec:var_vs_curvature}, in \cref{thm:general_lower_bound}
(note that \eqref{eq:assum:thm:general_lower_bound} holds when $L$ is a martingale) we will verify:
\medskip

\begin{theoI}[{\bf lower bounds}]
\label{statement:intro_sec:lower_osciallation}
Assume $\theta \in (0,1]$, a c\`adl\`ag martingale 
$L=(L_t)_{t\in [0,T)}\subseteq L_2$, and that $ \infty> \| [L;\tau]\|_{\BMO_1([0,T))}\to 0$ for
$\|\tau\|_1\to0$. Then the following assertions are equivalent:
\begin{enumerate}[{\rm(1)}]
\item \label{it:1:statement:intro_sec:lower_osciallation}
      $\inf_{t\in (0,T)} (T-t)^{\frac{1-\theta}{2}}\losc_t(L) >0$.
\item \label{it:2:statement:intro_sec:lower_osciallation}
      There is a $c >0$ such that  for all $\tau\in \cT$ one has 
      $\| [L;\tau]\|_{\BMO_1([0,T))} \geqslant c  \| \tau \|_\theta$.
\end{enumerate}
\end{theoI}
\bigskip

Comparing Theorems \ref{statement:intro_sec:var_vs_curvature} and \ref{statement:intro_sec:lower_osciallation},
one realizes a dichotomy in the behaviour of $\| [L;\tau]\|_{\BMO_1([0,T))}$ typical
for the gradient processes appearing on the L\'evy-It\^o space
(see Theorems \ref{statement:intro:gradient_LI_space_upper_boud} and \ref{statement:intro:gradient_LI_space_lower_boud}).
Regarding the equivalence between the RHS of \eqref{eqn:1:intro:RL_vs_intsqf} and \eqref{eqn:2:intro:RL_vs_intsqf}, 
the  proof of \cref{theo:intro_L2} reveals
\[ \eqref{eqn:2:intro:RL_vs_intsqf} \Longleftrightarrow \int_0^T  (T-t)^{-\theta}  \| L_t \|_{L_2}^2 \od t < \infty
   \sptext{1}{for}{1} \theta\in (0,1). \]
So, the next result, which is proven in
\cref{statement:properties_Besov-spaces}\eqref{item:2:statement:properties_Besov-spaces}, is natural:
\medskip

\begin{theoI}
\label{statement:intro:properties_Besov-spaces}
For $\theta \in (0,1)$, $\alpha:=\frac{1-\theta}{2}$, and a c\`adl\`ag martingale $L=(L_t)_{t\in [0,T)}$ one has
\begin{equation}\label{eqn:statement:intro:properties_Besov-spaces}
  \| \cI^\alpha L \|_{\BMO_2([0,T))}   \leqslant 3 \frac{\sqrt{2 \alpha}}{T^{\alpha}} 
     \left ( \int_0^T  (T-t)^{-\theta}  \| L_t \|_{L_\infty}^2 \od t \right )^\frac{1}{2}.
\end{equation}
\end{theoI}
\medskip

Later we use in \cref{statement:intro:gradient_LI_space_upper_boud}, inequality \eqref{eqn:1:statement:intro:gradient_LI_space_upper_boud}, 
the RHS of \eqref{eqn:statement:intro:properties_Besov-spaces} to handle gradient processes 
of functionals $f(X_T)$, where $X=(X_t)_{t\in [0,T]}$ is a L\'evy process and $f\in \hoelO{\eta,2}$.
This takes us back to real interpolation and provides a replacement 
for the implication \eqref{eqn:2:intro:RL_vs_intsqf} $\Rightarrow$ \eqref{eqn:1:intro:RL_vs_intsqf}. 
\medskip

\subsection{Gradient estimates on the Wiener space}
\label{subsec:intro:gradients_Wiener}
Let 
$W=(W_t)_{t\in [0,T]}$ be a standard Brownian motion, 
$Y=(Y_t)_{t\in [0,T]}$ be the geometric Brownian motion
$Y_t := \e^{W_t - \frac{t}{2}}$, and $g:(0,\infty) \to \R$ 
be a Borel function from $C_Y$ defined by \eqref{eqn:CY}. If $G(t,y):=\E g(y Y_{T-t})$, then
$G$ satisfies the backward parabolic equation
\[ \frac{\partial G}{\partial t} + \frac{y^2}{2} \frac{\partial^2 G}{\partial y^2} = 0
   \sptext{1}{on}{1} [0,T)\times (0,\infty), \]
so that by It\^o's formula 
\[ g(Y_T) - \E g(Y_T) =   \int_{(0,T)} \varphi_t \od Y_t 
                    =  \int_{(0,T)} Z_t \od W_t \mbox{ a.s.} \]
with 
\begin{equation}\label{eqn:intro:def_varphi_Z} 
   \varphi_t:= \frac{\partial G}{\partial y}(t,Y_t) 
   \sptext{1}{and}{1}
   Z_t:= Y_t \frac{\partial G}{\partial y}(t,Y_t) 
\sptext{1}{for}{1} t\in [0,T).
\end{equation}
Here $(Z_t)_{t\in [0,T)}$ becomes a martingale. For example, 
let $\theta \in (0,1)$ and $g^{(\theta)}(y):= h_{\theta,0}(y-e^{-\frac{T}{2}})$.
If 
\[ f^{(\theta)}(x):= g^{(\theta)}(e^{x-\frac{T}{2}})
   \sptext{1}{and}{1} F^{(\theta)}(t,x):= \E f^{(\theta)}(x+W_{T-t})
   \sptext{1}{for}{1}
   (t,x)\in [0,T]\times \R, \]
then $Z=Z^{(\theta)}$ satisfies 
$Z^{(\theta)}_t = \frac{\partial F^{(\theta)}}{\partial x}(t,W_t)$. 
Therefore, \cref{statement:lower_bound_complete}\eqref{item:2:statement:lower_bound_complete} yields
the lower bound
\[ \inf_{t\in (0,T)} (T-t)^{-\frac{1-\theta}{2}} \losc_t (Z^{(\theta)}) > 0. \]
On the other side, passing from $\hoelO{\theta,\infty}$ to the slightly smaller space $\hoelO{\theta,2}$
(which corresponds in our accompanying example to pass from $h_{\theta,0}$ to $h_{\theta,a}$ with $a>1/2$), we obtain: 
\medskip

\begin{theoI} 
\label{statement:intro:gradient_Wiener_space}
\begin{enumerate}[{\rm (1)}]
\item For $g\in C_Y$ the process $\varphi$ is of maximal oscillation.
\item If $\theta \in (0,1)$, $\Phi_t := Y_t^\theta$, and if $\cI^\frac{1-\theta}{2} Z$
      is defined by \eqref{eqn:intro:defin_RL}, then for all $p\in (0,\infty)$ the condition $g\in  \hoelO{\theta,2}$ implies that
      \medskip

      \begin{enumerate}[{\rm (a)}]
      \item $\cI^\frac{1-\theta}{2} Z -Z_0\in \BMO_p^\Phi([0,T))$,
      \item $\cI_T^\frac{1-\theta}{2} Z := \lim \limits_{t \uparrow T}\cI_t^\frac{1-\theta}{2} Z$
            in $L_p$ and a.s.
      \end{enumerate}
\end{enumerate}
\end{theoI}
\cref{statement:intro:gradient_Wiener_space} is a consequence of 
Theorems 
\ref{statement:oscillation_brownian_case},
\ref{statement:Hoelder_new_new}\eqref{item:3:statement:Hoelder_new_new}, and
\ref{statement:Z_like_a_martingale}.
The results in \cref{sec:application_brownian_case} are proven for 
a more general setting in which $Z$ is not necessarily a martingale.
The process $\varphi=(\varphi_t)_{t\in [0,T)}$ can be interpreted as trading strategy in the Black-Scholes model, so that its
maximal oscillation is of interest. The control
of $\cI^{\frac{1-\theta}{2}}Z$ in BMO is needed in our approximation problem in terms of 
\eqref{eqn:intro:BMO-estimate_vs_RL} and describes
a fractional smoothness of order 
$\frac{1+\theta}{2}$ in $\BMO_2^\Phi([0,T))$
by the same methodology as described after \eqref{eqn:2:intro:RL_vs_intsqf} for $L_2$.
\medskip

\subsection{Gradient estimates on the L\'evy-It\^o space}
\label{subsec:intro:gradients_Ito}
We use for the introduction the standing assumptions that
\begin{enumerate}
\item [--] $X=(X_t)_{t\in [0,T]}$ is a L\'evy process without Brownian part with L\'evy measure $\nu(\R)\in (0,\infty]$,
\item [--] $\rho$ is a probability measure on $(\R,\cB(\R))$ with $\rho(\{0\})=0$.
\end{enumerate}
Now we consider the questions from \cref{subsec:intro:gradients_Wiener} on the L\'evy-It\^o space.
What is the replacement of the processes $\varphi$ and $Z$?
As we shall use the L\'evy process $X$ as an integrator and not an exponential of it,  
$\varphi$ and $Z$ from  \cref{subsec:intro:gradients_Wiener} 
 can be considered to be the same (it would correspond to $Y=W$).
However, with an abuse of notation, we still consider two types of processes: 
the first process, again called $Z$ since its role is close to the role of $Z$ in 
\cref{subsec:intro:gradients_Wiener}, 
\[ Z_t:\Omega \to H:=L_2(\R,z^2 \nu (\od z)),\] 
takes values in a Hilbert space instead of being $\R$-valued.
The second type of processes are the scalar valued $\varphi$ processes (with a similar  
role as $\varphi$ in \cref{subsec:intro:gradients_Wiener} as integrand in a stochastic integral) derived from $Z$ by applying linear functionals.
\smallskip

The process $Z$ is obtained as follows:
if $f:\R\to \R$ is a Borel function with $f(X_T)\in L_2$, $t\in [0,T)$, then  
$\ce{\F_t}{f(X_T)}$ is Malliavin differentiable for $t\in [0,T)$ and for the derivative one has
\[ \bD \big (\ce{\F_t}{f(X_T)} \big )\in L_2((0,T]\times \R, \lambda\otimes z^2 \nu (\od z)), \]
where $(\F_t)_{t\in [0,T]}$ is the augmented natural filtration of $X$ and 
the Malliavin derivative is defined by the chaos expansion (cf. \cite[Section 2.2]{SUV:07a}
and \cref{sec:malliavin_calculus:basics}).
Under our conditions it will turn out that there is a version of $(\bD_{s,z} \big (\ce{\F_t}{f(X_T)} \big ))_{(s,z)\in (0,T]\times \R}$ 
which is constant in $s\in (0,t]$ and vanishes for  $s\in (t,T]$, so that we may use for $t\in (0,T)$
and 
\[ F(t,x) := \E f(x+X_{T-t}) \]
the Ansatz
\[  Z_t := \left \{ z \mapsto \frac{F(t,z+X_t)-F(t,X_t)}{z} \right \}
         \stackrel{\mbox{a.s.}}{=} \frac{1}{t} \int_0^t \bD_{s,\cdot} \big (\ce{\F_t}{f(X_T)} \big )\od s \in H, \]
where the almost sure equality is explained in \cref{chaos-decom}.
Our first objective is to understand the singularity properties of $(\langle Z_t,D \rangle )_{t\in [0,T)}$ in dependence 
on the functional $D$.

For this  we proceed as follows: in \cref{statement:definition_Gamma} we introduce a linear sub-space $\dom(\Gamma_\rho^0)$ of
Borel functions $f:\R\to \R$ and the operator
\[ \Gamma_{t,\rho}^0 :  \dom(\Gamma^0_\rho) \to \R
   \sptext{1}{with}{1}
   \langle f,\Gamma_{t,\rho}^0 \rangle
   := \int_{\RO} \frac{F(t,z)-F(t,0)}{z} \rho(\od z). \]
\smallskip
Our approach goes beyond the $L_2$-approach as
generally  $\bD (\ce{\F_t}{f(X_T)})$ is interpreted
in $L_2((0,T]\times \R, \lambda\otimes z^2 \nu (\od z))$ as equivalence class only. 
For example, choosing $\rho$ as Dirac measure, due to the choice of the range of definition $\dom(\Gamma^0_\rho)$, 
we can interpret and estimate  
$\frac{1}{t} \int_0^t \bD_{s,z} \big (\ce{\F_t}{f(X_T)} \big )\od s$ for a fixed $z\in\RO$.

The operators $\Gamma_{t,\rho}^0$ are
{\it deterministic} and  linear 
and allow therefore for the application of interpolation techniques from Banach space theory. To understand $\Gamma_{t,\rho}^0 $ as  mathematical object we associate
to the probability measure $\rho$  and to the L\'evy process $X$ a probability density $\gamma_{t,\rho}\in L_1(\R)$
for which it follows from \cref{statement:properties:Gamma1} that 
$\Gamma_{t,\rho}^0 = - \partial  \gamma_{t,\rho}$, i.e. $\Gamma_{t,\rho}^0 $ can be seen as a distributional derivative of a distribution of $L_1$-type.
In \cref{defi:uppper_bounds_X_rho} and \cref{defi:lower_bounds_X_rho} we introduce the following concepts,
where we let $c>0$ and  $A\subseteq \R$ be closed and non-empty:
\medskip
\begin{align}
     X\in \cU(\beta,A;c) \sptext{.9}{for}{.7} \beta \in (0,\infty] \quad 
&:\Longleftrightarrow  \sup_{z\in A} \sup_{s\in (0,T]} s^{\frac{1}{\beta}} |z|^{-1} \| \p_{z+X_s} - \p_{X_s} \|_{\rm TV} \leqslant  c \notag \\
     X\in \cL(\beta) \sptext{.4}{for}{.7} \beta \in (0,2] \quad
&:\Longleftrightarrow \exists r>0 \,\, \left (\inf\limits_{-r \leqslant a < b \leqslant r} 
     \inf\limits_{s\in (0,T]} \p \left ( s^{-\frac{1}{\beta}}X_s \in (a,b) \right )/(b-a) > 0\right ) \notag \\
     \rho\in U(\varepsilon;c) \sptext{.7}{for}{.7} \varepsilon \geqslant 0 \quad 
&:\Longleftrightarrow \forall  d \in [0,1] \,\, \Big ( \rho([-d,d]) \leqslant c d^\varepsilon \Big ) \label{eqn:intro:defin_Uepsilon}\\
     \rho\in L(\varepsilon)\sptext{.7}{for}{.7}  \varepsilon > 0 \quad 
&:\Longleftrightarrow  \liminf_{d\downarrow 0} \frac{\rho([-d,d])}{d^\varepsilon}>0 \notag
\end{align}
\smallskip

Here $\p_{z+{X_s}}$ and $\p_{X_s}$ are the laws of $z+{X_s}$ and $X_s$, 
and $\|\cdot\|_{\rm TV}$ stands for the total variation. 
The properties $U(\varepsilon;c)$ and $L(\varepsilon)$ are complementary small ball properties 
of the measure $\rho$, upper bounds for
$\| \p_{z+{X_s}} - \p_{X_s}\|_{\rm TV}$ were investigated in the literature (see \cite[Theorem 3.1]{SSW12})
in connection with the coupling property of L\'evy processes.
The properties $\cU(\beta,A;c)$ and $X\in \cL(\beta)$ are complementary as well. They appear to be the correct ones for the questions
we consider and are motivated by the 
$\beta$-stable processes as we show 
in \cref{statement:U-L-for-beta-like} that $X\in   \cU(\beta,\R;c) \cap \cL(\beta)$ for $\beta \in (0,2)$ and 
some $c>0$ if $X$ is a symmetric $\beta$-stable like process.
Now we let
\begin{equation}\label{eqn:intro:definition_varphi} D_\rho F(t,x) := \langle f(x+\cdot),\Gamma_{t,\rho}^0 \rangle
\sptext{.75}{and}{.75} \varphi(f,\rho) \mbox{ be a c\`adl\`ag version of }
(D_\rho F(t,X_t))_{t\in [0,T)},
\end{equation}
where we use that $(D_\rho F(t,X_t))_{t\in [0,T)}$ is a martingale (see \cref{statement:easy_properties_Gamma}).
The following two statements describe the regularity properties of $\varphi(f,\rho)$:
\medskip

\begin{theoI}[{\bf upper bounds for $D_\rho F$}]
\label{statement:intro:gradient_LI_space_upper_boud}
Let $\varepsilon\in [0,1)$, $(\eta,q,\beta)\in (0,1-\varepsilon)\times [1,\infty)\times (0,\infty)$, $(X_t)_{t\in [0,T]} \subseteq L_{\eta+\gamma}$ for some $\gamma>0$,
\[ \rho \in \bigcup_{c>0} U(\varepsilon;c),
   \sptext{1}{}{1} X \in \bigcup_{c>0} \cU(\beta,\supp(\rho);c),
   \sptext{1}{and}{1} \alpha:= \frac{1-(\varepsilon+\eta)}{\beta}.\]
Then there are  $c_\eqref{eqn:1:statement:intro:gradient_LI_space_upper_boud},c_\eqref{eqn:2:statement:intro:gradient_LI_space_upper_boud}>0$
such that for $f\in \hoel{\eta,q}$ and $f\in \hoel{\eta,2}$, respectively,
\begin{align}
 \left ( \int_0^T (T-t)^{\alpha q-1}  \sup_{x\in \R} |D_\rho F(t,x)|^q \od t \right )^\frac{1}{q} 
&\leqslant c_\eqref{eqn:1:statement:intro:gradient_LI_space_upper_boud} \, \| f \|_{\hoel{\eta,q}}, \label{eqn:1:statement:intro:gradient_LI_space_upper_boud}\\
           \sup_{a\in [0,T)} (T-a)^{\alpha} \| \varphi_a(f,\rho) \|_{L_\infty} 
             + \big \| \cI^\alpha \varphi(f,\rho) -  \varphi_0(f,\rho) \big \|_{\BMO_2([0,T))}
&\leqslant c_\eqref{eqn:2:statement:intro:gradient_LI_space_upper_boud} \| f \|_{\hoelO {\eta,2}}.\label{eqn:2:statement:intro:gradient_LI_space_upper_boud}
\end{align}
\end{theoI}
\medskip
\cref{statement:intro:gradient_LI_space_upper_boud} is verified at the end of
\cref{sec:application_levy_case}.
For our accompanying example it gives
that the LHS of \eqref{eqn:1:statement:intro:gradient_LI_space_upper_boud} is finite for 
$f=f_{\eta,a}$  with $a>1/q$. 
In the following lower bound the space $\lbvlocal \eta$ is defined in \cref{defi:lbvlocal}. 
A typical example for $f\in \lbvlocal \eta \cap \hoel \eta$ is $f=f_{\eta,0}$.
\medskip

\begin{theoI}[\bf lower bounds for $D_\rho F$]
\label{statement:intro:gradient_LI_space_lower_boud}
Let 
$(\varepsilon,\beta)\in (0,1)\times (0,2]$, 
$\eta\in (0,1-\varepsilon)$, 
$(X_t)_{t\in [0,T]}\subseteq L_\eta$,
\[ \rho \in L(\varepsilon),
   \sptext{1}{}{1}
   X \in \cL(\beta), \sptext{1}{and}{1} \alpha:= \frac{1-(\varepsilon+\eta)}{\beta}. \]
If $\|X_s\|_{{\rm TV}(\rho,\eta)}< \infty$ for $s\in (0,T]$
(see \cref{definition:TVrho})
 and $f\in \lbvlocal \eta \cap \hoel \eta$,
then 
\begin{enumerate}[{\rm (1)}]
\item $\varphi(f,\rho)$ is an $L_\infty$-martingale of maximal oscillation,
\item $\inf_{t\in (0,T)} (T-t)^\alpha \losc_t(\varphi(f,\rho))>0$.
\end{enumerate}
\end{theoI}

The statement is verified in
\cref{statement:maximal_oscillation_application}(\eqref{item:3:statement:maximal_oscillation_application},\eqref{item:4:statement:maximal_oscillation_application})
and \cref{statement:lower_bound_complete}\eqref{item:1:statement:lower_bound_complete}.
\bigskip

\subsection{An application to an approximation problem on the Wiener space}
\label{subsec:intro:approximation}
We return to the setting of \cref{subsec:intro:gradients_Wiener}. For $\tau\in \cT$ 
we define the approximation error for the Riemann approximation of the stochastic integral as
\begin{equation}\label{eqn:intro:error}
 E_t(g;\tau) := \int_0^t \varphi_s \od Y_s - \sum_{i=1}^n \varphi_{t_{i-1}} (Y_{t_i\wedge t} - Y_{t_{i-1}\wedge t})
 \sptext{1}{with}{1}
 \varphi_t = \frac{\partial G}{\partial y}(t,Y_t).
\end{equation}
The term $E_t(g;\tau)$ can be interpreted as discrete time hedging error 
based on the grid $\tau$ at time $t$ of an European option with pay-off $g(Y_T)$  in the Black-Scholes model. By It\^o's isometry one has 
\medskip
\[ \ce{\F_a}{|E_T(g;\tau)-E_a(g;\tau)|^2} = 
 \ce{\F_a}{\intsq{\varphi}{\tau}_T^\sigma-\intsq{\varphi}{\tau}_a^\sigma} \mbox{ a.s.}
  \sptext{1}{for}{1}
   \sigma=Y. \]
\medskip
One has $\|E_T(g;\tau)\|_{L_2} \geqslant \frac{c}{\sqrt{n}}$ for some $c>0$ for all
$\tau\in \cT$ of cardinality $n+1$
if there are no $a,b\in\R$ such that $g(Y_T)=a+ b Y_T$ a.s. (\cite[Theorem 2.5]{GG04}). 
To estimate $E_T(g;\tau)$ from above usually the $L_2$-setting is used to exploit orthogonality
(cf. \cite{Ge02,GG04,Geiss:Hujo:07} for the Wiener space and \cite{GGL13} for the L\'evy-It\^o space). 
On the Wiener space, the approximation in $L_p$ for $p\in [2,\infty)$ is considered in 
\cite{Geiss:Toivola:15}, whereas a different route is taken in \cite{Geiss:05} by exploiting 
weighted BMO (both for $T=1$):
\smallskip

\begin{enumerate}
\item [--] In \cite[Theorems 7 and 8]{Geiss:05} it is shown that 
      \begin{equation}\label{eqn:intro:BMO-estimate_vs_Lipschitz}
      \sup_{\tau\in\cT} \frac{\|E(g;\tau)\|_{\BMO_2^Y([0,T])} }{ \sqrt{\| \tau \|_1}  } <\infty
      \Longleftrightarrow g \mbox{ is (equivalent to) a Lipschitz function}.
      \end{equation}
\item [--] In \cite[Theorems 1.2 and 5.5]{Geiss:Toivola:15} it is shown for $\theta\in (0,1]$, $p\in [2,\infty)$,
      and $G(t,y)=\E g(y Y_{T-t})$ that 
       \begin{equation}\label{eqn:intro:Lp-estimate_vs_curvature}
      \sup_{\tau\in\cT} \frac{\|E_T(g;\tau)\|_{L_p}}{ \sqrt{\| \tau \|_\theta}  } <\infty
      \Longleftrightarrow 
      \left ( \int_0^T (T-t)^{1-\theta} \left | Y_t^2 \frac{\partial^2 G}{\partial y^2}(t,Y_t) \right |^2 \od t \right )^\frac{1}{2} \in L_p.
      \end{equation}
\end{enumerate}
\medskip

In this article we solve the missing case 
$(p,\theta)\in \{\infty\} \times (0,1)$
in \eqref{eqn:intro:BMO-estimate_vs_Lipschitz} and \eqref{eqn:intro:Lp-estimate_vs_curvature}
by the BMO-approach by extending the techniques from \cite{Geiss:05} from the Lipschitz case
to the H\"older case and the results in \cite{Geiss:Toivola:15} from $p<\infty$ to $p=\infty$:
besides that we consider slightly more general processes than in \cite{Geiss:05,Geiss:Toivola:15}, we 
prove for $\theta\in (0,1)$, and $\varphi$ and $Z$ from \eqref{eqn:intro:def_varphi_Z},
in \cref{statement:brownian_case_theta_new} that
\smallskip
\begin{equation}\label{eqn:intro:BMO-estimate_vs_RL}
 \sup_{\tau\in\cT} \frac{\|E(g;\tau)\|_{\BMO_2^{\Phi}([0,T))}}{\sqrt{\| \tau\|_\theta}} < \infty
	 \,\,\Longleftrightarrow\,\,
      \begin{cases}
      \cI^\frac{1-\theta}{2} Z - Z_0\in \BMO^\Phi_2([0,T))\\
                  (T-t)^{\frac{1-\theta}{2}} |\varphi_t-\varphi_0| 
         \leqslant \, c \, \Phi_t \, Y_t^{-1} \quad \mbox{a.s. for all } t\in [0,T) 
         \end{cases}
\end{equation}
\smallskip
under a monotonicity condition on $\Phi$ and an a-priori condition on $Z$.
As by Theorems \ref{statement:intro:gradient_Wiener_space} 
and \ref{statement:Hoelder_new_new},
\begin{equation}\label{eqn:implications_of_hoeltheta2}
g\in \hoel{\theta,2} \implies 
 \begin{cases}
      \cI^\frac{1-\theta}{2} Z - Z_0\in \BMO^{Y^\theta}_2([0,T))\\
                  (T-t)^{\frac{1-\theta}{2}} |\varphi_t| 
         \leqslant c \, Y_t^{\theta-1} \quad {\mbox{a.s. for all } t\in [0,T)}
         \end{cases}
\end{equation}
\smallskip

for $\theta \in (0,1)$, we get the right-hand side of \eqref{eqn:intro:BMO-estimate_vs_RL}
if $g\in \hoel{\theta,2}$ and $\Phi$ is chosen accordingly.
Using the John-Nirenberg theorem for the $\BMO_q^\Phi$-spaces we also verify 
in \cref{proof:statement:intro:tail_hoeleta2}:
\medskip

\begin{theoI}
\label{statement:intro:tail_hoeleta2}
For $\theta\in (0,1)$, $T>0$, and $g=g^{(\theta)}+g^{(1)}$ with $g^{(\theta)}\in \hoel {\theta,2}$ and $g^{(1)}:\R\to \R$ being 
Lipschitz there are $A,c>0$ and $\lambda_0\geqslant 1$
such that for $n\in \bN$ and 
$a\in  \tau_n^\theta$ with $a<T$ it holds
\[ \p_{\F_a} \left ( \sup_{t\in [a,T]} \left | E_t(g;\tau_n^\theta)-  E_a(g;\tau_n^\theta) \right|
    \geqslant c\, \lambda \frac{Y_a \vee Y_a^\theta}{\sqrt{n}} \right ) 
\leqslant A \e^{- \left ( \frac{\ln \lambda}{A} \right )^2} \mbox{ a.s.}
 \sptext{1}{for}{1} \lambda\geqslant \lambda_0.
\]
\end{theoI}
\medskip

So we may condition on $(Y_a \vee Y_a^\theta)$ and obtain uniform explicit error estimates along the time-nets $\tau_n^\theta$,
that are stronger than arising from an $L_p$-estimate.
In the context of stochastic finance this applies to power type options or options of powered type with exponent 
$\gamma\in (0,1)$ and can be exploited to derive shortfall probabilities for
the discrete time super-hedging. A typical example is a powered call option
$g_\gamma(Y_T):=((Y_T-K)^+)^\gamma$
for $K>0$.
This option interpolates between a call option and a binary option as
\[ g_1(Y_T)=(Y_T-K)^+ \sptext{1}{and}{1} \lim_{\gamma\downarrow 0} g_\gamma(Y_T)=\1_{\{Y_T>K\}}, \]
and we have $g_\gamma \in \hoel {\theta,2}+\hoell{1}{0}$ for 
$\theta\in (0,\gamma)$ (see \cref{statement:smoothness_powered_option}).

\subsection{An application to representations of H\"older functionals on the  L\'evy-It\^o space}
\label{subsec:intro:holder_Ito}
Let us additionally assume for the L\'evy process in 
\cref{subsec:intro:gradients_Ito}
that $(X_t)_{t\in [0,T]}$
is symmetric and square integrable. 
For $D\in L_2(\R,z^2 \nu(\od z))$ with 
$\| D\|_{L_2(\R,z^2 \nu(\od z))}=1$
we let
\[ X_t^D := \int_{(0,t]\times \R} D(x)x \wt{N}(\od s,\od x), \]
where $\wt{N}$ is the compensated Poisson random measure associated to $X$. 
In a sense, $X^D$ is the projection of $X$ into the direction of $D$.
In \cref{sec:orthogonal_decomposition_Levy-Ito} we prove the following decomposition:
\bigskip

\begin{theoI}
\label{statement:intro:decomposition_hoelder_functional}
Let 
$\beta \in (0,2)$, 
$\eta\in [0,1]$,
$\alpha:= \frac{1}{2}-\frac{\eta}{\beta}$,
and let $(D_j)_{j\in J}\subseteq  L_2(\R,z^2 \nu(\od z))$ be an orthonormal basis.
Then, for an $\eta$-H\"older continuous function \rm{(}$f\in \hoel \eta$\rm{)} 
with H\"older constant $|f|_\eta$ one has the orthogonal decomposition 
\begin{equation}\label{eqn:statement:intro:decomposition_hoelder_functional}
 f(X_T) = \E f(X_T) + \sum_{j\in J} \int_{(0,T)} \psi_{t-}^j \od X^{D_j}_t,
 \end{equation}
 where the sum is taken in $L_2$ if $|J|=\infty$, with 
 \[   \psi_t^j 
    =   \| D_j^+\|_{L_1(\R,z^2 \od \nu)} \varphi_t(f,\rho_j^+)
      - \| D_j^-\|_{L_1(\R,z^2 \od \nu)} \varphi_t(f,\rho_j^-)
   \sptext{.8}{and}{.8}
   \rho^{\pm}_j(\od z) := \frac{D_j^\pm(z)z^2 \nu(\od z)}{\| D_j^\pm\|_{L_1(\R,z^2 \od \nu)}} \]
 where we use the notation from \eqref{eqn:intro:definition_varphi}
 and omit corresponding terms when $\| D_j^+\|_{L_1(\R,z^2 \od \nu)}$ or $\| D_j^-\|_{L_1(\R,z^2 \od \nu)}$
 vanish. This decomposition satisfies:
 \begin{enumerate}[{\rm (1)}]
 \item There exists a $c>0$ such that for all $j\in J$, $f \in \hoel{\eta}$, and $t\in [0,T)$ one has
       \[ \|\psi_t^j\|_{L_\infty}
          \leqslant  c \,\, |f|_\eta \,\, 
 		\begin{cases}
 			(T-t)^{-\alpha}            &:  \eta \in [0,\frac{\beta}{2}) \\
 			\ln \left ( 1+(T-t)^{-\frac{1-\eta}{\beta}} \right ) &:  \eta = \frac{\beta}{2}  \\
 			1                                                      &:  \eta \in (\frac{\beta}{2},1]
 		\end{cases}. \]
 		
 \item Assume additionally $\eta\in (0,\frac{\beta}{2})$ and
       $q\in [1,\infty)$. Then there exists a $c>0$ such that, for all $j\in J$ and $f \in \hoel{\eta,q}$,
       \[ \left ( \int_0^T (T-t)^{\alpha q-1} \| \psi_t^j\|_{L_\infty}^q \od t \right )^\frac{1}{q}  
 	\leqslant  c \, \| f\|_{\hoelO{\eta,q}}. \]
   \end{enumerate}
\end{theoI}
\medskip
One main point of \cref{statement:intro:decomposition_hoelder_functional} is the degree of the singularity 
$\alpha=\frac{1}{2} -\frac{\eta}{\beta} \in (0,\frac{1}{2}]$
in the most singular case $\eta\in [0,\frac{\beta}{2})$, i.e. when the singularity is of form $(T-t)^{-\alpha}$ with
$\alpha>0$. To illustrate this, we let $\eta >0$ and observe 
that \cref{statement:intro:gradient_LI_space_lower_boud} provides the general lower bound 
\[ S(\varepsilon,\eta,\beta):= \frac{1-(\varepsilon+\eta)}{\beta} \]
for $(\varepsilon,\beta) \in (0,1)\times (0,2]$ and $\eta\in (0,1-\varepsilon)$. It turns out that, in general,
the singularity $S(\varepsilon,\eta,\beta)$ can be made arbitrarily large for small $\beta$. But, for example, to apply the approximation techniques 
from \cref{statement:intro_sec:var_vs_curvature} we would need to have $\alpha< \frac{1}{2}$.
What additionally holds true in the setting of
\cref{statement:intro:decomposition_hoelder_functional} is the following: the particular 
measures $\rho^\pm_j$ originate from $L_2$-variables
$D_j^\pm$ and {\it in this case} we get that $\rho^\pm_j\in \bigcup_{c>0} U\left (\varepsilon;c\right )$
(defined in \eqref{eqn:intro:defin_Uepsilon}) with
$\varepsilon:=\frac{2-\beta}{2}$ so that
\[ \alpha = \frac{1-(\varepsilon+\eta)}{\beta} 
          = \frac{1-(\frac{2-\beta}{2}+\eta)}{\beta} 
          = \frac{1}{2} - \frac{\eta}{\beta} < \frac{1}{2}.\]
Finally let us comment on \underline{\cref{sec:interpolation}}: This section provides with \cref{statement:abstract_interpolation_new} an interpolation result adapted
to gradient estimates in the L\'evy setting, which is formulated in a general context and is of possible independent interest. 
\medskip

The sections of the article interact as follows: 
\bigskip

\begin{tikzpicture}[node distance=1.6cm]
 \node(Preliminaries)[process]{\secref{sec:preliminaries}};
 \node(RL)[process,below of=Preliminaries,xshift=-5.5cm,draw, align=left]{\secref{sec:RL-operators}\\ 
                                                                \secref{sec:convergence_sqfunction}\\                                                                        
                                                                       \secref{sec:random_measures}
                                                                        };
 \node(Oscillation)[process,below of=Preliminaries,xshift=0cm]{\secref{sec:oscillation_general}};
 \node(BM)[process,xshift=0cm,below of=RL]{\secref{sec:application_brownian_case}};
 \node(Interpolation)[process,xshift=3.5cm,below of=Preliminaries]{\secref{sec:interpolation}}; 
 \node(Levy)[process,xshift=1.8cm,below of=Oscillation,align=left]{\secref{sec:application_levy_case}\\
                                                        \secref{sec:approximation_random_measure}};
 \draw [arrow] (Preliminaries) -- (RL);
 \draw [arrow] (Preliminaries) -- (Oscillation);
 \draw [arrow] (Preliminaries) -- (Interpolation);
 \draw [arrow] (Interpolation) -- (Levy);
 \draw [arrow] (RL) -- (BM);
 \draw [arrow] (Oscillation) -- (BM);
 \draw [arrow] (RL) -- (Levy);
 \draw [arrow] (Oscillation) -- (Levy);
\end{tikzpicture}
\bigskip
\bigskip


\newpage
\section{Preliminaries}
\label{sec:preliminaries}

\subsection{General notation}
We let $\bN:=\{1,2,\ldots\}$ and $\bN_0:= \{0,1,2,\ldots\}$. For $a, b \in \R$ we use $a\vee b : =\max\{a, b\}$,
$a\wedge b :=\min\{a, b\}$, $a^+:= a\vee 0$, $a^-:= (-a)\vee 0$, and for $A, B \geqslant 0$ and $c\geqslant 1$ the notation $A\sim_c B$ for $\frac{1}{c}B\leqslant A \leqslant c B$. 
The corresponding one-sided inequalities are abbreviated by $A \succeq_c B$ and
$A \preceq_c B$.
Given $x\in \R$, $\sign(x):=1$ for $x\geqslant 0$ and $\sign(x):=-1$ for $x< 0$ is the standard sign function,
and we agree about $0^0:=1$.
Given a metric space $M$, $\cB(M)$ denotes the Borel $\sigma$-algebra generated by the open sets.
For a probability space $(\Omega,\F,\p)$ and a measurable map $X:\Omega \to \R^d$, 
where $\R^d$ is equipped with the Borel $\sigma$-algebra $\cB(\R^d)$,
the law of $X$ is denoted by $\p_X$.
Given a measurable space $(\Omega,\F)$, we let  
$\cL_0(\Omega,\F)$ be the space of all $\F$-measurable maps
$X:\Omega \to \R$, and for
$p\in (0,\infty]$ and a measure space 
$(\Omega,\F,\mu)$ we use the standard Lebesgue spaces $L_p(\Omega,\F,\mu)$.
We drop parts of the corresponding measure space in the notation  
if there is no risk of confusion.
Given a (finite) signed measure $\mu$ on $(\R,\cB(\R))$, we denote by $|\mu|:=\mu^+ + \mu^-$ its \textit{variation} 
and by $\|\mu \|_{\textrm{TV}} := |\mu|(\R)$ its \textit{total variation}. 
The Lebesgue measure on $(\R,\cB(\R))$ will be denoted by $\lambda$.
For two measures $\mu$ and $\nu$ on a measurable space $(\Omega,\F)$ we write
$\nu\ll\mu$ if $\nu$ is absolutely continuous with respect to $\mu$.
For a set $A\in \F$ with $\mu(A)\in (0,\infty)$ we let $\mu_A$ be the normalized restriction
of $\mu$ to the trace $\sigma$-algebra $\F|_A$. 
For $0<p\leqslant q \leqslant \infty$, $\sigma$-finite measure spaces $(M,\Sigma,\mu)$ and
$(N,\cN,\nu)$, and a measurable map $f:M\times N\to [0,\infty)$ we use the inequality
\begin{equation}\label{eqn:interchange_Lp_Lq}
          \left \| \left \| f \right  \|_{L_p(\mu)} \right \|_{L_q(\nu)}
\leqslant \left \| \left \| f \right  \|_{L_q(\nu)} \right \|_{L_p(\mu)}.
\end{equation}

\subsection{Support of a measure}
Let $\mu$ be a measure on $\cB(\R^d)$, then $\supp(\mu)$ denotes the 
closed set 
$\{ x\in \R^d: \mu(U_\varepsilon(x))>0 \mbox{ for all } \varepsilon>0 \}$, where 
$U_\varepsilon(x)$ is the open euclidean ball centered at $x$ with radius $\varepsilon>0$.
Given a random variable $X:\Omega \to \R^d$, we let $\supp(X) := \supp(\p_X)$. One knows that 
$\p(\{ X \in \supp(X)\})=1$ and that for 
independent random variables $X:\Omega\to\R^m$ and $Y:\Omega\to\R^n$ it holds
$\supp((X,Y)) = \supp(X) \times \supp(Y)$.
Finally, for a random variable $X:\Omega \to \R^d$
and a Borel measurable $H:\R^d\to \R$ that is continuous on $\supp(X)$ (with respect to the induced topology) 
it holds that $\|H(X)\|_{L_\infty(\Omega,\F,\p)}= \sup_{x\in \supp(X)} |H(x)|$.

\subsection{Interpolation spaces}
\label{sub:sec:interpolation}
We will only consider Banach spaces over $\R$. Let $(E_0,E_1)$ be a couple of Banach spaces such that $E_0$ and $E_1$ are continuously embedding into 
some topological Hausdorff space $X$ ($(E_0,E_1)$  is called an interpolation couple).
We equip $E_0+E_1:= \{ x = x_0+x_1 : x_i \in E_i\}$ with the norm 
$\|x\|_{E_0+E_1} := \inf \{  \|x_0\|_{E_0}+\|x_1\|_{E_1}  : x_i\in E_i, x=x_0+x_1\}$
and $E_0 \cap E_1$ with the norm $\|x\|_{E_0\cap E_1} := \max \{ \|x\|_{E_0}, \|x\|_{E_1}\}$ to get Banach spaces 
$E_0\cap E_1 \subseteq E_0+E_1$.
For $x\in E_0+E_1$ and $v \in (0,\infty)$ we define the $K$-functional
$$ K(v,x;E_0,E_1) := \inf \{ \|x_0\|_{E_0} + v \|x_1\|_{E_1} : x = x_0+x_1 \}. $$
Given $(\theta,q)\in (0,1)\times [1,\infty]$ we set 
\[ (E_0,E_1)_{\theta,q} := \left \{ x \in E_0+E_1 : \|x\|_{(E_0,E_1)_{\theta,q}}:=\left \| v \mapsto  v^{-\theta} K(v,x;E_0,E_1) 
                              \right \|_{L_q\left ((0,\infty), \frac{\od v}{v }\right )} < \infty \right \}.\]
We obtain a family of Banach spaces $\left (  (E_0,E_1)_{\theta,q},\|\cdot\|_{(E_0,E_1)_{\theta,q}} \right )$ with the  lexicographical 
ordering
\[ (E_0,E_1)_{\theta,q_0} \subseteq (E_0,E_1)_{\theta,q_1} \sptext{1}{for all}{1} \theta \in (0,1)
                                                           \sptext{.5}{and}{.5} 1\leqslant  q_0 < q_1 \leqslant \infty, \]
and, under the additional assumption that $E_1 \subseteq E_0$ with $\|x\|_{E_0} \leqslant c \| x\|_{E_1}$
for some $c>0$,  
\[ (E_0,E_1)_{\theta_0,q_0} \subseteq (E_0,E_1)_{\theta_1,q_1}
   \sptext{1}{for all}{1}
   0< \theta_1 < \theta_0<1 \sptext{.5}{and}{.5} q_0,q_1\in [1,\infty]. \]
Given a linear operator $T:E_0+E_1\to F_0+F_1$ with $T:E_i\to F_i$ for $i=0,1$, we use that
the real interpolation method is an exact interpolation functor, i.e.
\begin{equation}\label{eq:functor_with_constant_1_real_interpolation} 
\| T:(E_0,E_1)_{\theta,q} \to (F_0,F_1)_{\theta,q} \|
   \leqslant 
  \| T:E_0 \to F_0 \|^{1-\theta} \| T:E_1\to F_1 \|^{\theta}.
\end{equation}
For more information about the real interpolation method the reader is referred to 
\cite{Bennet:Sharpley:88,Bergh:Loefstroem:76,Triebel:78}.
Given a Banach space $E$ and $(q,s)\in [1,\infty]\times \R$, we will use the Banach spaces 
\begin{equation}\label{eqn:definition_ellqsE}
 \ell_q^s(E) := \{ (x_k)_{k=0}^\infty \subseteq E : \| (x_k)_{k=0}^\infty \|_{\ell_q^s(E)} := 
                   \| (2^{ks} \| x_k\|_E)_{k=0}^\infty \|_{\ell_q} < \infty\} 
\end{equation}
and the notation $\ell_q(E):= \ell_q^0(E)$.
For $q_0,q_1,q\in [1,\infty]$ and $s_0,s_1\in \R$ with $s_0\not = s_1$, and $\theta  \in (0,1)$,
one has according to \cite[Theorem 5.6.1]{Bergh:Loefstroem:76} that
\begin{equation}\label{eq:interpol_ell_q^s}
 (\ell_{q_0}^{s_0}(E),\ell_{q_1}^{s_1}(E))_{\theta,q} = \ell_{q}^{s}(E)
   \sptext{1}{where}{1}
   s:= (1-\theta)s_0 + \theta s_1
\end{equation}
and there is a $c_\eqref{eq:interpol_ell_q^s_constant}\geqslant 1$ that depends at most on $(s_0,s_1,q_0,q_1,\theta,q)$ such that
\begin{equation}\label{eq:interpol_ell_q^s_constant}
                                          \| \cdot \|_{\ell_{q}^{s}(E)} 
\sim_{c_\eqref{eq:interpol_ell_q^s_constant}} \| \cdot \|_{(\ell_{q_0}^{s_0}(E),\ell_{q_1}^{s_1}(E))_{\theta,q}}.
\end{equation}

\subsection{Function spaces}
\label{sec:intro:function_spaces}

Given $\emptyset \not = A \in \cB(\R)$, we let $B_b(A)$ be the Banach space of bounded Borel functions 
$f\colon A \to \R$ with
$\| f\|_{B_b(A)} := \sup_{x\in A} |f(x)|$,
$C_b^0(\R)$ be the closed subspace of $B_b(\R)$ of continuous functions vanishing at zero,
and $C_b^\infty(\R)\subseteq B_b(\R)$ the infinitely often differentiable functions such that the
derivatives satisfy $f^{(k)}\in B_b(\R)$, $k \geqslant 1$. 
The space $C^1(\R)$ consists of differentiable functions with continuous derivative and 
$C^\infty(\R)$ of the functions that are infinitely often differentiable.
For $\eta\in [0,1]$ we  use the H\"older spaces
\begin{align}
\hoel \eta    &:=  \left\{f\in \cL_0(\R,\cB(\R)): \,\,\, |f|_{\eta}:=\sup_{- \infty <x<y<\infty}\frac{|f(x)-f(y)|}{|x-y|^\eta} <\infty \right\},
                   \label{eqn:definition:eta_Hoelder}\\
\hoell \eta 0 & := \{ f \in \hoel \eta : f(0)=0 \},\notag \\  
\hoelO{\eta,q}&:= (C_b^0(\R),\hoell{1}{0})_{\eta,q} \sptext{1}{for}{1}(\eta,q)\in (0,1)\times [1,\infty]. \notag
\end{align}
Note that we can define the Banach space $C_b^0(\R) +\hoell{1}{0}$, so that $(C_b^0(\R),\hoell{1}{0})$ forms an
interpolation pair.
If we use on $C_b^0(\R)$ the equivalent norm 
$\| f \|^0_{C_b^0(\R)}:= \sup \{|f(x)-f(y)| : x,y\in \R \}$, then $\frac{1}{2} \|f\|^0_{C_b^0(\R)} \leqslant \|f\|_{C_b^0(\R)} \leqslant  \|f\|^0_{C_b^0(\R)}$
and build with this norm the interpolation spaces $\hoelO{\eta,q}$ and denote the norms by 
$\| f \|_{\hoelO{\eta,q}}^0$, then we get the 'translation invariance' (useful later for us)
\[ \| f \|_{\hoelO{\eta,q}}^0 = \| f(x+\cdot) - f(x) \|_{\hoelO{\eta,q}}^0
   \sptext{1}{for all}{1}
   x\in \R.\]
We also note that for $0<\eta_0<\eta_1<1$ that 
\begin{equation}\label{eqn:inclusion_Hoelder_spaces}
 \hoelO{\eta_1,\infty} \cap C_b^0(\R) \subseteq \hoelO{\eta_0,1}
\end{equation}
as $\sup_{v\in [1,\infty)} K(v,f;C_b^0(\R),\hoell{1}{0}) < \infty$
for $f\in  \hoelO{\eta_1,\infty} \cap C_b^0(\R)$. Moreover,
by the reiteration theorem (see
\cite[Theorem 3.5.3]{Bergh:Loefstroem:76} or \cite[Theorem 5.2.4]{Bennet:Sharpley:88})
it follows  
\begin{equation}\label{eq:reiteration_for_Hoelder}
(\hoelO{\eta_0,q_0},\hoelO{\eta_1,q_1})_{\theta,q}
= \hoelO{\eta,q}
\end{equation}
for $\theta,\eta_0,\eta_1\in (0,1)$ with $\eta_0\not = \eta_1$, $q,q_0,q_1\in [1,\infty]$,
$\eta:=(1-\theta)\eta_0+\theta \eta_1$,
where the norms are equivalent up to a multiplicative constant.
By the above definitions we obtain Banach space by $(\hoell \eta 0,|\cdot|_\eta)$
and for $\eta\in (0,1)$ we have that $\hoelO{\eta,\infty} = \hoell \eta 0$ with equivalent norms
up to a multiplicative constant (a direct proof can be obtained by an adaptation of \cite[Lemma A.3]{Laukkarinen:19}, see also 
\cite[Theorem 2.7.2/1]{Triebel:78}). 

\subsection{Stochastic basis}
\label{subsec:stochastic_basis}

We fix a time horizon $T\in (0,\infty)$, let $(\Omega, \F, \p)$ be a complete probability space
equipped with a right continuous filtration $\bF=(\F_t)_{t\in[0,T]}$ such that $\F_0$ is generated by the 
$\p$-null sets and $\F=\F_T$. For 
\[ \timed =[0,T] \sptext{1}{or}{1} \timed = [0,T) \]
we denote by $\CL(\mathbb I)$ the set of $\bF$-adapted \textit{c\`adl\`ag} 
(right continuous with left limits) processes $Y=(Y_t)_{t\in \mathbb I}$, by $\CL^+(\mathbb I)$ 
the sub-set of $Y\in\CL(\mathbb I)$ with $Y_t(\omega) \geqslant 0$ on $\mathbb I \times \Omega$,
and by  $\CL_0(\mathbb I)$ the sub-set of $Y\in\CL(\mathbb I)$ with $Y_0\equiv 0$.
For $Y \in \CL(\mathbb I)$ we use 
\begin{enumerate}
\item $Y^*=(Y^*_t)_{t\in \mathbb I}$ with $Y^*_t := \sup_{s \in [0, t]} |Y_s|$,
\item $\Delta Y =(\Delta Y_t)_{t\in \mathbb I}$ with $\Delta Y_t:=Y_t-Y_{t-}$, where $Y_{0-} := Y_0$ and 
      $Y_{t-} := \lim_{s<t,\,s\uparrow t} Y_s$ for $t>0$.
\end{enumerate}

The collection of all stopping times $\rho \colon \Omega \to [0,t]$ is denoted by $\cS_t$. We write $\ce{\cG}{X}$ for the 
conditional expectation of $X$ given $\cG$ (where we exploit extended conditional expectations if $X$ is non-negative)
and use $\p_\cG(B):= \ce{\cG}{\1_B}$ for $B\in \cB(\R)$.
The usual conditions imposed on $\bF$ allow us to assume that every martingale adapted to this filtration is c\`adl\`ag.
Given a c\`adl\`ag $L_2$-martingale $X=(X_t)_{t\in \timed}$, the sharp bracket process
is denoted by $\langle X \rangle = (\langle X \rangle_t)_{t\in \timed}$ and the 
square bracket process by $[X] = ([X]_t)_{t\in \timed}$ (see \cite[Chapter VII]{DM82}).
Both processes are assumed to be non-negative, c\`adl\`ag, and non-decreasing on $\Omega$,
and $[X]_0\equiv 0$ if $X_0\equiv 0$.
In particular, the process $\langle X \rangle = (\langle X \rangle_t)_{t\in \timed}$ is 
the unique (up to indistinguishability) non-decreasing, predictable, c\`adl\`ag process with
$\langle X \rangle_0 \equiv 0$ such that $(X_t^2 - \langle X \rangle_t)_{t\in \timed}$ is a martingale.

\subsection{Bounded mean oscillation and regular weights}
We use the following weighted BMO spaces, where we agree about $\inf\emptyset := \infty$ in this subsection.

\begin{defi}\label{definition:weighted_bmo} 
Let $p\in(0,\infty)$.
\begin{enumerate}[(1)]
\item For $Y\in \CL_0(\timed)$ and $\Phi \in \CL^+(\timed)$
      we let $\|Y\|_{\BMO_p^{\Phi}(\timed)}:=\inf c$, where the infimum is taken over all $c\in [0,\infty)$ such that,
      for all $t\in \timed$ and $\rho\in  \cS_t$,
      \begin{equation} \label{eqn:definition:weighted_bmo:item_1}
      \ce{\F_\rho}{|Y_t-Y_{\rho-}|^p} \leqslant c^p \Phi_{\rho}^p \mbox{ a.s.}
      \end{equation}
\item For $Y\in \CL_0(\timed)$ and $\Phi \in \CL^+(\timed)$
      we let $\|Y\|_{\bmo_p^{\Phi}(\timed)}:=\inf c$, where the infimum is taken over all $c\in [0,\infty)$ such that,
      for all $t\in \timed$ and $\rho\in  \cS_t$,
      \begin{equation}\label{eqn:definition:weighted_bmo:item_2}
        \ce{\F_\rho}{|Y_t-Y_{\rho}|^p} \leqslant c^p \Phi_{\rho}^p \quad \mbox{a.s.}
      \end{equation}
\end{enumerate}
For $\|Y\|_\Theta<\infty$ we write $Y\in \Theta$ with $\Theta\in \{ \BMO_p^\Phi(\timed),\bmo_p^\Phi(\timed)\}$.
If $\Phi\equiv 1$, then we use the notation $\BMO_p(\timed)$ and $\bmo_p(\timed)$, respectively.
\end{defi}
\smallskip
If $Y_0\equiv 0$ is not necessarily satisfied, then we use the notation $\| Y-Y_0\|_{\BMO^\Phi(\timed)}$
for $\| (Y_t-Y_0)_{t\in \timed} \|_{\BMO^\Phi(\timed)}$.
If $Y\in \CL_0(\timed)$ has continuous paths a.s., then $\|Y\|_{\BMO_p^{\Phi}(\timed)}=\|Y\|_{\bmo_p^{\Phi}(\timed)}$.
The theory of classical non-weighted $\BMO$-martingales can be found in \cite[Ch.VII]{DM82} or \cite[Ch.IV]{Pr05}; 
non-weighted $\bmo$-martingales were mentioned in \cite[Ch.VII, Remark 87]{DM82} and used after that in \cite{CKS98, DMSSS97}. 
The $\BMO_p^\Phi$ space was introduced and discussed in \cite{Geiss:05}. Some relations between $\bmo_p^\Phi$
and $\BMO_p^\Phi$, that are necessary for us,  are discussed in the appendix below. 
Next we recall (and adapt) the class $\cSM_p$, introduced in \cite[Definition 3]{Geiss:05}:
\smallskip

\begin{defi}
\label{def:SM_p}
For $p\in (0,\infty)$ and $\Phi\in \CL^+(\timed)$ 
we let $\| \Phi\|_{\cSM_p(\timed)} := \inf c$, where the infimum is taken over all 
$c \in [1, \infty)$ such that for all stopping times $\rho:\Omega \to \timed$ one has
\[ \ce{\F_\rho}{\sup_{\rho \leqslant t \in \timed} \Phi_t^p} \leqslant c^p \Phi_\rho^p \quad \mbox{a.s.} \]
If $\|\Phi\|_{\cSM_p(\timed)} <\infty$, then we write $\Phi\in \cSM_p(\timed)$.
\end{defi}
\smallskip
By choosing $\rho\equiv 0$, $\Phi\in \cSM_p(\timed)$ implies that $\E\sup_{t\in \timed} \Phi_t^p<\infty$.
Moreover, it follows directly from the definition that $\cSM_p(\timed) \subseteq \cSM_r(\timed)$ whenever $0<r\leqslant p<\infty$.
Simplifications in \cref{definition:weighted_bmo}  and  \cref{def:SM_p} and relations between the $\BMO$- and $\bmo$-spaces are 
recalled in \cref{sec:general_properties_BMO}.
If $p\in (1,\infty)$ and $\Phi$ is a martingale, then 
$\Phi\in \CL^+(\timed)$ is equivalent to the standard reverse H\"older condition
$\ce{\F_a}{\Phi_t^p} \leqslant d^p \Phi_a^p$ a.s. for $0\leqslant a \leqslant t < T$.

\smallskip
\subsection{Uniform quantization and time-nets}\label{subsec:time-net}
For $\theta \in (0,1]$ and $n\in \bN$ we introduce the non-uniform time-nets 
$\tau_n^{\theta}=\{t_{i,n}^{\theta}\}_{i=0}^n$ with 
\begin{equation}\label{eqn:theta_time-nets}
   t_{i,n}^{\theta} := T - T \( 1 - \tfrac{i}{n} \)^\frac{1}{\theta}
\end{equation}
for $i=0,\ldots,n$, that are characterized by the uniform quantization property
\[ \frac{\theta}{T^\theta} \int_{t_{i-1,n}^\theta}^{t_{i,n}^\theta} (T-u)^{\theta-1} \od u = \frac{1}{n}
   \sptext{1}{for}{1}
   i=1,\ldots,n. \]
We define the set of all {\em deterministic} time-nets
\[ \mathcal T := \{\tau = \{t_i\}_{i=0}^n : 0=t_0< t_1 < \cdots < t_n=T, \,  n\in \bN\} \]
and, for $\theta\in (0,1]$ and $\tau=\{t_i\}_{i=0}^n \in {\mathcal T}$, 
\[
\|\tau\|_\theta := \sup_{i=1,\ldots,n} \frac{t_i - t_{i-1}}{(T-t_{i-1})^{1-\theta}}.
\]
Note that 
\begin{align}\label{eqn:upper_bound_adapted_nets}
\| \tau_n^\theta \|_1 \leqslant \frac{T}{\theta n}
\sptext{1}{and}{1}
\| \tau_n^\theta \|_\theta \leqslant \frac{T^\theta}{\theta n},
\end{align}
and
\begin{align}\label{eq:adapted_net_monotnicity}
\frac{t_i-u}{(T-u)^{1-\theta}} \leqslant \frac{t_i-t_{i-1}}{(T-t_{i-1})^{1-\theta}}
\sptext{1}{for}{1} u\in [t_{i-1},t_i]\cap [0,T).
\end{align}

\subsection{List of notation}
Finally we list some important notations used in the article.
\medskip

{\bf Operators and spaces}

\begin{tabular}{|l|l|l|}\hline 
$\hoelO {\eta,q}$ & H\"older spaces & \cref{sec:intro:function_spaces}\\
$\BMO_p^\Phi(\timed)$, $\bmo_p^\Phi(\timed)$ & BMO spaces & \cref{definition:weighted_bmo} \\
$\cSM_p(\timed)$ & regularity for weights $\Phi$ & \cref{def:SM_p} \\
$\cI^\alpha$ and $\cI^\alpha_t$  & Riemann-Liouville type operators & \cref{definition:RL} \\
$\B_{p,q}^\alpha$, $\B_{b,q}^\alpha$ & norms for martingales & Definitions \ref{definition:B_pq^alpha} and \ref{definition:B_bq^alpha} \\
$\bvlocal$, $\lbvlocal \eta$ &  spaces of bounded variation & Definitions \ref{definition:bvloc} and \ref{defi:lbvlocal} \\ 
$\defX$, $D_\rho F$, $\dom(\Gamma_\rho^i)$, $\Gamma_\rho^i $ & weighted difference operators
                                         & \cref{statement:definition_Gamma} \\ \hline
\end{tabular}
\bigskip

{\bf Typical meaning of parameters}

\begin{tabular}{|l|l|l|} \hline
$\alpha$ & smoothing parameter for Riemann-Liuoville operator & \\ 
$\theta$ & mesh-size parameter in $\| \tau \|_\theta$ for time-net $\tau\in \cT$ and & \cref{subsec:time-net} \\
         &  smoothness of terminal condition $g(Y_T)$ & \cref{sec:application_brownian_case} \\
$\varepsilon$ & behaviour of measure $\rho$ around zero & \cref{sec:application_levy_case} and
\cref{sec:approximation_random_measure} \\
$\beta$ & small jump behaviour  of L\'evy process  $X$ & \cref{sec:application_levy_case} and
\cref{sec:approximation_random_measure} \\
$\eta$ & smoothness of terminal condition $f(X_T)$ & \cref{sec:application_levy_case} and
\cref{sec:approximation_random_measure} \\ \hline
\end{tabular}
\bigskip

{\bf Square functions}

\begin{tabular}{|l|l|}\hline 
$\intsq{\varphi}{\tau}_a$ & \cref{def:integrated_square_function} \\
$\intsqw{\varphi}{\pi}{\tau}_a$ & \cref{ass:random_measures} \\
$\intsqw{\varphi}{\sigma}{\tau}_a$  & \cref{ass:random_measures_simplified}\\ \hline
\end{tabular}
\bigskip

{\bf Oscillation of process, characteristics of L\'evy process $X$ and 
functional $\rho$}

\begin{tabular}{|l|l|l|}\hline 
$\losc_t(\varphi)$, $\uosc_t(\varphi)$ & oscillation of processes & \cref{definition:oscillation} \\
${\rm TV}(\rho,\eta)$ & weighted total variation & \cref{definition:TVrho} \\
$U(\varepsilon;c)$, $\cU(\beta,A;c)$ & upper bounds & \cref{defi:uppper_bounds_X_rho} \\
$L(\varepsilon)$, $\cL(\beta)$ & lower bounds & \cref{defi:lower_bounds_X_rho} \\ \hline
\end{tabular}


\section[Riemann-Liouville type operators]{Riemann-Liouville type operators}
\label{sec:RL-operators}

Riemann-Liouville operators are a central object and tool in fractional calculus. It is natural and useful 
to extend them to random frameworks. There are two principal approaches:
Directly translating the approach from fractional calculus, that uses Volterra kernels, leads to the notion of fractional processes,
in particular fractional martingales.
In our setting one would take a c\`adl\`ag process $K$ and would consider
\[ t \mapsto \int_0^t (t-u)^{\alpha-1} K_u \od u. \] 
This yields to an approach natural for path-wise fractional calculus of stochastic processes and is used, for example, for Gaussian processes 
\cite{Hu_etal:09}. For our purpose we use the different approach
\[ t\mapsto \int_0^T (T-u)^{\alpha-1} K_{u\wedge t} \od u \]  
to define $\cI_t^\alpha K$ in Definition \ref{definition:RL} below. 
The idea behind the operator $\cI^\alpha$ is to 
remove or reduce singularities of a c\`adl\`ag process $(K_t)_{t\in [0,T)}$ when $t\uparrow T$.
As we see in  \cref{theo:intro_L2} below, this approach is the right one to handle fractional smoothness in the Malliavin sense 
and in the sense of interpolation theory.
One basic difference to the Volterra-kernel approach is that, starting with a (sub-, super-) martingale $L$, we again obtain a (sub-, super-) martingale 
$\cI_t^\alpha L$.
This second approach was used in \cite[Definition 4.2]{Geiss:Toivola:09}, \cite[Section 4]{Geiss:Toivola:15}, and \cite{Applebaum:Banuelos:15},
and relates to fractional integral transforms of martingales  
(see, for example, \cite{Arai_etal:20}). This  
corresponds to equation \eqref{eqn:1:statement:RL_for_martinagles} of our \cref{statement:RL_for_martinagles-II}.

\begin{defi}
\label{definition:RL}
For $\alpha>0$ and a c\`adl\`ag function $K:[0,T)\to \R$ we define 
$\cI^\alpha K := (\cI_t^\alpha K)_{t\in [0,T)}$ 
by
\[ \cI_t^\alpha K:= \frac{\alpha}{T^\alpha} \int_0^T (T-u)^{\alpha-1} K_{u\wedge t} \od u. \]
Moreover, for $\alpha=0$ we let $\cI_t^0 K:=K_t$.
\end{defi}
\medskip

The c\`adl\`ag property implies the  boundedness of $K$ on any compact interval of $[0, T)$. 
Therefore, $\cI^\alpha K$ is well-defined and c\`adl\`ag on $[0, T)$.
The above definition can be re-formulated in terms of the classical Riemann-Liouville operator 
$\mathcal{R}_a^\alpha(f) := \frac{1}{\Gamma(\alpha)} \int_0^a  (a-u)^{\alpha -1} f(u) \od u$
by
\begin{align*}
\mathcal{R}_T^\alpha(K^{(t)}) & = \frac{T^\alpha}{\Gamma(\alpha+1)} \cI_t^\alpha K
  \sptext{1}{with}{1} 
  K^{(t)}_u:= K_{u\wedge t}
\end{align*}
where we compute the Riemann-Liouville operator, applied to the function $u\mapsto K^{(t)}_u$,
at $a=T$.
We use a different normalisation as we want to interpret the kernel in the Riemann-Liouville integral
as density of a probability measure. 
To shorten the  notation we will also call $\cI_t^\alpha K$ Riemann-Liouville operator although 
it is variant  of the classical Riemann-Liouville operator.
It follows directly from the definition that we have, for $\alpha \geqslant 0$, 
\begin{align}\label{eqn:RL-split}
\cI_t^\alpha K =  \frac{\alpha}{T^\alpha} \int_0^t (T-u)^{\alpha -1} K_u \od u + \left ( \frac{T-t}{T} \right )^\alpha K_t.
\end{align}
In the following we only need $\cI^\alpha K$ for $\alpha \geqslant 0$. However, to derive an inversion formula we extend the
definition by \eqref{eqn:RL-split} to the case $\alpha<0$ and prove that there is a group structure behind:
\smallskip

\begin{prop}
\label{statemant:RL_group_structure}
Define for $\alpha<0$, a c\`adl\`ag function $K:[0,T)\to \R$, and $t \in [0,T)$,
$\cI_t^\alpha K$ by formula \eqref{eqn:RL-split}. Then 
\begin{enumerate}[{\rm (1)}]
\item \label{item:1:statemant:RL_group_structure}
      $\cI^\alpha_t (\cI^\beta_\cdot K) = \cI^{\alpha+\beta}_t K$ for all $\alpha,\beta\in \R$,
\item \label{item:2:statemant:RL_group_structure}
      $\cI_t^{-\alpha} (\cI^\alpha_ \cdot K) = K_t$ for all $\alpha\in \R$.
\end{enumerate}
\end{prop}

\begin{proof}
As \eqref{item:2:statemant:RL_group_structure} follows from \eqref{item:1:statemant:RL_group_structure},
we only verify \eqref{item:1:statemant:RL_group_structure}, which follows from

\begin{align*}
\cI_t^{\alpha}(\cI^\beta K) & = \frac{\alpha}{T^\alpha} \int_0^t (T-u)^{\alpha -1} \cI^\beta_u K \od u + \(\frac{T-t}{T}\)^\alpha \cI_t^\beta K \\
& = \frac{\alpha}{T^\alpha} \int_0^t (T-u)^{\alpha -1} \(\frac{\beta}{T^\beta} \int_0^u (T-v)^{\beta -1} K_v \od v + \(\frac{T-u}{T}\)^\beta K_u\) \od u \\
& \qquad + \(\frac{T-t}{T}\)^\alpha \( 
  \frac{\beta}{T^\beta} \int_0^t (T-u)^{\beta -1} K_u \od u + \(\frac{T - t}{T}\)^\beta K_t\)\\
& = \frac{\alpha \beta}{T^{\alpha +\beta}}\int_0^t (T-u)^{\alpha -1} \int_0^u (T-v)^{\beta -1} K_v \od v \od u + \frac{\alpha}{T^{\alpha +\beta}} \int_0^t (T-u)^{\alpha + \beta -1} K_u \od u\\
& \qquad + \frac{\beta (T-t)^\alpha}{T^{\alpha + \beta}} \int_0^t (T-u)^{\beta -1} K_u \od u + \(\frac{T-t}{T}\)^{\alpha + \beta} K_t\\
& = \frac{\beta}{T^{\alpha +\beta}}\int_0^t (T-v)^{\beta -1} K_v \((T-v)^\alpha - (T-t)^\alpha\) \od v + \frac{\alpha}{T^{\alpha +\beta}} \int_0^t (T-u)^{\alpha + \beta -1} K_u \od u\\
& \qquad + \frac{\beta (T-t)^\alpha}{T^{\alpha + \beta}} \int_0^t (T-u)^{\beta -1} K_u \od u + \(\frac{T-t}{T}\)^{\alpha + \beta} K_t\\ 
& = \frac{\beta}{T^{\alpha +\beta}}\int_0^t (T-v)^{\alpha + \beta -1} K_v \od v - \frac{\beta(T-t)^\alpha}{T^{\alpha + \beta}}\int_0^t (T-v)^{\beta -1} K_v \od v \\
& \quad + \frac{\alpha}{T^{\alpha +\beta}} \int_0^t (T-u)^{\alpha + \beta -1} K_u \od u
         + \frac{\beta (T-t)^\alpha}{T^{\alpha + \beta}} \int_0^t (T-u)^{\beta -1} K_u \od u + \(\frac{T-t}{T}\)^{\alpha + \beta} K_t\\ 
& = \cI^{\alpha + \beta}_t K. \qedhere
\end{align*}
\end{proof}
We continue with some more structural properties:
\medskip

\begin{prop}
\label{statemant:RL_general}
For a c\`adl\`ag function $K:[0,T)\to \R$ and $t \in [0,T)$ one has:
\begin{enumerate}[{\rm (1)}]
\item \label{item:1:statemant:RL_general}
      $\lim_{\alpha\downarrow 0} \cI_t^\alpha K = K_t$.
\item \label{item:2:statemant:RL_general}
      $\lim_{\alpha\uparrow \infty} \cI_t^\alpha K = K_0$.
\item \label{item:5:statemant:RL_general}
      $\Delta \cI_t^\alpha K = \left ( \frac{T-t}{T} \right )^\alpha \Delta K_t$ for $\alpha\in \R$.
\end{enumerate}
\end{prop}

\begin{proof}
\eqref{item:1:statemant:RL_general} and \eqref{item:5:statemant:RL_general} follow from \cref{eqn:RL-split},
and \eqref{item:2:statemant:RL_general} from the c\`adl\`ag property of $K$.
\end{proof}

With the next statement we derive properties of $K$ from properties of $\cI_t^\alpha K$:

\begin{prop}\label{statement:properties_inverse_RL}
For $\alpha>0$, a c\`adl\`ag function $K:[0,T)\to \R$, and $0 \leqslant s < t < T$ we have 
\begin{equation}\label{eqn:statement:properties_inverse_RL}
 K_t -K_s =   \left ( \frac{T}{T-t} \right )^\alpha \left ( \cI_t^\alpha K - \cI_s^\alpha K \right ) 
              - \alpha T^\alpha \int_s^t (T-u)^{-\alpha-1} \left ( \cI_u^\alpha K - \cI_s^\alpha K \right ) \od u.
\end{equation}
Consequently the following holds:
\begin{enumerate}[{\rm (1)}]
\item \label{item:1:statement:properties_inverse_RL}
      $|K_t -K_s| \leqslant 2 \left ( \frac{T}{T-t} \right )^\alpha \sup_{u \in [s,t]}  \left | \cI_u^\alpha K - \cI_s^\alpha K \right |$.
\item \label{item:2:statement:properties_inverse_RL}
      If $\cI_T^\alpha K:=\lim_{t\uparrow T} \cI_u^\alpha K\in \R$ does exist, then $\lim_{t\uparrow T} (T-t)^\alpha K_t = 0$.
\end{enumerate}
\end{prop}

\begin{proof}
To verify relation \eqref{eqn:statement:properties_inverse_RL} we define $L_t := \cI_t^\alpha K$ for $t\in [0,T)$, 
express $K_t-K_s$ as $\cI_t^{-\alpha} L - \cI_s^{-\alpha} L$, and use \eqref{eqn:RL-split} to get
\[    \cI_t^{-\alpha} L - \cI_s^{-\alpha} L 
   = \left ( \frac{T}{T-t} \right )^\alpha (L_t-L_s) - \alpha T^\alpha \int_s^t (T-u)^{-\alpha-1} (L_u-L_s) \od u.\]
\eqref{item:1:statement:properties_inverse_RL} follows from \eqref{eqn:statement:properties_inverse_RL}.
\eqref{item:2:statement:properties_inverse_RL}
We use 
\[    \left ( \frac{T-t}{T} \right )^\alpha K_t  
   =  \cI_t^\alpha K - \int_0^t  (\cI_u^\alpha K) \left (  \alpha \frac{(T-t)^\alpha}{(T-u)^{\alpha+1}} \right ) \od u, \]
$\lim_{t\uparrow T} \int_0^t  \left (  \alpha \frac{(T-t)^\alpha}{(T-u)^{\alpha+1}} \right ) \od u =1$, and that 
$\lim_{t\uparrow T} \sup_{u\in [0,v]} \frac{(T-t)^\alpha}{(T-u)^{\alpha+1}} =0$ for all $v\in [0,T)$.
\end{proof}
\bigskip

The particular case that the function $K$ is a path of a c\`adl\`ag martingale $L$ is of our interest. 
The following statement is obvious, but useful:

\begin{prop}
\label{statement:RL_for_martinagles-I}
If $\alpha \geqslant 0$ and $L=(L)_{t\in [0, T)}$ is a c\`adl\`ag martingale (c\`adl\`ag super-, or sub-martingale), 
then $(\cI_t^\alpha L)_{t\in [0, T)}$
is a c\`adl\`ag martingale (c\`adl\`ag super-, or sub-martingale).
\end{prop}

The following functional $\intsq{L}{\tau}$ measures the oscillation of a martingale along a time-net in terms of an area.
Besides the functional occurs in various approximation problems for stochastic integrals, it is particularly designed 
to deal with the oscillation of non-closable martingales.
In \cref{theo:intro_L2} we characterize by the behaviour of this functional 
the degree of singularity of a martingale not closable in $L_2$.
Moreover, in \cref{sec:convergence_sqfunction} we investigate 
the convergence of this functional to a classical square function as the time-nets refine.

\begin{defi}
\label{def:integrated_square_function}
For a deterministic  time-net $\tau=\{t_i\}_{i=0}^n$, $0=t_0 < \cdots < t_n =T$, $a \in [0,T)$,
and a c\`adl\`ag process $L=(L_t)_{t\in [0,T)}$ we let
\[ \intsq{L}{\tau}_a := \int_0^a \left| L_u - \sum_{i=1}^n L_{t_{i-1}} \1_{(t_{i-1},t_i]}(u) 
   \right|^2 \od u \in [0,\infty). \]
Moreover, we define $\intsq{L}{\tau}_T := \lim_{a\uparrow T} \intsq{L}{\tau}_a\in [0,\infty]$. 
\end{defi}
\medskip

Now we give in \cref{theo:intro_L2} a first link between the Riemann-Liouville operators $\cI^\alpha$, real interpolation,
and the square function $\intsq{L}{\tau}$. To do this as simple as possible, we replace a martingale 
$L=(L_t)_{t\in [0,T)}$ by its discrete time version
\[ L^d := (L_{t_k})_{k=0}^\infty
   \sptext{1}{with}{1}
   t_k:=T\left (1-\frac{1}{2^k} \right ). \]
For the vector-valued interpolation we use $H:=L_2(\Omega,\F,\p)$ and
the end-point spaces
\begin{align*}
L^d \in \ell_2^{-\frac{1}{2}}(H) & \Longleftrightarrow \int_0^T \| L_t\|_{L_2}^2 \od t < \infty, \\
L^d \in \ell_\infty(H) & \Longleftrightarrow \|L^d\|_{\ell_\infty(H)} = \sup_{t\in [0,T)} \|L_t\|_{L_2} < \infty,
\end{align*}
where the first equivalence follows from \eqref{eq:discrete_interpol_continuous_interpol} below
and the spaces $\ell_q^s(H)$ and $\ell_\infty(H)$ were introduced in Section \ref{sub:sec:interpolation}.
The first condition, $\int_0^T \| L_t\|_{L_2}^2 \od t < \infty$, 
is a typical condition on martingales that appear as gradient processes. The other end-point,
$\sup_{t\in [0,T)} \| L_t \|_{L_2} < \infty$, consists of the martingales $L$ that are closable in $L_2$.
We will interpolate between these two end-points by the real interpolation method:
\smallskip

\begin{theo}
\label{theo:intro_L2}
For $\theta \in (0,1)$, $\alpha:=\frac{1-\theta}{2}$, and  a c\`adl\`ag martingale $L=(L_t)_{t\in [0,T)} \subseteq L_2$ 
the following assertions are equivalent:
\begin{enumerate}[{\rm (1)}]
\item \label{item:1:theo:intro_L2}
      $L^d \in (\ell_2^{-\frac{1}{2}}(H),\ell_\infty(H))_{\theta,2}$
      where  $\ell_2^{-\frac{1}{2}}(H)$ is defined in \eqref{eqn:definition_ellqsE}.
\item \label{item:2:theo:intro_L2}
      $(\cI_t^\alpha L)_{t\in [0,T)}$ is closable in $L_2$.
\item \label{item:3:theo:intro_L2}
      There is a $c>0$ such that  $\E [L;\tau]_T \leqslant c \|\tau\|_\theta$ for all $\tau \in \cT$.
\end{enumerate}
\end{theo}
\medskip

Before we prove \cref{theo:intro_L2} let us comment on it:

\begin{rema}
From Item \eqref{item:2:theo:intro_L2} we get for all $\varepsilon>0$ a $t(\varepsilon)\in [0,T)$ 
such that for $s\in [t(\varepsilon),T)$ one has
\begin{equation}\label{eqn:convergence_RL_martingale}
 \E \sup_{t\in [s,T)} \left | \int_s^T  ( L_{u\wedge t} - L_s )  (T-u)^{\alpha-1} \frac{\alpha}{T^\alpha} \od u \right |^2 < \varepsilon.
\end{equation}
Without the supremum the left-hand side is equal to $\E|\cI_t^\alpha L- \cI_s^\alpha L|^2$,  
the statement including the supremum follows from Doob's maximal inequality. The convergence in \eqref{eqn:convergence_RL_martingale} 
when $s\uparrow T$ is the replacement of the $L_2$- and a.s. convergence of $L$ in the case $L$ 
would be closable in $L_2$.
\end{rema}

For the proof of \cref{theo:intro_L2} and later in the article we need the following \cref{statement:RL_for_martinagles-II}.
We remark that \cref{statemant:RL_group_structure}\eqref{item:1:statemant:RL_group_structure} for $\alpha,\beta \geqslant 0$ 
can be also understood from equation  \eqref{eqn:1:statement:RL_for_martinagles} of \cref{statement:RL_for_martinagles-II} in the martingale setting.
\bigskip

\begin{prop}
\label{statement:RL_for_martinagles-II}
For $\alpha>0$, a c\`adl\`ag martingale $L=(L_t)_{t\in [0,T)}\subseteq L_2$ and $0\leqslant a < t < T$ one has, a.s.,
\begin{align}
    \cI_t^\alpha L 
& = L_0 + \int_{(0,t]}   \left ( \frac{T-u}{T} \right )^{\alpha} \od L_u,
    \label{eqn:1:statement:RL_for_martinagles}\\
    \ce{\F_a}{\left |\cI^{\alpha}_t  L - \cI^{\alpha}_a L\right |^2} 
& = 2  \alpha\ce{\F_a}{\int_a^T |L_{u\wedge t} - L_a|^2 \left (\frac{T-u}{T} \right)^{2\alpha -1} \frac{\od u}{T}}, 
    \label{eqn:3:statement:RL_for_martinagles}\\
    \ce{\F_a}{\left |\cI^{\alpha}_t  L - \cI^{\alpha}_a L\right |^2}
    + \left ( \frac{T-a}{T} \right )^{2 \alpha} |L_a|^2
& = 2 \alpha \ce{\F_a}{\int_a^T |L_{u\wedge t}|^2 \left (\frac{T-u}{T} \right)^{2\alpha -1} \frac{\od u}{T}}.
    \label{eqn:5:statement:RL_for_martinagles}
\end{align}
\end{prop}

\begin{proof}
Relation \eqref{eqn:1:statement:RL_for_martinagles}:
We apply partial integration to $\left ( \left ( \frac{T-t}{T} \right )^{\alpha} L_t\right )_{t\in [0,T)}$ and obtain,
for $t\in [0,T)$, that
\[    \left ( \frac{T-t}{T} \right )^{\alpha} L_t 
   =  \left ( \frac{T-0}{T} \right )^{\alpha} L_0 
                + \int_{(0,t]}   \left ( \frac{T-u}{T} \right )^{\alpha} \od L_u
                - \frac{\alpha}{T^\alpha} \int_{(0,t]}  (T-u)^{\alpha-1} L_u \od u \mbox{ a.s.} \]
Taking the last term  to the left side, we obtain \eqref{eqn:1:statement:RL_for_martinagles}.
For \eqref{eqn:3:statement:RL_for_martinagles} we use It\^o's isometry to get, a.s.,
\begin{align}
    \ce{\F_a}{\left | \cI_t^\alpha L - \cI_a^\alpha L\right|^2}
& = \ce{\F_a}{\int_{(a,t]} \left ( \frac{T-u}{T} \right )^{2\alpha} \od [L]_u} \label{eqn:squarefunction_RL} \\
& = \frac{1}{2\alpha T^{2\alpha}} \ce{\F_a}{\int_{(a,t]} \int_{[u,T)} (T-v)^{2\alpha -1} \od v \od [L]_u} \notag \\
& = \frac{1}{2\alpha T^{2\alpha}} \ce{\F_a}{\int_{(a,T)} \int_{(a,v\wedge t]}\od [L]_u  (T-v)^{2\alpha -1} \od v } \notag \\
& = \frac{1}{2\alpha T^{2\alpha}} \ce{\F_a}{\int_{(a,T)} |L_{v\wedge t} - L_a|^2 (T-v)^{2\alpha -1} \od v }. \notag
\end{align}
\smallskip
Relation \eqref{eqn:5:statement:RL_for_martinagles} follows directly from
\eqref{eqn:3:statement:RL_for_martinagles} and the orthogonality of $L_{u\wedge t}-L_a$ and $L_a$.
\end{proof}

\begin{proof}[Proof of \cref{theo:intro_L2}]
Because $(\|L_{t_k}\|_H)_{k=0}^\infty$ is non-decreasing we observe for $s\in \R$ that 
\begin{equation}\label{eq:discrete_interpol_continuous_interpol}
  \frac{\|(L_{t_k})_{k=0}^\infty\|_{\ell_2^s(H)}^2}{2 T^{2s}} =
        \sum_{k=0}^\infty (T-t_k)^{-1-2s} (t_{k+1}-t_k) \|L_{t_k}\|^2_H
   \sim_{c_{T,s}} \int_0^T(T-t)^{-1-2s}\|L_t\|_H^2 \od t 
\end{equation}
for some $c_{T,s}\geqslant 1$. For $s:=(1-\theta) \left ( -\frac{1}{2} \right) + \theta 0$
(so that $-1-2s= -\theta$) we use \cref{statement:RL_for_martinagles-II} (equation \eqref{eqn:5:statement:RL_for_martinagles})
with $a=0$ to get 
\[     \int_0^T (T-t)^{-\theta} \|L_t\|_H^2 \od t 
    = \sup_{t\in [0,T)} \frac{T^{2\alpha}}{2\alpha} \E [|\cI^\alpha_t L - L_0|^2+|L_0|^2]
    = \sup_{t\in [0,T)} \frac{T^{2\alpha}}{2\alpha} \E |\cI^\alpha_t L|^2. \]
Now the equivalence \eqref{item:1:theo:intro_L2} $\Leftrightarrow$ \eqref{item:2:theo:intro_L2}
follows from \eqref{eq:interpol_ell_q^s} and \eqref{eq:discrete_interpol_continuous_interpol}. 
The equivalence \eqref{item:2:theo:intro_L2} $\Leftrightarrow$ \eqref{item:3:theo:intro_L2}
follows from equation \eqref{eq:item:1:statemtent:equivalence_simplified_setting_new_global}
applied to $M:=L$ and $\sigma\equiv 1$.
\end{proof}

With the next definition we introduce the norms 
$\| \cdot \|_{\B_{p,q}^\alpha}$ for martingales. One origin of these norms 
is the $K$-method in real interpolation theory.
In \cref{statement:properties_Besov-spaces} we relate this norm to the BMO-norm of a martingale after transformed by the
Riemann-Liouville operator. As the latter BMO-norms serve in our approximation problem as upper bounds,
\cref{statement:properties_Besov-spaces} is the   key to relate real interpolation spaces to approximation
properties later.

\begin{defi}
\label{definition:B_pq^alpha}
For $\alpha>0$, $p,q\in [1,\infty]$, and a martingale $L=(L_t)_{t\in [0,T)}$ 
we let
\[    \| L \|_{\B_{p,q}^\alpha}
   := \left \|  t\mapsto (T-t)^\alpha  \| L_t \|_{L_p} \right \|_{L_q([0,T),\frac{\od t}{T-t})}. \]
\end{defi}
Because $[0,T)\ni t \mapsto \| L_t \|_{L_p}\in [0,\infty]$ is non-decreasing, this map is measurable and   
$ \| L \|_{\B_{p,q}^\alpha}$ is well-defined.
Moreover, for $q\in [1,\infty)$ it follows that
\begin{equation}\label{eqn:relations_between_Bpqalpha}
\| L \|_{\B_{p,\infty}^\alpha} \leqslant \sqrt[q]{\alpha q} \| L \|_{\B_{p,q}^\alpha}.
\end{equation}
In dependence on the fine-tuning index $q\in \{1,2,\infty\}$  in 
$\| L \|_{\B_{p,q}^\alpha}$ we have the following relations:
\medskip

\begin{theo}
\label{statement:properties_Besov-spaces}
For $\alpha>0$ and a c\`adl\`ag martingale $L=(L_t)_{t\in [0,T)}$ one has
\begin{enumerate}[{\rm (1)}]
\item \label{item:1:statement:properties_Besov-spaces} 
      $\| \sup_{t\in [0,T)} |\cI_t^\alpha L| \|_{L_\infty} \leqslant \frac{\alpha}{T^\alpha}  \| L \|_{\B_{\infty,1}^\alpha}$,
\item \label{item:2:statement:properties_Besov-spaces} 
      $\| \cI^\alpha L \|_{\BMO_2([0,T))}   \leqslant 3 \frac{\sqrt{2 \alpha}}{T^{\alpha}}  \| L \|_{\B_{\infty,2}^\alpha}$
      if $L_0\equiv 0$,
\item \label{item:3:statement:properties_Besov-spaces} 
      $\| L_t \|_{L_\infty}          \leqslant \frac{1}{(T-t)^\alpha} \| L \|_{\B_{\infty,\infty}^\alpha}$
      for $t\in [0,T)$.
\end{enumerate}
\end{theo}

\begin{proof}
\eqref{item:3:statement:properties_Besov-spaces} re-writes the definition, 
\eqref{item:1:statement:properties_Besov-spaces} is obvious because
\[
             \left \| \sup_{t\in [0,T)} |\cI_t^\alpha L| \right \|_{L_\infty}
   \leqslant \frac{\alpha}{T^\alpha} \int_0^T (T-t)^\alpha \|L_t\|_{L_\infty} \frac{\od t}{T-t}. \]
\eqref{item:2:statement:properties_Besov-spaces} 
From \cref{statement:RL_for_martinagles-II} it follows that, a.s.,
\begin{align*}
\ce{\F_a}{|\cI_t^\alpha L - \cI_a^\alpha L|^2} + \left ( \frac{T-a}{T} \right )^{2\alpha} |L_a|^2 
&    =     2 \alpha \ce{\F_a}{\int_a^T |L_{u\wedge t}|^2 \left ( \frac{T-u}{T} \right )^{2 \alpha - 1} \frac{\od u}{T}} \\
&\leqslant \frac{2 \alpha}{T^{2 \alpha}} \int_a^T (T-u)^{2\alpha} \|L_u\|_{L_\infty}^2 \frac{\od u}{T-u} \\
&\leqslant \frac{2 \alpha}{T^{2 \alpha}} \| L \|_{\B_{\infty,2}^\alpha}^2.
\end{align*}
Therefore there is a c\`adl\`ag modification $\tilde L$ of $L$ such that 
$|\tilde L_t| \leqslant \sqrt{2\alpha} (T-t)^{-\alpha} \| L \|_{\B_{\infty,2}^\alpha}$ on $[0,T)\times \Omega$.
In order to prove \eqref{item:2:statement:properties_Besov-spaces} we may assume w.l.o.g. that $L=\tilde L$.
For $a\in [0,T)$ this implies
\[          | \Delta \cI_a^\alpha L | 
       =     \left ( \frac{T-a}{T} \right )^{\alpha} |\Delta L_a|
   \leqslant 2 \frac{\sqrt{2 \alpha}}{T^{\alpha}} \| L \|_{\B_{\infty,2}^\alpha}
\]
and, by \cref{statement:bmo_determinstic} and \cref{statement:relation_BMO_bmo}\eqref{item:1:relation-bmo},
$\| \cI^\alpha L \|_{\BMO_2([0,T))}   \leqslant 3 \frac{\sqrt{2 \alpha}}{T^{\alpha}} \| L \|_{\B_{\infty,2}^\alpha}$.
\end{proof}

\begin{rema}
\label{statement:sharpness_Besov_to_RL}
Regarding the sharpness of parameters in \cref{statement:properties_Besov-spaces} we have the following
(partially preliminary) results for the case $\alpha\in \left (0,\frac{1}{2} \right )$
(and $T=1$).

\cref{statement:properties_Besov-spaces}\eqref{item:1:statement:properties_Besov-spaces}:
      For all $q \in (1,\infty]$ there is a  continuous martingale $L^{(q)}$ such that
      \[   \| L^{(q)} \|_{\B_{\infty,q}^\alpha} < \infty 
           \sptext{1}{but}{1} 
           \int_0^1 (1-t)^{\alpha-1} \| L_t^{(q)} \|_{L_\infty} \od t 
         = \| L^{(q)} \|_{\B_{\infty,1}^\alpha}
         = \infty. \]
      In fact, assume $\eta \in (0,1)$ such that $\alpha = (1-\eta)/2$. Using \cite[Corollary 4.4]{Geiss:20}, for all $q\in (1,\infty]$ there is an 
      $f^{(q)} \in \hoelO{\eta,q}$ such that 
      $f^{(q)} \not \in \B_{2,1}^\eta(\R,\gamma)$, where
      $\B_{2,1}^\eta(\R,\gamma)$ is the Gaussian Besov space of fractional order $\eta$ used in
      \cite{Geiss:Toivola:15,Geiss:20} and obtained by real interpolation, i.e.
      $\B_{2,1}^\eta(\R,\gamma)=(L_2(\R,\gamma),\D_{1,2}(\R,\gamma))_{\eta,1}$.
      Because $f^{(q)}\in L_2(\R,\gamma)$ we can use \cite[Theorem 3.1]{Geiss:Toivola:15} and must have
      \[ \int_0^1 (1-t)^{\alpha-1} \left \| \frac{\partial F^{(q)}}{\partial x}(t,W_t) \right \|_{L_2} \od t  = \infty \]
      with $F^{(q)}(t,x) := \E f^{(q)}(x+W_{1-t})$. As martingale $L^{(q)}$ we use 
      $L_t^{(q)} := \frac{\partial F^{(q)}}{\partial x}(t,W_t)$, $t\in [0,1)$, and obtain 
      $\int_0^1 (1-t)^{\alpha-1} \| L_t^{(q)} \|_{L_\infty} \od t = \infty$,
      whereas $ \| L^{(q)} \|_{\B_{\infty,q}^\alpha} < \infty$ by
      \eqref{eqn:Bb_bound_H_part_1} and \eqref{eqn:item:3:statement:consistence-upper_bound_sigma>0}.
\medskip

\cref{statement:properties_Besov-spaces}\eqref{item:2:statement:properties_Besov-spaces}:
      For all $q \in (2,\infty]$ there is a continuous martingale $L^{(q)}$ with 
      $\| L^{(q)}\|_{\B_{\infty,q}^\alpha}<\infty$ and $L^{(q)}_0 \equiv 0$, but
      \begin{equation}\label{eqn:item:2:statement:sharpness_Besov_to_RL}
      \| \cI^\alpha L^{(q)} \|_{\BMO_2([0,T))} \geqslant \sup_{t\in [0,T)} \| \cI_t^\alpha L^{(q)} \|_{L_2}  = \infty.
      \end{equation}
      To verify this, we take $f$ and $F$ from the proof of \cite[Theorem 5.1]{Geiss:20}, choose 
      $L_t^{(q)} := \frac{\partial F}{\partial x}(t,W_t)-\frac{\partial F}{\partial x}(0,0)$, and proceed,
      taking again into the account \eqref{eqn:Bb_bound_H_part_1} and \eqref{eqn:item:3:statement:consistence-upper_bound_sigma>0},
      as for \cref{statement:properties_Besov-spaces}\eqref{item:1:statement:properties_Besov-spaces}.
      
\end{rema}


\section{$L_1$-convergence of $\intsq{L}{\tau_n^{\theta}}_b$} 
\label{sec:convergence_sqfunction}
If $L=(L)_{t\in [0,T)}$ is a semi-martingale with $L_0\equiv 0$, then its 
quadratic variation has the natural representation
\[ [L]_b = \lim_{n\to \infty} \sum_{i=1}^n |L_{t_i^n\wedge b} - L_{t_{i-1}^n\wedge b}|^2
   \sptext{1}{if}{1} b\in [0,T) \]
and the limit is taken in probability, for any sequence of deterministic nets
$\tau_n=\{ 0=t_0^n < \cdots < t_n^n = T \}$ with 
$\lim_{n\to \infty} \| \tau_n \|_1=0$. 
Investigating the same question for $\intsq{L}{\tau_n}_b$
to get a geometric interpretation of scaling limits of
$\intsq{L}{\tau_n}_b$, the natural candidate for a limit 
is the quadratic variation $[ \cI^{\frac{1-\theta}{2}} L ]_b$ of the fractional integral 
$\cI^{\frac{1-\theta}{2}} L$. And indeed, after an obvious re-scaling in $n$ we obtain 
the desired result, however one needs to choose the nets $\tau_n$ according to the parameter 
$\theta$. Moreover, conditions on $[L]$ must be imposed in case deterministic nets are taken, 
otherwise such a result cannot hold: take a martingale $L=(L)_{t\in [0,T)}$ with $L_0\equiv 0$ that is constant
on all $[s_{l-1},s_l)$, $0=s_0<\cdots <s_N=T$, and take nets $\tau_n$, $n\geqslant 1$, such that
$\{ s_l\}_{l=0}^N \subseteq \tau_n$. Then
$\intsq{L}{\tau_n}_b\equiv 0$ for $b\in [0,T)$.
This is due to the fact that the particular time-nets $\tau_n$ minimize $\intsq{L}{\tau_n}_b$,
but do not maximize  $\intsq{L}{\tau_n}_b$.
The aim of this section is to present two cases where we have the desired convergence. In \cref{statement:pointwise_convergence_sq}
we assume a regularity  of the quadratic variation $[L]$, whereas in \cref{statement:pointwise_convergence_sq_random}
we exploit random nets in order to avoid conditions on $[L]$.
\medskip

\begin{theo}
\label{statement:pointwise_convergence_sq}
For a continuous martingale $L=(L_t)_{t\in [0,T)}\subseteq L_2$ with $L_0\equiv 0$ and such that we have a representation
\begin{equation}\label{eqn:regularity_bracket}
   [L]_b(\omega) = \int_0^b D(u,\omega) \od u<\infty
   \sptext{1}{for}{1} (b,\omega)\in [0,T)\times \Omega,
\end{equation}
where $D:[0,T)\times \Omega \to [0,\infty)$ is progressively measurable, one has for $\theta\in (0,1]$ that 
\[ \lim_{n\to \infty}
    \left \| \frac{2\theta n}{T}  \intsq{L}{\tau_n^{\theta}}_b - [ \cI^{\frac{1-\theta}{2}}     
             L ]_b \right \|_{L_1} = 0 \sptext{1}{for}{1} b\in [0,T).\]
\end{theo}
\bigskip

In the case we do not have a regularity as in \eqref{eqn:regularity_bracket} we need to apply a randomization
of the time-nets. 
For $r\in [0,1)$ and $\theta\in (0,1]$
we introduce the time-nets
\[ \tau_n^{\theta,r} =\{ t_{i,n}^{\theta,r} \}_{i=0}^{n+1}
   := \left\{0, r \, t_{1,n}^\theta, t_{1,n}^\theta + r\, (t_{2,n}^\theta-t_{1,n}^\theta),\ldots,t_{n-1,n}^\theta + r\, (T-t_{n-1,n}^\theta),T\right \}. \]
Note that \eqref{eqn:upper_bound_adapted_nets} implies that 
\begin{equation}\label{eqn:size_random_nets}
 \| \tau_n^{\theta,r}\|_1 \leqslant  \| \tau_n^{\theta}\|_1 \leqslant \frac{T}{\theta n}.
\end{equation}
\smallskip

\begin{theo}
\label{statement:pointwise_convergence_sq_random}
For a c\`adl\`ag martingale $L=(L_t)_{t\in [0,T)}\subseteq L_2$ with $L_0\equiv 0$,
$\theta \in (0,1]$, and $b\in (0,T)$ one has that
\[ \lim_{n\to \infty}
    \left \| \frac{2\theta n}{T} \int_{[0,1)} \intsq{L}{\tau_n^{\theta,r}}_b \od r - [ \cI^{\frac{1-\theta}{2}}     
             L ]_{b-} \right \|_{L_1} = 0.\]
\end{theo}
\bigskip
In the remaining part of this section we prove \cref{statement:pointwise_convergence_sq} and \cref{statement:pointwise_convergence_sq_random}.
For a time-net $\tau=\{t_i\}_{i=0}^n\in \cT$ and $u\in (0,T]$ we let
\[ \overline{u}_\tau := t_i \sptext{1}{if}{1} 
   u \in (t_{i-1},t_i].\]

\begin{lemm}\label{statement:convergence_vs_nets}
For $\theta \in (0,1]$ and $b\in (0,T)$ one has 
\begin{align}
\sup_{r\in [0,1)} \sup_{n\in \bN}\sup_{u\in (0,b]} \,\, \frac{2\theta n}{T} ((\overline u_{\tau_n^{\theta,r}}\wedge b) -u ) 
  & \leqslant 2, \label{eqn:1:statement:convergence_vs_nets}\\
\lim_{n\to \infty} \frac{2\theta n}{T}\int_{[0,1)} \left ((\overline u_{\tau_n^{\theta,r}}\wedge b) -u \right ) \od r 
                       & = \left ( \frac{T-u}{T} \right )^{1-\theta} \sptext{5.7}{for}{.7} u\in (0,b), \label{eqn:2:statement:convergence_vs_nets}\\
\lim_{n\to \infty} \frac{2\theta n}{T} \int_0^b \left ((\overline{u}_{\tau_n^\theta}\wedge b) -u\right ) D(u) \od u 
  & = \int_0^b \left ( \frac{T-u}{T} \right )^{1-\theta} D(u) \od u  \sptext{1}{for}{.7} D\in L_1([0,b]).
\label{eqn:3:statement:convergence_vs_nets}
\end{align}
\end{lemm}

\begin{proof}
\eqref{eqn:1:statement:convergence_vs_nets} follows from
$             (\overline u_{\tau_n^{\theta,r}}\wedge b) -u  
   \leqslant  \| \tau_n^{\theta,r}\|_1 
   \leqslant  \frac{T}{\theta n}$.
\medskip

\eqref{eqn:2:statement:convergence_vs_nets} 
Because $u<b$ we only need to prove 
$\lim_{n\to \infty} \frac{2\theta n}{T}\int_{[0,1)} (\overline u_{\tau_n^{\theta,r}}-u ) \od r 
 = \left ( \frac{T-u}{T} \right )^{1-\theta}$.
Let $i(u,n)\in \{1,\ldots,n\}$ such that 
$u\in (t_{i(u,n)-1,n}^\theta,t_{i(u,n),n}^\theta]$ and assume that $i(u,n)<n$ which holds for $n\ge n(u,\theta)\geqslant 1$. Then one gets
that
\[           \frac{1}{2} \left (t_{i(u,n)+1,n}^\theta-t_{i(u,n),n}^\theta \right )
   \leqslant \int_{[0,1)} (\overline u_{\tau_n^{\theta,r}}-u ) \od r 
   \leqslant \frac{1}{2} \left (t_{i(u,n),n}^\theta-t_{i(u,n)-1,n}^\theta \right ).
\]
Finally, by the mean value theorem 
\[   \frac{\theta}{T} \lim_{n\to \infty} n \left (t_{i(u,n)+1,n}^\theta-t_{i(u,n),n}^\theta \right )
   = \frac{\theta}{T} \lim_{n\to \infty} n \left (t_{i(u,n),n}^\theta-t_{i(u,n)-1,n}^\theta \right )
   = \left ( 1 - \frac{u}{T} \right )^{1-\theta}. \]
\eqref{eqn:3:statement:convergence_vs_nets}
We define the measure $\mu_{n,\theta}$ on $\cB([0,T))$ by  
\[    \mu_{n,\theta}(\od u) 
   := \left [ \frac{2 \theta n}{T} \sum_{i=1}^n \1_{   (t_{i-1,n}^\theta,t_{i,n}^\theta]}(u) (\ov{u}_{\tau_n^\theta} - u)\right ] \od u \]
and observe that $\mu_{n,\theta}(\od u)$ converges weakly to $\mu_\theta(\od u):=\left (\frac{T-u}{T}\right )^{1-\theta}\od u$ on each interval 
$[0,b]\subset [0,T)$ as $n\to \infty$ (one has $\lim_n \mu_{n,\theta}([0,a]) = \mu_\theta([0,a])$ for $a\in [0,T)$ which is a consequence of 
$\lim_n \mu_{n,\theta}([0,a]) = \lim_n \frac{\theta n}{T} \sum_{i\geqslant 1: t_{i,n}^\theta \leqslant a} (t_{i,n}^\theta - t_{i-1,n}^\theta)^2$
and $t_{i,n}^\theta - t_{i-1,n}^\theta = \frac{T^\theta}{n\theta} (T-\xi_{n,i}^\theta)^{ 1-\theta}$ for some
$\xi_{n,i}^\theta\in [t_{i-1,n}^\theta,t_{i,n}^\theta]$).
For $a\in (0,b)$ and $F\in C([0,a])$ this implies that 
\begin{equation}\label{eqn:weak_convergence}
   \lim_{n\to \infty} \frac{2\theta n}{T} \int_0^a \left ((\overline{u}_{\tau_n^\theta}\wedge b) -u\right ) F(u) \od u 
   = \lim_{n\to \infty} \frac{2\theta n}{T} \int_0^a \left ( \overline{u}_{\tau_n^\theta} -u\right ) F(u) \od u 
   = \int_0^a \left ( \frac{T-u}{T} \right )^{1-\theta} F(u) \od u.
\end{equation}
For $D\in L_1([0,b])$ and $\varepsilon>0$ there is an $a\in (0,b)$ (to be chosen first) and 
an $F\in C([0,a])$ with
\begin{equation}\label{eqn:approximation_L1}
\int_a^b |D(u)|\od u < \varepsilon  
\sptext{1}{and}{1}  
\int_0^a |D(u)-F(u)|\od u < \varepsilon.
\end{equation}
Now using 
\eqref{eqn:1:statement:convergence_vs_nets},
$\left ( \frac{T-u}{T} \right )^{1-\theta} \leqslant 1$,
\eqref{eqn:approximation_L1}, and \eqref{eqn:weak_convergence} we obtain
\eqref{eqn:3:statement:convergence_vs_nets}.
\end{proof}
\medskip

\begin{lemm}
\label{statement:exact_formula_approximation}
For a c\`adl\`ag martingale $L=(L_t)_{t\in [0,T)}\subseteq L_2$ and
$0 \leqslant a < b < T$ one has that
\[
    \int_a^b (L_u-L_a)^2 \od u - \int_{(a,b]} (b-v) \od [L]_v 
    =  2 \int_{(a,b]} (b-v) (L_{v-} - L_a) \od L_v \mbox{ a.s.} \]
\end{lemm}

\begin{proof}
We fix $a \in [0,T)$. Then we get 
\[ (L_u-L_a)^2 = [L]_u - [L]_a + 2 \int_{(a,u]} (L_{v-}-L_a) \od L_v 
   \sptext{1}{for}{1}
   u\in [a,T) 
   \mbox{ a.s.} \]
where all terms are c\`adl\`ag in $u$. Integration over $(a,b]$ yields, a.s.,
\begin{align*}
    \int_{(a,b]} (L_u-L_a)^2 \od u 
& = \int_{(a,b]} \int_{(a,u]} \od [L]_v \od u 
    + 2 \int_{(a,b]}  \int_{(a,u]} (L_{v-}-L_a) \od L_v \od u \\
& = \int_{(a,b]} (b-v) \od [L]_v  
    + 2 \int_{(a,b]}  (b-v)(L_{v-}-L_a) \od L_v 
\end{align*}
where the stochastic Fubini argument for the last term can be verified  using
integration by parts of $(b-u)\int_{(a,u]}(L_{v-}-L_a) \od L_v$ in the bounds $u=a$ and $u=b$. 
\end{proof}

\begin{lemm}
\label{statement:elimination_first_term}
For a c\`adl\`ag martingale $L=(L_t)_{t\in [0,T)}\subseteq L_2$ with $L_0\equiv 0$,
a probability measure $\rho$ on $\cB([0,1))$, $\theta \in (0,1]$, and $b\in (0,T)$ one has that
\begin{multline*}
 \limsup_{n\to \infty}
    \left \| \frac{2\theta n}{T} \int_{[0,1)} \intsq{L}{\tau_n^{\theta,r}}_b \rho(\od r) - [ \cI^{\frac{1-\theta}{2}}     
             L ]_{b-} \right \|_{L_1} \\ \leqslant 
 \limsup_{n\to \infty} \left \|  \int_{(0,b)} \left [ \frac{2\theta n}{T}  \int_{[0,1)} \left ((\overline{u}_{\tau_n^{\theta,r}}\wedge b) - u 
              \right )  \rho(\od r)  - 
              \left ( \frac{T-u}{T} \right )^{1-\theta} \right ] \od [L]_u \right \|_{L_1}.
\end{multline*}
\end{lemm}

\begin{proof}
We start by
\begin{align*}
&            \left \| \frac{2\theta n}{T} \int_{[0,1)} \intsq{L}{\tau_n^{\theta,r}}_b \rho(\od r) - [ \cI^{\frac{1-\theta}{2}}     
             L ]_{b-} \right \|_{L_1} \\
& \leqslant  \left \| \frac{2\theta n}{T} \int_{[0,1)} 
              \left [  \intsq{L}{\tau_n^{\theta,r}}_b  -
              \int_{(0,b)}  \left ((\overline{u}_{\tau_n^{\theta,r}}\wedge b) - u \right ) \od [L]_u  \right ] \rho(\od r)
             \right \|_{L_1} \\ 
& +          \left \| \frac{2\theta n}{T}  \int_{(0,b)} \int_{[0,1)} \left ((\overline{u}_{\tau_n^{\theta,r}}\wedge b) - u \right )  \rho(\od r)      
             \od [L]_u - [ \cI^{\frac{1-\theta}{2}} L]_{b-} \right \|_{L_1} \\
& =  \left \| \frac{2\theta n}{T} \int_{[0,1)} 
              \left [  \intsq{L}{\tau_n^{\theta,r}}_b  -
              \int_{(0,b]}  \left ((\overline{u}_{\tau_n^{\theta,r}}\wedge b) - u \right ) \od [L]_u  \right ] \rho(\od r)
              \right \|_{L_1} \\ 
& +           \left \|  \int_{(0,b)} \left [ \frac{2\theta n}{T}  \int_{[0,1)} \left ((\overline{u}_{\tau_n^{\theta,r}}\wedge b) - u 
              \right )  \rho(\od r)  - 
              \left ( \frac{T-u}{T} \right )^{1-\theta} \right ] \od [L]_u \right \|_{L_1},
\end{align*}
where we used \eqref{eqn:1:statement:RL_for_martinagles}. To continue the proof  we let
\[ s_{i,n}^{\theta,r} := t_{i,n}^{\theta,r}\wedge b 
   \sptext{1}{and}{1}
   J_{i,n}^{\theta,r} := (s_{i-1,n}^{\theta,r},s_{i,n}^{\theta,r}]    \sptext{1}{so that}{1}
   \bigcup_{i=1}^{n+1}  J_{i,n}^{\theta,r} = (0,b]. \]
We fix $r\in [0,1)$. By \cref{statement:exact_formula_approximation}, the Burkholder-Davis-Gundy inequalities (here with the
constant $\beta_1>0$), and \eqref{eqn:size_random_nets} we obtain that 
\begin{align*}
& \left \|  \intsq{L}{\tau_n^{\theta,r}}_b  -
               \int_{(0,b]} \left ((\overline{u}_{\tau_n^{\theta,r}}\wedge b) - u\right ) \od [L]_u  
             \right \|_{L_1}\\
& = \left \| \sum_{i=1}^{n+1} \left [ \int_{J_{i,n}^{\theta,r}}
            \left ( L_u - L_{s_{i-1,n}^{\theta,r}} \right )^2 \od u - 
  \int_{J_{i,n}^{\theta,r}} \left ( s_{i,n}^{\theta,r} - u\right ) \od [L]_u  
  \right ]
             \right \|_{L_1} \\
& = 2 \left \| \sum_{i=1}^{n+1} 
  \left [ 
  \int_{J_{i,n}^{\theta,r}}
  \left ( s_{i,n}^{\theta,r}  - v                       \right ) 
  \left (  L_{v-} -  L_{s_{i-1,n}^{\theta,r}} \right )
  \od L_v 
  \right ]
             \right \|_{L_1} \\          
& \leqslant 2 \beta_1  \left \| \left (  \sum_{i=1}^{n+1} 
    \int_{J_{i,n}^{\theta,r}}
  \left ( s_{i,n}^{\theta,r}  - v                       \right )^2 
  \left ( L_{v-} -  L_{s_{i-1,n}^{\theta,r}} \right )^2
  \od [L]_v 
  \right )^\frac{1}{2}
             \right \|_{L_1} \\    
& \leqslant  \frac{2 \beta_1 T}{\theta n}  \left \| \left (  \sum_{i=1}^{n+1} 
    \int_{J_{i,n}^{\theta,r}}
  \left (  L_{v-} -  L_{s_{i-1,n}^{\theta,r}} \right )^2
  \od [L]_v 
  \right )^\frac{1}{2}
             \right \|_{L_1}.
\end{align*}
To prove our lemma it is sufficient to show that
\[ \lim_{n\to \infty}  \int_{[0,1)\times \Omega} 
 \left (  \sum_{i=1}^{n+1} 
    \int_{J_{i,n}^{\theta,r}}
  \left (  L_{v-} -  L_{s_{i-1,n}^{\theta,r}} \right )^2
  \od [L]_v 
  \right )^\frac{1}{2}
              (\rho\otimes \p)(\od (r,\omega)) = 0.\]
This is done by dominated convergence: On the one side we have that 
\[  \left (  \sum_{i=1}^{n+1} 
    \int_{J_{i,n}^{\theta,r}}
  \left (  L_{v-} -  L_{s_{i-1,n}^{\theta,r}} \right )^2
  \od [L]_v 
  \right )^\frac{1}{2}
   \leqslant 2 L_b^*  [L]_b^\frac{1}{2}  \in L_1 \]
where we use $\E L_b^*  [L]_b^\frac{1}{2} \leqslant \|L_b^*\|_{L_2} (\E [L]_b)^\frac{1}{2}<\infty$.
So it suffices to show, for fixed $(r,\omega) \in [0,1)\times \Omega$, that 
\[ \lim_{n\to \infty} \left ( \sum_{i=1}^{n+1} 
    \int_{J_{i,n}^{\theta,r}}
  \left (  L_{v-}(\omega) -  L_{s_{i-1,n}^{\theta,r}}(\omega) \right )^2
  \od [L]_v(\omega) \right )  = 0.\]
This follows from
$\lim_{v\downarrow t} (L_{v-}(\omega) -  L_t(\omega)) = 0$ for $t\in [0,T)$
because of the c\`adl\`ag paths of $L$ 
and dominated convergence on $(0,b]$ with respect to the measure 
$\od [L]_v(\omega)$ for a fixed $\omega\in \Omega$ (note that
$\sup_{v\in (0,b]} \left | L_{v-}(\omega) -  L_{s_{i-1,n}^{\theta,r}}(\omega) \right | \leqslant 2 L_b^*(\omega)
<\infty$).
\end{proof}

\begin{proof}[Proof of \cref{statement:pointwise_convergence_sq_random}]
We use $\rho(\od r)=\od r$,  \cref{statement:elimination_first_term}, and \cref{statement:convergence_vs_nets}. 
\end{proof}

\begin{proof}[Proof of \cref{statement:pointwise_convergence_sq}]
For $\rho=\delta_0$ we get
\begin{align}
& \limsup_{n\to \infty} \left \|  \int_{(0,b)} \left [ \frac{2\theta n}{T}  \int_{[0,1)} \left ((\overline{u}_{\tau_n^{\theta,r}}\wedge b) - u 
              \right )  \rho(\od r)  - 
              \left ( \frac{T-u}{T} \right )^{1-\theta} \right ] \od [L]_u \right \|_{L_1} \notag \\
& = \limsup_{n\to \infty} \E \left | \int_{(0,b)} \left [ \frac{2\theta n}{T}  \left ((\overline{u}_{\tau_n^{\theta}}\wedge b) - u 
              \right )   - 
              \left ( \frac{T-u}{T} \right )^{1-\theta} \right ] D(u,\cdot)  \od u \right |. \label{eqn:remaining_term_D}
\end{align}
Because of inequality \eqref{eqn:1:statement:convergence_vs_nets} we have that
\[ \left | \int_{(0,b)} \left [ \frac{2\theta n}{T}  \left ((\overline{u}_{\tau_n^{\theta}}\wedge b) - u 
              \right )   - 
              \left ( \frac{T-u}{T} \right )^{1-\theta} \right ] D(u,\cdot)  \od u \right |
\leqslant 2 \int_{(0,b)} D(u,\cdot)  \od u 
    =     2 [L]_b\in L_1 \]
and  by \eqref{eqn:3:statement:convergence_vs_nets} for all $\omega\in \Omega$ that 
\[ \lim_{n\to \infty} \left | \int_{(0,b)} \left [ \frac{2\theta n}{T}  \left ((\overline{u}_{\tau_n^{\theta}}\wedge b) - u 
              \right )   - 
              \left ( \frac{T-u}{T} \right )^{1-\theta} \right ] D(u,\omega)  \od u \right | = 0. \]
Consequently, by dominated convergence the term  \eqref{eqn:remaining_term_D} vanishes.
\end{proof}


\section[Riemann-Liouville op's \& approximation]{Riemann-Liouville type operators and approximation}
\label{sec:random_measures}

Various $L_p$-approximation problems in stochastic integration theory can be translated by the 
Burkholder-Davis-Gundy inequalities into problems about quadratic variation processes. 
In the special case of $L_2$-approximations this is particularly useful as there is a chance to turn 
the approximation problem into -in a sense- more deterministic problem by Fubini's theorem 
when the interchange of the integration in time and in $\omega$ is possible. When $p\not=2$ this does 
not work (at least) in this straight way, see for example \cite{Geiss:Toivola:15}. However, passing from 
global $L_2$-estimates to weighted local $L_2$-estimates, i.e. weighted bounded mean oscillation estimates, 
and exploiting a weighted John-Nirenberg type theorem, gives a natural approach to $L_p$- and exponential 
estimates.
\smallskip

The plan of this section is as follows:
\begin{enumerate}[(A)]
\item \cref{statement:upper_bound_randommeasure} and \cref{statement:lower_bounds_randommeasure}
      are the key to exploit these local $L_2$-estimates in our article later.
      It turned out that one can naturally formulate these theorems in the general setting of 
      random measures $(\Pi,\Upsilon)$.  For this one needs a replacement \eqref{eqn:ass:random_measures} of 
      {\em orthogonality}, which will be
      relation \eqref{eqn:ass:random_measures}. 
\item By \cref{ass:random_measures_simplified} we specialize the setting given in \cref{ass:random_measures} so that
      the measure $\Pi$ will describe the quadratic variation of the driving process of the stochastic integral to be approximated 
      and $\Upsilon$ will describe some kind of {\em curvature} of the stochastic integral.
      As results we obtain \cref{statemtent:equivalence_simplified_setting_new} and \cref{equivalence_martingale}.
\item In \cref{statement:intsq_vs_Binfty2} we bring the upper bound of \cref{statemtent:equivalence_simplified_setting_new} into a
      functional analytic form applied in \cref{sec:approximation_random_measure}.
\end{enumerate}
\medskip

So let us start by introducing the random measures and the quasi-orthogonality where we use extended 
conditional expectations for non-negative random variables.

\begin{assumption}
	\label{ass:random_measures}
	We assume random measures 
	\[  \Pi,\Upsilon\colon \Omega \times \cB((0,T))\to [0,\infty], \]
        a progressively measurable process $(\varphi_t)_{t\in [0,T)}$, and a constant $\kappa\geqslant 1$, such that 
	\[ \Pi(\omega,(0,b]) + \Upsilon(\omega, (0,b])+\sup_{t\in [0,b]}|\varphi_t(\omega)| <\infty \]
    for  $(\omega,b) \in \Omega\times (0,T)$
    and such that, for  $0\leqslant s \leqslant a < b < T$,
	\begin{align}\label{eqn:ass:random_measures}
		\ce{\F_a}{\int_{(a,b]} \left|\varphi_u - \varphi_s\right|^2 \Pi(\cdot,\od u)}  
		& \sim_\kappa \ce{\F_a}{ \left|\varphi_a-\varphi_s\right|^2 \Pi (\cdot,(a,b]) 
			+ \int_{(a,b]} (b-u) \Upsilon(\cdot,\od u)} \mbox{ a.s.}
		\end{align}
When \eqref{eqn:ass:random_measures} holds with $\preceq_\kappa$, then we denote the inequality by $\eqref{eqn:ass:random_measures}^\leqslant$,
		in case of $\succeq_\kappa$, by $\eqref{eqn:ass:random_measures}^\geqslant$.
\end{assumption}
\smallskip

To simplify the notation in some situations we extend $\Pi$ and $\Upsilon$ to
$\Pi,\Upsilon\colon \Omega \times \cB((0,T])\to [0,\infty]$
by $\Pi(\omega,\{T\})= \Upsilon(\omega,\{T\})=0$ for all $\omega\in \Omega$.
\smallskip

\begin{defi}
\label{defi:R_process}
For a random measure $\Pi\colon \Omega \times \cB((0,T))\to [0,\infty]$ and
a progressively measurable process $(\varphi_t)_{t\in [0,T)}$ such that 
$\Pi(\omega,(0,b])+\sup_{t\in [0,b]}|\varphi_t(\omega)| <\infty$ for  $(\omega,b) \in \Omega\times (0,T)$
we define for $\tau=\{t_i\}_{i=0}^n \in \mathcal{T}$ the non-negative, non-decreasing, and c\`adl\`ag 
process $\intsqw{\varphi}{\pi}{\tau}=(\intsqw{\varphi}{\pi}{\tau}_a)_{a\in [0,T)}$ by 
	\[ \intsqw{\varphi}{\pi}{\tau}_a  := \int_{(0,a]} \left| \varphi_u - \sum_{i=1}^n \varphi_{t_{i-1}} \1_{(t_{i-1},t_i]}(u) 
	\right|^2 \Pi(\cdot, \od u)\in [0,\infty)
	\]
and let $\intsqw{\varphi}{\pi}{\tau}_T := \lim_{a\uparrow T} \intsqw{\varphi}{\pi}{\tau}_a\in [0,\infty]$.
\end{defi}
\medskip

The next two statements, Theorems \ref{statement:upper_bound_randommeasure} and \ref{statement:lower_bounds_randommeasure},
develop further ideas from \cite[Lemma 3.8]{Geiss:Hujo:07} and \cite[Lemma 5.6]{Geiss:Toivola:15} to a general conditional 
setting using random measures we exploit in the sequel. 
For $\tau=\{t_i\}_{i=0}^n\in \cT$ and  $a\in [t_{k-1},t_k)$ we let 
\begin{equation}\label{eqn:under-overline-a}
   \underline{a}(\tau) := t_{k-1} 
   \sptext{1}{and}{1}
   \overline{a}(\tau) := t_k.
\end{equation}

\begin{theo}[Upper bound]
\label{statement:upper_bound_randommeasure}
Suppose \cref{ass:random_measures} with $\eqref{eqn:ass:random_measures}^\leqslant$. If $(\theta,a)\in (0,1]\times [0,T)$,
$\tau\in \cT$, and $(\un{a},\ov{a}]:=(\un{a}(\tau),\ov{a}(\tau)]$,
then
\[ \frac{\ce{\F_a}{\intsqw{\varphi}{\pi}{\tau}_T-\intsqw{\varphi}{\pi}{\tau}_a}}{\| \tau\|_\theta} 
   \leqslant \kappa  
   \begin{cases}
            \ce{\F_a}{\int_{(a,T)} (T-u)^{1-\theta} \Upsilon(\cdot,\od u) + \frac{(T-\un{a})^{1-\theta}}{\ov{a}-\un{a}}
            |\varphi_a - \varphi_{\un{a}}|^2 \Pi(\cdot,(a,\ov{a}])}   \\
            \ce{\F_a}{\int_{(a,T)} (T-u)^{1-\theta} \Upsilon(\cdot,\od u)} \sptext{2}{if}{1} a \in \tau 
\end{cases} \hspace*{-.7em}\mbox{a.s.} \]
\end{theo}

\begin{theo}[Lower bounds]
\label{statement:lower_bounds_randommeasure}
Suppose \cref{ass:random_measures} with $\eqref{eqn:ass:random_measures}^\geqslant$ and $(\theta,a)\in (0,1]\times [0,T)$. 
\begin{enumerate}[{\rm (1)}]
\item \label{item:1:statement:lower_bounds_randommeasure}
      If $\tau\in \cT$, $(\un{a},\ov{a}]:=(\un{a}(\tau),\ov{a}(\tau)]$, and 
      $\| \tau\|_\theta = \frac{\ov{a}-\un{a}}{(T-\underline{a})^{1-\theta}}$, then
\[ \frac{\ce{\F_a}{\intsqw{\varphi}{\pi}{\tau}_{\ov a} - \intsqw{\varphi}{\pi}{\tau}_a}}{\|\tau \|_\theta}
   \geqslant \frac{1}{\kappa}  \ce{\F_a}{\frac{(T-\un{a})^{1-\theta}}{\ov{a}-\un{a}}
   |\varphi_a - \varphi_{\un{a} }|^2 \Pi(\cdot,(a,\ov{a}])} \mbox{ a.s.}
\]
\item \label{item:2:statement:lower_bounds_randommeasure}
There exist $\tau_n\in\cT$, $n\in\bN$,  with $a\in \tau_n$ and $\lim_n \|\tau_n\|_\theta = 0$ such that
\[ \liminf_n
\frac{\ce{\F_a}{\intsqw{\varphi}{\pi}{\tau_n}_T - \intsqw{\varphi}{\pi}{\tau_n}_a}}{\|\tau_n \|_\theta}
\geqslant \frac{1}{\kappa 2^{\frac{1}{\theta}+2}} \ce{\F_a}{\int_{(a,T)} (T-u)^{1-\theta} \Upsilon(\cdot,\od u)} \mbox{ a.s.}
\]
\end{enumerate} 
\end{theo}
\medskip

\begin{proof}[Proof of \cref{statement:upper_bound_randommeasure}]
To simplify the notation we set $\varphi_T:=0$. It is obvious that we only need to show the first inequality.
For $\tau = \{t_i\}_{i=0}^n\in \mathcal{T}$, $(t_{k-1},t_k]=(\un{a}(\tau),\ov{a}(\tau)]$, and 
$s_i := t_i \vee a$ one has, a.s.,
\begin{align*}
	&\ce{\F_a}{\int_{(a, T]}\left|\varphi_{u} - \sum_{i= 1}^n \varphi_{t_{i-1}}\1_{(t_{i-1}, t_{i}]}(u) \right|^2 \Pi(\cdot,\od u)}  \\
	&= \ce{\F_a}{\int_{ (a, t_k]}\left|\varphi_u- \varphi_{t_{k-1}}\right|^2 \Pi(\cdot,\od u) + \sum_{i=k+1}^n\int_{(t_{i-1}, t_{i}]} \left|\varphi_{u} -
		\varphi_{t_{i-1}}\right|^2 \Pi(\cdot,\od u)}\\
    &\leqslant \kappa \ce{\F_a}{\left|\varphi_{a}- \varphi_{t_{k-1}}\right|^2 \Pi(\cdot,(a,t_k]) + \sum_{i=k}^n\int_{(s_{i-1}, s_{i}]} (s_i-u) \Upsilon (\cdot,\od u)} \\
&\leqslant \kappa \E^{\F_a} \Bigg [ \frac{t_k-t_{k-1}}{(T-t_{k-1})^{1-\theta}}\left|\varphi_{a}- \varphi_{t_{k-1}}\right|^2 \frac{(T-t_{k-1})^{1-\theta}}{t_k-t_{k-1}}\Pi(\cdot,(a,t_k])\\ & \hspace*{15em}+ \sum_{i=k}^n\int_{(s_{i-1}, s_{i}]} 
    \frac{s_i-u}{(T-u)^{1-\theta}} (T-u)^{1-\theta} \Upsilon (\cdot,\od u)\Bigg ] \\
&\leqslant \kappa \|\tau\|_\theta \ce{\F_a}{\left|\varphi_{a}- \varphi_{t_{k-1}}\right|^2 \frac{(T-t_{k-1})^{1-\theta}}{t_k-t_{k-1}}\Pi(\cdot,(a,t_k]) + \int_{(a,T)} 
  (T-u)^{1-\theta} \Upsilon (\cdot,\od u)}
\end{align*}
where we use \eqref{eq:adapted_net_monotnicity}.
\end{proof}

\begin{proof}[Proof of \cref{statement:lower_bounds_randommeasure}]
\eqref{item:1:statement:lower_bounds_randommeasure}
Beginning the proof as for \cref{statement:upper_bound_randommeasure}
with $(t_{k-1},t_k]=(\un{a}(\tau),\ov{a}(\tau)]$, we get, a.s.,
\begin{align*}
           \ce{\F_a}{\int_{(a, t_k]}\left|\varphi_{u} - \sum_{i= 1}^n \varphi_{t_{i-1}}\1_{(t_{i-1}, t_{i}]}(u) \right|^2 \Pi(\cdot,\od u)} 
& =     \ce{\F_a}{\int_{ (a, t_k]}\left|\varphi_u- \varphi_{t_{k-1}}\right|^2 \Pi(\cdot,\od u)}\\
&\geqslant \frac{1}{\kappa}  \ce{\F_a}{\left|\varphi_{a}- \varphi_{t_{k-1}}\right|^2 \Pi(\cdot,(a, t_k])}.
\end{align*}
Dividing by
$\| \tau\|_\theta = \frac{t_k-t_{k-1}}{(T-t_{k-1})^{1-\theta}}$ we obtain the desired statement.
\smallskip

\eqref{item:2:statement:lower_bounds_randommeasure}
We partition the interval $[a,T]$ with
\begin{align*}	
 u_{i,n}^{\theta,a} & := a + (T-a) \left[ 1 - \( 1 - \tfrac{i}{n} \)^\frac{1}{\theta} \right],  \quad i=0, \ldots,n, \\
 r_{i,n}^{\theta,a} & := a + (T-a) \left[ 1 - \( 1 - \tfrac{2i-1}{2n} \)^\frac{1}{\theta} \right], \quad i=1, \ldots,n,
\end{align*} 
and add  $r_{0,n}^{\theta,a}: = a$ and $r_{n+1,n}^{\theta,a}: =T$. Choosing for both nets the remaining time-knots on $[0,a]$ fine enough,
we obtain nets  $\tau_n^{\theta}(a)$ and $\wt \tau_n^{\theta}(a)$ satisfying
\[ \| \tau_n^{\theta}(a)\|_\theta = \sup_{i=1,\ldots,n} \frac{u_{i,n}^{\theta,a} - u_{i-1,n}^{\theta,a}}{(T-u_{i-1,n}^{\theta,a})^{1-\theta}}
   \sptext{1}{and}{1}
    \|\wt \tau_n^{\theta}(a)\|_\theta 
	= \sup_{i=0,1,\ldots,n} \frac{r_{i+1,n}^{\theta,a} - r_{i,n}^{\theta,a}}{(T-r_{i,n}^{\theta,a})^{1-\theta}}.\]
By a computation, we have for $i=1, \ldots, n$ and $u\in (u_{i-1,n}^{\theta,a}, r_{i,n}^{\theta,a}]$ that
	\begin{align}\label{eq:thm-time-net-ineq1}
	\frac{(T-a)^\theta}{\theta 2^{\frac{1}{\theta}+1} n} 
	\leqslant \frac{u_{i,n}^{\theta,a} - r_{i,n}^{\theta,a}}{(T-r_{i,n}^{\theta,a})^{1-\theta}}
	\leqslant \frac{u_{i,n}^{\theta,a} - u}{(T-u)^{1-\theta}}
	\leqslant \frac{u_{i,n}^{\theta,a} - u_{i-1,n}^{\theta,a}}{(T-u_{i-1,n}^{\theta,a})^{1-\theta}} 
	\leqslant \frac{(T-a)^\theta}{\theta n},
	\end{align}
and for $i= 1,\ldots,n-1$ and $u\in (r_{i,n}^{\theta,a}, u_{i,n}^{\theta,a}]$ that
	\begin{align}\label{eq:thm-time-net-ineq2}
	\frac{(T-a)^\theta}{\theta 2^{\frac{1}{\theta}+1} n} 
	\leqslant \frac{r_{i+1,n}^{\theta,a} - u_{i,n}^{\theta,a}}{(T-u_{i,n}^{\theta,a})^{1-\theta}}
	\leqslant \frac{r_{i+1,n}^{\theta,a} - u}{(T-u)^{1-\theta}}
	\leqslant \frac{r_{i+1,n}^{\theta,a} - r_{i,n}^{\theta,a}}{(T-r_{i,n}^{\theta,a})^{1-\theta}} 
	\leqslant \frac{(T-a)^\theta}{\theta n},
	\end{align}
where the last inequality holds for $i\in \{0,n\}$ as well.
By the above relations we obtain, a.s.,
\begin{align*}
& \ce{\F_a}{\int_{(a,r_{n,n}^{\theta,a}]}  (T-u)^{1-\theta} \Upsilon(\cdot,\od u)} \\
& =
	\sum_{i=1}^n \ce{\F_a}{\int_{(u_{i-1,n}^{\theta,a},\, r_{i,n}^{\theta,a}]} (T-u)^{1-\theta}
		\Upsilon(\cdot,\od u)} 
 + \sum_{i=1}^{n-1} \ce{\F_a}{\int_{(r_{i,n}^{\theta,a},\, u_{i,n}^{\theta,a}]} (T-u)^{1-\theta}
		\Upsilon(\cdot,\od u)}\\
& \leqslant \frac{\theta 2^{\frac{1}{\theta}+1} n}{(T-a)^{\theta}} \Bigg [
	\sum_{i=1}^n \ce{\F_a}{\int_{(u_{i-1,n}^{\theta,a},\, r_{i,n}^{\theta,a}]} (u_{i,n}^{\theta,a} -u)
		\Upsilon(\cdot,\od u)} 
+ \sum_{i=1}^{n-1} \ce{\F_a}{\int_{(r_{i,n}^{\theta,a},\, u_{i,n}^{\theta,a}]} (r_{i+1,n}^{\theta,a}-u) \Upsilon(\cdot,\od u)} \Bigg ] \\
&\leqslant (\kappa 2^{\frac{1}{\theta}+1})  \ce{\F_a}{
             \frac{\intsqw{\varphi}{\pi}{    \tau_n^{\theta}(a)}_T -  \intsqw{\varphi}{\pi}{\tau_n^{\theta}(a)}_a}
             {\|\tau_n^{\theta}(a)\|_\theta} 
           + \frac{\intsqw{\varphi}{\pi}{\wt \tau_n^{\theta}(a)}_T -  \intsqw{\varphi}{\pi}{\wt\tau_n^{\theta}(a)}_a}
                  {\|\wt \tau_n^{\theta}(a)\|_\theta}},
\end{align*}
where for the last inequality we first use  \eqref{eqn:ass:random_measures}, that gives the factor $\kappa$, and then
\eqref{eq:thm-time-net-ineq1} and \eqref{eq:thm-time-net-ineq2} that give 
$\|\tau_n^{\theta}(a)\|_\theta \leqslant (T-a)^\theta/(\theta n)$ and 
$\|\wt \tau_n^{\theta}(a) \|_\theta \leqslant (T-a)^\theta/(\theta n)$.
For each $n$ we choose the time-net that gives the larger quotient and obtain the desired nets.
To obtain the final statement we observe that $r_{n,n}^{\theta,a}\uparrow T$.
\end{proof}
\smallskip

Now we specialize Assumption \ref{ass:random_measures} to the settings that will be used in Sections \ref{sec:application_brownian_case} and \ref{sec:application_levy_case}:

\begin{assumption}
\label{ass:random_measures_simplified}
We assume that there are
\begin{enumerate}
\item \label{item:1:simplified_random_measures}
      a positive, c\`adl\`ag, and adapted process $(\sigma_t)_{t\in [0,T]}$ such that $\sigma_T^*\in L_2$
      and such that there is a $c_\sigma\geqslant 1$ with
      \[ \ce{\F_a}{\frac{1}{b-a}\int_a^b \sigma_u^2 \od u} \sim_{c_\sigma} \sigma_a^2 \mbox{ a.s.}
         \sptext{1}{for all}{1} 0 \leqslant a < b \leqslant T, \] 
\item \label{item:2:simplified_random_measures}
      a c\`adl\`ag square integrable martingale $M=(M_t)_{t\in [0,T)}$ with $M_0\equiv 0$,
\item let $\Pi(\omega,\od u) := \sigma^2_u(\omega) \od u$  and $\Upsilon(\omega,\od u):= \od \langle  M \rangle_u (\omega)$
      for $u\in [0,T)$, where  $\langle  M \rangle$ is the conditional square-function (see \cref{subsec:stochastic_basis}), 
\item \label{item:3:simplified_random_measures}
      a $\varphi\in \CL([0,T))$, 
\item assume that \cref{eqn:ass:random_measures} is satisfied, and 
\item let $\intsqw{\varphi}{\sigma}{\tau} := \intsqw{\varphi}{\pi}{\tau}$.
\end{enumerate}
\end{assumption}

\begin{rema}
\cref{ass:random_measures_simplified}, the equality
\[ \ce{\F_a}{\int_{(a,b]} |M_u-M_a|^2 \od u}
   = \ce{\F_a}{\int_{(a,b]}(b-u) \od \langle M \rangle_u} \mbox{ a.s.} \]
for $0\leqslant a < b <T$ yields, for $0\leqslant s \leqslant a < b <T$ and with $\kappa' := \kappa c_\sigma$, to
\begin{align}\label{eq:ass:random_measures_simplified}
\ce{\F_a}{\int_a^b \left|\varphi_u - \varphi_s\right|^2 \sigma_u^2 \od u}
& \sim_{\kappa'} (b-a)\left|\varphi_a-\varphi_s\right|^2  \sigma_a^2 
+  \ce{\F_a}{\int_a^b |M_u-M_a|^2 \od u} \mbox{ a.s.}
\end{align}
\end{rema}
\bigskip

Now we transfer \cref{statement:upper_bound_randommeasure} and \cref{statement:lower_bounds_randommeasure}
into the setting of \cref{ass:random_measures_simplified}, where we remind the reader on the notation 
$\underline{a}(\tau)$ and $\overline{a}(\tau)$ introduced in \cref{eqn:under-overline-a}.
\bigskip

\begin{theo}
\label{statemtent:equivalence_simplified_setting_new}
Suppose $(\theta,a)\in (0,1]\times [0,T)$ and that \cref{ass:random_measures_simplified} holds, and 
define
$c_\eqref{eq:item:1:statemtent:equivalence_simplified_setting_new} := \kappa (4 T^{1-\theta} \vee c_\sigma)$,
$c_\eqref{eq:2:item:2:statemtent:equivalence_simplified_setting_new} := \kappa c_\sigma$, and
$c_\eqref{eq:1:item:2:statemtent:equivalence_simplified_setting_new} := T^{\theta-1}16 \kappa 2^{\frac{1}{\theta}}$.
Then one has for all $\tau\in \cT$, a.s., 
      \begin{equation}\label{eq:item:1:statemtent:equivalence_simplified_setting_new}
       \frac{\ce{\F_a}{\intsqw{\varphi}{\sigma}{\tau}_T - \intsqw{\varphi}{\sigma}{\tau}_a}}{\| \tau\|_\theta}
         \leqslant  c_\eqref{eq:item:1:statemtent:equivalence_simplified_setting_new}  
         \left ( \ce{\F_a}{\sup_{t\in [a,T)} \left |\cI_t^{\frac{1-\theta}{2}} M - \cI_a^{\frac{1-\theta}{2}} M \right |^2}
         + \frac{T-a}{(T-\underline{a}(\tau))^\theta} |\varphi_a -\varphi_{\underline{a}(\tau)}|^2 \sigma_a^2 \right ).
      \end{equation}
Moreover, there are $(\tau_n)_{n\in \bN} \subset \cT$ with $a\in \tau_n$ and $\lim_n \| \tau_n\|_\theta =0$
and for all $s\in [0,a]$ there is a net $\tau \in \cT$ with $s=\underline{a}(\tau)$, such that, for all $\sigma$-algebras $\cG\subseteq \F_a$,
one has
      \begin{align}
          \frac{\ce{\F_a}{\intsqw{\varphi}{\sigma}{\tau}_T - \intsqw{\varphi}{\sigma}{\tau}_a}}{\| \tau\|_\theta}
         &\geqslant \frac{1}{c_\eqref{eq:2:item:2:statemtent:equivalence_simplified_setting_new}} \,\,
          \frac{T-a}{(T-\underline{a}(\tau))^\theta} |\varphi_a -\varphi_{\underline{a}(\tau)}|^2 \sigma_a^2 \mbox{ a.s.},
           \label{eq:2:item:2:statemtent:equivalence_simplified_setting_new} \\
         \liminf_n \frac{\ce{\cG}{\intsqw{\varphi}{\sigma}{\tau_n}_T - \intsqw{\varphi}{\sigma}{\tau_n}_a}}{\| \tau\|_\theta}
         & \geqslant \frac{1}{c_\eqref{eq:1:item:2:statemtent:equivalence_simplified_setting_new} }\,\,
           \ce{\cG}{\sup_{t\in [a,T)} \left |\cI_t^{\frac{1-\theta}{2}} M - \cI_a^{\frac{1-\theta}{2}} M \right |^2} \mbox{ a.s.}
           \label{eq:1:item:2:statemtent:equivalence_simplified_setting_new}
       \end{align}
\end{theo}
\medskip
\begin{proof}
For the following we let  $(\un{a},\ov{a}]:=(\un{a}(\tau),\ov{a}(\tau)]$ for $\tau\in\cT$. \smallskip

Inequality \eqref{eq:item:1:statemtent:equivalence_simplified_setting_new}:
For $0\leqslant \un{a} \leqslant a < \ov{a} \leqslant T$, \cref{ass:random_measures_simplified} implies that 
\[ \ce{\F_a}{\frac{(T-\un{a})^{1-\theta}}{\ov{a}-\un{a}}|\varphi_a - \varphi_{\un{a}}|^2 \Pi(\cdot,(a,\ov{a}])}
   \sim_{c_\sigma} |\varphi_a - \varphi_{\un{a}}|^2 \frac{(T-\un{a})^{1-\theta}}{\ov{a}-\un{a}} \sigma_a^2
   (\ov{a}-a) \mbox{ a.s.} \]
Maximizing the right-hand side over $\ov{a}\in (a,T]$ yields 
$\frac{T-a}{(T-\un{a})^\theta} |\varphi_a - \varphi_{\un{a}}|^2 \sigma_a^2
   \mbox{ a.s.}$
Moreover, we have
\begin{align*}
     \ce{\F_a}{\int_{(a,T)} \left ( \frac{T-u}{T} \right )^{1-\theta} \Upsilon(\cdot,\od u)}
& =  \ce{\F_a}{\int_{(a,T)} \left ( \frac{T-u}{T} \right )^{1-\theta} \od \langle M \rangle_u} \\
& =  \ce{\F_a}{\int_{(a,T)} \left ( \frac{T-u}{T} \right )^{1-\theta} \od [M]_u} \\
&   \sim_4  
   \ce{\F_a}{\sup_{t\in [a,T)} \left | \cI_t^{\frac{1-\theta}{2}}M - \cI_a^\frac{1-\theta}{2} M\right |^2} 
\mbox{ a.s.}
\end{align*}
by Doob's maximal inequality (and \cref{statement:RL_for_martinagles-II}\eqref{eqn:1:statement:RL_for_martinagles}
for $\theta \in (0,1)$).  \cref{statement:upper_bound_randommeasure} implies, a.s.,
\begin{align*}
 \frac{\ce{\F_a}{\intsqw{\varphi}{\sigma}{\tau}_T - \intsqw{\varphi}{\sigma}{\tau}_a}}{\| \tau\|_\theta}
&\leqslant  \kappa  
            \ce{\F_a}{\int_{(a,T)} (T-u)^{1-\theta} \Upsilon(\cdot,\od u) + \frac{(T-\un{a})^{1-\theta}}{\ov{a}-\un{a}}
            |\varphi_a - \varphi_{\un{a}}|^2 \Pi(\cdot,(a,\ov{a}])} \\
&\leqslant  \kappa \left ( 4 T^{1-\theta}
         \ce{\F_a}{\sup_{t\in [a,T)} \left |\cI_t^{\frac{1-\theta}{2}} M - \cI_a^{\frac{1-\theta}{2}} M \right |^2}
         +  c_\sigma \frac{T-a}{(T-\underline{a})^\theta} |\varphi_a -\varphi_{\underline{a}}|^2 \sigma_a^2 \right ).
\end{align*}
Inequality \eqref{eq:2:item:2:statemtent:equivalence_simplified_setting_new}:
We choose a net $\tau\in \cT$ such that $(s,T]=(\un{a},\ov{a}]$ and 
      $\| \tau\|_\theta = \frac{\ov{a}-\un{a}}{(T-\underline{a})^{1-\theta}}$, so that, by
\cref{statement:lower_bounds_randommeasure}\eqref{item:1:statement:lower_bounds_randommeasure},
\[ \frac{\ce{\F_a}{\intsqw{\varphi}{\sigma}{\tau}_{\ov a} - \intsqw{\varphi}{\sigma}{\tau}_a}}{\|\tau \|_\theta}
   \geqslant \frac{1}{\kappa c_\sigma}  \frac{(T-\un{a})^{1-\theta}}{\ov{a}-\un{a}}
   |\varphi_a - \varphi_{\un{a} }|^2 (\ov{a}-a) \sigma_a^2 
   =  \frac{1}{\kappa c_\sigma}  \frac{T-a}{(T-\underline{a})^\theta}|\varphi_a - \varphi_{\un{a} }|^2 \sigma_a^2 
   \mbox{ a.s.}
\]
Inequality \eqref{eq:1:item:2:statemtent:equivalence_simplified_setting_new} follows
for $\cG=\F_a$ from  \cref{statement:lower_bounds_randommeasure} 
where we also use the previous computations on the constant.
In the case of $\cG \subsetneqq \F_a$ we get by Fatou's lemma that, a.s.,
\begin{align*}
  \ce{\cG}{\sup_{t\in [a,T)} \left |\cI_t^{\frac{1-\theta}{2}} M - \cI_a^{\frac{1-\theta}{2}} M \right |^2} 
& \leqslant c_\eqref{eq:1:item:2:statemtent:equivalence_simplified_setting_new} \E^\cG \left [ \liminf_n \ce{\F_a}{
            \frac{\intsqw{\varphi}{\sigma}{\tau_n}_T - \intsqw{\varphi}{\sigma}{\tau_n}_a}{\| \tau_n\|_\theta}} \right ] \\
& \leqslant c_\eqref{eq:1:item:2:statemtent:equivalence_simplified_setting_new} \liminf_n  \E^\cG \left [\ce{\F_a}{
            \frac{\intsqw{\varphi}{\sigma}{\tau_n}_T - \intsqw{\varphi}{\sigma}{\tau_n}_a}{\| \tau_n\|_\theta}} \right ] \\
& = c_\eqref{eq:1:item:2:statemtent:equivalence_simplified_setting_new} \liminf_n  \E^\cG \left [
            \frac{\intsqw{\varphi}{\sigma}{\tau_n}_T - \intsqw{\varphi}{\sigma}{\tau_n}_a}{\| \tau_n\|_\theta} \right ].
    \qedhere
\end{align*}
\end{proof}

For $a=0$, $\cG:= \{ \emptyset,\Omega\}$, and $c_\eqref{eq:item:1:statemtent:equivalence_simplified_setting_new_global}:= 16 \kappa 2^{\frac{1}{\theta}}$
we may deduce from \cref{eq:item:1:statemtent:equivalence_simplified_setting_new}, where only the first term is needed, 
and \cref{eq:1:item:2:statemtent:equivalence_simplified_setting_new} that 
\begin{equation}\label{eq:item:1:statemtent:equivalence_simplified_setting_new_global}
\sup_{\tau \in \cT}
\frac{\|\intsqw{\varphi}{\sigma}{\tau}_T \|_{L_2}}{\| \tau\|_\theta}
         \sim_{c_\eqref{eq:item:1:statemtent:equivalence_simplified_setting_new_global}} 
         T^{1-\theta} \left \|\sup_{t\in [0,T)} \left |\cI_t^{\frac{1-\theta}{2}} M  \right | \right \|_{L_2}^2.
\end{equation}

First we deduce from \cref{statemtent:equivalence_simplified_setting_new} a complete characterization of
$\|\intsqw {\varphi}{\sigma}{\tau}\|_{\BMO_1^{\Phi^2}([0,T))} \leqslant c^2 \|\tau\|_\theta$:

\begin{coro}\label{equivalence_martingale} 
Assume that \cref{ass:random_measures_simplified} is satisfied.
Then for $\theta \in (0,1]$ and $\Phi \in \CL^+([0,T))$
the following assertions are equivalent:
\begin{enumerate}[{\rm (1)}]
\item \label{item:2:equivalence_martingale} 
       One has $\cI^{\frac{1-\theta}{2}} M \in \bmo_2^{\Phi}([0,T))$ and there 
       is a $c_{\eqref{eq:item:2:equivalence_martingale}} >0$ such that one has
 \begin{align}\label{eq:item:2:equivalence_martingale}
     |\varphi_a -\varphi_s| \sigma_a \leqslant c_{\eqref{eq:item:2:equivalence_martingale}}
                 \frac{(T -s)^\frac{\theta}{2}}
                      {(T -a)^\frac{1}{2}} \Phi_a
        \sptext{1}{for}{1}
        0\leqslant   s < a < T \mbox{ a.s.} 
         \end{align}

\item \label{item:1:equivalence_martingale} 
      There is a constant $c_{\eqref{eq:item:1:equivalence_martingale} } >0$ such that, 
      for all time-nets $\tau \in \mathcal T$,
	  \begin{align}\label{eq:item:1:equivalence_martingale}
	\|\intsqw {\varphi}{\sigma}{\tau}\|_{\BMO_1^{\Phi^2}([0,T))} \leqslant c^2_{\eqref{eq:item:1:equivalence_martingale}} \|\tau\|_\theta.
	\end{align}
\end{enumerate}
\medskip
If $\Phi=(\sigma_t \Psi_t)_{t\in [0,T)}$, where $\Psi\in \CL^+([0,T))$ is non-decreasing, then \eqref{eq:item:2:equivalence_martingale}
is equivalent to the existence of $c_{\eqref{eq:equivalence_local_part_theta(0,1)}}, c_{\eqref{eq:equivalence_local_part_theta1}}>0$ such that
\begin{align}
           |\varphi_a -\varphi_0| 
&\leqslant c_{\eqref{eq:equivalence_local_part_theta(0,1)}}(T-a)^{\frac{\theta-1}{2}} \Psi_a \sptext{4}{for}{1}  0 \leqslant a < T \mbox{ a.s.}
            \sptext{3}{if}{1} \theta\in (0,1), \label{eq:equivalence_local_part_theta(0,1)} \\
           |\varphi_a -\varphi_s| 
&\leqslant c_{\eqref{eq:equivalence_local_part_theta1}}
       		\left(1+ \ln \frac{T -s}
                      {T -a}\right)
                         \Psi_a
        \sptext{1.2}{for}{1}
        0 \leqslant  s < a < T \mbox{ a.s.}  \sptext{1.1}{if}{1}  \theta = 1. \label{eq:equivalence_local_part_theta1}
\end{align}
\end{coro}
\smallskip

\begin{proof}  
The equivalence between 
\eqref{item:2:equivalence_martingale} and \eqref{item:1:equivalence_martingale} follows directly 
from \cref{statemtent:equivalence_simplified_setting_new}
and \cref{statement:bmo_determinstic}. The equivalence between 
\eqref{eq:item:2:equivalence_martingale} and
\eqref{eq:equivalence_local_part_theta(0,1)}-\eqref{eq:equivalence_local_part_theta1} follows from 
\cref{lemm:upper-bound-intergrand-equivalence} below.
\end{proof}

Finally, we derive a functional analytic variant of 
\eqref{eq:item:1:statemtent:equivalence_simplified_setting_new} from \cref{statemtent:equivalence_simplified_setting_new}: 
\smallskip

\begin{coro}
\label{statement:intsq_vs_Binfty2}
For $\theta \in (0,1)$, $\alpha:= \frac{1-\theta}{2}$,
and a c\`adl\`ag martingale $\varphi=(\varphi_t)_{t\in [0,T)}$ 
with $\| \varphi \|_{\B_{\infty,2}^\alpha}<\infty$ it holds
\begin{equation}\label{eqn:statement:intsq_vs_Binfty2}
	\left \| \intsq{\varphi}{\tau} \right \|_{\BMO_1([0,T))} 
	\leqslant c_\eqref{eqn:statement:intsq_vs_Binfty2}
	\| \tau\|_\theta \,\, \| \varphi -  \varphi_0 \|_{\B_{\infty,2}^\alpha}^2
\end{equation}
where $c_\eqref{eqn:statement:intsq_vs_Binfty2}>0$ depends at most on $(\theta,T)$.
\end{coro}

\begin{proof}
We apply \cref{statemtent:equivalence_simplified_setting_new} with $\sigma \equiv 1$, $c_\sigma=1$,
$M_t=\varphi_t-\varphi_0$, and get from inequality \eqref{eq:item:1:statemtent:equivalence_simplified_setting_new} that
\begin{align}
	\left \| \intsq{\varphi}{\tau} \right \|_{\BMO_1([0,T))} 
	& = \left \| \intsq{\varphi}{\tau} \right \|_{\bmo_1([0,T))}  \notag \\
	&\le c_{\eqref{eq:item:1:statemtent:equivalence_simplified_setting_new}} \| \tau\|_\theta
	\left [ 4 \| \cI^{\frac{1-\theta}{2}} \varphi -  \varphi_0 \|_{\bmo_2([0,T))}^2
	+ \sup_{0\leqslant s \leqslant a < T} \frac{T-a}{(T-s)^\theta} \| \varphi_a -  \varphi_s \|_{L_\infty}^2  \right ]
        \label{eqn:proof:statement:intsq_vs_Binfty2}
\end{align}
where we used Doob's maximal inequality for the first term. For the second term we continue with
\begin{align*}
\sup_{0\leqslant s \leqslant a < T} \sqrt{\frac{T-a}{(T-s)^\theta}} \| \varphi_a -  \varphi_s \|_{L_\infty}
& \leqslant \sup_{0\leqslant s \leqslant a < T} \sqrt{\frac{T-a}{(T-s)^\theta}}  
\left ( \| \varphi_a -  \varphi_0 \|_{L_\infty} + \| \varphi_0 -  \varphi_s \|_{L_\infty} \right ) \\
& \leqslant 2 \sup_{0\leqslant s \leqslant a < T} \sqrt{\frac{T-a}{(T-s)^\theta}}  
\| \varphi_a -  \varphi_0 \|_{L_\infty} \\
&  =        2  \sup_{0\leqslant a < T} (T-a)^\frac{1-\theta}{2} \| \varphi_a -  \varphi_0 \|_{L_\infty} \\
&  =        2 \| \varphi -  \varphi_0 \|_{\B_{\infty,\infty}^{\alpha}} \\
& \leqslant 2 \sqrt{2 \alpha}  \| \varphi -  \varphi_0 \|_{\B_{\infty,2}^{\alpha}},
\end{align*}
where we used inequality \eqref{eqn:relations_between_Bpqalpha}. If we use  
\cref{statement:relation_BMO_bmo}\eqref{item:2:relation-bmo} and
\cref{statement:properties_Besov-spaces}\eqref{item:2:statement:properties_Besov-spaces} 
for the first term in \cref{eqn:proof:statement:intsq_vs_Binfty2}, then we may conclude
\begin{align*}
\left \| \intsq{\varphi}{\tau} \right \|_{\BMO_1([0,T))} 
& \leqslant 8 \alpha c_{\eqref{eq:item:1:statemtent:equivalence_simplified_setting_new}}
\left [ \frac{9}{T^{2\alpha}} + 1 \right ] \|\tau\|_\theta 
\| \varphi -  \varphi_0 \|_{\B_{\infty,2}^\alpha}^2. \qedhere
\end{align*}
\end{proof}


\section[Oscillation]{Oscillation of stochastic processes and lower bounds}
\label{sec:oscillation_general}

In this section we consider lower bounds for the oscillation of stochastic processes
and use them in  \cref{sec:application_brownian_case} (Case (C1)) and \cref{sec:application_levy_case}.
As such, the approach is intended for stochastic processes 
$(\varphi_t)_{t\in [0,T)}\subseteq L_\infty$ with a blow-up 
of $\|\varphi_t \|_{L_\infty}$ if $t\uparrow T$. This is a typical case for the gradient processes we consider.
The quantities, we are interested in, concern the degree of the oscillation of the process measured in $L_\infty$,  here denoted 
by $\losc_t(\varphi)$ and  $\uosc_t(\varphi)$. In order to get lower bounds for these oscillatory quantities, 
we use the concept of {\it maximal oscillation}. The above mentioned concepts are introduced
in \cref{definition:oscillation} below. The maximal oscillation is verified in 
\cref{exam:oscillation_Markov_process} and \cref{exam:oscillation_Levy_process} below. 
The application to $\intsq{\varphi}{\tau}$  is given in \cref{thm:general_lower_bound}.
\cref{exam:oscillation_Markov_process}  and \cref{thm:general_lower_bound}
will be used in \cref{sec:application_brownian_case}, and
\cref{exam:oscillation_Levy_process}  and \cref{thm:general_lower_bound}
will be used in the L\'evy case in \cref{sec:application_levy_case}. 
Let us start to introduce our concept:
\smallskip

\begin{defi}
\label{definition:oscillation}
If $\varphi=(\varphi_t)_{t\in [0,T)}$ is a stochastic process and $t\in (0,T)$, then we let
\[ \losc_t(\varphi) := \inf_{s \in [0,t)} \| \varphi_t  -\varphi_s \|_{L_\infty} \in [0,\infty]
   \sptext{1}{and}{1}
   \uosc_t(\varphi) := \inf_{s\in [0,t)} \sup_{u\in [s,t]}  \| \varphi_t - \varphi_u \|_{L_\infty} \in [0,\infty]. \]
The process is called of {\it maximal oscillation} with constant $c\geqslant 1$ if for all $t\in (0,T)$ one has
\[ \losc_t(\varphi) \geqslant \frac{1}{c} \| \varphi_t - \varphi_0 \|_{L_\infty}.\]
If both sides equal infinity, then we use $c=1$ (however, this case is not of relevance for us).
\end{defi}

\begin{lemm}
\label{statement:relations_oscillation}
For a stochastic process $\varphi=(\varphi_t)_{t\in [0,T)}$ the following holds:
\begin{enumerate}[{\rm (1)}]
\item \label{item:1:statement:relations_oscillation}
      One has $\losc_t(\varphi) \leqslant \uosc_t(\varphi)$ for $t\in (0,T)$.
\item \label{item:2:statement:relations_oscillation}
      One has $\uosc_t(\varphi) \leqslant 2 \losc_t(\varphi)$ for $t\in (0,T)$ if $\varphi$ is a martingale.
\item \label{item:3:statement:relations_oscillation}
      If $\varphi_a \equiv \1_{\Q\cap [0,T)}(a)$ for $a\in [0,T)$, then
      $0=\losc_t(\varphi) <\uosc_t(\varphi)=1$ for all $t\in (0,T)$.
\end{enumerate}
\end{lemm}
\medskip

\begin{proof}
\eqref{item:1:statement:relations_oscillation} follows from the definition. 
\eqref{item:2:statement:relations_oscillation} If $\varphi$ is a martingale and 
$0 \leqslant s < t<T$, then we have
\[           \sup_{u\in [s,t]}  \| \varphi_t - \varphi_u \|_{L_\infty}
   \leqslant \| \varphi_t - \varphi_s \|_{L_\infty} + \sup_{u\in [s,t]} \| \varphi_u - \varphi_s \|_{L_\infty}
   \leqslant 2 \| \varphi_t - \varphi_s \|_{L_\infty}. 
\]
Taking the infimum on both sides over $s\in [0,t)$ yields the assertion.
Item \eqref{item:3:statement:relations_oscillation} is obvious.
\end{proof}

\begin{rema}
In the sequel we do not need the following two statements,
so that we state them without proof:
\begin{enumerate}
\item It is possible to construct examples such that for a given $c\in [1,\infty)$ the constant $c$ is optimal in the definition of
      maximal oscillation.
\item Again by examples one can see that the constant $2$ in \cref{statement:relations_oscillation}\eqref{item:2:statement:relations_oscillation}
      is optimal.
\end{enumerate}
\end{rema}
 
To verify a  maximal oscillation we make use of the following observation:

\begin{lemm}
\label{statement:basic_lemma_oscillation}
Assume two random variables $A,B:\Omega\to \R$ on $(\Omega,\F,\p)$.
Assume a probability measure 
$\Q$, that is absolutely continuous with respect to $\p$ {\rm (}$\Q \ll\p${\rm )}, such that
$\E^{\Q} |B|<\infty$ and $\E^{\Q} B =0$. Then 
\[ \| B - A \|_{L_\infty(\p)}\geqslant \inf_{a\in \R} \| B - a \|_{L_\infty(\p)}
   \sptext{1}{implies}{1}
   \| B - A \|_{L_\infty(\p)} \geqslant \frac{1}{2} \| B \|_{L_\infty(\p)}.\]
\end{lemm}
\begin{proof}
We may assume that $ \| B - A \|_{L_\infty(\p)}<\infty$, otherwise there is nothing to prove.
Because of our assumption, for all $\varepsilon >0$ there is an $a_\varepsilon\in \R$ such that we have 
\[ \| B - A \|_{L_\infty(\p)} \geqslant \| B \|_{L_\infty(\p)} - |a_\varepsilon| - \varepsilon
   \sptext{.7}{and}{.7}
   \| B - A \|_{L_\infty(\p)} \geqslant \E^\Q | B - a_{\varepsilon} | - \varepsilon
                      \geqslant | \E^\Q B - a_{\varepsilon} | - \varepsilon
                           =    |a_{\varepsilon} | - \varepsilon. \]
The combination of the inequalities implies
\[ \| B - A \|_{L_\infty(\p)} \geqslant \| B \|_{L_\infty(\p)} - |a_\varepsilon| - \varepsilon
                      \geqslant \| B \|_{L_\infty(\p)} - \| B - A \|_{L_\infty(\p)}  - 2\varepsilon \]
so that 
$2 \| B - A \|_{L_\infty(\p)} \geqslant \| B \|_{L_\infty(\p)} - 2\varepsilon$.
By $\varepsilon \downarrow 0$ we get our statement. 
\end{proof}
\bigskip

Now we consider two examples relevant for us:

\begin{exam}[Markov type processes, \cref{sec:application_brownian_case}]
\label{exam:oscillation_Markov_process}
Let $(Y_t)_{t\in [0,T]}$ be a process with values in $\defg$, where
$\defg =\R$ or $\defg = (0,\infty)$, and $Y_0\equiv y_0 \in \defg$.
Assume transition densities
$\Gamma_Y: \{ (s,t): 0 \leqslant s < t \leqslant T \} \times \defg \times \defg \to (0,\infty)$,
jointly continuous in the space variables for fixed $0 \leqslant s < t \leqslant T$, such that
\begin{equation}\label{eqn:derived_from_Markov_SDE}
 \p(Y_t\in B|Y_s) = \int_B \Gamma_Y(s,t;Y_s,y)\od y \mbox{ a.s.} 
\end{equation}
for $B\in \cB(\defg)$ and $0\leqslant s < t \leqslant T$. 
Then, for $0<s<t \leqslant T$ and continuous $H,\tilde H:\defg \to \R$, one has
\[ \| H(Y_t) - \tilde H(Y_s)  \|_{L_\infty} \geqslant \| H(Y_t) - \tilde H(y_0) \|_{L_\infty}.\]
This follows from the fact that the density $D_{s,t}:\defg\times \defg \to [0,\infty)$ of 
$law(Y_s,Y_t)$ with respect to the Lebesgue measure $\lambda\otimes \lambda|_{\defg\times\defg}$
is the positive and continuous function
\[ D_{s,t}(y_1,y_2):=\Gamma_Y(0,s;y_0,y_1)\Gamma_Y(s,t;y_1,y_2). \]
Consequently, if there is a probability measure $\Q \ll\p$ and
if for all $t\in [0,T)$ one has that $H(t,\cdot):\defg\to \R$ is continuous, $\E^\Q |H(t,Y_t)|<\infty$, and
$\E^\Q (H(t,Y_t)-H(0,y_0))=0$, then $(H(t,Y_t)-H(0,y_0))_{t\in [0,T)}$ is of maximal 
oscillation with constant $2$ according to \cref{statement:basic_lemma_oscillation}.
\end{exam}
\medskip

\begin{exam}[L\'evy processes, \cref{sec:application_levy_case}]
\label{exam:oscillation_Levy_process}
Let $(X_t)_{t\in [0, T]}$, $X_t:\Omega \to \R$, be a L\'evy process. 
By \cite[Theorem 61.2]{Sa13} there are $\ell\in \R$ and a closed  
$Q\subseteq \R$ such that 
\[ 0\in Q,  \quad Q+Q=Q, \sptext{1}{and}{1} \supp(X_t)= Q+\ell t \sptext{.5}{for}{.5} t\in (0,T]. \]
The set $Q$ is unique. Indeed, assume representations $(\ell,Q)$ and $(\ell',Q')$ so that 
$t_n\ell + Q = t_n \ell' + Q'$ for $T\geqslant t_n \downarrow 0$. By the closeness of $Q$ and 
$Q'$ one deduces $Q=Q'$. If $Q=\R$, then any $\ell\in \R$ can be chosen and w.l.o.g. we agree about 
$\ell:=0$. Assuming $Q\subsetneqq \R$ and $\ell\not = \ell'$ leads to a contradiction: by induction
$t\ell + Q= t\ell' + Q$, $t\in (0,T]$, implies $k (\ell-\ell') (0,T]\subseteq  Q$ and
$k (\ell'-\ell) (0,T] \subseteq Q$ for all $k\in \bN$, which would yield to $Q=\R$
(we have always $0\in Q$). Now define
\begin{equation}\label{eqn:definition_Y_Levy}
Y_t := (X_t -\ell t) \1_{\{ X_t \in \supp (X_t) \}}
\end{equation}
so that $Y_t(\Omega) \subseteq Q$ and $\supp(Y_t)=Q$ for all $t\in (0,T]$. 
Let $0<s<t\leqslant T$ and $H,\tilde H:Q\to \R$ be continuous on $Q$ . Then
\[ \| H(Y_t) - \tilde H(Y_s)  \|_{L_\infty} \geqslant \| H(Y_t) - \tilde H(0) \|_{L_\infty}.\]
This can be seen from
\begin{multline*}
\| H(Y_t) - \tilde H(Y_s)  \|_{L_\infty} = \| H(Y_s + (Y_t-Y_s)) - \tilde H(Y_s)  \|_{L_\infty}
                                     = \sup_{y,y'\in Q} |H(y'+y) - \tilde H(y')| \\
                             \geqslant \sup_{y\in Q} |H(y) - \tilde H(0)|
                                     = \| H(Y_t) - \tilde H(0) \|_{L_\infty}.
\end{multline*}
Consequently, if there is a probability measure $\Q\ll\p$ and if for all $t\in [0,T)$ one has that 
$H(t,\cdot):Q\to \R$ is continuous, $\E^\Q |H(t,Y_t)|<\infty$, and
$\E^\Q (H(t,Y_t)-H(0,0))=0$, then $(H(t,Y_t)-H(0,0))_{t\in [0,T)}$ is of maximal oscillation with constant $2$
according to \cref{statement:basic_lemma_oscillation}.
\end{exam}
\bigskip

Now we connect the notion of oscillation to the behaviour of
$\intsq{\varphi}{\tau}$, where we recall that we use extended conditional expectations for non-negative random variables.

\begin{theo}\label{thm:general_lower_bound}
Assume $\theta \in (0,1]$, $c_\eqref{eq:assum:thm:general_lower_bound}>0$, and an adapted c\`adl\`ag process $(\varphi_t)_{t\in [0,T)}$ such that
\begin{equation}\label{eq:assum:thm:general_lower_bound}
            \frac{1}{c_\eqref{eq:assum:thm:general_lower_bound}} |\varphi_a - Z|^2
   \leqslant \ce{\F_a}{\frac{1}{b-a}\int_a^b |\varphi_u - Z |^2 \od u} \mbox{ a.s.}
\end{equation}
for all $0 \leqslant a < b <T$ and all $\F_a$-measurable $Z:\Omega\to \R$.
Consider the following assertions:
\begin{enumerate}[{\rm(1)}]
\item \label{item:1:thm-general-lower}
      $\inf_{t\in (0,T)} (T-t)^{\frac{1-\theta}{2}}\losc_t(\varphi) >0$.
\item \label{item:2:thm-general-lower}
      There is  a $c_{\cref{thm-general-lower-2}} >0$ such that for all 
      $\tau=\{t_i\}_{i=0}^n \in \mathcal T$ with $\|\tau\|_\theta = \frac{t_k-t_{k-1}}{(T-t_{k-1})^{1-\theta}}$
      one has 
      \begin{align}\label{thm-general-lower-2}
      \inf_{\vartheta_{i-1}\in \cL_0(\F_{t_{i-1}})} \sup_{a\in [t_{k-1},t_k)}
      \left \|\ce{\F_a}{\int_a^T \left| \varphi_u - \sum_{i=1}^n \vartheta_{i-1} \1_{(t_{i-1},t_i]}(u)  \right|^2 \od u} 
      \right \|_{L_\infty}
	  \geqslant c_{\cref{thm-general-lower-2}}^2  \| \tau \|_\theta.
	  \end{align}
\item \label{item:3:thm-general-lower}
	There is  a constant $c_{\cref{thm-general-lower-3}} >0$ such that  for all time-nets $\tau \in \mathcal T$ one has 
	\begin{align}\label{thm-general-lower-3}
	\| [\varphi;\tau]\|_{\BMO_1([0,T))} \geqslant c_{\cref{thm-general-lower-3}}^2  \| \tau \|_\theta.
	\end{align}
\item \label{item:4:thm-general-lower}
      $\inf_{t\in (0,T)} (T-t)^{\frac{1-\theta}{2}}\uosc_t(\varphi) >0$.
\end{enumerate}
\medskip
Then we have $\eqref{item:1:thm-general-lower} \Rightarrow \eqref{item:2:thm-general-lower} \Rightarrow \eqref{item:3:thm-general-lower}$.
If $\|\intsq{\varphi}{\tau}\|_{\BMO_1([0,T))}<\infty$ for all $\tau\in \cT$ and $ \| [\varphi;\tau]\|_{\BMO_1([0,T))}\to 0$ for
$\|\tau\|_1\to0$, then \eqref{item:3:thm-general-lower} $\Rightarrow$ \eqref{item:4:thm-general-lower}.
\end{theo}
\smallskip

\begin{proof}
\eqref{item:1:thm-general-lower} $\Rightarrow$ \eqref{item:2:thm-general-lower}
If $\delta:=\inf_{t\in (0,T)} (T-t)^{\frac{1-\theta}{2}}\losc_t(\varphi)>0$ 
and $a\in [t_{k-1},t_k)$, then
for any $\vartheta_{i-1}\in \cL_0(\F_{t_{i-1}})$ one has, a.s., 
\begin{align*}
            \ce{\F_a}{\int_a^T \left| \varphi_u - \sum_{i=1}^n \vartheta_{i-1} \1_{(t_{i-1},t_i]}(u)  \right|^2 \od u} 
&\geqslant  \ce{\F_a}{\int_a^{t_k} \left| \varphi_u - \sum_{i=1}^n \vartheta_{i-1} \1_{(t_{i-1},t_i]}(u)  \right|^2 \od u}\\
& =         \ce{\F_a}{\int_a^{t_k} \left| \varphi_u - \vartheta_{k-1}  \right|^2 \od u}\\
&\geqslant \frac{1}{c_\eqref{eq:assum:thm:general_lower_bound}} (t_k - a)|\varphi_a - \vartheta_{k-1}|^2.
		\end{align*}
We apply this inequality to $a=t_{k-1}$ and $a=a_0:=\frac{1}{2}(t_{k-1}+t_k)$, observe that 
\begin{align*}
 \frac{1}{2} \left [ (t_k - t_{k-1}) \|\varphi_{t_{k-1}} - \vartheta_{k-1}\|_{L_\infty}^2 
                                        + (t_k - a_0          ) \|\varphi_{a_0} - \vartheta_{k-1}\|_{L_\infty}^2 \right ] 
&\geqslant  \frac{t_k-a_0}{4}\|\varphi_{a_0} - \varphi_{t_{k-1}}\|_{L_\infty}^2 \\
&\geqslant  \frac{t_k-t_{k-1}}{8} \losc^2_{a_0}(\varphi),
\end{align*}
and deduce
\begin{align*}
 \sup_{a\in [t_{k-1},t_k)}
 \left \| \ce{\F_a}{\int_a^T \left| \varphi_u - \sum_{i=1}^n \vartheta_{i-1} \1_{(t_{i-1},t_i]}(u)  \right|^2 \od u} \right \|_{L_\infty}^2 
&\geqslant \frac{1}{c_\eqref{eq:assum:thm:general_lower_bound}}\frac{\delta^2}{8} \frac{t_k-t_{k-1}}{(T-t_{k-1})^{1-\theta}} 
 =  \frac{1}{c_\eqref{eq:assum:thm:general_lower_bound}} \frac{\delta^2}{8} \|\tau\|_\theta.
\end{align*}

\eqref{item:2:thm-general-lower} $\Rightarrow$ \eqref{item:3:thm-general-lower} with 
$c_{\cref{thm-general-lower-3}}:=c_{\cref{thm-general-lower-2}}$ is obvious
because we can choose $\vartheta_{i-1}:=\varphi_{t_{i-1}}$.
\medskip

\eqref{item:3:thm-general-lower} $\Rightarrow$ \eqref{item:4:thm-general-lower}
For $a\in [0,T)$ and $0\leqslant s<t<T$, a time-net $\tau=\{t_i\}_{i=0}^n$ such that $s=t_{k-1}<t_k=t$
and  
\begin{equation}\label{eqn:choice_theta_net} 
\frac{t - s}{(T-s)^{1-\theta}} = \|\tau\|_\theta
\end{equation}
we get 
\begin{align*}
    \left \| \ce{\F_a}{\int_{(a,T) \cap (s, t]} \left| \varphi_u - \sum_{i=1}^n \varphi_{t_{i-1}} 
			\1_{(t_{i-1},t_{i}]}(u)  \right|^2 \od u} \right \|_{L_\infty} 
& = \left \| \ce{\F_a}{\int_{(a,T) \cap (s, t]}|\varphi_u - \varphi_{s}|^2\od u} \right \|_{L_\infty} \\
& \leqslant (t- s)\sup_{u \in (s, t]}\|\varphi_u - \varphi_{s}\|_{L_\infty}^2
\end{align*}
and
\begin{align*}
    c_{\cref{thm-general-lower-3}} \sqrt{\frac{t - s}{(T-s)^{1-\theta}}}
  = c_{\cref{thm-general-lower-3}} \|\tau\|_\theta^\frac{1}{2} 
&\leqslant \sup_{a \in [0,T)}
     \left \| \ce{\F_a}{\int_a^T \left| \varphi_u - \sum_{i=1}^n \varphi_{t_{i-1}} \1_{(t_{i-1},t_{i}]}(u)  
     \right|^2 \od u} \right \|_{L_\infty}^\frac{1}{2}  \\
& \leqslant  \sup_{a \in [0,T)} \Bigg [ \left \| \ce{\F_a}{\int_{(a,T) \cap  (s,t]} \left| \varphi_u - \sum_{i=1}^n \varphi_{t_{i-1}} \1_{(t_{i-1},t_{i}]}(u)  \right|^2 \od u}    
		\right \|_{L_\infty}^\frac{1}{2} \notag \\
		& \quad + \left \| \ce{\F_a}{\int_{(a,T) \backslash (s,t]} \left| \varphi_u - \sum_{i=1}^n \varphi_{t_{i-1}} \1_{(t_{i-1},t_{i}]}(u)  \right|^2 \od u}    
		\right \|_{L_\infty}^\frac{1}{2} \Bigg ] \\
& \leqslant \sqrt{t- s}\sup_{u \in (s, t]}\|\varphi_u - \varphi_{s}\|_{L_\infty} \\
& \quad + \sup_{a \in [0,T)}  \left \| \ce{\F_a}{\int_{(a,T) \backslash (s,t]} \left| \varphi_u - \sum_{i=1}^n \varphi_{t_{i-1}} \1_{(t_{i-1},t_{i}]}(u)  \right|^2 \od u}    
		\right \|_{L_\infty}^\frac{1}{2}.
\end{align*}
Assume a time-net $\tilde{\tau}$ that coincides with $\tau$ outside the interval $(s,t)$. Then
\[
    c_{\cref{thm-general-lower-3}} \sqrt{\frac{t - s}{(T-s)^{1-\theta}}}
\leqslant \sqrt{t- s}\sup_{u \in (s, t]}\|\varphi_u - \varphi_{s}\|_{L_\infty}  + \| [\varphi;\tilde{\tau}]\|_{\BMO_1([0,T))}^\frac{1}{2}. \]
Choosing a sequence $(\tau_n,\tilde \tau_n)$ of $(\tau,\tilde \tau)$ with \eqref{eqn:choice_theta_net} as above,
such that $\|\tilde{\tau}_n\|_{1}\to 0$, we conclude with
\[
    c_{\cref{thm-general-lower-3}} \sqrt{\frac{t - s}{(T-s)^{1-\theta}}}
\leqslant \sqrt{t- s}\sup_{u \in (s, t]}\|\varphi_u - \varphi_{s}\|_{L_\infty} \]
and hence
\[ \sup_{u \in (s, t]}\|\varphi_u - \varphi_{s}\|_{L_\infty}^2
\geqslant   c_{\cref{thm-general-lower-3}}^2 (T-s)^{\theta-1}. \]
For $s\in \left [ (2t -T)^+,t\right )$ this gives
$\sup_{u \in (s, t]}\|\varphi_u - \varphi_{s}\|_{L_\infty}^2 \geqslant   c_{\cref{thm-general-lower-3}}^2 2^{\theta-1} (T-t)^{\theta-1}$
and therefore
\[
c_{\cref{thm-general-lower-3}} 2^{\frac{\theta-1}{2}} (T-t)^{\frac{\theta-1}{2}}
\leqslant \| \varphi_t -\varphi_s \|_{L_\infty} + \sup_{u \in (s, t]}\|\varphi_t - \varphi_u\|_{L_\infty}
\leqslant 2  \sup_{u \in [s, t]}\|\varphi_t - \varphi_u\|_{L_\infty}.\]
This implies
$\uosc_t(\varphi) \geqslant   c_{\cref{thm-general-lower-3}} 2^{\frac{\theta-3}{2}} (T-t)^{\frac{\theta-1}{2}}$.
\end{proof}


\section[Gradient estimates: Brownian setting]{Brownian setting: Gradient estimates and approximation}
\label{sec:application_brownian_case}

We suppose additionally that $\F=\F_T$
and that $(\F_t)_{t\in [0,T]}$ is the augmentation of the natural filtration of a standard one-dimensional 
Brownian motion $W=(W_t)_{t\in [0,T]}$ with continuous paths and starting in zero for all $\omega\in \Omega$.
We recall the setting from \cite{GG04} and start with the stochastic differential equation (SDE)
\begin{equation}\label{eqn:sde:X}
\od X_t = \hat{\sigma}(X_t) \od W_t + \hat{b}(X_t) \od t 
\sptext{1}{with}{1} 
X_0\equiv x_0\in \R
\end{equation}
where
$0< \varepsilon_0 \leqslant \hat{\sigma} \in C_b^\infty(\R)$ for some constant $\ep_0$ and 
$\hat{b} \in C_b^\infty(\R)$
and where all paths of $X$ are assumed to be continuous.
From this equation we derive the SDE
\[ \od Y_t = \sigma(Y_t) \od W_t  
\sptext{1}{with}{1} 
Y_0\equiv y_0\in \defg \]
where two settings are used simultaneously:
\medskip
	
Case (C1): $Y:=X$ with $\sigma \equiv \hat{\sigma}$, $\hat{b}\equiv 0$, and $\defg := \R$.

Case (C2): $Y := \e^X$ with 
           $\sigma(y):= y \hat{\sigma}(\ln y)$,
           $\hat{b}(x) := -\frac{1}{2}\hat{\sigma}^2(x)$, and $\defg := (0, \infty)$.
\medskip

In both cases, we let $C_Y$ be the set of all Borel functions $g \colon \defg \to \R$ such that
\begin{equation}\label{eqn:CY}
\sup_{x\in \R} \e^{-m|x|} \int_\R |g(\alpha(x+ t y))|^2 \e^{-y^2} \od y < \infty 
	\sptext{1}{for all}{1} t>0
\end{equation}
	for some $m>0$, where $ \alpha(x)= x$ in the case (C1) and $ \alpha(x)= \e^x$ in the case (C2).
Under (C1) any polynomially bounded $g$ belongs to $C_Y$, under (C2) a bound
$|g(y)|\leqslant  A + B|y|$ gives $g\in C_Y$.
Let us denote by $(Y_s^{t,y})_{s\in [t,T]}$ the diffusion $Y$ started at time $t\in [0,T]$ in $y\in \defg$
and let us define, for $g \in C_Y$,
\[ G(t,y):= \E g(Y_T^{t,y}) 
   \sptext{1}{for}{1}
   (t,y)\in [0,T]\times \defg. \]

\begin{rema}
\label{statement:facts_from_GG04}
We collect some facts we shall use  and that hold in both cases, (C1) and (C2):
\begin{enumerate}[{\rm (A)}]
\item \label{item:1a:statement:facts_from_GG04} 
      $\| \sigma' \|_{B_b(\defg)} + \| \sigma \sigma{''} \|_{B_b(\defg)} < \infty$.
\item \label{item:1b:statement:facts_from_GG04} 
      In the case (C2) we have $\sigma(y)\sim_c y$ for $y\in\defg$ and some $c\geqslant 1$.
\item \label{item:2:statement:facts_from_GG04} 
      One has $G\in C^\infty([0,T)\times \defg)$ and 
      $\frac{\partial G}{\partial t} + \frac{\sigma^2}{2} \frac{\partial^2 G}{\partial y^2} = 0$
      on $[0,T)\times \defg$.
\item \label{item:3:statement:facts_from_GG04} 
      $\E \left [ |g(Y_T)|^2 +  \sup_{t\in [0,b]} \left | \left (\sigma \frac{\partial G}{\partial y} \right )(t,Y_t) \right |^2 \right ]< \infty$ for all
      $b\in [0,T)$.  
\item \label{item:4:statement:facts_from_GG04}
      The process $\left ( \left (\sigma^2 \frac{\partial^2 G}{\partial y^2}\right ) (t,Y_t) \right )_{t\in [0,T)}$ is an $L_2$-martingale.
\item \label{item:5:statement:facts_from_GG04}
      The process $X$ has a transition density $\Gamma_X$ in the sense of \cref{statement:friedman}.
\end{enumerate}
Items \eqref{item:1a:statement:facts_from_GG04} and \eqref{item:1b:statement:facts_from_GG04}
are obvious,
\eqref{item:2:statement:facts_from_GG04} is contained in \cite[Preliminaries]{GG04},
\eqref{item:3:statement:facts_from_GG04} follows from the definition of $C_Y$, \cref{statement:friedman}, and
\cite[Lemma 5.2]{GG04}, and \eqref{item:4:statement:facts_from_GG04} is \cite[Lemma 5.3]{GG04}.
\end{rema}
\medskip

This yields to the following setting:

\begin{setting} \label{set:application-BM}
In the notation of \cref{ass:random_measures_simplified} we set
\begin{enumerate}
\item \hspace*{.5em}$\sigma := (\sigma(Y_t))_{t\in [0,T]}$,
\item $M := \left (\int_0^t \left (\sigma^2 \frac{\partial^2 G}{\partial y^2}\right )(u,Y_u) \od W_u\right )_{t\in [0,T)}$
      (with $M_0\equiv 0$ and the continuity of all paths),
\item \hspace*{.2em} $\varphi := \left ( \frac{\partial G}{\partial y}(t,Y_t)\right )_{t\in [0, T)}$.
\end{enumerate}
\end{setting}    
\cref{statement:regularity_sigma_varphi}   and \cite[Corollary 3.3]{Ge02}
imply that \cref{ass:random_measures_simplified} is fulfilled.
To shorten the notation at some places we use
\[ Z_t := \sigma_t \varphi_t, \quad
   \varphi(t,y) := \frac{\partial G}{\partial y}(t,y), 
   \sptext{1}{and}{1}
   H_t  := \sigma_t^2 \frac{\partial^2 G}{\partial y^2}(t,Y_t)
   \sptext{1}{for}{1} (t,y) \in [0, T)\times \defg. \]
For $\alpha>0$ and $0\leqslant a \leqslant t <T$ we get
by \eqref{eqn:1:statement:RL_for_martinagles} and \eqref{eqn:squarefunction_RL}, a.s.,
\begin{equation}\label{eqn:cIalphaM_vs_Malpha}
\cI_t^\alpha M = \int_{(0,t]} \left ( \frac{T-u}{T} \right )^\alpha H_u \od W_u 
\sptext{1}{and}{1}
  \ce{\F_a}{\left | \cI_t^\alpha M - \cI_a^\alpha M \right |^2}
= \ce{\F_a}{\int_a^t \left ( \frac{T-u}{T} \right )^{2\alpha} H_u^2 \od u}.
\end{equation}
Moreover, as $M$ and $Z$ are continuous, by  \cref{statemant:RL_general}\eqref{item:5:statemant:RL_general}
$(\cI^\alpha_t M)_{t\in [0,T)}$ and $(\cI^\alpha_t Z)_{t\in [0,T)}$ are continuous as well and the $\bmo_p^\Phi([0,T))$ and
$\BMO_p^\Phi([0,T))$ are interchangeable when applied to these processes.

Denote by $E(g;\tau) = (E_t(g;\tau))_{t\in [0, T]}$ the error process resulting from the difference between 
the stochastic integral and its Riemann approximation associated with the time-net $\tau= \{t_i\}_{i=0}^n \in \mathcal T$, i.e.
\begin{align*}
E_t(g;\tau) := \int_{(0,t]} \varphi_s \od Y_s - \sum_{i=1}^n \varphi_{t_{i-1}}(Y_{t_i \wedge t} - Y_{t_{i-1} \wedge t})
               \sptext{1}{for}{.5} t \in [0, T].
\end{align*}
For any $0 \leqslant a \leqslant t \leqslant T$, we apply the conditional  It\^o's isometry to obtain that, a.s.,
\begin{align}\label{eq:err_vs_square_function}
    \ce{\F_a}{|E_t(g; \tau) - E_a(g; \tau)|^2} 
& = \ce{\F_a}{\int_a^t \left| \varphi_{u} - \sum_{i=1}^n \varphi_{t_{i-1}} \1_{(t_{i-1}, t_i]}(u)\right|^2 \sigma_u^2 \od u}
  = \ce{\F_a}{\intsqw{\varphi}{\sigma}{\tau}_t-\intsqw{\varphi}{\sigma}{\tau}_a}.
\end{align}
Using \cref{statement:bmo_determinstic} this implies, for $\Phi\in \CL^+([0,T))$, that

\begin{align}\label{eqn:equivalence_E_R}
\left \| \left ( E_t(g;\tau) \right )_{t\in [0,T)} \right \|_{\bmo_2^\Phi([0,T))}^2 
&= \left \| \left ( \intsqw{\varphi}{\sigma}{\tau} \right )_{t\in [0,T)} \right \|_{\bmo_1^{\Phi^2}([0,T))},
\end{align}
where $\intsqw{\varphi}{\sigma}{\tau}$ is given in \cref{ass:random_measures_simplified}. 
Moreover, $\bmo_2^\Phi([0,T))$ and $\bmo_1^{\Phi^2}([0,T))$ above can be replaced by 
$\BMO_2^\Phi([0,T))$ and $\BMO_1^{\Phi^2}([0,T))$, respectively, due to the path continuity of $E(g;\tau)$ and $\intsqw{\varphi}{\sigma}{\tau}$.
To be in accordance with the previous sections we use in  
\eqref{eqn:equivalence_E_R} the time interval $[0,T)$ instead of $[0,T]$.
\smallskip

\subsection{The results}
\label{subsec:results:BM}
In this section we formulate the results, they 
are verified in Section \ref{subsec:results:BM:proof}. 
The first result shows that all gradient processes   $(\varphi(t,Y_t))_{t\in [0,T)}$ have  
a large oscillation:

\begin{theo}
\label{statement:oscillation_brownian_case}
For $g\in C_Y$ the process $(\varphi_t)_{t\in [0,T)} = (\frac{\partial G}{\partial y}(t,Y_t))_{t\in [0,T)}$ 
is of maximal oscillation with constant $2$ in the sense of \cref{definition:oscillation}.
\end{theo}
\bigskip

Now we discuss cases in which we get equivalences by choosing the weight $\Phi$ accordingly.
For $\theta=1$ we obtain a characterization in terms of Lipschitz functions that extends \cite[Theorem 8]{Geiss:05}:
\smallskip

\begin{theo}
\label{thm:pde_for_theta_equals_one}
For $g\in C_Y$ and $\Phi = \sigma$ the following assertions are equivalent:
\begin{enumerate}[\rm (1)]
\item \label{item:1:thm:pde_for_theta_equals_one}
      There exists a Lipschitz function $\tilde g \colon \defg \to \R$ such that $g=\tilde g$ a.e. on $\defg$
      with respect to the Lebesgue measure.
\item \label{item:2:thm:pde_for_theta_equals_one}
      There is a constant $c>0$ such that $\|E(g;\tau) \|_{\BMO_2^{\Phi}([0,T))} \leqslant c \sqrt{\| \tau\|_1}$ 
      for all $\tau \in \mathcal{T}$.
\end{enumerate}
\end{theo}
\bigskip

Before we investigate \cref{thm:pde_for_theta_equals_one} for H\"older functions 
we show that the process $(\cI^\alpha_t Z)_{t\in [0,T)}$ shares the limit behaviour of a martingale.
To do so and for later purpose we use the following a-priori condition:
given $\Phi \in \CL^+([0,T))$, there is  $c_\eqref{eq:a_priori_asumption_phi_new}>0$
such that 
\begin{equation}\label{eq:a_priori_asumption_phi_new}\tag{$C_\Phi$}
            (T-t)^{\frac{1}{2}} |Z_t| 
 \leqslant c_\eqref{eq:a_priori_asumption_phi_new}  \Phi_t \mbox{ a.s.}
  \sptext{1}{for}{.75} t\in [0,T).
\end{equation}
In the cases, we are interested in, this condition will be satisfied and corresponds to a known
a-priori condition from the theory of parabolic PDEs.
We will also use the condition $\Phi \in \cSM_2([0,T))$ with $\sup_{t\in [0,T)} \Phi_t \in L_q$ for some 
$q\in [2,\infty)$ and want to note that $\sup_{t\in [0,T)} \Phi_t \in L_q$ is redundant if $q=2$ 
as this follows from $\Phi \in \cSM_2([0,T))$.

\medskip

\begin{theo}
\label{statement:Z_like_a_martingale}
Let 
$(\alpha,q)\in (0,\infty) \times [2,\infty)$,
$\Phi \in \cSM_2([0,T))$ with $\sup_{t\in [0,T)} \Phi_t \in L_q$, and 
$g\in C_Y$ such that \eqref{eq:a_priori_asumption_phi_new} is satisfied. If
$\cI^\alpha  Z - Z_0\in \BMO^\Phi_2([0,T))$,
then 
\[ \cI_T^\alpha Z := \lim_{t \uparrow T} \cI_t^\alpha Z
   \sptext{1}{exists in}{1} L_q \quad \mbox{and a.s.}\]
\end{theo}
\medskip

For $\theta\in (0,1)$ we obtain in a first step as counterpart to \cref{thm:pde_for_theta_equals_one} 
a general equivalence in terms of the Riemann--Liouville type integral (introduced in \cref{sec:RL-operators}) 
of the gradient process:
\medskip

\begin{theo}
\label{statement:brownian_case_theta_new}
Let $\theta\in (0,1)$, $\Phi=(\sigma_t \Psi_t)_{t\in [0,T)} \in \cSM_2([0,T))$ where 
$\Psi\in \CL^+([0,T))$ is path-wise non-decreasing,
and assume the a-priori estimate \eqref{eq:a_priori_asumption_phi_new}.
Then the following is equivalent:
\smallskip

\begin{enumerate}[\rm (1)]
\item \label{item:1:statement:brownian_case_theta_new}
      One has  $\cI^\frac{1-\theta}{2} Z - Z_0\in \BMO^\Phi_2([0,T))$ and there is a constant $c>0$ such that
      \begin{equation}\label{eq:item:1:statement:brownian_case_theta_new}
                  (T-t)^{\frac{1-\theta}{2}} |\varphi_t-\varphi_0|
         \leqslant c \Psi_t \quad {\mbox{a.s.}}  \sptext{1}{for}{.75} t\in [0,T).
      \end{equation}
\item \label{item:2:statement:brownian_case_theta_new}
      There is a constant $c>0$ such that 
      $\|E(g;\tau)\|_{\BMO_2^{\Phi}([0,T))}  
	 \leqslant c \sqrt{\| \tau\|_\theta}$ for all $\tau \in \mathcal{T}$.
\end{enumerate}
\end{theo}
\bigskip 

The next result takes us finally to the H\"older functions that form a natural class of functions 
that can be used in \cref{statement:brownian_case_theta_new}. 
\medskip

\begin{theo}
\label{statement:Hoelder_new_new}
For $(\theta,p) \in [0,1]\times (0,\infty)$ and $g\in \hoel{\theta}$ one has $g|_\defg\in C_Y$ and it holds:
\begin{enumerate}[\rm(1)]
\item \label{item:1:statement:Hoelder_new_new} 
      $\sigma(\theta):=(\sigma_t^\theta)_{t\in [0,T)} \in \cSM_p([0,T))$.
\item \label{item:2:statement:Hoelder_new_new}
      There is a constant $c>0$ such that $(T-t)^{\frac{1-\theta}{2}} |\varphi_t| \leqslant c \sigma_t^{\theta-1}$ a.s. for all $t \in [0,T)$.
\item \label{item:3:statement:Hoelder_new_new}
      If $\theta\in (0,1)$ and $g\in \hoelO{\theta,2}$, then we have the following:
      \begin{enumerate}[{\rm (a)}]
      \item \label{item:3a:statement:Hoelder_new_new}
            $\cI^\frac{1-\theta}{2} Z -Z_0\in \BMO_p^{\sigma(\theta)}([0,T))$.
      \item \label{item:3b:statement:Hoelder_new_new}
            There is a constant $c>0$ such that for all $\tau\in \cT$ one has 
            \[ \| E(g;\tau)\|_{\BMO_2^{\Phi(\tau,\theta)}([0,T))} \leqslant c \sqrt{\| \tau\|_\theta}
               \sptext{1}{with}{1} \Phi(\tau,\theta):=(\sigma_a^\theta+\sigma_{\underline{a}(\tau)}^{\theta-1}\sigma_a)_{a\in [0,T)}\in \CL^+([0,T))\]
            where for $a\in [t_{k-1},t_k)$ we let $\underline{a}(\tau) := t_{k-1}$. 
      \end{enumerate}
\end{enumerate}
\end{theo}
\medskip

\begin{rema}
\label{rema:statement:Hoelder_new_new}
In the case (C1) the weight $\Phi(\tau,\theta)$ in \cref{statement:Hoelder_new_new} behaves like a constant. For the case (C2),
relevant for option pricing, we have the following:
\begin{enumerate}[{\rm (1)}]
\item For $a\in [0,T)\cap \tau$ it holds
\begin{align*} 
  \sup_{t\in [a,T)} \Phi_t(\tau,\theta) 
& \leqslant  2 \sigma_a^\theta   \sup_{a \leqslant u \leqslant t < T} 
      \left \{ \left (\frac{\sigma_u}{\sigma_a}\right )^{\theta-1} \frac{\sigma_t}{\sigma_a} \right \} 
  \leqslant c \sigma_a^\theta \e^{2 \sup_{t\in [a,T]} \left | \int_{[a,t]} \widehat{\sigma}(X_u) \od W_u\right |} 
  =: c \sigma_a^\theta R_a,
\end{align*}
where $c>0$ can be chosen to depend on $\widehat{\sigma}$, $\widehat{b}$, and $T$ only. Moreover, 
using (the arguments of) \cite[Lemma 6.2(ii)]{GY19} for the exponent in $R_a$, there is an $A=A(\widehat{\sigma},T)>0$ such that
\begin{equation}\label{eqn:tail_RAa}
   \p_{\F_a}(R_a > \lambda) \leqslant A \e^{-(\frac{\ln \lambda}{A})^2} \mbox{ a.s.}
   \sptext{1}{for}{1} \lambda \geqslant 1.
\end{equation}

\item 
For $\Phi=(\sigma_t \Psi_t)_{t\in [0,T)}$ with 
      $\Psi_t:= \sup_{s\in [0,t]} (\sigma_s^{\theta-1})$ one has 
      $2 \Phi_t \geqslant \Phi_t(\tau,\theta) 
                \geqslant \sigma_t^\theta$ and  $\Phi \in \cSM_q([0,T))$.
Therefore \cref {statement:Hoelder_new_new} is applicable to \cref{statement:brownian_case_theta_new}.
The fact $\Phi \in \cSM_q([0,T))$ follows from 
\cref{statement:regularity_sigma_varphi}\eqref{item:1:statement:regularity_sigma_varphi}
\end{enumerate}
\end{rema}

\subsection{Preparations to prove the results of Section \ref{subsec:results:BM}} 
\label{subsec:results:BM:proof}
We collect some lemmas we need.
\smallskip

\begin{lemm}
\label{statement:abstract:pde_new}
Assume  that $\theta\in (0,1]$, $g\in C_Y$, $\tau\in \cT$, and $\Phi \in \CL^+([0, T))$ such that, for $a\in [0,T)$,
\[  \ce{\F_a}{\sup_{t\in [a,T)} \left |\cI_t^{\frac{1-\theta}{2}} M- \cI_a^{\frac{1-\theta}{2}} M \right |^2}
    + \frac{T-a}{(T-\underline{a}(\tau))^\theta} \left | \varphi_a-\varphi_{\underline{a}(\tau)} \right |^2\sigma^2_a
      \leqslant \Phi_a^2 \quad \mbox{a.s.}
\]
Then $\|E(g;\tau)\|_{\BMO_2^{\Phi}([0,T))}\leqslant  \sqrt{c_\eqref{eq:item:1:statemtent:equivalence_simplified_setting_new}}  \sqrt{\| \tau \|_\theta}$, 
where $c_\eqref{eq:item:1:statemtent:equivalence_simplified_setting_new}>0$ is taken
from inequality \eqref{eq:item:1:statemtent:equivalence_simplified_setting_new}.
\end{lemm}

\begin{proof}
The statement follows directly from \cref{eqn:equivalence_E_R} and
inequality \eqref{eq:item:1:statemtent:equivalence_simplified_setting_new}.
\end{proof}

\begin{lemm}
\label{statement:regularity_sigma_varphi}
The following assertions hold true:
\begin{enumerate}[{\rm (1)}]
\item \label{item:1:statement:regularity_sigma_varphi}
      In the case (C2) one has 
      $(Y_t^{\beta_0} (Y^{\beta_1})^*_t)_{t\in [0,T]}\in \cSM_p([0,T])$ for $p\in (0,\infty)$ and $\beta_0,\beta_1\in\R$.
\item \label{item:2:statement:regularity_sigma_varphi}
      There is a constant $c_{\eqref{eqn:statement:reverse_hoelder_for_sigma}}>0$ such that, 
      for all $0\leqslant a < b \leqslant T$,
      \begin{equation}\label{eqn:statement:reverse_hoelder_for_sigma}
      \ce{\F_a}{\frac{1}{b-a}\int_a^b \sigma_u^2 \od u} \sim_{c_\eqref{eqn:statement:reverse_hoelder_for_sigma}^2} \sigma_a^2
      \quad\mbox{ a.s.}
      \end{equation}	
\item \label{item:3:statement:regularity_sigma_varphi}
      For $g\in C_Y$ one has $\E \sup_{u\in [a,T]} |\varphi_a \sigma_u|^2 <\infty$ 
      for $a\in [0,T)$.
\end{enumerate}
\end{lemm}

\begin{proof}
\eqref{item:1:statement:regularity_sigma_varphi}
Because $\hat\sigma\in B_b(\R)$, for all $\alpha\in \R$ there is a constant
$c_\eqref{eq:item:1:statement:regularity_sigma_varphi}=c_\eqref{eq:item:1:statement:regularity_sigma_varphi}(\alpha,T,\hat{\sigma})>0$
such that
\begin{equation}\label{eq:item:1:statement:regularity_sigma_varphi}
    \ce{\F_a}{\sup_{t\in [a,T]} \e^{\alpha \int_{(a,t]} \hat{\sigma}(X_s) \od W_s}} 
\le c_\eqref{eq:item:1:statement:regularity_sigma_varphi} \mbox{ a.s.}
\end{equation}
for $a\in [0,T]$.
Because $\hat{b}$ is bounded this implies that $(Y_t^\beta)_{t\in [0,T]}\in \cSM_p([0,T])$ for all $p\in (0,\infty)$
and $\beta\in\R$ by \cref{statement:smp_determinstic}.
Therefore we may conclude by items \eqref{item:3:statement:SM-properties} and \eqref{item:4:statement:SM-properties} of
\cref{statement:SM-properties}.
\smallskip

\eqref{item:2:statement:regularity_sigma_varphi}
We only need to check the case (C2) where we
replace $\sigma$ by $Y$ due to \eqref{item:1b:statement:facts_from_GG04}.
As $Y$ is a martingale we get $\ce{\F_a}{\int_a^b Y_u^2 \od u} \geqslant (b-a) Y_a^2$ a.s.,
otherwise $\ce{\F_a}{\int_a^b Y_u^2 \od u} \leqslant \| Y\|^2_{\cSM_2([0,T])} (b-a) Y_a^2$ a.s.
\smallskip

\eqref{item:3:statement:regularity_sigma_varphi}
Because of \eqref{item:3:statement:facts_from_GG04} we only need to check  
 (C2), use again \eqref{item:1b:statement:facts_from_GG04} to replace $\sigma$ by $Y$, and obtain 
\[
             \E \sup_{u\in [a,T]} |\varphi_a Y_u|^2 
 = \E  \left [ |\varphi_a|^2 \ce{\F_a}{ \sup_{u\in [a,T]} Y_u^2} \right ]
\leqslant \|Y\|_{\cSM_2([0,T]}^2 \E |\varphi_a Y_a|^2 < \infty. \qedhere\]
\end{proof}
\bigskip

\begin{lemm} \label{lemma:ito_brownian_case_new}
For $\alpha \geqslant 0$ and $t\in [0,T)$  one has, a.s.,
	\begin{align*}
	(T-t)^\alpha Z_t  = &  T^\alpha Z_0
	+ \int_{(0,t]} (T-u)^\alpha H_u \od W_u 
	+ \int_{(0,t]} (T-u)^\alpha \sigma'(Y_u) Z_u \od W_u  \\
	& - \alpha \int_{(0,t]} (T-u)^{\alpha-1} Z_u \od u 
	+ \frac{1}{2} \int_{(0,t]} (T-u)^\alpha (\sigma \sigma{''})(Y_u) Z_u \od u.
	\end{align*}
\end{lemm}

\begin{proof}
The assertion follows by It\^o's formula applied to the function 
$(t,y)\mapsto (T-t)^\alpha \left ( \sigma\frac{\partial G}{\partial y}\right )(t,y)$
with $Y_t$ inserted into the $y$-component, where we use the PDE from
\eqref{item:2:statement:facts_from_GG04}.
\end{proof}
\bigskip

\begin{lemm}
\label{lemma:estimates_partialG}
There exists a constant $c_{\eqref{eq:lemma:estimates_partialG}}=c_{\eqref{eq:lemma:estimates_partialG}}(c_\eqref{eq:statement:friedman},T)>0$
such that for $\theta \in [0,1]$ and $g\in \hoel \theta$ one has
\begin{equation}\label{eq:lemma:estimates_partialG}
         \left | \frac{\partial G}{\partial y}(u,y) \right | 
\leqslant c_{\eqref{eq:lemma:estimates_partialG}} \, |g|_\theta \, \sigma(y)^{\theta-1} (T-u)^{\frac{\theta-1}{2}}
          \sptext{1}{for}{.5} (u,y)\in [0,T)\times \defg.
\end{equation}
\end{lemm} 
\begin{proof} 
Set $f:=g$ and $F:=G$ in case (C1) and $f(x) := g(\e^x)$ and $F(u, x) := G(u,\e^x)$ for $(u, x) \in [0, T) \times \R$
in case (C2). Let us fix $u\in [0,T)$. In both cases, (C1) and (C2), we have
\[   \frac{\partial F}{\partial x}(u,x) 
   = \int_\R \frac{\partial \Gamma_X}{\partial x}(T-u,x,\xi) f(\xi) \od\xi 
   = \int_\R \frac{\partial \Gamma_X}{\partial x}(T-u,x,\xi) (f(\xi)-f(x)) \od\xi\]
where we use \eqref{item:5:statement:facts_from_GG04} with the transition density $\Gamma_X$ from \cref{statement:friedman}. 
For $t>0$ denote $\gamma_t(x): = \frac{1}{\sqrt{2\pi t}} \e^{-\frac{x^2}{2t}}$. 
In the case (C1) we derive for $y=x$ that
\begin{align*}
    \left | \frac{\partial G}{\partial y}(u,y) \right |
  = \left | \frac{\partial F}{\partial x}(u,x) \right | 
&\leqslant |g|_\theta \int_\R \left | \frac{\partial \Gamma_X}{\partial x}(T-u,x,\xi)\right|  |\xi- x|^\theta \od\xi \\
&\leqslant |g|_\theta  \int_\R c_\eqref{eq:statement:friedman}  (T-u)^{-\frac{1}{2}} 
           \gamma_{c_\eqref{eq:statement:friedman} (T-u)}(x-\xi) |\xi-x|^\theta \od\xi \\
& =  |g|_\theta  (T-u)^{\frac{\theta-1}{2}}  
     \int_\R c_\eqref{eq:statement:friedman} \gamma_{c_\eqref{eq:statement:friedman} }(\eta) |\eta|^\theta \od\eta \\
&\leqslant |g|_\theta  (T-u)^{\frac{\theta-1}{2}}  
     \int_\R c_\eqref{eq:statement:friedman} \gamma_{c_\eqref{eq:statement:friedman} }(\eta) (1+|\eta|) \od\eta
\end{align*}
where we use 
$                              \int_\R \frac{\partial \Gamma_X}{\partial x}(T-u,x,\xi) \od\xi 
 = \frac{\partial}{\partial x} \int_\R \Gamma_X(T-u,x,\xi) \od\xi
= 0$.
For $y=\e^x$ we get for (C2) that  
\equa 
\left | y \frac{\partial   G}{\partial y  }(u,y)\right | =  
	 \left | \frac{\partial F}{\partial x}(u,x) \right |
	&\leqslant& |g|_\theta \int_\R \left | \frac{\partial \Gamma_X}{\partial x}(T-u,x,\xi)\right|  |\e^\xi- \e^x|^\theta \od\xi \\
	& =& |g|_\theta  \e^{x \theta} \int_\R \left | \frac{\partial \Gamma_X}{\partial x}(T-u,x,\xi)\right|  |\e^{\xi-x}- 1|^\theta \od\xi \\
	&\leqslant& |g|_\theta  \e^{x \theta} \int_\R c_\eqref{eq:statement:friedman}  (T-u)^{-\frac{1}{2}} \gamma_{c_\eqref{eq:statement:friedman} (T-u)}(x-\xi) |\e^{\xi-x}- 1|^\theta \od\xi.
	\tion
	We conclude by 
	\begin{align*}
     \int_\R \gamma_{c_\eqref{eq:statement:friedman} (T-u)}(x-\xi) |\e^{\xi-x}- 1|^\theta \od\xi 
    & \leqslant \int_\R \gamma_{c_\eqref{eq:statement:friedman} (T-u)}(\xi) |\xi|^\theta \e^{\theta|\xi|} \od\xi \\
    & \leqslant (T-u)^\frac{\theta}{2} \int_\R \gamma_{c_\eqref{eq:statement:friedman} }(\eta) |\eta|^\theta \e^{\theta\sqrt{T} |\eta|} \od\eta \\
    &\leqslant (T-u)^\frac{\theta}{2} \int_\R \gamma_{c_\eqref{eq:statement:friedman} }(\eta) (1+|\eta|) \e^{\sqrt{T} |\eta|} \od\eta
< \infty. \qedhere
    \end{align*}
\end{proof}
\bigskip

\begin{prop}
\label{statment:IZ_vs_HdW_new}
For $\alpha>0$ and $\Phi\in \cSM_2([0,T))$, such that
\eqref{eq:a_priori_asumption_phi_new} holds, one has:
\begin{enumerate}[{\rm (1)}]

\item \label{it:a}
      $(\cI_t^\alpha Z - Z_0)_{t\in [0,T)} \in \BMO_2^\Phi([0,T))$ if and only if
      $\cI^\alpha  M \in \BMO_2^\Phi([0, T))$.

\item \label{it:b}
     If $q\in [2,\infty)$ and $\sup_{t\in [0,T)} \Phi_t \in L_q$, then 
     as $t\uparrow T$ one has that $\cI^\alpha Z$ converges (is bounded) in $L_q$ 
     if and only if $\cI^\alpha M$ does (is).

\item \label{it:c}
      As $t\uparrow T$ one has that $\cI^\alpha Z$ converges a.s. if and only if $\cI^\alpha M$ does. 
      \footnote{For $\cI^\alpha M$ the $L_q$-boundedness and the convergence in $L_q$ is equivalent 
                because of the martingale property.}

\end{enumerate}
\end{prop}

\begin{proof}
\eqref{it:a} For $t\in [0,T)$ the relation 
\[   \alpha\int_0^T (T-u)^{\alpha-1} Z_{u\wedge t} \od u 
   = \alpha \int_0^t (T-u)^{\alpha - 1} Z_u \od u  + (T-t)^\alpha Z_t \]
and \cref{lemma:ito_brownian_case_new} imply that
\begin{align*}
     \alpha \int_{(0,T]} (T-u)^{\alpha-1} Z_{u\wedge t} \od u   
= &  T^\alpha Z_0
	+ \int_{(0,t]} (T-u)^\alpha H_u \od W_u 
	+ \int_{(0,t]} (T-u)^\alpha \sigma'(Y_u) Z_u \od W_u  \\
	& 
	+ \frac{1}{2} \int_{(0,t]} (T-u)^\alpha (\sigma \sigma{''})(Y_u) Z_u \od u \mbox{ a.s.}
	\end{align*}
Denote $b_u(\omega): = \frac{1}{2}(\sigma \sigma'')(Y_u(\omega))$ and 
$B: = \frac{1}{2}\|\sigma \sigma''\|_{B_b(\defg)}<\infty$. Dividing both sides of the equality above by $T^\alpha$ 
and using \eqref{eqn:cIalphaM_vs_Malpha}, gives 
\begin{equation}\label{eqn:representation_Z}
     (\cI_t^\alpha Z -  Z_0) - \cI_t^\alpha M 
   = \int_{(0,t]}  \left ( \frac{T-u}{T}\right )^\alpha Z_u ( \sigma'(Y_u) \od W_u + b_u \od u ) \quad \mbox{a.s.}
\end{equation}
Next we observe that, for $0\leqslant a <t < T$, \mbox{a.s.,}
\begin{align*}
&   \left ( \ce{\F_a}{\left | \int_{(a,t]}  \left ( \frac{T-u}{T}\right )^\alpha Z_u \sigma'(Y_u) \od W_u \right |^2} 
     \right )^\frac{1}{2} + 
     \left ( \ce{\F_a}{\left | \int_{(a,t]}  \left ( \frac{T-u}{T}\right )^\alpha |Z_u b_u| \od u 
          \right |^2} \right )^\frac{1}{2} \\
& \leqslant (\|\sigma'\|_{B_b(\defg)} + B \sqrt{T}) 
     \left ( \ce{\F_a}{\int_{(a,t]}  \left ( \frac{T-u}{T}\right )^{2\alpha} |Z_u|^2 \od u} 
     \right )^\frac{1}{2} \\
& \leqslant c_\eqref{eq:a_priori_asumption_phi_new}
     (\|\sigma'\|_{B_b(\defg)} + B \sqrt{T}) 
     \left ( \ce{\F_a}{\sup_{u\in [a,T)} \Phi_u^2 \int_{(a,t]}  \left ( \frac{T-u}{T}\right )^{2\alpha} (T-u)^{-1}\od u} 
     \right )^\frac{1}{2} \\
& \leqslant \frac{c_\eqref{eq:a_priori_asumption_phi_new}
     (\|\sigma'\|_{B_b(\defg)} + B \sqrt{T}) }{\sqrt{2\alpha}}  
     \| \Phi\|_{\cSM_2([0, T))}   \left ( \frac{T-a}{T}\right )^\alpha  \Phi_a.
\end{align*}
Because of \eqref{eqn:representation_Z} item \eqref{it:a} follows.

\eqref{it:b} We further conclude that the martingale
$(\int_{(0,t]}  \left ( \frac{T-u}{T}\right )^\alpha Z_u \sigma'(Y_u) \od W_u)_{t\in [0,T)}$ converges in $L_q$ and a.s. because of
$\Phi\in \cSM_2([0,T))$, $\sup_{t\in [0,T)} \Phi_t \in L_q$,
and \eqref{eqn:2:statement:Lp-BMO_new} of \cref{statement:Lp-BMO_new}. Again by
\eqref{eqn:2:statement:Lp-BMO_new}, the non-negative and non-decreasing process
$\left (\int_0^t  \left ( \frac{T-u}{T}\right )^\alpha |Z_u b_u| \od u \right )_{t\in [0,T)}$ 
converges in $L_q$ and a.s. For this reason
$\left (\int_0^t  \left ( \frac{T-u}{T}\right )^\alpha Z_u b_u \od u \right )_{t\in [0,T)}$
converges in $L_q$ and a.s. as well. 
So, again using \eqref{eqn:representation_Z}, item \eqref{it:b} follows.
\smallskip

\eqref{it:c}
This part follows from the proof of \eqref{it:b} for $q=2$ as $\Phi\in \cSM_2([0,T))$ gives $\sup_{t\in [0,T)} \Phi_t \in L_2$.
\end{proof}
\smallskip

\begin{lemm}
\label{lemma:varphi_martingale}
Let $\od \hat{\p}:= L \od \p$  with
$L:=\exp \left (\int_{(0,T]} \sigma'(Y_t) \od W_t - \frac{1}{2} \int_{(0,T]} |\sigma'(Y_t)|^2 \od t\right )$
and $g \in C_Y$. Then the process $(\varphi(t,Y_t))_{t\in [0,T)}
=(\frac{\partial G}{\partial y}(t,Y_t))_{t\in [0,T)}$ is a $\hat{\p}$-martingale.
\end{lemm}              

\begin{proof}
Applying the PDE from \eqref{item:2:statement:facts_from_GG04}
we get that
\[ \frac{\partial \varphi}{\partial t}(t,y) + (\sigma \sigma')(y) \frac{\partial \varphi}{\partial y}(t,y) +
   \frac{\sigma(y)^2}{2} \frac{\partial^2 \varphi}{\partial y^2}(t,y) =
   \frac{\partial}{\partial y} \left [ \frac{\partial G}{\partial t}(t,y) + \frac{\sigma(y)^2}{2} \frac{\partial^2 G}{\partial y^2}(t,y)
                              \right ]
  = 0 \]
on $[0,T)\times \defg$. By It\^o's formula this implies that
\[ \varphi(t,Y_t) = \varphi(0,y_0) + 
                    \int_{(0,t]} \left( \sigma \frac{\partial \varphi}{\partial y}\right) (u,Y_u) \left [ \od W_u - \sigma'(Y_u) \od u \right ] \mbox{ a.s.}\]
for $t\in [0,T)$. Because of \eqref{item:1a:statement:facts_from_GG04} and Girsanov's theorem
we obtain a $\hat{\p}$ standard Brownian motion 
$\hat{W}_t := W_t - \int_{(0,t]} \sigma'(Y_u) \od u$,
$t\in [0,T]$. Moreover, for $ t\in [0,T)$, 
$p,q\in (1,\infty)$, $1=\frac{1}{p}+\frac{1}{p'} = \frac{1}{q}+\frac{1}{q'}$, and with $p'q'=2$ we have that 
\begin{align*}
&\E^{\hat{\p}}  \left ( \int_0^t \left | \left(\sigma\frac{\partial \varphi}{\partial y}\right )(u,Y_u)  \right |^2 \od u \right )^\frac{1}{2} \\
&\leqslant (\E^\p L^p )^\frac{1}{p} 
           \left ( \E^\p \left ( \int_0^t \left | \left( \sigma\frac{\partial \varphi}{\partial y}\right)(u,Y_u)  \right |^2 \od u 
                         \right )^\frac{p'}{2} \right )^\frac{1}{p'} \\
&\leqslant (\E^\p L^p )^\frac{1}{p} 
           \left ( \E^\p\left ( \sup_{u\in [0,T]} \sigma_u^{-p'}\right )
           \left ( \int_0^t \left | \left( \sigma^2\frac{\partial \varphi}{\partial y}\right)(u,Y_u)  \right |^2 \od u \right )^\frac{p'}{2} 
           \right )^\frac{1}{p'} \\
&\leqslant (\E^\p L^p )^\frac{1}{p} 
           \left ( \E^\p\left ( \sup_{u\in [0,T]} \sigma_u^{-p'}\right )^q \right )^\frac{1}{p'q}
           \left ( \E \left ( \int_0^t \left | \left( \sigma^2\frac{\partial \varphi}{\partial y}\right)(u,Y_u)  \right |^2 \od u \right )^\frac{p'q'}{2} 
           \right )^\frac{1}{p'q'} \\
& = (\E^\p L^p )^\frac{1}{p} 
           \left ( \E^\p\left ( \sup_{u\in [0,T]} \sigma_u^{-p'}\right )^q \right )^\frac{1}{p'q}
           \left ( \E \int_0^t \left | \left( \sigma^2\frac{\partial \varphi}{\partial y}\right)(u,Y_u)  \right |^2 \od u  
           \right )^\frac{1}{2}.
\end{align*}
The last term is finite because of \eqref{item:4:statement:facts_from_GG04}, 
the first term is finite as $\sigma'$ is bounded, the second term is finite in the case $(C1)$, but also 
finite in the case $(C2)$ as then $\sigma_u\sim \e^{X_u}$ and $\widehat{\sigma}$ and $\widehat {b}$ are bounded. 
As by the Burkholder-Davis-Gundy inequalities applied to continuous local martingales we also have
\[ \E^{\hat{\p}}  \left | \int_{(0,t]} \left | \left(\sigma\frac{\partial \varphi}{\partial y}\right)(u,Y_u)  \right |^2 \od u \right |^\frac{1}{2}
   \sim_c \E^{\hat{\p}} \sup_{s\in [0,t]} \left | \int_{(0,s]} \left( \sigma\frac{\partial \varphi}{\partial y}\right)  (u,Y_u)  \od\hat{W}_u \right |
  \]
for some absolute constant $c\geqslant 1$ and $t\in [0,T)$, we get that 
$(\varphi(t,Y_t))_{t\in [0,T)}$ is a $\hat{\p}$-martingale.
\end{proof}

\subsection{Proof of \cref{statement:oscillation_brownian_case}}
According to \cref{lemma:varphi_martingale} there is an equivalent probability measure $\hat{\p}\sim\p$ such that 
$(\varphi(t,Y_t))_{t\in [0,T)}$ is a $\hat{\p}$-martingale.
The transition density of $Y$ under $\p$ computes as
\begin{equation}\label{eq:transition_density_Y}
\Gamma_Y(s,t;y_1,y_2) = \frac{1}{y_2} \Gamma_X(s,t;\ln(y_1),\ln(y_2))
\end{equation}
in the case (C2), otherwise $\Gamma_Y=\Gamma_X$, where $\Gamma_X$ is taken from
\cref{statement:friedman} in both cases.
We conclude by \cref{exam:oscillation_Markov_process},
where relation \eqref{eqn:derived_from_Markov_SDE} follows from \cref{statement:friedman},
the uniqueness in law of the SDE \eqref{eqn:sde:X}, and the theory of 
Markov processes.
\qed
        
\subsection{Proof of \cref{thm:pde_for_theta_equals_one}}
\eqref{item:1:thm:pde_for_theta_equals_one} $\Rightarrow$ \eqref{item:2:thm:pde_for_theta_equals_one}
We may assume that $g \colon \defg \to \R$ is Lipschitz. By  \cref{lemma:estimates_partialG}
we have
\[           \left | \frac{\partial G}{\partial y}(u,y) \right | 
\leqslant c_{\eqref{eq:lemma:estimates_partialG}} |g|_1
\sptext{1}{and}{1}
\left | Z_u \right | 
\leqslant c_{\eqref{eq:lemma:estimates_partialG}} |g|_1 \sigma_u
\sptext{1}{for}{1} (u,y) \in [0,T) \times \defg
\]
Let $0 \leqslant a < t <T$. From \cref{lemma:ito_brownian_case_new} with $\alpha=0$ we get that
\[ Z_t  =   Z_a + \int_{(a,t]}  H_u \od W_u  + \int_{(a,t]}  \sigma'(Y_u) Z_u \od W_u  
+ \frac{1}{2} \int_{(a,t]} (\sigma \sigma{''})(Y_u) Z_u \od u \quad \mbox{a.s.} \]
Then one has, a.s.,
\begin{align*}
&    \sqrt{\ce{\F_a}{\int_a^t  H_u^2 \od u}}\\
&   \leqslant  \sqrt{\ce{\F_a}{|Z_t-Z_a|^2}} 
+ \|\sigma'\|_{B_b(\defg)} \sqrt{\ce{\F_a}{\int_a^t  Z_u^2 \od u}}
+ \frac{1}{2} \| \sigma \sigma{''}\|_{B_b(\defg)} \sqrt{\ce{\F_a}{\left | \int_a^t  |Z_u| \od u\right |^2}} \\
&   \leqslant  \sqrt{\ce{\F_a}{|Z_t-Z_a|^2}} 
+ \left [ \|\sigma'\|_{B_b(\defg)} + \frac{\sqrt{T}}{2}  \| \sigma \sigma{''}\|_{B_b(\defg)} \right ]
\sqrt{\ce{\F_a}{\int_a^t  Z_u^2 \od u}} \\
&   \leqslant  c_{\eqref{eq:lemma:estimates_partialG}} |g|_1  \left [ \sqrt{\ce{\F_a}{\sigma_t^2}} + \sigma_a \right ] 
+  c_{\eqref{eq:lemma:estimates_partialG}}  |g|_1 \left [ \|\sigma'\|_{B_b(\defg)} + \frac{\sqrt{T}}{2}  \| \sigma \sigma{''}\|_{B_b(\defg)} \right ] \sqrt{T}
\sqrt{\ce{\F_a}{\sup_{u\in (a,T]} \sigma_u^2}} \\
&   \leqslant  c_{\eqref{eq:lemma:estimates_partialG}} |g|_1
\left [ 2+  \sqrt{T} \|\sigma'\|_{B_b(\defg)} + \frac{T}{2}  \| \sigma \sigma{''}\|_{B_b(\defg)} \right ]
\sqrt{\ce{\F_a}{\sup_{u\in [a,T]} \sigma_u^2}} \\
&   \leqslant  c_{\eqref{eq:lemma:estimates_partialG}} |g|_1
\left [ 2+  \sqrt{T} \|\sigma'\|_{B_b(\defg)} + \frac{T}{2}  \| \sigma \sigma{''}\|_{B_b(\defg)} \right ]
\|\sigma\|_{\cSM_2([0, T])} \sigma_a
\end{align*}
and hence
\begin{equation}\label{eqn:upper_bound_for_H_Lip}
          \sqrt{\ce{\F_a}{|M_t-M_a|^2}}
     =    \sqrt{\ce{\F_a}{\int_{(a,t]}  H_u^2 \od u}}
\leqslant c_{\eqref{eqn:upper_bound_for_H_Lip}} |g|_1 \|\sigma \|_{\cSM_2([0, T])}
          \sigma_a \mbox{ a.s.,}
\end{equation}
for some $c_{\eqref{eqn:upper_bound_for_H_Lip}}>0$. Applying 
\cref{lemma:estimates_partialG} for $\theta=1$ and 
\eqref{eqn:upper_bound_for_H_Lip} (together with Doob's maximal inequality)
to \cref{statement:abstract:pde_new}
for $\Phi_a=c\sigma_a$ for some appropriate $c>0$ and $\theta=1$
(note that $(T-a)/(T-\underline{a}(\tau)) \leqslant 1$),
we derive \eqref{item:2:thm:pde_for_theta_equals_one}.
\medskip

\eqref{item:2:thm:pde_for_theta_equals_one} $\Rightarrow$ \eqref{item:1:thm:pde_for_theta_equals_one} 
Given $a\in (0,T)$, exploiting \eqref{eq:2:item:2:statemtent:equivalence_simplified_setting_new}
of Theorem \ref{statemtent:equivalence_simplified_setting_new} and
\eqref{eq:err_vs_square_function} give 
\begin{equation}
\label{eqn:local_boundedness_theta_one}
\sup_{s\in [0,a]}
\frac{T-a}{T-s} |\varphi_a-\varphi_s |^2 \leqslant c_{\eqref{eqn:local_boundedness_theta_one}}^2 \mbox{ a.s.}
\end{equation}
For $a\in \left (\frac{T}{2},T\right )$ we choose $s\in (0,a)$ such that 
$\frac{T-a}{T-s} = \frac{1}{2}$. Therefore we may continue to 
\[           \left | \frac{\partial G}{\partial y}(a,y_a) \right | 
\leqslant \left | \frac{\partial G}{\partial y}(s,y_s) \right |
+ \sqrt{2}  c_{\eqref{eqn:local_boundedness_theta_one}}
\sptext{1}{for all}{1} y_a,y_s\in \defg\]
where we use the positivity and continuity of the transition density 
$\Gamma_Y$ (for (C2) see \eqref{eq:transition_density_Y}) to get
$\supp((Y_s,Y_a))=\defg\times\defg$ and the 
continuity of $\frac{\partial G}{\partial y}(t,\cdot):\defg\to\R$ for $t\in [0,T)$.
Applying \cref{lemma:varphi_martingale}, we have $\E^{\hat{\p}} \varphi(s,Y_s) = \varphi(0,Y_0)$ for $s\in [0,T)$.
Therefore, for each $s\in [0,T)$ there are $\omega_s^0,\omega_s^1\in \Omega$ such that for
$y_s^i:= Y_s(\omega_s^i)\in \defg$ we have 
$\varphi(s,y_s^0) \leqslant \varphi(0,Y_0) \leqslant \varphi(s,y_s^1)$.
Because $y\to \frac{\partial G}{\partial y}(s,y)$ is continuous on $\defg$ we find a $y_s\in \defg$ such that 
$\varphi(s,y_s) = \varphi(0,y_0)$. Therefore,
\begin{equation}\label{eq:Lipschitz_constant_gradient}
             \left | \frac{\partial G}{\partial y}(a,y) \right | 
   \leqslant \left | \frac{\partial G}{\partial y}(0,y_0) \right | + 
             \sqrt{2}  c_{\eqref{eqn:local_boundedness_theta_one}} =: c_\eqref{eq:Lipschitz_constant_gradient}
   \sptext{1}{for all}{1} (a,y)\in (0,T)\times \defg.
\end{equation}
Let $\Omega_g\in\F$ be of measure one such that for all $\omega \in \Omega_g$ one has
\[ \lim_{t\uparrow T} G(t,Y_t(\omega)) = g(Y_T(\omega)). \]
Let $I_g := Y_T(\Omega_g) \subseteq \defg$. Then $g$ is Lipschitz  
on $I_g$ with Lipschitz constant $c_\eqref{eq:Lipschitz_constant_gradient}$, 
and since $I_g$ is dense in $\defg$, the function $g|_{I_g}$ can be extended to 
$\tilde g \colon \defg\to \R$ as a Lipschitz function on $\defg$. Moreover, 
$\p(\{ \omega \in \Omega : g(Y_T(\omega)) =  \tilde g(Y_T(\omega))\})
\geqslant \p(\Omega_g) = 1$. \qed
\smallskip

\subsection{Proof of \cref{statement:Z_like_a_martingale}}
From $\sup_{t\in [0,T)} \Phi_t \in L_q$, $\cI^\alpha Z-Z_0\in \BMO_2^\Phi([0,T))$, and 
\eqref{eqn:2:statement:Lp-BMO_new} of \cref{statement:Lp-BMO_new}
we deduce that $\sup_{t\in [0,T)}|\cI_t^\alpha Z|\in L_q$.
By \cref{statment:IZ_vs_HdW_new}\eqref{it:b} we conclude that $\sup_{t\in [0,T)} \left \| \cI^\alpha_t M \right \|_{L_q} < \infty$
and obtain from the martingale property the $L_q$- and a.s. convergence of $\cI^\alpha M$.
Now we can use \eqref{it:b} and \eqref{it:c} to obtain the $L_q$- and a.s. convergence of $(\cI^\alpha Z_t)_{t\in [0,T)}$.
\qed
\smallskip

\subsection{Proof of \cref{statement:brownian_case_theta_new}}
For $\alpha:=\frac{1-\theta}{2}$ \cref{statment:IZ_vs_HdW_new}\eqref{it:a} implies that 
\[ \cI^\alpha Z - Z_0\in \BMO_2^\Phi([0,T))
   \Leftrightarrow 
   \cI^\alpha M \in \BMO_2^\Phi([0,T)). \]
Now \eqref{item:1:statement:brownian_case_theta_new} $\Leftrightarrow$ \eqref{item:2:statement:brownian_case_theta_new}
follows from \cref{equivalence_martingale} with 
\eqref{eq:equivalence_local_part_theta(0,1)}.
\qed
\smallskip

\subsection{Proof of \cref{statement:Hoelder_new_new}}
\eqref{item:1:statement:Hoelder_new_new} 
We only need to check the case (C2) and in this case we have $\sigma(y)\sim y$ so that we can use
\cref{statement:regularity_sigma_varphi}\eqref{item:1:statement:regularity_sigma_varphi}.
\eqref{item:2:statement:Hoelder_new_new} follows directly from \cref{lemma:estimates_partialG}.
\medskip

\eqref{item:3a:statement:Hoelder_new_new} 
We fix $a\in [0,T)$, a set $A\in \F_a$ of positive measure.
First we observe that by \eqref{eqn:ass:random_measures} (applied to $s=a$ and with $b\uparrow T$), 
\cref{lemma:estimates_partialG} for $\theta=0$, and 
\cref{statement:regularity_sigma_varphi}, 
\begin{align*}
 \sqrt{\int_A \int_a^T (T-u) H_u^2 \od u \od \p} 
& \sim_{\sqrt{\kappa}} 
  \sqrt{\int_A \int_a^T \left | \varphi_u - \varphi_a \right |^2 
              \sigma_u^2 \od u \od \p}\\
& \leqslant 
  \sqrt{\int_A \int_a^T Z_u^2 \od u \od \p} +
  \sqrt{\int_A \varphi_a^2 \int_a^T \sigma_u^2 \od u \od \p}\\
&\leqslant \sqrt{\int_A g(Y_T)^2 \od \p} + 
           \sqrt{\int_A \left [c_\eqref{eq:lemma:estimates_partialG}^2 |g|_0^2  \sigma_a^{-2} (T-a)^{-1}      \right ] \!\!\!
                        \left [c_\eqref{eqn:statement:reverse_hoelder_for_sigma}^2 (T-a) \sigma_a^2 \right ] \od \p} \\
&\leqslant c_0 \| g \|_{B_b(\defg)} \sqrt{\p(A)}.
\end{align*}
On the other hand \eqref{eqn:upper_bound_for_H_Lip} gives 
\begin{equation*}
          \sqrt{\int_A \int_a^T H_u^2 \od u \od \p}
\leqslant c_{\eqref{eqn:upper_bound_for_H_Lip}} |g|_1 \| \sigma \|_{\cSM_2([0,T])}
           \sqrt{\int_A  \sigma_a^2 \od \p}. 
\end{equation*}
For the linear map $T: g\mapsto \left ( H_u \right )_{u\in [a,T)}$
we get
\begin{align}
           \left \| T : C_b^0(\R) \to L_2([a,T)\times A, ((T-\cdot) \lambda\otimes \p_A))    \right \| 
& \leqslant c_0, \label{eqn:interpol:0} \\
           \left \| T : \hoell{1}{0} \to L_2([a,T)\times A, \lambda\otimes \p_A)  \right \| 
& \leqslant c_1 \sqrt{\int_A  \sigma_a^2 \od \p_A}, \label{eqn:interpol:1}
 \end{align}
where $\p_A$ is the normalized restriction of $\p$ to $A$.
Applying the Stein-Weiss interpolation theorem \cite[Theorem 5.4.1]{Bergh:Loefstroem:76} to 
\eqref{eqn:interpol:0} and \eqref{eqn:interpol:1} yields
\begin{align}\label{eqn:interpol:theta}
    \left \| T : (C_b^0(\R), \hoell{1}{0})_{\theta,2}  \to L_2([a,T)\times A, ((T-\cdot)^{1-\theta} \lambda\otimes \p_A))
    \right \| & \leqslant c_{\eqref{eqn:interpol:theta}} \left (\int_A  \sigma_a^2 \od \p_A \right )^\frac{\theta}{2},
 \end{align}
with $c_{\eqref{eqn:interpol:theta}}:=  C c_0^{1-\theta} c_1^\theta$. In other words, we did prove
\begin{equation*}
\left ( \int_A \int_a^T (T-u)^{1-\theta} H_u^2 \od u \od \p_A\right )^\frac{1}{2}
\leqslant c_{\eqref{eqn:interpol:theta}} \left (\int_A  \sigma_a^2 \od \p_A \right )^\frac{\theta}{2}
\|g\|_{\hoelO{\theta,2}}.
\end{equation*}
For $\delta \in (0,1)$ and $l\in \Z$ define $A_l := \{ \delta^{l+1} < \sigma_a^2 \leqslant \delta^l \}$. 
Then
\begin{align*}
    \int_A \int_a^T (T-u)^{1-\theta} H_u^2 \od u \od \p_A
& = \sum_{\p(A\cap A_l)>0} \left ( \int_{A\cap A_l} \int_a^T (T-u)^{1-\theta} H_u^2 \od u \od \p_{A\cap A_l}\right ) \p_A(A\cap A_l)\\
& \leqslant c_{\eqref{eqn:interpol:theta}}^2 \sum_{\p(A\cap A_l)>0}  \left (\int_{A\cap A_l}  \sigma_a^2 \od \p_{A\cap A_l} \right )^\theta \p_A(A\cap A_l)
            \,\,\|g\|_{\hoelO{\theta,2}}^2 \\
& \leqslant c_{\eqref{eqn:interpol:theta}}^2 \sum_{\p(A\cap A_l)>0}  \delta^{l\theta}  \p_A(A\cap A_l) 
            \,\,\|g\|_{\hoelO{\theta,2}}^2 \\
& \leqslant c_{\eqref{eqn:interpol:theta}}^2 \delta^{-\theta} \int_A \sigma_a^{2\theta} \od \p_A
            \,\,\|g\|_{\hoelO{\theta,2}}^2.
\end{align*}
As $\delta \in (0,1)$ was arbitrary, we derive
\begin{equation*}
          \int_A \int_a^T (T-u)^{1-\theta} H_u^2 \od u \od \p_A
\leqslant c_{\eqref{eqn:interpol:theta}}^2 \int_A  \sigma_a^{2\theta} \od \p_A \,\,\|g\|_{\hoelO{\theta,2}}^2
\end{equation*}
and, for $0 \leqslant a \leqslant t < T$,
\begin{equation*}
          T^{1-\theta} \ce{\F_a}{\left | \cI^{\frac{1-\theta}{2}}_t M - \cI^{\frac{1-\theta}{2}}_a M\right |^2}
\leqslant \ce{\F_a}{\int_a^T (T-u)^{1-\theta} H_u^2 \od u }
\leqslant c_{\eqref{eqn:interpol:theta}}^2 \sigma_a^{2\theta} \|g\|_{\hoelO{\theta,2}}^2 \mbox{ a.s.}
\end{equation*}
by \eqref{eqn:cIalphaM_vs_Malpha}.
We use \cref{statment:IZ_vs_HdW_new}\eqref{it:a}
to conclude $\left ( \cI_t^\frac{1-\theta}{2} Z -Z_0\right )_{t\in [0,T)}\in \BMO_2^{\sigma(\theta)}([0,T))$
and finish by $\sigma(\theta)\in \cSM_p([0,T))$ and \cref{statement:Lp-BMO_new}\eqref{item:1:statement:Lp-BMO}
for $p\in (2,\infty)$, whereas $p\in (0,2]$ is obvious.

\eqref{item:3b:statement:Hoelder_new_new} 
The previous part combined with Doob's maximal inequality implies
\[ \ce{\F_a}{\sup_{t\in [a,T)} \left | \cI^{\frac{1-\theta}{2}}_t M - \cI^{\frac{1-\theta}{2}}_a M\right |^2}
\leqslant 4 T^{\theta-1} 
c_{\eqref{eqn:interpol:theta}}^2 \sigma_a^{2\theta} \|g\|_{\hoelO{\theta,2}}^2 \mbox{ a.s.}\]
Moreover, 
\begin{align*}
          \frac{T-a}{(T-\underline{a}(\tau))^\theta} \left | \varphi_a-\varphi_{\underline{a}(\tau)} \right |^2\sigma^2_a 
&\leqslant 
         2 (T-a)^{1-\theta}  (\varphi_a^2 +\varphi_{\underline{a}(\tau)}^2) \sigma^2_a \\
&\leqslant 
         2 (T-a)^{1-\theta} c^2 \left ( \sigma_a^{2(\theta-1)} (T-a)^{\theta-1}
                               +        \sigma_{\underline{a}(\tau)}^{2(\theta-1)} (T-\underline{a}(\tau))^{\theta-1} \right ) \sigma_a^2 \\
&\leqslant 
         2 c^2 \left ( \sigma_a^{2(\theta-1)} + \sigma_{\underline{a}(\tau)}^{2(\theta-1)} \right ) \sigma_a^2.
\end{align*}
Now we finish with \cref{statement:abstract:pde_new}.
\qed

\subsection{Proof of \cref{statement:intro:tail_hoeleta2}}
\label{proof:statement:intro:tail_hoeleta2}
By \cref{thm:pde_for_theta_equals_one} and \cref{statement:Hoelder_new_new}\eqref{item:3b:statement:Hoelder_new_new}
we get constants $c_\theta,c_1>0$ such that 
\[  \| E(g^{(1)};\tau)\|_{\BMO_2^{\sigma}([0,T))} \leqslant c_1 \sqrt{\| \tau\|_1}
    \sptext{1}{and}{1}
\| E(g^{(\theta)};\tau)\|_{\BMO_2^{\Phi(\tau,\theta)}([0,T))} \leqslant c_\theta \sqrt{\| \tau\|_\theta}
       \]
for all $\tau\in \cT$. Using $\|\tau\|_1 \leqslant T^{1-\theta}\|\tau\|_\theta$, $\overline{\Phi} := \sigma+\Phi(\tau,\theta)$, 
and $\overline{c}:= c_\theta \vee (T^{\frac{1-\theta}{2}} c_1)$ we derive
\[  \| E(g;\tau)\|_{\BMO_2^{\overline{\Phi}}([0,T))} \leqslant \overline{c} \sqrt{\| \tau\|_\theta}.\]
Fix $a\in [0,T)\cap\tau$ and $B\in \F_a$ of positive measure. Then
\[ \left \| \left ( \frac{E_t(g;\tau)-E_a(g;\tau)}{\overline{d}(\sigma_a \vee \sigma_a^\theta)} \1_B \right )_{t\in [a,T)}
    \right \|_{\BMO_2^{\left ( \frac{\overline{\Phi}_t \1_B}{\overline{d}(\sigma_a \vee \sigma_a^\theta)} \right )_{t\in [a,T)}} ([a,T))}
   \leqslant   \overline{c} \sqrt{\| \tau\|_\theta}. \]
As in \cref{rema:statement:Hoelder_new_new} with the same notation one checks  
$\sup_{t\in [a,T)} \overline{\Phi}_t\leqslant \overline{d} (\sigma_a \vee \sigma_a^\theta) R_a$ a.s.
for some $\overline{d}>0$. Finally from \eqref{eqn:1:statement:Lp-BMO_new} of \cref{statement:Lp-BMO_new}
we deduce for $\mu >0$ and $\nu \ge 1$ that
\begin{align*}
           \p_B\left (\sup_{t\in [a,T)} \left |\frac{E_t(g;\tau)-E_a(g;\tau)}{\overline{d}(\sigma_a \vee \sigma_a^\theta)}\right | 
            > b  (\overline{c} \sqrt{\|\tau\|_\theta}) \mu \nu \right ) 
&\leqslant \e^{1-\mu} + \beta \p_B\left (\sup_{t\in [a,T)}\frac{\overline{\Phi}_t}{\overline{d}(\sigma_a \vee \sigma_a^\theta)}
           >  \nu \right ) \\
& \leqslant \e^{1-\mu} + \beta \p_B\left (R_a >  \nu \right ) \\
& \leqslant \e^{1-\mu} + \beta A \e^{-\left ( \frac {\ln \nu}{A} \right )^2},
\end{align*}
where $b,\beta>0$ are absolute constants, the last inequality follows from 
\cref{eqn:tail_RAa}, and $\p_B$ is the normalized restriction of $\p$ to $B$.
Now we choose $\mu=\nu=\sqrt{\lambda}$ and estimate the right-hand side by
$\p(B) \leqslant \overline{A}  \e^{-\left ( \frac {\ln \lambda}{\overline{A}} \right )^2}$ for
$\lambda \geqslant \lambda_0 \geqslant 1$.
\qed

\begin{rema}\label{statement:smoothness_powered_option}
Given $K>0$, we decompose $g_\gamma(y):= ((y-K)^+)^\gamma$ into
\[ g_\gamma(y) =  [g_\gamma(y)\wedge 1] + [g_\gamma(y)- (g_\gamma(y)\wedge 1)] 
=: g_\gamma^{(b)}(y) + g_\gamma^{(1)}(y) ,\]
so that $g_\gamma^{(b)} \in \hoelO {\gamma,\infty} \cap C_b^0(\R)$ and $g_\gamma^{(1)}\in \hoell{1}{0}$. 
Therefore, $g_\gamma\in \hoelO{\theta,1} + \hoell{1}{0}$ by
\eqref{eqn:inclusion_Hoelder_spaces}.
\end{rema}


\section[Interpolation]{A general interpolation result}
\label{sec:interpolation}

The following interpolation result is designed to prove \cref{statement:upper_bounds_Drho}, but of independent interest.
For this section we assume
\begin{enumerate}
\item $\kappa_0,\kappa_1\in [0,1]$ and $0\leqslant \gamma_0 < \gamma_1 < \infty$, 
\item a finite measure space $(R,\cR,\pi)$ with $\pi(R)>0$,
\item an interpolation couple of Banach spaces $(E_0,E_1)$ and a Banach space $E$, 
\item measurable maps $A_0,A_1:R\to [0,\infty]$ with 
      \[ \int_R (A_i \wedge w) \od \pi \leqslant c_i w^{1-\kappa_i} 
         \sptext{1}{for}{1}
         w \geqslant T^{-\gamma_i},\]
\item for $(t,r)\in [0,T)\times R$ linear operators $T_{t,\pi},T_{t,r} : E_0+E_1 \to E$ such that
      \begin{enumerate}
      \item $\| T_{t,r} x \|_E \leqslant  d_i \min \{ A_i(r), (T-t)^{-\gamma_i} \} \| x\|_{E_i}$ for $x\in E_i$ and $r\in R$, 
      \item if $\|T_{t,\cdot}x \|_E \leqslant P(\cdot)$ on $R$, where $P:R\to [0,\infty)$ is measurable and $x\in E_0+E_1$, then
            $\|T_{t,\pi}x \|_E \leqslant \int_R P(r) \pi(\od r)$,
      \end{enumerate}
\item \label{item:6:general_interpolation}
      $\|T_{s,\pi}x \|_E \leqslant  \|T_{t,\pi}x \|_E$ for all $0\leqslant s < t< T$ and $x\in E_0+E_1$,
\end{enumerate}
where $c_0,c_1,d_0,d_1>0$ are constants. 
Note that the map $[0,T)\ni t \mapsto  \| T_{t,\pi}x\|_E$ is measurable by assumption \eqref{item:6:general_interpolation}.
Under the above assumptions the following statement holds:
\bigskip

\begin{theo}\label{statement:abstract_interpolation_new}
Assume that $(1-\kappa_0)\gamma_0 \not = (1-\kappa_1)\gamma_1$. Then 
for all $(\delta,q)\in (0,1)\times [1,\infty]$ there is a 
$c_\eqref{eqn:statement:abstract_interpolation} =
 c_\eqref{eqn:statement:abstract_interpolation}(\kappa_0,\kappa_1,\gamma_0,\gamma_1,\delta,q)>0$
such that, for  $\alpha:= (1-\delta)(1-\kappa_0)\gamma_0 + \delta (1-\kappa_1)\gamma_1$,
\begin{equation}\label{eqn:statement:abstract_interpolation}
 \big \| t\mapsto (T-t)^\alpha \| T_{t,\pi}  x \|_E \big \|_{L_q([0,T),\frac{\od t}{T-t})}
  \leqslant c_\eqref{eqn:statement:abstract_interpolation} (c_0 d_0)^{1-\delta} (c_1 d_1)^{\delta} \| x\|_{(E_0,E_1)_{\delta,q}} 
  \sptext{1}{for}{1}
  x\in (E_0,E_1)_{\delta,q}.
\end{equation}
\end{theo}
\smallskip

\begin{proof}
For $t\in [0,T)$ we observe that 
\[  \| T_{t,\pi}x \|_E
    \leqslant d_i \int_R \min \{ A_i(r), (T-t)^{-\gamma_i} \} \pi(\od r) \| x\|_{E_i} 
    \leqslant d_i c_i (T-t)^{(\kappa_i-1)\gamma_i}  \| x\|_{E_i}.
\]
For $t_k := T(1-\frac{1}{2^k})$, $k\in \bN_0$, and $\alpha_i := (1-\kappa_i)\gamma_i$ this gives 
\begin{equation}\label{eqn:interpolation_start}
\|  (T_{t_k,\pi}x)_{k\in \bN_0} \|_{\ell_\infty^{-\alpha_i}(E)}  \leqslant c_i d_i T^{-\alpha_i} \| x\|_{E_i}.
\end{equation}
Using real interpolation, \eqref{eq:interpol_ell_q^s_constant}, and \eqref{eq:functor_with_constant_1_real_interpolation}, we derive
\[ \|  (T_{t_k,\pi}x)_{k\in \bN_0} \|_{\ell_q^{-\alpha}(E)}  
   \leqslant c_{\alpha_0,\alpha_1,\delta,q} \, (c_0 d_0)^{1-\delta} (c_1 d_1)^{\delta} \, T^{-\alpha} \| x\|_{(E_0,E_1)_{\theta,q}}. \]
The assertion follows, because assumption \eqref{item:6:general_interpolation} implies that 
\[ \big \| t\mapsto (T-t)^{\alpha} \| T_{t,\pi} x \|_E \big \|_{L_q([0,T),\frac{\od t}{T-t})}
   \sim_c  T^{\alpha} \|  (T_{t_k,\pi} x)_{k\in \bN_0} \|_{\ell_q^{-\alpha}(E)}, \]
where $c\geqslant 1$ depends at most on $(\alpha,q)$.
\end{proof}
\bigskip


\section[Gradient estimates: L\'evy setting]{L\'evy setting: Directional gradient estimates and applications}
\label{sec:application_levy_case}

\subsection{Setting}
\label{sec:setting_levy_case}
Let $X=(X_t)_{t \in [0,T]}$ be a  L\'evy process defined on a complete probability space $(\Omega, \F, \p)$, i.e. $X_0\equiv 0$, 
$X$ has stationary and independent increments, and c\`adl\`ag trajectories. Assume that 
$\bF=(\F_t)_{t \in [0,T]}$ is the augmented natural filtration of $X$ and $\F=\F_T$. The Poisson random measure $N$ associated to $X$ is defined 
by $N(E) := \#\{t\in (0, T] : (t,\Delta X_t) \in E\}\in \bN_0\cup\{\infty\}$
for $E\in \cB((0,T]\times \R)$.
For $B\in \cB(\R)$ with $(-\varepsilon,\varepsilon)\cap B =\emptyset$ for some $\varepsilon>0$ we define 
$\nu(B) := \frac{1}{T}\E N((0, T] \times B)$. 
By $\varepsilon\downarrow 0$ this measure has a unique extension to a $\sigma$-finite Borel measure $\nu$ on $\cB(\R)$ such that $\nu(\{0\})=0$
and $\int_\R (x^2 \wedge 1) \nu (\od x) < \infty$. The measure $\nu$ is the L\'evy measure associated to 
$X=(X_t)_{t \in [0,T]}$.
Let $\sigma \geqslant 0$ 
be the coefficient of the standard Brownian motion $W$ in the L\'evy--It\^o decomposition of $X$ (see, e.g., \cite[Theorem 19.2]{Sa13}). 
We define the $\sigma$-finite measure $\mu$ on $\mathcal B(\R)$ by 
\begin{align} \label{eqn:mu}
\mu(\od x) :=\sigma^2 \delta_{0}(\od x)+x^2\nu(\od x).
\end{align}
For this section we assume the following setting:
\begin{enumerate}[(1)]
\item \label{it:1:L}
      $X=(X_t)_{t\in [0,T]}$ is a L\'evy process with $\mu(\R) \in (0,\infty]$.
\item We use the description $\supp(X_t)=Q+\ell t$, 
      $t\in (0,T]$, from \cref{exam:oscillation_Levy_process} and
      \[ Y_t := (X_t -\ell t) \1_{\{ X_t \in \supp (X_t) \}}
         \sptext{1}{for}{.5} t\in [0,T]. \]
     As outlined in \cref{exam:oscillation_Levy_process},
     the set $Q$ is unique, and $\ell$ is unique as well if  
     $Q\not = \R$. In the case $Q=\R$ we let $\ell:=0$ so that 
     $Y_t=X_t$.
\item \label{it:3:L}
      $\rho$ is a probability measure on $\cB(\R)$.
\end{enumerate}


\bigskip
\subsection{Definition of $D_\rho F(t,x)$}
\label{subsec:Drho}

To define the operator $D_\rho$ and its range of definition we first recall a class of functions that are of local bounded variation:
\smallskip

\begin{defi}
\label{definition:bvloc}
A Borel function $f:\R\to\R$ belongs to $\bvlocal$ provided that $f$ is right-continuous and
there are Borel measures $\mu_f^+$ and $\mu_f^-$ on
$\cB(\R)$, finite on each compact interval, and disjoint $S^+,S^-\in \cB(\R)$ with $S^+\cup S^-=\R$ and $\mu_f^+(S^-)=\mu_f^-(S^+)=0$, such that
\[ f(b) -f(a) = \mu_f^+((a,b]) - \mu_f^-((a,b])
   \sptext{1.2}{for all}{.7}
   -\infty < a < b < \infty.\]
Furthermore, we let $|f'|:= \mu_f^+ + \mu_f^-$ and, for a Borel function
$g:\R\to \R$ with $\int_\R |g(x)| |f'|(\od x)<\infty$,
\[ \int_\R g(x) f'(\od x):=\int_\R g(x) \mu_f^+(\od x)-\int_\R g(x) \mu_f^-(\od x). \]
\end{defi}   
The pair of measures $(\mu^+_f,\mu_f^-)$ is unique and we will identify 
\begin{equation}\label{eqn:f'_measure}
f' = (\mu^+_f,\mu^-_f). 
\end{equation}
The space $\bvlocal$ consists of functions that are of bounded variation of on each compact interval
(cf. \cite[Chapter 8]{Rudin:real_and_complex:70}).
Now we can define the operator $D_\rho$ and its range of definition:

\begin{defi}
\label{statement:definition_Gamma}
\begin{enumerate}[(1)]
\item A Borel function $f\colon \R \to \R$ belongs to $\defX$ if  
      $\E|f(x+X_{s})|<\infty$ for all $(s,x)\in [0,T]\times \R$. For $f\in \defX$ we define
      $F \colon [0,T]\times \R\to\R$ by
      \begin{equation}\label{eqn:def_F}
      F(t,x) := \E f(x+X_{T-t}).
      \end{equation}

\item \label{item:2:statement:definition_Gamma}
      Given $f\in \defX$ such that 
      \[ \int_{\RO} \frac{|F(t,x+z) - F(t,x)|}{|z|} \rho(\od z)<\infty 
         \sptext{1}{for all}{1} (t,x)\in [0,T) \times \R, \]
      we define  $D_\rho F:[0,T)\times \R\to \R$ by
      \[ D_{\rho} F(t,x)  := \int_{\RO} \frac{F(t,x+z) - F(t,x)}{z} \rho(\od z). \]
\item  For $t\in [0,T)$ we define $\Gamma_{t,\rho}^0:\dom(\Gamma_{\rho}^0)\to \R$ by
\begin{align*}
\dom(\Gamma_{\rho}^0) & :=  \Bigg \{ f \in \defX : 
               \forall s\in [0,T) \,\, \forall \, 0\leqslant \delta \leqslant s <T \,\,
                              \forall x\in \R \mbox{ one has:} \\
                      &  \hspace*{8 em}  \E \int_{\RO} \left | \frac{F(s,x+X_\delta + z) 
                          - F(s,x+ X_\delta)}{z} \right |  \rho(\od z) < \infty \Bigg \}, \\
\langle f,\Gamma_{t,\rho}^0 \rangle &:= D_\rho F(t,0).
\end{align*}
\item For $t\in [0,T)$ we define the Borel function $\gamma_{t,\rho}:\R \to [0,\infty]$
      and $\Gamma_{t,\rho}^1:\dom(\Gamma_{\rho}^1)\to \R$ by
\begin{align*}
\gamma_{t,\rho}(v) & := \int_{\R\setminus \{ 0\}}  \frac{\p (X_{T-t} \in J(v;z))}{|z|} \rho(\od z)
                                \sptext{1}{with}{1} J(v;z):=  v + [-z^+,z^-), \\
   \dom(\Gamma_{\rho}^1) & := \Bigg \{ f\in \defX\cap \bvlocal \mbox{ and } \forall \, 0\leqslant \delta \leqslant s <T  \,\,
                               \forall x\in \R: \\
                         &  \hspace*{17em}  \E \int_\R \gamma_{s,\rho}(v-x-X_\delta) |f'|(\od v) < \infty \Bigg \},\\    
\langle f, \Gamma^1_{t,\rho} \rangle & = \langle f', \gamma_{t,\rho} \rangle := \int_\R \gamma_{t,\rho}(v) \, f'(\od v).
\end{align*} 
\end{enumerate}
\end{defi}
We shall use the following properties. As they are obvious or standard we omit their proofs:

\begin{lemm}
\label{statement:easy_properties_Gamma}
For $f\in \dom(\Gamma_\rho^0)$ the following holds:
\begin{enumerate}[{\rm (1)}] 
\item $D_\rho F(t,x) = \langle f(x+\cdot),\Gamma_{t,\rho}^0 \rangle$.
\item \label{item:2:statement:easy_properties_Gamma} 
      One has that $t\mapsto d(t):=\|D_\rho F(t,\cdot)\|_{B_b(\R)} \in [0,\infty]$ is non-decreasing. 
\item The process $(D_\rho F(t,X_t))_{t\in [0,T)}$ is a martingale.   
\item \label{item:4:statement:easy_properties_Gamma} 
      There exists a c\`adl\`ag modification $\varphi=(\varphi_t)_{t\in [0,T)}$ of $(D_\rho F(t,X_t))_{t\in [0,T)}$
      such that 
      \[ |\varphi_t| \leqslant d(t+) \sptext{1}{on}{1} [0,T) \times \Omega. \]  
\end{enumerate}
\end{lemm}
\medskip

\begin{conv}
If we consider $(D_\rho F(t,X_t))_{t\in [0,T)}$ as stochastic process,
then we always choose a c\`adl\`ag modification from \cref{statement:easy_properties_Gamma}\eqref{item:4:statement:easy_properties_Gamma}.
\end{conv}

In \cref{sec:upper_estimate_gradient_general_levy} we aim for estimates of type
\[    \| D_\rho F(t,\cdot)\|_{B_b(\R)} \le c(t) |f |_E 
   \sptext{1}{for all}{1}
   f \in E \]
where $E$ is a semi-normed space. This is achieved by the following 
\cref{statement:Bb_vs_computation_in_zero} and  \cref{def:Gamma_functional_semi-norm}:
\medskip

\begin{lemm}
\label{statement:Bb_vs_computation_in_zero}
Assume a linear space $E$ of functions $f:\R\to \R$ equipped with a semi-norm $|\cdot|_E$.
Define $E^0 := \{ g : \R\to \R \mbox{ with } g(x):=f(x_0+x) -f(x_0): f\in E, x_0\in \R\}$ and suppose that 
\[ E^0\subseteq E \subseteq \dom(\Gamma_\rho^0)
   \sptext{1}{with}{1}
  | x \mapsto f(x_0+x) -f(x_0) |_E = |f|_E.\]
Then, given $c>0$ and $t\in [0,T)$, one has 
\begin{align*}
   \| D_\rho F(t,\cdot)\|_{B_b(\R)} \le c |f |_E 
   \sptext{.4}{for all}{.4}
   f \in E
&\sptext{1.3}{if and only if}{1.3}
   | D_\rho F(t,0)| \le c |f |_E 
   \sptext{.4}{for all}{.4}
   f \in E^0 \\
&\sptext{1.3}{if and only if}{1.3}
   | D_\rho F(t,0)| \le c |f |_E 
   \sptext{.4}{for all}{.4}
   f \in E.
\end{align*}
\end{lemm}

The proof of \cref{statement:Bb_vs_computation_in_zero} is obvious so that we leave it out as well.

\medskip

\begin{defi}
\label{def:Gamma_functional_semi-norm}
For $t\in [0,T)$ and a linear space $E\subseteq \dom(\Gamma_{\rho}^{0})$ 
equipped with a semi-norm $|\cdot|_E$ we let
$\| \Gamma_{t,\rho}^{0}\|_{E^*} := \inf c$, where the infimum is taken over all $c>0$ such that 
\[ |D_\rho F(t,0)| = |\langle f,\Gamma_{t,\rho}^{0} \rangle|\leqslant c |f|_E
   \sptext{1}{for all}{1}
   f\in E. \]
\end{defi}
\medskip 


\subsection{Interpretation as directional gradient}
\label{subsec:directional_gradient}
In this section we explain that $D_\rho F(t,X_t)$, defined with 
\cref{statement:definition_Gamma}\eqref{item:2:statement:definition_Gamma},
is actually a directional Malliavin derivative. 
This fact is also behind \cref{sec:approximation_random_measure} 
where we consider orthogonal decompositions on the L\'evy-It\^o space into stochastic
integrals with a control of the singularity of the gradients appearing as integrands.
\smallskip
Assume that $f\in \defX$ with $\E|f(X_T)|^2<\infty$ and recall that $F:[0,T]\times \R\to \R$ is given by
$F(t,x):= \E f(x+X_{T-t})$. We  define the vector-valued gradient
\[ \partial F : [0,T) \times \R  \to \cL_0(\R,\cB(\R)) \]
by
\[  (\partial F(t,x))(z):= \frac{\partial F}{\partial x}(t,x) \1_{\{0\}}(z) + \frac{F(t, x + z) - F(t, x)}{z}\1_{\RO}(z) \]
where the first term on the right-hand side is omitted if $\sigma=0$ and is well defined if $\sigma>0$
(see \cref{statement:consistence-upper_bound_sigma>0}). 
\smallskip

Let us fix $t\in (0,T)$ for 
the remainder of \cref{subsec:directional_gradient}. We obviously 
have $F(t,X_t) = \ce{\F_t}{f(X_T)}$ a.s., but the gradient also satisfies for
$\lambda\otimes\mu\otimes\p$ almost all $(s,z,\omega)\in (0,T]\times \R\times \Omega$ the relation 
\[   \1_{(0,t]}(s) (\partial F(t,X_t(\omega)))(z)  
  =  (D_{s,z} F(t,X_t))(\omega)
\]
and both sides are square-integrable with respect to $\lambda\otimes\mu\otimes\p$.
This relation together with the notion of the  Malliavin derivative 
$(D_{s,z} \xi)_{(s,z)\in (0,T]\times \R}$ for a $\xi \in \D_{1,2}$
is recalled in \cref{sec:malliavin_calculus:basics}. 
\bigskip

Assume $\sigma=0$, $D\in L_1(\R,\mu) \cap L_2(\R,\mu)$ with $D\geqslant 0$ and $\int_\R D \od \mu=1$, and
$\od \rho := D\od \mu$. Then we verify in \eqref{eqn:integrated:eija_equation} below
that there is a null-set  $N_t$ such that outside this null-set one has
\[     \frac{1}{t} \int_0^t\int_{\RO} D_{s,z} F(t,X_t)  D(z) \mu(\od z) \od s
     = D_\rho F(t,X_t). \] 
\smallskip
There is also a connection to PDIEs: following \cite[Section 4.1]{Nualart:Schoutens:01} for $\sigma=0$ and $\int_\R |z| \nu(\od z) < \infty$
one has
\[ \frac{\partial F}{\partial t}(t,x) + \int_\R (F(t,x+z)-F(t,x)) \nu(\od z) = 0 \]
under certain regularity assumptions on $F$. This non-local equation constitutes the counterpart to the 
parabolic equation in \cref{statement:facts_from_GG04}\eqref{item:2:statement:facts_from_GG04}. As we do not 
use the above PDIE we leave out the technical details.


\subsection{Duality between  $\Gamma_{t,\rho}^0$ and $\Gamma_{t,\rho}^1$}
We clarify the relation between $\Gamma_{t,\rho}^0$ and $\Gamma_{t,\rho}^1$ and interpret 
the operator $\Gamma_{t,\rho}^\rho$ as a distribution in \eqref{eqn:Gamma_and_distributions}
(and we use the proof of the statement below in \cref{statement:lower_bounds_levy_new} as well).
\smallskip

\begin{theo}[Properties of the functional $\Gamma_{t,\rho}^1$]
\label{statement:properties:Gamma1}
Let $t \in [0, T)$.
\begin{enumerate}[{\rm (1)}]
\item \label{item:1:statement:properties:Gamma1}
      One has $\int_{\R} \gamma_{t,\rho}(v) \od v = \rho(\RO)$.
\item \label{item:2:statement:properties:Gamma1}
      One has $\dom(\Gamma^1_\rho)\subseteq \dom(\Gamma^0_\rho)$ and,
      for $f\in \dom(\Gamma^1_{t,\rho})$ and $x\in\R$, that
      \[ D_{\rho} F(t,x) = \langle f^x, \Gamma^1_{t,\rho} \rangle = \langle f^x  , \Gamma^0_{t,\rho} \rangle
         \sptext{1}{if}{1} f^x(\cdot) := f(\cdot +x). \]
\item \label{item:3:statement:properties:Gamma1}
      If $q,r\in [1,\infty]$ and $s := \min \{q,r\}$, and if $X_{T-t}$ has a density $p_{T-t}\in L_r(\R)$, then
      \[ \| \gamma_{t,\rho}\|_{L_{q}(\R)}
          \leqslant  \left  \|  p_{T-t} \right \|_{L_s(\R)}\int_{\RO} |z|^{\frac{1}{q}-\frac{1}{s}}  \rho(\od z). \]
\end{enumerate}
\end{theo}

\begin{proof}
Recall the notation $J(v; z) = v + [-z^+, z^-)$.
\eqref{item:1:statement:properties:Gamma1}
follows from
\[    \int_{\R} \gamma_{t,\rho}(v) \od v
    = \int_{\R\setminus \{ 0\}}  \int_\Omega \left [ \int_{\R}   \1_{\{X_{T-t}\in J(v;z)\}}\frac{1}{|z|} \od v \right ]  \od \p \rho(\od z) 
    = \int_{\R\setminus \{ 0\}}  \int_\Omega  \od \p \rho(\od z) 
    = \rho(\RO).
\]
 
\eqref{item:2:statement:properties:Gamma1} 
For $f\in \dom(\Gamma_\rho^1)$ and $x\in \R$ we observe that 
\begin{align*} 
     \int_{\R\setminus \{ 0\}} \left | \frac{F(t,x+z)-F(t,x)}{z} \right | \rho (\od z)  
&\leqslant \int_{\R\setminus \{ 0\}} \left [\int_\R \left | \frac{f(x+z+y)-f(x+y)}{z} \right | \p_{X_{T-t}}(\od y)\right ]  \rho (\od z)  \\
&\leqslant \int_{\R\setminus \{ 0\}} \int_\R \frac{1}{|z|} \int_{(x+y-z^-, x+y+z^+]} |f'|(\od v) \p_{X_{T-t}}(\od y)  \rho (\od z) \\
& =   \int_\R \left [ \int_{\R\setminus \{ 0\}} \frac{1}{|z|} \int_\R \1_{(x+y-z^-, x+y+z^+]}(v) \p_{X_{T-t}}(\od y) \rho (\od z) \right ] \\
& \hspace*{23em} |f'|(\od v) \\
& =   \int_\R \left [ \int_{\R\setminus \{ 0\}}  \frac{\p (X_{T-t} \in J(v-x;z))}{|z|}
      \rho (\od z) \right ] |f'|(\od v) \\
& =   \int_\R \gamma_{t,\rho}(v-x) |f'| (\od v)
\end{align*}
which implies $\dom(\Gamma_\rho^1) \subseteq \dom(\Gamma_\rho^0)$ and also enables us to compute,
exactly along the previous computation,
\begin{align*} 
     \int_{\R\setminus \{ 0\}} \frac{F(t,x+z)-F(t,x)}{z} \rho (\od z)  
& =   \int_\R \left [ \int_{\R\setminus \{ 0\}} \frac{\p (X_{T-t} \in J(v;z))}{|z|}
      \rho (\od z) \right ] (f^x)'(\od v).
\end{align*}
\eqref{item:3:statement:properties:Gamma1}
Let $z\not = 0$. Then the assertion follows from 
\begin{align*}
        \left \|  \p (X_{T-t} \in J(\cdot;z)) \right \|_{L_q(\R)}
& = |z| \left \|  \frac{1}{|z|}\int_{J(\cdot;z)} p_{T-t}(y) \od y \right \|_{L_q(\R)} \\
&\leqslant |z|^{1-\frac{1}{s}} \left \|  v \mapsto\left \| \1_{J(v;z)} p_{T-t} \right \|_{L_s(\R)} \right \|_{L_{q}(\R)} \\
&\leqslant |z|^{1-\frac{1}{s}} \left \|  y \mapsto \left \| \1_{J(\cdot;z)}(y) p_{T-t}(y) \right \|_{L_{q}(\R)}
                                                                                \right \|_{L_s(\R)} \\
& = |z|^{1-\frac{1}{s}} \left \| y \mapsto  \left \| \1_{J(\cdot;z)}(y) \right \|_{L_{q}(\R)} p_{T-t}(y) 
                                                                                \right \|_{L_s(\R)} \\
& = |z|^{1-\frac{1}{s}+\frac {1}{q}} \left \|  p_{T-t} \right \|_{L_s(\R)},
\end{align*}
where we use H\"older's inequality for the first inequality and \eqref{eqn:interchange_Lp_Lq} in the second one.
\end{proof}
\bigskip

In \cref{statement:properties:Gamma1} we proved $\int_\R \gamma_{t,\rho}(v)\od v = \rho(\RO)$,  
$\dom(\Gamma_{\rho}^1)\subseteq \dom(\Gamma_{\rho}^0)$, and that 
\[ \langle f', \gamma_{t,\rho} \rangle = \langle f  , \Gamma^0_{t,\rho} \rangle
\sptext{1}{for}{1} f\in \dom(\Gamma_{\rho}^1).\]
If $D(\R)$ is the space of test functions that consists of 
$f\in C^\infty(\R)$ with compact support, then $D(\R)\subseteq \dom(\Gamma_{\rho}^1)$
(for $f\in D(\R)$ we have $f'(\od v)=f'(v)\od v$ and $|f'|(\od v)=|f'(v)| \od v$, where $f'$ on the right-hand sides 
is the classical derivative).
If we consider $\gamma_{t,\rho}\in L_1(\R)$ as distribution 
$\gamma_{t,\rho}\in D'(\R)$ (see \cite[Section 6.11]{Rudin:FA}), then we have the interpretation
\begin{equation}\label{eqn:Gamma_and_distributions}
\partial \gamma_{t,\rho} = -\Gamma_{t,\rho}^0,
 \end{equation}
see \cite[Section 6.12]{Rudin:FA} and $\Gamma^0_{t,\rho}$ can be seen as distributional derivative of a distribution of $L_1$-type.
\smallskip


\subsection{Upper bounds for the gradient process}
\label{sec:upper_estimate_gradient_general_levy} 

Gradient estimates in the L\'evy setting are studied in different ways in the literature.
In \cite[Theorem 1.1 and Remark 2.4]{Chen:Hao:Zhang:20} H\"older regularities are studied, where one looks 
for an improvement of the H\"older regularity caused by the transition group. In a way, this is opposite to 
our question. 
 The result from the literature we contribute to is 
\cite[Theorem 1.3]{SSW12} (see \cref{remark:relation_to_SSW12} below).
Finally, \cite{Laukkarinen:19} investigates when $f(X_T)$ belongs to $\D_{1,2}$ or $(L_2,\D_{1,2})_{\theta,\infty}$ 
in dependence on $f\in \hoelO{\eta,\infty}$ and properties of the underlying L\'evy process $X$. In 
our article we look for $L_\infty$- and $\BMO$-bounds for vector-valued gradient processes generated by an $f(X_T)$ when 
$f\in \hoelO{\eta,q}$,
where we do not need and therefore not consider any Malliavin smoothness of $f(X_T)$ itself. Moreover, 
for a given $f(X_T)$ the fractional smoothness of the gradient process depends on the direction in which the 
gradient process is tested.
\smallskip

As a counterpart to \cref{statement:Hoelder_new_new} proved on the Wiener space
we shall prove \cref{statement:consistence-upper_bound_sigma>0} and \cref{statement:upper_bounds_Drho} in this section
as final results.
To start with, we introduce a variational quantity that is one key 
for us to obtain upper bounds for gradient processes:
\smallskip

\begin{defi}
\label{definition:TVrho}
For $\eta\in [0,1]$ and $s\in [0,T]$ we let
\[ \|X_s\|_{{\rm TV}(\rho,\eta)} := \inf_P \left \{ \int_{\RO} P(z)^{1-\eta} \rho(\od z) \right \}
   \in [0,\infty], \]
where the infimum is taken over all measurable $P:\RO\to [0,\infty)$ such that 
\[ \frac{\|\p_{z+{X_s}} - \p_{X_s}\|_{\rm TV}}{|z|} \leqslant P(z) \sptext{1}{for}{1} z\in \RO. \]
\end{defi}
\bigskip

We use the potentials $P$ to avoid a discussion about the measurability 
of the map $z\mapsto \|\p_{z+{X_s}} - \p_{X_s}\|_{\rm TV}$ (which would not be necessary for us).
We have the following special cases:
\begin{enumerate}
\item $ \|X_s\|_{{\rm TV}(\rho,1)}=\rho(\RO)<\infty$ for $s\in [0,T]$.
\item $ \|X_0\|_{{\rm TV}(\rho,\eta)}=2^{1-\eta}  \int_{\RO} |z|^{\eta-1} \rho(\od z)$ for $\eta \in [0,1]$.
\item $ \|X_s\|_{{\rm TV}(\delta_z,\eta)}=\left ( \frac{\|\p_{z+{X_s}} - \p_{X_s}\|_{\rm TV}}{|z|}\right )^{1-\eta}<\infty$,
      $\eta\in [0,1]$, if $\delta_z$ is the Dirac measure in $z\in \RO$.
\item $\|X_s\|_{{\rm TV}(\delta_0,\eta)}=0$ for $\eta\in [0,1]$.
\end{enumerate}

\begin{rema}
\begin{enumerate}
\item In the following we will exploit the behaviour of
      $\|X_s\|_{{\rm TV}(\rho,\eta)}$ for $s\in (0,T]$. This  enables us to obtain
      the correct blow-up of gradient processes when considering $\beta$-stable-like processes. 
\item The quantities $\|X_s\|_{{\rm TV}(\delta_z,0)}$ are well-studied in the literature in 
      connection to the coupling property of a L\'evy process $X$ (see \cite[Theorem 1.1]{Boetcher:Schilling:Wang:11}
      and \cite[Theorem 3.1]{SSW12}), \cref{statement:X-rho_to_TV} provides new upper bounds for our purpose.
\item At some places we only use '$\|X_s\|_{{\rm TV}(\rho,\eta)}<\infty$ for $s\in (0,T]$',
      which is satisfied in our relevant situations. However, there are examples where this fails: 
      let $(X_t)_{t\in [0,T]}$ be a Poisson process with drift and $\rho$ with $\rho((-1,1)\setminus \{0\})=1$.
      Then $\|X_s\|_{{\rm TV}(\rho,\eta)} = \int_{(-1,1)\setminus \{0\}} \left ( 2/|z| \right )^{1-\eta} \rho(\od z)$
      which might be infinite if $\eta \in [0,1)$.   
      A separate investigation for '$\|X_s\|_{{\rm TV}(\rho,\eta)}<\infty$ for $s\in (0,T]$' is not done here (for example,
      whether this property depends on $s$).
\end{enumerate}
\end{rema}
\smallskip

To start with, we introduce a natural modification of $\| \cdot \|_{\B_{\infty,q}^\alpha}$ from \cref{definition:B_pq^alpha}:

\begin{defi}\label{definition:B_bq^alpha} 
For $(\alpha,q)\in (0,\infty)\times [1,\infty]$ and $G:[0,T)\times \R \to \R$ with
$\|G(s,\cdot)\|_{B_b(\R)}  \leqslant \|G(t,\cdot)\|_{B_b(\R)}\in  [0,\infty]$ for
$0 \leqslant s < t < T$ we define
\begin{align*}
    \| G \|_{\B_{b,q}^\alpha}
&:= \left \|  t\mapsto (T-t)^\alpha \| G(t,\cdot)\|_{B_b(\R)} \right \|_{L_q([0,T),\frac{\od t}{T-t})}.
\end{align*}
\end{defi} 

In the case a martingale is of the form $\varphi_t = G(t,X_t)$ we have that
\[  \| \varphi  \|_{\B_{\infty,q}^\alpha} \leqslant \| G \|_{\B_{b,q}^\alpha}.\]
\medskip

To obtain lower quantitative bounds for our directional gradients 
we also exploit the following concept:

\begin{defi}
\label{definition:consistent_family}
The function $H\colon [0,T)\times  Q\to \R$ is $Y$-\textit{consistent} provided that
      \begin{enumerate}[{\rm (1)}]
      \item \label{item:1:definition:consistent_family}
            $H(t,\cdot)$ is continuous on $Q$ for all $t\in [0, T)$,
      \item \label{item:2:definition:consistent_family} 
            $\E |H(t,y+ Y_{t-s})| < \infty$ for all $0\leqslant s \leqslant t < T$ and $y\in Q$,
      \item \label{item:3:definition:consistent_family} 
            $\E H(t,y+Y_{t-s}) = H(s,y)$ for all $0\leqslant s \leqslant t <T$ and $y\in Q$.
      \end{enumerate}
\end{defi}

As counterpart to \cref{statement:easy_properties_Gamma}\eqref{item:2:statement:easy_properties_Gamma}
we obtain:

\begin{lemm}\label{statement:consistent_increasing_suprema}
If $H:[0,T)\times Q \to \R$ is $Y$-consistent, then
\[ \sup_{y\in Q} |H(s,y)| \leqslant \sup_{y\in Q} |H(t,y)| 
   \sptext{1}{for}{1}
   0 \leqslant s < t < T.\]
\end{lemm}
\begin{proof}
In fact,  for $y\in Q$ we have
\[
   |H(s,y)| 
   \leqslant \sup_{\omega \in \Omega} |H(t,y+Y_{t-s}(\omega))| 
   \leqslant \sup_{y'\in Q}  |H(t,y+y')| 
   \leqslant \sup_{y''\in Q} |H(t,y'' )|. \qedhere \]
\end{proof}

Before we consider upper bounds for the directional gradient that explores the jump part of a L\'evy process,
we start with the directional gradient that explores the Brownian motion part only:

\begin{theo}[{\bf upper bounds for $\partial F/\partial x$}]
\label{statement:consistence-upper_bound_sigma>0}
Let $\sigma >0$. Then $Q=\R$ and the following holds:
\begin{enumerate}[{\rm (1)}]
\item \label{item:1:statement:consistence-upper_bound_sigma>0}
      If  $f \colon \R \to \R$ is a Borel function with $\E |f(X_T)|^q < \infty$ for some $q>1$,
      then  $\E|f(x+X_{T-t})|<\infty$ for $(t,x)\in [0,T]\times \R$. If 
      $F(t,x):= \E f(x+X_{T-t})$ on $[0,T]\times \R$, then $F(t,\cdot)\in C^\infty(\R)$ 
      for $t\in [0,T)$ and we obtain an $X$-consistent function $G:[0,T)\times \R\to\R$ by 
      \[ G(t,x) := \frac{\partial F}{\partial x}(t,x)
         \sptext{1}{with}{1}
         G(t,x)= \frac{1}{\sigma}\E \left [f(x+X_{T-t})\frac{W_{T-t}}{T-t} \right ]. \]
\item \label{item:2:statement:consistence-upper_bound_sigma>0}
      If $f\in \hoel{\eta}$ for some $\eta \in [0,1]$ (and  $\E |f(X_T)|^q < \infty$ for some $q>1$ if $\eta \in (0,1]$), then 
      \begin{equation}\label{eqn:Bb_bound_H_part_1}
                 \left \| G(t,\cdot) \right \|_{B_b(\R)} 
      \leqslant  |f|_\eta \sigma^{\eta-1} (T-t)^{\frac{\eta-1}{2}} \int_\R |x|^{\eta+1} \e^{-\frac{x^2}{2}} \frac{\od x}{\sqrt{2\pi}}.
      \end{equation}
\item \label{item:3:statement:consistence-upper_bound_sigma>0}
      If $(\eta,q) \in (0,1) \times [1,\infty)$ and $X\subseteq L_{\eta+\gamma}$ for some $\gamma>0$, then there is a 
      $c_\eqref{eqn:item:3:statement:consistence-upper_bound_sigma>0}>0$ such that, for 
      $\alpha:= \frac{1-\eta}{2}$,  
      \begin{equation}\label{eqn:item:3:statement:consistence-upper_bound_sigma>0}
                 \left \| G \right \|_{\B_{b,q}^\alpha}
      \leqslant c_\eqref{eqn:item:3:statement:consistence-upper_bound_sigma>0} \| f \|_{\hoel{\eta,q}}
      \sptext{1}{for}{1}
      f\in \hoel{\eta,q}.
      \end{equation}
\end{enumerate}
\end{theo}
\medskip

From \cref{statement:consistent_increasing_suprema} we get that 
$\left\| G(t,\cdot) \right\|_{B_b(\R)}$ is non-decreasing in $t$ so that 
$\left \| G \right \|_{\B_{b,q}^\alpha}$ is well-defined.
Item \eqref{item:2:statement:consistence-upper_bound_sigma>0} corresponds to \eqref{item:3:statement:consistence-upper_bound_sigma>0}
for $q=\infty$. We separated the cases as \eqref{item:2:statement:consistence-upper_bound_sigma>0} also covers 
$\eta\in \{0,1\}$.
\medskip

\begin{proof}[Proof of \cref{statement:consistence-upper_bound_sigma>0}]
$Q=\R$ follows from \cite[Theorem 24.10]{Sa13}.
\eqref{item:1:statement:consistence-upper_bound_sigma>0}
Let $f\geqslant 0$, $J: =X-\sigma W$, and fix $t\in [0, T)$. 

(a) Since 
$\E|f(\sigma W_T + J_t + (J_T - J_t))|^q = \E|f(X_T)|^q <\infty$,
independence and Fubini's theorem yield to 
$\E|f(\sigma W_T + a_t + (J_T-J_t))|^q <\infty$ for some $a_t \in \R$.
If 
\[ N^{(t)}:= \{\xi\in \R: \E|f(\sigma \xi + a_t + (J_T-J_t))|^q =\infty\}, \]
then $N^{(t)}$ is a Borel set of Lebesgue measure zero. We define
\[ f^{(t)}(\xi):= \1_{\{ \xi \not \in N^{(t)}\}}(\xi) \,\,
                \E f(\sigma \xi+a_t + (J_T-J_t)) \] 
so that $\E |f^{(t)}(W_T)|^q<\infty$. Now we can apply \cite[Lemma A.2]{Ge02} to $f^{(t)}$ and get for 
$(s,\xi)\in [0,T)\times \R$ and $F^{(t)}(s,\xi):= \E f^{(t)}(\xi+W_{T-s})$ that 
\[ F^{(t)}(s,\cdot)\in C^\infty(\R)
   \sptext{1}{and}{1}
     \frac{\partial F^{(t)}}{\partial \xi}(s,\xi)
   = \E \left [ f^{(t)}(\xi + W_{T-s})\frac{W_{T-s}}{T-s} \right ].\]
(b) Because $f \geqslant 0$, $N^{(t)}$ has Lebesgue measure zero, and $T-t>0$, we verify by Fubini's theorem 
(regardless of the finiteness of the integrals) that 
\[ F(t,x)=F^{(t)}\left (t,\frac{x-a_t}{\sigma}\right ) < \infty 
   \sptext{1}{so that}{1}
   F(t,\cdot)\in C^\infty(\R). \]
(c) We choose $\tilde q\in (1,q)$ so that 
$\E \left [ f(x+X_T-X_s)|W_T-W_t| \right ]\leqslant c_{\tilde q,T-t}  
\| f(x+X_T-X_s) \|_{L_{\tilde q}} < \infty$ where the finiteness of the last term is obtained as in (a-b) by starting with 
the function $x\mapsto f(x)^{\tilde q}$. This moment estimate enables us to apply Fubini's theorem in the sequel.
\smallskip

(d)  Using \cite[Lemma A.2]{Ge02} we deduce that
\begin{align*}
    \sigma \frac{\partial F}{\partial x}(t,x)  
  = \frac{\partial F^{(t)}}{\partial \xi}\left (t,\frac{x-a_t}{\sigma}\right )
& = \E \left [ f^{(t)}\left (\frac{x-a_t}{\sigma} + W_{T-t}\right )\frac{W_{T-t}}{T-t} \right ] \\
& = \E \left [ (\wt \E f(x + \sigma  W_{T-t} + \wt{J}_{T-t}))\frac{W_{T-t}}{T-t} \right ]
  = \E \left [ f(x + X_{T-t})\frac{W_{T-t}}{T-t} \right ].
\end{align*}
(e) To check $\E \frac{\partial F}{\partial x}(t,x+X_{t-s})= \frac{\partial F}{\partial x}(s,x)$ for
$s\in [0,t]$ we have to verify
\[  \E \left [ f(x+X_T-X_s)\frac{W_T-W_t}{T-t} \right ]
   =\E \left [ f(x+X_T-X_s)\frac{W_T-W_s}{T-s} \right ]. \]
As $\frac{W_T-W_t}{T-t}-\frac{W_T-W_s}{T-s}$ is of mean zero and independent of $X_T-X_s$, the last equality is true.
To conclude the proof of \eqref{item:1:statement:consistence-upper_bound_sigma>0} we remove the assumption 
$f \geqslant 0$ by considering the positive and negative parts separately.
\smallskip

 \eqref{item:2:statement:consistence-upper_bound_sigma>0}
Now we additionally assume that $f\in \hoel \eta$. Assume that $t\in [0,T)$ and $x=\sigma \xi +a_t$ with $\xi\not\in N^{(t)}$
and $N^{(t)}$ defined as in step (a). Then $\E |f(x+J_{T-t})|= \E |f(\sigma \xi + a_t+J_{T-t})|<\infty$ and
\begin{align*}
           \left | \frac{\partial F}{\partial x}(t,x) \right |
& =        \left | \frac{1}{\sigma}\E \left [ f(x+X_{T-t}) \frac{W_{T-t}}{T-t} \right ] \right | 
  =        \frac{1}{\sigma}\E \left | \left [ \left ( f(x+X_{T-t}) - f(x+J_{T-t}) \right ) \frac{W_{T-t}}{T-t} \right ] \right | \\
&\leqslant \frac{|f|_\eta}{\sigma} \E \left [ \left | \sigma W_{T-t} \right |^\eta \frac{|W_{T-t}|}{T-t} \right ] 
  =        |f|_\eta \sigma^{\eta-1} (T-t)^{\frac{\eta-1}{2}} \E |g|^{\eta+1}, 
\end{align*}
where $g:\Omega\to\R$ has a standard normal distribution.
Because $\lambda(N^{(t)})=0$ and $x\mapsto \frac{\partial F}{\partial x}(t,x)$ continuous, 
the estimate is true for all $x\in\R$.

\eqref{item:3:statement:consistence-upper_bound_sigma>0}
We choose $0<\eta_0<\eta<\eta_1 < \eta + \gamma$ and obtain from 
\eqref{item:2:statement:consistence-upper_bound_sigma>0} the estimates 
\[            \left \| G \right \|_{\B_{b,\infty}^{\frac{1-\eta_i}{2}}}
    \leqslant c_i \| f \|_{\hoel{\eta_i,\infty}}
    \sptext{1}{for}{1}
    f \in \hoel{\eta_i,\infty}\]
for some $c_i>0$ with $i=0,1$, where we choose $q_i := \frac{\eta+\gamma}{\eta_i}>1$. 
If $\delta \in (0,1)$ is chosen such that $\eta = (1-\delta)\eta_0 + \delta \eta_1$, then 
 $\frac{1-\eta}{2} = (1-\delta) \frac{1-\eta_0}{2} + \delta \frac{1-\eta_1}{2}$ as well.
Using the same argument as in the proof of \cref{statement:abstract_interpolation_new}
(starting with inequality \eqref{eqn:interpolation_start}) 
and \eqref{eq:reiteration_for_Hoelder} gives the statement.
\end{proof}

Now we aim for the counterpart(s) to \cref{statement:consistence-upper_bound_sigma>0} as follows:
\begin{enumerate}[(A)]
\item In \cref{statement:TV_to_Gamma} we prove the consistency and an upper bound for the directional 
      gradient in terms of  $\|X_{T-t}\|_{{\rm TV}(\rho,\eta)}$.
\item In \cref{statement:X-rho_to_TV} we provide upper bounds for $\|X_{T-t}\|_{{\rm TV}(\rho,\eta)}$
      in terms of the conditions $\rho \in U(\varepsilon;c)$ and
      $X \in \cU(\beta,\supp(\rho);d)$.
\item In \cref{statement:upper_bounds_Drho} we finally deduce upper bounds for $\| D_\rho F \|_{\B_{b,q}^\alpha}$,
       where the case $q=2$ gives the connection to the Riemann-Liouville operators.
\end{enumerate}

\begin{theo}[{\bf from $\|\cdot\|_{{\rm TV}(\rho,\eta)}$ to $\Gamma_{t,\rho}$}]
\label{statement:TV_to_Gamma}
Assume a L\'evy process $(X_t)_{t\in [0,T]}$, $\eta \in [0,1]$, 
\begin{enumerate}[{\rm (a)}]
\item $(X_t)_{t\in [0,T]}\subseteq L_\eta$,
\item $\|X_s\|_{{\rm TV}(\rho,\eta)}<\infty$ for $s\in (0,T]$, where $\rho$ is a probability measure on
      $(\R,\cB(\R))$.
\end{enumerate}
Then the following holds:
\begin{enumerate}[{\rm (1)}]
\item \label{item:1:statement:TV_to_Gamma}
      One has $\hoel \eta \subseteq \dom(\Gamma_{\rho}^0)$ with
      \begin{equation}\label{eq:item:1:statement:TV_to_Gamma}
      \left \| \Gamma_{t,\rho}^0 \right \|_{(\hoel \eta)^*} 
                    \leqslant \|X_{T-t}\|_{{\rm TV}(\rho,\eta)},
      \end{equation}
     where $\hoel \eta$ is equipped with the semi-norm
     $|f|_{\eta,\infty}:= \| f-f(0)\|_{\hoelO{\eta,\infty}}$ if $\eta\in (0,1)$.
\item \label{item:2:statement:TV_to_Gamma}
      If $f\in \hoel \eta$ and $K(t,y):=F(t,y+\ell t)$ for $(t,y)\in [0,T]\times Q$, and
      \[ H(t,y) := \int_{Q\setminus \{0\}} \frac{K(t,y+z)-K(t,y)}{z}\rho(\od z)
         \sptext{1}{for}{1}
         (t,y)\in [0,T)\times Q, \]
      then $H$ is $Y$-consistent with $Y$ given in \eqref{eqn:definition_Y_Levy}, 
      where for $\eta=\sigma=0$ we additionally assume that
      $y \mapsto f(y+\ell T)$ is continuous as map from $Q$ to $\R$.
\end{enumerate}
\end{theo}

For the proof we need the following lemma:

\begin{lemm}
\label{statement:consistence-sigma=0_new} 
We assume 
\begin{enumerate}[{\rm (1)}]
\item \label{it:1:statement:consistence-sigma=0_new}
      that $k\colon Q \to \R$ is a Borel function with 
      $\E|k(y+Y_s)|<\infty$ for $(s,y)\in [0,T]\times Q$ and that $K \colon [0,T)\times Q\to\R$ with
      $K(t,y) := \E k(y+Y_{T-t})$ satisfies
      \[ \E \int_{Q\setminus \{0\}} \left | \frac{K(t,y+Y_\delta + z) 
                          - K(t,y+ Y_\delta)}{z} \right |  \rho(\od z) < \infty
         \sptext{1}{for}{1} 
         0 \leqslant \delta \leqslant t < T, \]
\item \label{it:2:statement:consistence-sigma=0_new}
      that $y\mapsto K(t,y)$ is continuous on $Q$ for $t\in [0,T)$,
\item \label{it:3:statement:consistence-sigma=0_new}
      that for all $(t,y)\in [0,T) \times Q$ there is an $\varepsilon >0$ such the family  
      of functions $z\mapsto \frac{K(t,y'+z)-K(t,y')}{z}$, indexed by 
      $y'\in Q$ with $|y-y'|<\varepsilon$, is uniformly integrable on $(Q\setminus \{0\},\rho)$.
\end{enumerate}
Then we obtain a $Y$-consistent function by \begin{align*}
	 H(t,y) := \int_{Q \backslash\{0\}} \frac{K(t,y+z)-K(t,y)}{z} \rho(\od z)
    \sptext{1}{for}{1} (t,y)\in [0,T)\times Q.
        \end{align*}
\end{lemm}

\begin{proof}
We check the requirement of \cref{definition:consistent_family}:
\smallskip
(a) Taking $\delta=0$ in assumption \eqref{it:1:statement:consistence-sigma=0_new}, we see
that $H(t,y)$ is well-defined, and for $\delta:= t-s$ we obtain that
\begin{equation}\label{eq:statement:consistence-sigma=0_new}
     \E|H(t,y+Y_{t-s})|
 \leqslant \E \int_{Q \backslash\{0\}} \left | \frac{K(t,y+Y_{t-s}+z)-K(t,y+Y_{t-s})}{z} \right | \rho(\od z) 
 < \infty.
\end{equation}
(b)
Because of \eqref{eq:statement:consistence-sigma=0_new} we can apply Fubini's theorem to get
\begin{align*}
\E H(t,y+ Y_{t-s})
& =  \E  \int_{Q \backslash\{0\}} \frac{K(t,y+Y_{t-s}+z)-K(t,y+Y_{t-s})}{z}  \rho(\od z) \\
& =  \int_{Q \backslash\{0\}} \frac{\E K(t,y+Y_{t-s}+z)-\E K(t,y+Y_{t-s})}{z}  \rho(\od z) \\
& =  \int_{Q \backslash\{0\}} \frac{K(s,y+z)-K(s,y)}{z}  \rho(\od z)\\
& = H(s, y)
\end{align*}
where we use $\E K(t,y+Y_{t-s})=K(s,y)$.
\smallskip

(c)
If we have $y_n,y\in Q$ with $y_n \to y$, then we take $\varepsilon = \varepsilon(t,y)>0$ from assumption 
\eqref{it:3:statement:consistence-sigma=0_new} and obtain 
$\lim_n H(t,y_n) = H(t,y)$ by the uniform integrability 
imposed in \eqref{it:3:statement:consistence-sigma=0_new} and assumption 
\eqref{it:2:statement:consistence-sigma=0_new}.
\end{proof}

\begin{proof}[Proof of \cref{statement:TV_to_Gamma}]
\eqref{item:1:statement:TV_to_Gamma}
First we remark that $(X_t)_{t\in [0,T]}\subseteq L_\eta$ implies that
$\hoel \eta \subseteq \defX$. Moreover, for
fixed $z\in \RO$, $t\in [0,T)$, and $f\in \hoel 1$ we obtain the estimate
\begin{equation}\label{eq:interpol_1}
\left | \frac{F(t,x+z) - F(t,x)}{z} \right |  \leqslant |f|_1
\end{equation}
and, for $f\in B_b(\R)$ and $x'\in \R$,
\begin{align*}
 \left | F(t,x+z) - F(t,x) \right |  
 & =  \left | \int_\R (f(x+y)-f(x')) \p_{z+X_{T-t}}(\od y) -
 \int_\R (f(x+y)-f(x'))     \p_{X_{T-t}}(\od y) \right | \\
 &\leqslant   \int_\R |f(x+y)-f(x')|\,|\p_{z+X_{T-t}} - \p_{X_{T-t}}|(\od y)  \\
 &\leqslant  \| f-f(x')\|_{B_b(\R)}  \| \p_{z+{X_{T-t}}} - \p_{X_{T-t}}\|_{\rm TV}.
 \end{align*}
Therefore, 
\begin{equation}\label{eq:interpol_0}
\left | \frac{F(t,x+z)-F(t,x)}{z}\right |
\leqslant c_\eqref{eq:interpol_0} \frac{ \| \p_{z+{X_{T-t}}} - \p_{X_{T-t}}\|_{\rm TV}}{|z|}
\end{equation}
for 
$c_\eqref{eq:interpol_0} := \|f\|_{C_b^0(\R)}$ if $f\in C_b^0(\R)$ (take $x'=0$) and
$c_\eqref{eq:interpol_0} :=  |f|_0$   if $f\in \hoel 0$   (take the supremum over $x'\in\R$ on the 
                                                                    right-hand side).
Moreover, real interpolation between 
\eqref{eq:interpol_0} for $C_b^0(\R)$ and 
\eqref{eq:interpol_1} for $\hoell 1 0$ 
(for fixed $x$ and $z$) implies that
\begin{equation} \label{eq:interpol_between_01}
   \left | \frac{F(t,x+z) - F(t,x)}{z} \right |  
   \leqslant \|f\|_{\hoelO{\eta,\infty}}
   \left [ \frac{\| \p_{z+{X_{T-t}}} - \p_{X_{T-t}}\|_{\rm TV}}{|z|} \right ]^{1-\eta}
\end{equation}
for $\eta\in (0,1)$ by \eqref{eq:functor_with_constant_1_real_interpolation}.
From \eqref{eq:interpol_0} and \eqref{eq:interpol_1} we deduce 
$\hoel \eta \subseteq \dom(\Gamma_\rho^{0})$ and \eqref{eq:item:1:statement:TV_to_Gamma}
for $\eta\in \{0,1\}$. 
If $\eta \in (0,1)$, then \eqref{eq:interpol_between_01} implies 
$\hoelO {\eta,\infty} \subseteq \dom(\Gamma_\rho^{0})$ and \eqref{eq:item:1:statement:TV_to_Gamma}
with $\hoel \eta$ replaced by $\hoelO {\eta,\infty}$. But if $f\in \hoel \eta$, then we replace 
$f$ by $f_0:=f-f(0)\in \hoelO {\eta,\infty}$ and get
\eqref{eq:interpol_between_01} with constant $\|f-f(0)\|_{\hoelO{\eta,\infty}}$.
This concludes the proof of \eqref{item:1:statement:TV_to_Gamma}.

\medskip

\eqref{item:2:statement:TV_to_Gamma}
We verify the assumptions of \cref{statement:consistence-sigma=0_new}:
\smallskip

\underline{Assumption \eqref{it:1:statement:consistence-sigma=0_new}} follows from
$K(t,y) = F(t,y+\ell t)$ and $f\in\dom(\Gamma_\rho^0)$.
\smallskip

\underline{Assumption \eqref{it:2:statement:consistence-sigma=0_new}}:
Let $k(y):= f(y+\ell T)$ for $y\in \R$.  If $\eta\in (0,1]$ or $\eta=\sigma=0$ we know that $k:Q\to \R$ 
is continuous. Assume $t\in [0,T)$ and $y_n\to y$ with $y_n,y\in Q$. Then 
$(k(y_n+Y_{T-t}))_{n=1}^\infty$ is uniformly integrable. For $\eta=0$ this is obvious,
for $\eta\in (0,1]$ this follows from $(Y_t)_{t\in [0,T]} \subseteq L_\eta$ and $k\in \hoel \eta$. In both cases the continuity of 
$k$ implies that $\lim_{n\to \infty} K(t,y_n) = K(t,y)$. For $\eta=0$ and $\sigma>0$ we use 
\cref{statement:consistence-upper_bound_sigma>0}.
\smallskip

\underline{Assumption \eqref{it:3:statement:consistence-sigma=0_new}}:
Re-writing \eqref{eq:interpol_between_01} for $K$ we get
\begin{equation}\label{eqn:Bb_bound_H_part_2}
\left | \frac{K(t,y+z)-K(t,y)}{z} \right |
   \leqslant c_\eta(f) \left [ \frac{\| \p_{z+{X_{T-t}}} - \p_{X_{T-t}}\|_{\rm TV}}{|z|}\right ]^{1-\eta}
\end{equation}
for $z\in Q\setminus \{ 0 \}$. Therefore $\| X_{T-t} \|_{\rm TV(\rho,\eta)}<\infty$, $t\in [0,T)$, implies that 
the family of maps $Q\setminus \{0\} \ni z \mapsto  \frac{K(t,y+z)-K(t,y)}{z}\in \R$, indexed by $y\in \R$, is uniformly integrable.
\end{proof}

\begin{defi}\hspace{0em}
\label{defi:uppper_bounds_X_rho}
\begin{enumerate}
\item For a probability measure $\rho$ on $(\R,\cB(\R))$, $\varepsilon \geqslant 0$, and $c>0$, we let 
      $\rho\in U(\varepsilon;c)$ if one has
      \[ \rho([-d,d]\setminus \{0\}) \leqslant c d^\varepsilon
         \sptext{1}{for}{1} d \in [0,1]. \]
\item For a L\'evy process $X=(X_t)_{t\in [0,T]}$, $\beta \in (0,\infty]$,  a closed non-empty $A\subseteq \R$, and $c>0$ we let
      $X\in \cU(\beta,A;c)$ if for $z \in A$ and $s\in (0,T]$ one has
      \[ \| \p_{z+X_s} - \p_{X_s} \|_{\rm TV} \leqslant  c |z| s^{-\frac{1}{\beta}}. 
         \]
\end{enumerate}
\end{defi} 
\medskip

\begin{rema}
\label{statement:Uepsc_vs_limsup}
For all  probability measures $\rho$ one has  $\rho\in U(0;1)$. If $\varepsilon>0$, then
the condition '$\rho\in U(\varepsilon;c)$ for some $c>0$' is equivalent to 
\[ \limsup_{d\downarrow 0}\frac{\rho([-d,d]\setminus \{0\})}{d^\varepsilon} < \infty. \]
\end{rema}
\medskip

\begin{theo}[{\bf upper bounds for $\|\cdot\|_{{\rm TV}(\rho,\eta)}$}] 
\label{statement:X-rho_to_TV}
Assume a L\'evy process $X=(X_t)_{t\in [0,T]}$ and $\eta\in [0,1]$.
\begin{enumerate}[{\rm (1)}]
\item \label{item:1:statement:X-rho_to_TV}
      If $(\varepsilon,\beta) \in [0,\infty)\times (0,\infty]$, $c,c'>0$, and  
      \[ \rho \in U(\varepsilon;c) \sptext{1}{and}{1}
         X \in \cU(\beta,\supp(\rho);c'), \]
      then there is a $c_{\eqref{eqn:statement:end_point_estimate_levy_new}}>0$, that depends at most on 
      $(\varepsilon,\beta,\eta,c',T)$,
      such that for $s\in (0,T]$ one has 
      \begin{equation}\label{eqn:statement:end_point_estimate_levy_new}
             \|X_s\|_{{\rm TV}(\rho,\eta)} 
      \leqslant  c_{\eqref{eqn:statement:end_point_estimate_levy_new}} (1\vee c) \,
      \begin{cases}
           s^{-\frac{1-(\varepsilon+\eta)}{\beta}}  &:  \eta \in [0,1-\varepsilon) \\
            \ln \left ( 1+s^{-\frac{1-\eta}{\beta}} \right )                         &:  \eta = 1 - \varepsilon < 1 \\
           1                                     &:  \eta \in (1-\varepsilon,1) \cup \{ 1 \}
      \end{cases}.
      \end{equation}

\item \label{item:2:statement:X-rho_to_TV}
      If $s\in (0,T]$ and $X_s$ has a density $p_s \in C^1(\R)$, then
      \begin{align*}
                  \|X_s\|_{{\rm TV}(\rho,\eta)}    
        \leqslant \int_{\RO} \left (  \min \left \{ \frac{2}{|z|}, 
                  \left \| \frac{\partial p_s}{\partial y} \right \|_{L_1(\R)} \right \} \right )^{1-\eta} \rho(\od z).
      \end{align*}
\item \label{item:3:statement:X-rho_to_TV}
      If $\sigma>0$ and $s\in (0,T]$, 
      then $p_s \in C^1(\R)$ with
      $\left \| \frac{\partial p_s}{\partial y} \right \|_{L_1(\R)} 
      \leqslant \sqrt{\frac{2}{\pi \sigma^2}} s^{-\frac{1}{2}}$
      and $X \in \cU\left (2,\R;\sqrt{\frac{2}{\pi \sigma^2}}\right )$.
\end{enumerate}
\end{theo}
\medskip
\begin{rema}\hspace{0em}
The existence of the transition density $p_{T-t}$ of a L\'evy process is a time 
dependent distributional property (see, e.g., \cite[Ch.5]{Sa13}).
In \cref{statement:properties:Gamma1} and \cref{statement:X-rho_to_TV}
we used $\left \| \partial p_{T-t}/\partial x \right \|_{L_1(\R)}$ and $\|p_{T-t}\|_{L_s(\R)}$. 
Results for $\partial p_t/\partial x$ and $p_t$ for a L\'evy process can be found, for example, in
\cite{KS13, KS18, SSW12, Sz11}.
\end{rema}

For the proof we need the following estimate:

\begin{lemm}
\label{statement:upper_bound_function_A}
Let $\varepsilon\geqslant 0$, $\eta \in [0,1)$, $c,w_0>0$, and 
$\rho \in U(\varepsilon;c)$. Define 
$A:\RO \to [0,\infty)$ by $A(z) := (2/|z|)^{1-\eta}$. Then there is a 
$  c_{\eqref{eqn:statement:upper_bound_function_A}} 
 = c_{\eqref{eqn:statement:upper_bound_function_A}}(\varepsilon,\eta,w_0)>0$
such that for $w\geqslant w_0$ one has
\begin{equation}\label{eqn:statement:upper_bound_function_A}
 \int_\RO  (A\wedge w)  \od \rho
\leqslant c_{\eqref{eqn:statement:upper_bound_function_A}} (1 \vee c)
   \begin{cases}
   w^{1-\frac{\varepsilon}{1-\eta}} &:  \varepsilon + \eta < 1 \\
   \ln (1+w)                       &:  \varepsilon + \eta = 1 \\
   1                                &:  \varepsilon + \eta > 1
\end{cases}.
\end{equation}
\end{lemm}
\begin{proof}
It is sufficient to prove the inequality for any fixed $w_0>0$ and obtain the inequality for other $w_0>0$ by a change 
of the constant $c_{\eqref{eqn:statement:upper_bound_function_A}}>0$. Here we choose $w_0>0$ such that 
$2 w_0^{-\frac{1}{1-\eta}}=1$ and obtain for $w\geqslant w_0$ that
\[ w \rho (A>w) 
   \leqslant c w \left (  2 w^{-\frac{1}{1-\eta}} \right )^\varepsilon 
   =  (c  2^\varepsilon) \, w^{1-\frac{\varepsilon}{1-\eta}}.\]
For the remaining term and $w\geqslant w_0$ we get
\begin{align*}
          2^{\eta-1} \int_{\{A \leqslant w \}}  A(z) \rho(\od z)
& =       \int_{\left \{  |z| \geqslant 2 w^{-\frac{1}{1-\eta}}  \right \}} |z|^{\eta-1} \rho(\od z) \\
&\leqslant 1 + \int_{\left \{1 >  |z| \geqslant 2 w^{-\frac{1}{1-\eta}}  \right \}} |z|^{\eta-1} \rho(\od z) \\
&\leqslant 1 + \sum_{n=1}^\infty 
           \1_{ \left \{     2^{1-n} \geqslant 2 w^{-\frac{1}{1-\eta}}    \right \} } 
\rho \left ( \frac{1}{2^n} \leqslant |z| < \frac{1}{2^{n-1}} \right ) 2^{n(1-\eta)} \\
&\leqslant 1 + c \sum_{n = 1}^\infty 
           \1_{ \left \{     2^{1-n} \geqslant 2 w^{-\frac{1}{1-\eta}}    \right \} } 
           \left ( \frac{1}{2^{n-1}} \right )^\varepsilon 2^{n(1-\eta)} \\
&          =  1 + c 2^\varepsilon \sum_{n = 1}^\infty 
           \1_{ \left \{     2^{n} \leqslant  w^{\frac{1}{1-\eta}}    \right \} }  2^{n(1-(\varepsilon+\eta))} \\
&\leqslant (1 \vee c) \left (1 + 2^\varepsilon \sum_{n = 1}^\infty 
           \1_{ \left \{     2^{n} \leqslant  w^{\frac{1}{1-\eta}}    \right \} }  2^{n(1-(\varepsilon+\eta))}\right ).
\end{align*}
We conclude with 
\[ \sum_{n = 1}^\infty 
   \1_{ \left \{     2^{n} \leqslant  w^{\frac{1}{1-\eta}}    \right \} }  2^{n(1-(\varepsilon+\eta))}
   \leqslant 
    \begin{cases}
     d \, w^{\frac{1-(\varepsilon+\eta)}{1-\eta}}               &: \varepsilon + \eta < 1 \\
    \frac{1}{1-\eta}   \log_2 (1+w)                             &: \varepsilon + \eta = 1\\
    \sum_{n =1}^\infty 2^{n(1-(\varepsilon+\eta))} < \infty     &: \varepsilon + \eta > 1
    \end{cases}
    \]
for some $d=d(\varepsilon,\eta)>0$.
\end{proof}

\begin{proof}[Proof of \cref{statement:X-rho_to_TV}]
\eqref{item:1:statement:X-rho_to_TV}
For $\eta=1$ the estimate is obvious, so let us assume $\eta \in [0,1)$ and $s\in (0,T]$. Set
$\gamma := \frac{1-\eta}{\beta}\in [0,\infty)$ and $w_0 := T^{-\gamma}$. In the notation of \cref{statement:upper_bound_function_A}
we deduce
\begin{align*}
             \|X_s\|_{{\rm TV}(\rho,\eta)}
 & \leqslant \int_{\supp(\rho)\setminus \{0\}}  \left ( \min\left \{ \frac{2}{|z|},  c' s^{-\frac{1}{\beta}} \right \} \right )^{1-\eta} \rho(\od z) \\ 
 &   =       \int_{\supp(\rho)\setminus \{0\}} \hspace*{-2em} \min \{ A(z),  (c')^{1-\eta}  s^{-\gamma} \} \rho(\od z) \\
 &\leqslant  (1\vee (c')^{1-\eta}) \int_{\supp(\rho)\setminus \{0\}} \hspace*{-2em} (A \wedge  s^{-\gamma})\od \rho \\
 &\leqslant  (1\vee (c')^{1-\eta}) \int_\RO (A \wedge  s^{-\gamma})\od \rho.
\end{align*}
We conclude with \cref{statement:upper_bound_function_A}.

\eqref{item:2:statement:X-rho_to_TV}
We observe that
\begin{align}
    \|\p_{z+X_s} - \p_{X_s}\|_{\rm TV} 
& = \left \| p_s(\cdot -z) - p_s\right \|_{L_1(\R)} 
  = \int_\R \left |\int_{x-z}^x \frac{\partial p_s}{\partial y}(y) \od y \right | \od x \notag \\
& \leqslant \sign(z) \int_\R \int_{x-z}^x\left | \frac{\partial p_s}{\partial y}(y) \right | \od y \od x  \notag \\
& = |z| \int_\R \left | \frac{\partial p_s}{\partial y}(y) \right | \od y. \label{eqn:TV_vs_gradient_density}
\end{align}
As we  have $\|\p_{z+X_s} - \p_{X_s}\|_{\rm TV}\leqslant 2$ as well, the claim follows.
\medskip

\eqref{item:3:statement:X-rho_to_TV}
If $\sigma>0$ and $s\in (0,T]$, then the density of $X_s$ is given by
$p_s(y):=\E p_{\sigma W_s}(y-J_s)$ where 
$J_s:= X_s - \sigma W_s$, and
$p_{\sigma W_s}$ is the $C^\infty(\R)$-density of $\sigma W_s$
and satisfies
\[
          \left \| \frac{\partial p_s}{\partial y} \right \|_{L_1(\R)} 
     =    \left \| \E \frac{\partial p_{\sigma W_s}}{\partial y}(\cdot-J_s) \right \|_{L_1(\R)} 
\leqslant \left \| \frac{\partial p_{\sigma W_s}}{\partial y}\right \|_{L_1(\R)} 
     =    \sqrt{\frac{2}{\pi \sigma^2}} s^{-\frac{1}{2}}. 
\]
Part \eqref{item:2:statement:X-rho_to_TV} implies that 
\[           \| \p_{z+X_s} - \p_{X_s} \|_{\rm TV}
   \leqslant |z|  \sqrt{\frac{2}{\pi \sigma^2}} s^{-\frac{1}{2}} \]
and therefore $X \in \cU\left (2,\R;\sqrt{\frac{2}{\pi \sigma^2}}\right )$.
\end{proof}

\begin{theo}[{\bf upper bounds for $D_\rho F$}]
\label{statement:upper_bounds_Drho}
Assume a L\'evy process $X=(X_t)_{t\in [0,T]}$, $\eta\in [0,1]$,
\begin{enumerate}[{\rm (a)}]
\item $(X_t)_{t\in [0,T]}\subseteq L_\eta$,
\item $(\varepsilon,\beta) \in [0,\infty)\times (0,\infty]$,  $c,c'>0$, and 
      $\rho \in U(\varepsilon;c)$  and $X \in \cU(\beta,\supp(\rho);c')$,
\item $\alpha:= \frac{1-(\varepsilon+\eta)}{\beta}$.
\end{enumerate}
Then the following holds:
\begin{enumerate}[{\rm (1)}]
\item \label{item:1:statement:upper_bounds_Drho}
      One has $\hoel \eta \subseteq \dom(\Gamma_\rho^0)$ and 
      there is a 
      $  c_{\eqref{eqn:item:1:statement:upper_bounds_Drho}}>0$, that depends at most on 
      $(\varepsilon,\beta,\eta,c',T)$,
      such that for $t\in [0,T)$ one has
      \begin{equation}\label{eqn:item:1:statement:upper_bounds_Drho}
           \| \Gamma_{t,\rho}^0 \|_{(\hoel \eta)^\ast} 
      \leqslant c_{\eqref{eqn:item:1:statement:upper_bounds_Drho}} \, (1\vee c) \,
      \begin{cases}
           (T-t)^{-\frac{1-(\varepsilon+\eta)}{\beta}}            &:  \eta \in [0,1-\varepsilon) \\
            \ln \left ( 1+(T-t)^{-\frac{1-\eta}{\beta}} \right ) &:  \eta = 1 - \varepsilon < 1 \\
           1                                                      &:  \eta \in (1-\varepsilon,1) \cup \{ 1 \}
      \end{cases}. 
      \end{equation}

\item \label{item:2:statement:upper_bounds_Drho}
      If $\varepsilon\in [0,1)$, $(\eta,q,\beta)\in (0,1-\varepsilon)\times [1,\infty)\times (0,\infty)$, and 
      $X\subseteq L_{\eta+\gamma}$ for some $\gamma>0$, then there is a 
      $ c_\eqref{eqn:item:2:statement:upper_bounds_Drho}
              =c_\eqref{eqn:item:2:statement:upper_bounds_Drho}(\varepsilon,\beta,\eta,\gamma,q,c',T)>0$
      such that 
      \begin{equation}\label{eqn:item:2:statement:upper_bounds_Drho}
                 \left \| D_\rho F \right \|_{\B_{b,q}^\alpha}
      \leqslant c_\eqref{eqn:item:2:statement:upper_bounds_Drho} \, (1\vee c) \, \| f \|_{\hoel{\eta,q}}
      \sptext{1}{for}{1}
      f\in \hoel{\eta,q}.
      \end{equation}
\end{enumerate}
\end{theo}
\bigskip

\begin{rema}\label{remark:relation_to_SSW12}
\cref{statement:upper_bounds_Drho}\eqref{item:1:statement:upper_bounds_Drho} for $\varepsilon=\eta=0$
corresponds to the upper bound of 
Schilling, Sztonyk, and Wang
\cite[Theorem 1.3]{SSW12} as we get in this case
$\rho\in U(0;1)$ for any probability measure $\rho$ and
\[              \| \Gamma_{t,\rho}^0 \|_{(\hoel 0)^\ast} 
      \leqslant c_{\eqref{eqn:statement:end_point_estimate_levy_new}}
                (T-t)^{-\frac{1}{\beta}}. \]
\end{rema}
\medskip

\begin{proof}[Proof of \cref{statement:upper_bounds_Drho}]
\eqref{item:1:statement:upper_bounds_Drho} 
follows from combining
\cref{statement:TV_to_Gamma}\eqref{item:1:statement:TV_to_Gamma} and
\cref{statement:X-rho_to_TV}\eqref{item:1:statement:X-rho_to_TV} 
with the constant 
$c_{\eqref{eqn:statement:end_point_estimate_levy_new}}(\varepsilon,\beta,\eta,c',T)>0$ 
if $\hoel \eta$ is equipped with the semi-norm
      $|f|_{\eta,\infty}:= \| f-f(0)\|_{\hoelO{\eta,\infty}}$ if $\eta\in (0,1)$.
Using the remarks in \cref{sec:intro:function_spaces}, $|f|_{\eta,\infty}$ is equivalent to 
$|f|_\eta$ if $\eta \in (0,1)$ up to a multiplicative constant depending on $\eta$ only.

\eqref{item:2:statement:upper_bounds_Drho}
We choose $0<\eta_0 < \eta < \eta_1 < \min \{1-\varepsilon,\eta +\gamma\}$,
define $E_i:= \hoelO{\eta_i,\infty}$, $E:= B_b(\R)$, and 
\[ (T_{t,z} f)(x) := \frac{F(t,x+z) - F(t,x)}{z}
   \sptext{1}{for}{1}
   z\in \RO. \]
To have less dependencies on the parameters we choose specifically
$\eta_0:=\eta/2$ and $\eta_1:=  (\eta + \min \{1-\varepsilon,\eta +\gamma\})/2$,
so that $(\eta_0,\eta_1)$ depends on $(\eta,\varepsilon,\gamma)$ only.
For $z\in \supp(\rho)\setminus \{0\}$ it follows from \eqref{eq:interpol_between_01} that 
\begin{align*}
           \| T_{t,z} f\|_E
&\leqslant \|f\|_{E_i} \min \{ A_i(z), (c')^{1-\eta_i} (T-t)^{-\gamma_i} \}
\end{align*}
with $\gamma_i:= \frac{1-\eta_i}{\beta}$ and $A_i(z):=(2/|z|)^{1-\eta_i}$.
Moreover, for $\kappa_i:= \frac{\varepsilon}{1-\eta_i}$  \cref{statement:upper_bound_function_A}  gives 
\[ \int_{\supp(\rho)\setminus \{0\}} (A_i \wedge w)  \od \rho  
         \leqslant c^{(i)}_{\eqref{eqn:statement:upper_bound_function_A}} \, (1\vee c) \,
         w^{1-\kappa_i} 
         \sptext{1}{for}{1}
         w \geqslant T^{-\gamma_i}\]
where 
$ c^{(i)}_{\eqref{eqn:statement:upper_bound_function_A}}
 = c_{\eqref{eqn:statement:upper_bound_function_A}}(\varepsilon,\eta_i,T^{-\frac{1-\eta_i}{\beta}})>0$.
If $\delta\in (0,1)$ is such that $\eta=(1-\delta)\eta_0+\delta \eta_1$, then 
$(1-\kappa_0)\gamma_0 \not = (1-\kappa_1)\gamma_1$ 
and we apply \cref{statement:abstract_interpolation_new}
with $(R,\cR,\pi)=(\supp(\rho)\setminus \{0\}, \cB(\supp(\rho)\setminus \{0\}),\rho)$,
$T_{t,\pi} f := D_\rho F(t,\cdot)$, 
\[ c_i :=  c^{(i)}_{\eqref{eqn:statement:upper_bound_function_A}} (1\vee c), 
   \sptext{1}{and}{1}
   d_i :=  1 \vee (c')^{1-\eta_i}, \]
we note that
\[ \alpha=(1-\delta)(1-\kappa_0)\gamma_0 + \delta (1-\kappa_1)\gamma_1, \]
and finally use the reiteration theorem in the form of \eqref{eq:reiteration_for_Hoelder}.
\end{proof}
\medskip


\subsection{Lower bounds for the gradient processes}
\label{sec:oscillation_gradient_levy}
In this section we focus on lower bounds for the oscillation of the directional gradients.
The main quantitative result is \cref{statement:lower_bound_complete}. To obtain this result 
we proceed as follows:
\begin{enumerate}[(A)]
\item In \cref{statement:lower-bound-Levy-gen_new} we provide a more general result 
      when a martingale is of {\it maximal oscillation} and satisfies a certain {\it forward uniqueness},
      where we use the concept of {\it consistency} introduced in \cref{definition:consistent_family}.
\item In \cref{statement:maximal_oscillation_application} we apply 
      \cref{statement:lower-bound-Levy-gen_new} to the directional gradients to derive
      {\it maximal oscillation} and {\it forward uniqueness}.
\item In \cref{statement:lower_bounds_levy_new} we provide lower $L_\infty$-bounds for the directional gradients.
\item Our point is to combine the lower $L_\infty$-bounds from \cref{statement:lower_bounds_levy_new}
      with the {\it maximal oscillation} and {\it forward uniqueness} from \cref{statement:maximal_oscillation_application}
      to prove the main result \cref{statement:lower_bound_complete}.
\end{enumerate}
\bigskip

\begin{theo}[\bf maximal oscillation and forward uniqueness under consistency]
\label{statement:lower-bound-Levy-gen_new} 
Let $H$ be $Y$-consistent and $\varphi_t := H(t,Y_t)$, $t\in [0,T)$. 
Then $\varphi=(\varphi_t)_{t\in [0,T)}$ is a martingale of maximal oscillation with constant $2$.
Moreover, if for all $t\in [0,T)$ there is an $\overline{t}\in (t,T)$ such that 
$H(t,Y_{\overline{t}})\in L_2$, then the following assertions are equivalent:
\begin{enumerate}[{\rm (1)}]
\item \label{item:1:prop:lower-bound-Levy-gen_new}
      $\inf_{t\in (0,T)} \losc_t(\varphi)=0$.
\item \label{item:2:prop:lower-bound-Levy-gen_new}
      $\varphi_t=\varphi_0$ a.s. for all $t\in (0,T)$.
\item \label{item:3:prop:lower-bound-Levy-gen_new}
      $\varphi_t=\varphi_0$ a.s. for some $t\in (0,T)$.
\end{enumerate}
\end{theo}

The implication 
\eqref{item:3:prop:lower-bound-Levy-gen_new}$\Rightarrow$\eqref{item:2:prop:lower-bound-Levy-gen_new}
states a so-called {\it forward uniqueness}.

\begin{proof}[Proof of \cref{statement:lower-bound-Levy-gen_new}]
The martingale property follows by the definition and the maximal oscillation with constant $2$
follows from \cref{exam:oscillation_Levy_process}.
Regarding the equivalences,
\eqref{item:2:prop:lower-bound-Levy-gen_new}$\Rightarrow$\eqref{item:3:prop:lower-bound-Levy-gen_new} is obvious
and \eqref{item:3:prop:lower-bound-Levy-gen_new}$\Rightarrow$\eqref{item:1:prop:lower-bound-Levy-gen_new}
follows from $\losc_{t}(\varphi)=0$ if $\varphi_t=\varphi_0$ a.s. So we 
only need to show \eqref{item:1:prop:lower-bound-Levy-gen_new}$\Rightarrow$\eqref{item:2:prop:lower-bound-Levy-gen_new}.
For $0 < s < t<T$, $y'_1, y'_2 \in Q$ and $\omega \in (Y_t - Y_s)^{-1}(Q)$ we obtain that
\begin{align*}
           \|\varphi_t -\varphi_s\|_{L_\infty} 
& =    \sup_{y,y' \in Q}  | H(t, y+ y') - H(s, y')|\\
&\geqslant   |H(t, y'_1+(Y_t-Y_s)(\omega) + y'_2+ Y_s(\omega)) -H(s, y'_2 + Y_s(\omega)) |\\
&=  |H(t, y'_1+ y'_2+ Y_t(\omega)) - H(s, y'_2 + Y_s(\omega)) |,
\end{align*}
where the first inequality comes from 
$\varphi_t-\varphi_s=H(t,Y_t-Y_s + Y_s)-H(s,Y_s)$,
$\supp(Y_t-Y_s,Y_s)=Q\times Q$, and from the continuity of
$Q\times Q \ni (y,y')\mapsto H(t,y+y')-H(s,y')$.
This implies 
\begin{align*}
\|\varphi_t -\varphi_s\|_{L_\infty} 
 \geqslant \sup_{y,y'\in Q}  |\E H(t, y+ y' + Y_{t}) - \E H(s, y' + Y_{s})|
& = \sup_{y,y'\in Q}  | H(0,y+ y') - H(0,y')| \\
&\geqslant  \sup_{y\in Q}   | H(0,y) - H(0,0)|.
\end{align*} 
For $s=0$ we use the same idea with $y'=y'_2=0$ to get 
$\|\varphi_t -\varphi_0\|_{L_\infty} \geqslant  \sup_{y\in Q}   | H(0,y) - H(0,0)|$.
So \eqref{item:1:prop:lower-bound-Levy-gen_new} yields to
$C:=H(0,0)= H(0,y)$ for all $y\in Q$. Fix $0\leqslant t < \overline{t} <T$ as in our assumption.
According to \cref{chaos-decom}, we have a chaos expansion
$$H(t, Y_{\overline{t}})=\E H(t, Y_{\overline{t}}) + \sum_{n=1}^\infty I_n\(\tilde h_n\1_{[0, \overline{t}]}^{\otimes n}\)$$
with $\tilde h_n \in L_2(\mu^{\otimes n})$. For $\Delta t := \overline{t}-t>0$ this implies
$\ce{\F_{\Delta t}}{H(t,Y_{\ov{t}})} \stackrel{a.s.}{=} \wt{\E} H(t,Y_{\Delta t}+\wt{Y}_t)= H(0,Y_{\Delta t})=C$,
where $\wt{Y}_t$ is an independent copy of $Y_t$ with corresponding expectation $\wt{\E}$, and
\[
  C = \ce{\F_{\Delta t}}{H(t, Y_{\overline{t}})} 
    = \E H(t, Y_{\overline{t}}) + \sum_{n=1}^\infty I_n\(\tilde h_n\1_{[0, \Delta t]}^{\otimes n}\)  \;\mbox{a.s.}
\]
Therefore, $\tilde h_n=0 $ in $L_2(\mu^{\otimes n})$ for all $n\geqslant 1$, which yields
$H(t, Y_{\overline{t}})= C$ a.s.
Since $\supp(Y_{\overline{t}}) = Q = \supp(Y_t)$, together with the continuity of $H(t, \cdot)$ on $Q$, we derive 
that $H(t, y) = C$ for all $y\in Q$. Therefore $\varphi_t  = H(t, Y_t)=C$ a.s.
\end{proof}
\smallskip

In the next statement we show that our directional gradients are of {\it maximal oscillation:}
\medskip

\begin{theo}[{\bf maximal oscillation and forward uniqueness of gradient}]
\label{statement:maximal_oscillation_application}
We assume a L\'evy process $(X_t)_{t\in [0,T]}$, $\eta\in [0,1]$,
\begin{enumerate}[{\rm (a)}]
\item \label{ass:a:statement:maximal_oscillation_application}
      $(X_t)_{t\in [0,T]}\subseteq L_\eta$,
\item \label{ass:b:statement:maximal_oscillation_application}
      $\|X_s\|_{{\rm TV}(\rho,\eta)}< \infty$
      for $s\in (0,T]$, where $\rho$ is a probability measure on $(\R,\cB(\R))$ 
      with $\rho(\{0\})=0$ if $\sigma=0$. 
      \footnote{Regarding the condition $\|X_s\|_{{\rm TV}(\rho,\eta)}< \infty$ for $s\in (0,T]$, 
                see also the remarks following \cref{definition:TVrho}.}
\end{enumerate}
Assume further that $f\in \hoel\eta$ with the additional assumptions that 
\begin{enumerate}[{\rm (a)}]
\item [{\rm (c)}]  \label{ass:c:statement:maximal_oscillation_application}
      $y \mapsto f(y+ \ell T)$ is continuous as map from $Q$ to $\R$ if $\eta=\sigma=0$,
\item [{\rm (d)}]  \label{ass:d:statement:maximal_oscillation_application}
      $\E|f(X_T)|^q<\infty$ for some $q>1$ if $\rho(\{0\})>0$ and $\eta \in (0,1]$.
\end{enumerate}
For $(t,x) \in [0,T)\times \R$ define  
\[ \opD_\rho F(t,x) := \rho(\{0\}) \frac{\partial F}{\partial x}(t,x) + D_\rho F(t,x)
   \sptext{1}{and}{1}
   \varphi_t  := \opD_\rho F(t,X_t), \]
where we omit the first term in $\opD_\rho F(t,x) $ if $\rho(\{0\})=0$.
Then the following holds:
\begin{enumerate}[{\rm (1)}]
\item \label{item:1:statement:maximal_oscillation_application}
      $f\in \dom(\Gamma_\rho^0)$.
\item \label{item:2:statement:maximal_oscillation_application}
      $F(t,\cdot)\in C^\infty(\R)$ for $t\in [0,T)$ if $\rho(\{0\})>0$.
\item \label{item:3:statement:maximal_oscillation_application}
      $\|\varphi_t\|_{L_\infty} = \sup_{x\in \supp(X_t)} |\opD_\rho F(t,x)|<\infty$ for $t\in [0,T)$.
\item \label{item:4:statement:maximal_oscillation_application}
      $(\varphi_t)_{t\in [0,T)}$ is a martingale of maximal oscillation with constant 2.
\item \label{item:5:statement:maximal_oscillation_application}
      Unless  $\varphi_t = \varphi_0$ a.s. for some $t\in [0,T)$, one has
      $\inf_{t\in (0,T)} \losc_t(\varphi)>0$.
\end{enumerate}
\end{theo}
\medskip

\begin{proof}
\eqref{item:1:statement:maximal_oscillation_application} 
follows from \cref{statement:TV_to_Gamma}\eqref{item:1:statement:TV_to_Gamma}
and assumptions \eqref{ass:a:statement:maximal_oscillation_application} and \eqref{ass:b:statement:maximal_oscillation_application}.
\smallskip

\eqref{item:2:statement:maximal_oscillation_application}
Assumption  \eqref{ass:b:statement:maximal_oscillation_application} implies $\sigma>0$.
Therefore $Q=\R$ by \cref{statement:consistence-upper_bound_sigma>0},
and \eqref{item:2:statement:maximal_oscillation_application} follows from (d) and
\cref{statement:consistence-upper_bound_sigma>0}\eqref{item:1:statement:consistence-upper_bound_sigma>0}. 
\smallskip

\eqref{item:3:statement:maximal_oscillation_application}
For $(t,y)\in [0,T)\times Q$ let us set 
\begin{equation*}
k(y)   := f(y+\ell T), \quad 
K(t,y) :=F(t,y+\ell t),  
  \sptext{1}{and}{1}
H(t,y) := \opD_\rho F(t,y + \ell t).
\end{equation*}
For $(t,y)\in [0,T)\times Q$ this 
yields to 
\[   H(t,y)
   = \opD_\rho F(t,y + \ell t) 
   = \rho(\{0\}) \frac{\partial K}{\partial y}(t,y) + \int_{Q\setminus \{0\}} \frac{K(t,y+z)-K(t,y)}{z}\rho(\od z).
   \]
We also have $\sup_{y\in Q}|H(t,y)|< \infty$ for all $t\in [0,T)$ because of \eqref{eqn:Bb_bound_H_part_1}, 
\eqref{eqn:Bb_bound_H_part_2}, and the assumption (b).
What is left is to show
$\|\varphi_t\|_{L_\infty} = \sup_{x\in \supp(X_t)} |\opD_\rho F(t,x)|$ for $t\in [0,T)$.
For $t=0$ we simply have
$\| \varphi_0 \|_{L_\infty} =  |\opD_\rho F(0,0)|$ and $\supp(X_0)=\{0\}$.
So let us assume that $t\in (0,T)$.
By  \cref{statement:consistence-upper_bound_sigma>0}\eqref{item:1:statement:consistence-upper_bound_sigma>0} 
and \cref{statement:TV_to_Gamma}\eqref{item:2:statement:TV_to_Gamma} and assumptions (c) and (d) 
we obtain that $H$ is $Y$-consistent. This implies 
\[ \| H(t,Y_t)\|_{L_\infty} = \sup_{y\in Q} |H(t,y)| \]
and
\begin{align*}
     \| \varphi_t          \|_{L_\infty}
& =  \| \opD_\rho F(t,X_t) \|_{L_\infty} 
  =  \| \opD_\rho F(t,Y_t+\ell t) \|_{L_\infty} 
  =  \|           H(t,Y_t       ) \|_{L_\infty} 
  = \sup_{y\in Q} |H(t,y)| \\
& = \sup_{y\in Q} |\opD_\rho F(t,y+\ell t)| 
  = \sup_{x\in \supp(X_t)} |\opD_\rho F(t,x)|.
\end{align*}
\eqref{item:4:statement:maximal_oscillation_application} and \eqref{item:5:statement:maximal_oscillation_application}
follows from \cref{statement:lower-bound-Levy-gen_new} because $\varphi_t = H(t,Y_t)$ a.s. for $t\in [0,T)$.
\end{proof}
\medskip

Now we provide in \cref{statement:lower_bound_complete} the corresponding lower bounds to \cref{statement:upper_bounds_Drho}.
Here we use \cref{defi:lower_bounds_X_rho} which is the counterpart to \cref{defi:uppper_bounds_X_rho}
and use with \cref{defi:lbvlocal} functions that have a certain lower variation at zero:
\medskip

\begin{defi}
\label{defi:lower_bounds_X_rho}
\hspace{0em}
\begin{enumerate}
\item For a probability measure $\rho$ on $(\R,\cB(\R))$ and $\varepsilon > 0$ we let
      $\rho\in L(\varepsilon)$ if one has
      \[ \liminf_{d\downarrow 0} \frac{\rho([-d,d]\setminus \{0\})}{d^\varepsilon}>0.\]
\item For a L\'evy process $X=(X_t)_{t\in [0,T]}$ and $\beta \in (0,2]$ we let $X\in \cL(\beta)$ 
      provided that there are $c,r>0$ such that
      for all $-r \leqslant a < b \leqslant r$ and $s\in (0,T]$ one has
      \begin{equation}\label{eqn:condition_Lbeta}
      \p \left ( s^{-\frac{1}{\beta}}X_s \in (a,b) \right ) \geqslant \frac{1}{c}(b-a).
      \end{equation}
\end{enumerate}
\end{defi}

\begin{rema}\hspace{0em}
\begin{enumerate}
\item In \cref{defi:lower_bounds_X_rho} the parameter $\varepsilon=0$ is not possible
      as there are no measures $\rho\in L(0)$.
\item In an adapted form the condition $\rho\in L(\varepsilon)$
      is used for L\'evy measures in the context of infinitely divisible distributions in 
      \cite[Condition $D_\beta$]{Orey:68} (cf. \cite[Proposition 28.3]{Sa13})
      to derive  distributional properties.
\end{enumerate}
\end{rema}

\begin{defi}
\label{defi:lbvlocal}
For $\eta \in [0,1]$ and $f\in \bvlocal$ we let $f\in \lbvlocal \eta$ if $f'=(\mu_f,0)$ in the sense of \eqref{eqn:f'_measure} with
\[ \liminf_{d\downarrow 0} \frac{\mu_f([-d,d])}{d^\eta}
   = \liminf_{d\downarrow 0} \frac{f(d)-f(d-)}{d^\eta}>0. \]
\end{defi}

The functions $f\in \lbvlocal \eta$ are non-decreasing and right-continuous.
Now we are ready to formulate our lower bounds:
\medskip

\begin{theo}[\bf size of oscillation of gradient]
\label{statement:lower_bound_complete}
For $\eta\in (0,1)$, a L\'evy process $(X_t)_{t\in [0,T]}\subseteq L_\eta$,  
$f\in \lbvlocal \eta \cap \hoel \eta$,
and $F: [0,T] \times \R \to \R$ with $F(t,x): = \E f(x+X_{T-t})$ one has:
\bigskip

\begin{enumerate}[{\rm (1)}]
\item \label{item:1:statement:lower_bound_complete}
      If $(\varepsilon,\beta)\in (0,1)\times (0,2]$, 
      $\eta \in (0,1-\varepsilon)$,
      $\|X_s\|_{{\rm TV}(\rho,\eta)}< \infty$ for $s\in (0,T]$, 
      $\rho \in L(\varepsilon)$ with $\rho(\{0\})=0$, and $X \in \cL(\beta)$, then 
      $\varphi=(D_\rho F(t,X_t))_{t\in [0,T)}$ 
      is an $L_\infty$-martingale with
      \[ \inf_{t\in (0,T)} (T-t)^\alpha \losc_t(\varphi)>0
         \sptext{1}{for}{1}
         \alpha:= \frac{1-(\varepsilon+\eta)}{\beta}\in \left (0,\frac{1}{\beta}\right ). \]

\item \label{item:3:statement:lower_bound_complete_new}
      If $\sigma>0$, then $X\in \cL(2)$.

\item \label{item:2:statement:lower_bound_complete}
      If $\sigma>0$, $q>1$, and $\E |f(X_T)|^q<\infty$,
      then $\varphi=(\frac{\partial F}{\partial x} (t,X_t))_{t\in [0,T)}$
      is an $L_\infty$-martingale with  
      \[ \inf_{t\in (0,T)} (T-t)^\alpha \losc_t(\varphi)>0
         \sptext{1}{for}{1}
         \alpha:= \frac{1-\eta}{2}\in \left (0,\frac{1}{2}\right ). \]
\end{enumerate}         
\end{theo}
\medskip

\begin{rema}
The current formulation of \cref{statement:lower_bound_complete} excludes $\eta=0$:
On the one side, $\lbvlocal 0$ contains only functions, that are not continuous, as $f\in \lbvlocal 0$ with $f'=(\mu_f,0)$ implies
$\mu_f(\{0\})>0$. On the other side, \cref{statement:maximal_oscillation_application}, used in the 
proof of \cref{statement:lower_bound_complete}, requires 
formally continuity properties when $\eta=0$. So far, we did not investigate this further.
\end{rema}
\medskip

For the proof of \cref{statement:lower_bound_complete} we use the following lemma:

\begin{lemm}
\label{statement:lower_bounds_levy_new}
Assume a L\'evy process $(X_t)_{t\in [0,T]}$, $\eta \in [0,1]$, $f\in \lbvlocal{\eta}\cap \defX$,
and let $F: [0,T] \times \R \to \R$ be given by $F(t,x): = \E f(x+X_{T-t})$.
\begin{enumerate}[{\rm (1)}]
\item  \label{item:1:statement:lower_bounds_levy_new}
       If $\rho \in L(\varepsilon)$ and $X \in \cL(\beta)$ with 
       $(\varepsilon,\beta)\in (0,\infty)\times (0,2]$, then
       \[ \inf_{t\in [0,T)} (T-t)^{\frac{1-(\varepsilon+\eta)}{\beta}}  D_\rho F(t,0) \in (0,\infty]. \]

\item  \label{item:2:statement:lower_bounds_levy_new}
       If $\sigma>0$ and additionally $\E |f(X_T)|^q < \infty$ for some $q>1$, then
       $X\in \cL(2)$ and
       \[ \inf\limits_{t\in [0,T)} (T-t)^{\frac{1- \eta}{2}}  \frac{\partial F}{\partial x}(t,0) \in (0,\infty]. \]
\end{enumerate}
\end{lemm}
\smallskip

\begin{proof}
For the proof we use the notation $f'=(\mu_f,0)$.

\eqref{item:1:statement:lower_bounds_levy_new}
For $t\in [0,T)$ we observe that 
$\frac{F(t,x+z)-F(t,x)}{z}\geqslant 0$ for all $z\in \RO$ as $f$ is non-decreasing by assumption. 
For this reason we can use the proof of item \eqref{item:2:statement:properties:Gamma1} of \cref{statement:properties:Gamma1} 
{\it without} checking integrability assumptions to derive, for $d>0$, that
\begin{align*}
&   \int_{\RO} \frac{F(t,z)-F(t,0)}{z} \rho(\od z) \\
&  =       \int_\R \int_{\RO} \frac{\p(X_{T-t} \in J(v;z))}{|z|}  \rho(\od z) f'(\od v) \\
&\geqslant \int_{|v| \leqslant d} \int_{0<|z| \leqslant d}  \frac{\p(X_{T-t} \in J(v;z))}{|z|}  \rho(\od z) f'(\od v) \\ 
&\geqslant \inf_{|v|\leqslant d, 0<|z|\leqslant d} \frac{\p(X_{T-t} \in J(v;z))}{|z|}
           \rho([-d,d]\setminus \{0\}) \mu_f([-d,d])  \\
&\geqslant \inf_{-2d \leqslant a < b \leqslant 2d} \frac{\p(X_{T-t} \in [a,b))}{b-a}
           \rho([-d,d]\setminus \{0\}) \mu_f([-d,d]) \\
&= (T-t)^{-\frac{1}{\beta}} \left [ \inf_{-2d (T-t)^{-\frac{1}{\beta}} \leqslant a < b \leqslant 2d (T-t)^{-\frac{1}{\beta}}}
                   \frac{\p\left ( (T-t)^{-\frac{1}{\beta}} X_{T-t} \in [a,b) \right )}{b-a} \right ]  
           \rho([-d,d]\setminus \{ 0 \}) \mu_f([-d,d]).
\end{align*}
By assumption there are $c,r>0$ such that 
\[ \inf_{-r\leqslant a < b \leqslant r}
    \frac{\p\left ( (T-t)^{-\frac{1}{\beta}} X_{T-t} \in [a,b) \right )}{b-a} 
    \geqslant \frac{1}{c}.\]
Let $d_0:= \frac{r}{2} T^\frac{1}{\beta}$. Then there are $c',c''>0$ such that
\[ \rho([-d,d]\setminus \{0\}) \geqslant \frac{1}{c'} d^\varepsilon
   \sptext{1}{and}{1}
   \mu_f([-d,d]) \geqslant \frac{1}{c''} d^\eta \]
for $d\in (0,d_0]$. Therefore for $d:= \frac{r}{2} (T-t)^\frac{1}{\beta}\in (0,d_0]$
we can continue to 
\[            \int_{\RO} \frac{F(t,z)-F(t,0)}{z} \rho(\od z) 
   \geqslant (T-t)^{-\frac{1}{\beta}} \, \frac{1}{c c' c''} d^{\varepsilon+\eta} 
     = \left ( \frac{r}{2} \right )^{\varepsilon+\eta} \frac{1}{c c' c''} 
       (T-t)^{\frac{\varepsilon+\eta-1}{\beta}}.
\]
(\ref{item:2:statement:lower_bounds_levy_new}a) 
For $s\in (0,T]$ we let $p_s = p_{\sigma W_s} * \p_{J_s}$ be the continuous density of the law of $X_s$ (see \cref{statement:X-rho_to_TV}). Then we have
\begin{align}\label{eq:lower_bound_density_sigma_positiv}
          p_s(x)
   =      \frac{1}{\sigma \sqrt{2\pi s}} \int_{\R}   \e^{-\frac{(x-z)^2}{2\sigma^2 s}}\p_{J_{s}}(\od z)
\geqslant \frac{1}{\sigma \sqrt{2\pi s}} \e^{-\frac{x^2}{\sigma^2 s}} \int_\R \e^{-\frac{z^2}{\sigma^2 s}}\p_{J_{s}}(\od z)
\geqslant c_\eqref{eq:lower_bound_density_sigma_positiv} s^{-\frac{1}{2}} \e^{-\frac{x^2}{\sigma^2 s}} 
\end{align}
with 
$   c_\eqref{eq:lower_bound_density_sigma_positiv}
 := (\sigma\sqrt{2\pi})^{-1} \inf_{u\in (0,T]} \int_\R \e^{-\frac{z^2}{\sigma^2 u}} \p_{J_u}(\od z)$.
It remains to verify that
\[   \inf_{u\in (0,T]} \int_\R \e^{-\frac{z^2}{\sigma^2 u}} \p_{J_u}(\od z) >0. \]
For this distributional estimate we may simply assume that $J_t=J_t^{(0)} + J_t^{(1)}$, where 
$(J_t^{(i)})_{t\in [0,T]}$ are independent  L\'evy processes without Brownian part with L\'evy measures $\nu^{(i)}$ satisfying 
$\nu^{(0)}(\{ z\in \R: |z|  > 1\})=0$ and $\nu^{(1)}(\{ z\in \R:|z| \leqslant 1\})=0$. Then
\[ \E \e^{-\frac{J_u^2}{\sigma^2 u}}
   \geqslant  \E \e^{- 2 \frac{(J_u^{(0)})^2}{\sigma^2 u} - 2 \frac{(J_u^{(1)})^2}{\sigma^2 u}} 
       =      \E \e^{- 2 \frac{(J_u^{(0)})^2}{\sigma^2 u}} \E \e^{ - 2 \frac{(J_u^{(1)})^2}{\sigma^2 u}}. \]
For the first term we get by Jensen's inequality that
\[  \E \e^{- 2 \frac{(J_u^{(0)})^2}{\sigma^2 u}}  \geqslant \e^{- 2 \frac{\E (J_u^{(0)})^2}{\sigma^2 u}} \]
and observe that $\sup_{u\in (0,T]} \frac{\E (J_u^{(0)})^2}{u} < \infty$ (see \cite[Example 25.12]{Sa13}).
Regarding the second term, $(J_t^{(1)})_{t\in [0,T]}$ is a compound Poisson process 
plus a possible linear drift \cite[Theorem 21.2]{Sa13} so that 
\[ \lim_{u\downarrow 0}\frac{(J_u^{(1)})^2}{u} = 0 \mbox{ a.s.} \]
As $J^{(1)}$ has a jump at a given deterministic time with probability zero, this also implies that the map 
$[0,T] \ni u \mapsto \E \e^{- 2 \frac{(J_u^{(1)})^2}{\sigma^2 u}} \in (0,\infty)$ is continuous,
where the value at $u=0$ is defined as limit from the right-hand side,
and has therefore a positive infimum. This proves $c_\eqref{eq:lower_bound_density_sigma_positiv}>0$.
Finally, since $p_s(x) \geqslant c_\eqref{eq:lower_bound_density_sigma_positiv} s^{-\frac{1}{2}} \e^{-\frac{1}{\sigma^2}}$
for $|x| \leqslant s^\frac{1}{2}$ this implies
condition \eqref{eqn:condition_Lbeta} with 
\begin{equation}\label{eqn:parameters_L_positive_sigma}
\beta=2, r=1, \sptext{.5}{and}{.5}
c:=c_\eqref{eq:lower_bound_density_sigma_positiv}^{-1} \e^{\frac{1}{\sigma^2}}.
\end{equation}

(\ref{item:2:statement:lower_bounds_levy_new}b) 
As $\E|f(X_T)|^q<\infty$ for some $q>1$ we apply \cref{statement:consistence-upper_bound_sigma>0}
to deduce $F(t,\cdot)\in C^\infty(\R)$ for $t\in [0,T)$. For $\varepsilon\in (0,1]$ we define the probability measures 
$\rho_\varepsilon (\od x) := \varepsilon \1_{(0,1]}(x) x^{\varepsilon-1} \od x$ that satisfy
$\rho_\varepsilon([-d,d]\setminus \{0\}) = d^\varepsilon$ for $d\in (0,1]$.
The above computation in \eqref{item:1:statement:lower_bounds_levy_new} with the parameters from \eqref{eqn:parameters_L_positive_sigma}
gives 
\[           \int_{\RO} \frac{F(t,z)-F(t,0)}{z} \rho_\varepsilon(\od z)
   \geqslant \frac{1}{2^{\varepsilon+\eta} c c''} 
             \,\, (T-t)^{\frac{\varepsilon+\eta-1}{2}}. \] 
The statement follows by $\varepsilon \downarrow 0$.
\end{proof}
\smallskip

\begin{rema}
\hspace{0em}
\begin{enumerate}
\item In \cref{statement:lower_bounds_levy_new}
      Item \eqref{item:2:statement:lower_bounds_levy_new} can be seen as the limiting case of \eqref{item:1:statement:lower_bounds_levy_new}
      for $\varepsilon=0$.
\item If $\mu_f$ in \cref{statement:lower_bounds_levy_new}
      is a finite measure, then $f$ is bounded and therefore $f\in \defX$ and $\E |f(X_T)|^q < \infty$ for all $q \in (0,\infty)$.
\item Given $\eta \in [0,1]$, a typical example for $f\in \lbvlocal \eta \cap \hoel \eta$
      is $f(x) := (x\wedge 1)^\eta \1_{[0,\infty)}(x)$ (with the convention $0^0=1$).
\end{enumerate}
\end{rema}
\bigskip

\begin{proof}[Proof of \cref{statement:lower_bound_complete}]
\eqref{item:1:statement:lower_bound_complete}
The assumptions of \cref{statement:maximal_oscillation_application} are satisfied and the
assumptions of \cref{statement:lower_bounds_levy_new}\eqref{item:1:statement:lower_bounds_levy_new}
are satisfied as well, where we use that $f\in \hoel \eta$ and $X\subseteq L_\eta$ imply $f\in \defX$.
\smallskip

Now $X\in \cL(\beta)$ implies $0\in \supp(X_s)$ for $s\in (0,T]$ so that 
\begin{equation}\label{eqn:lower_bound_oscillation}
             \inf_{t\in (0,T)} (T-t)^{\frac{1-(\varepsilon+\eta)}{\beta}} \|\varphi_t \|_{L_\infty}
   \geqslant \inf_{t\in (0,T)} (T-t)^{\frac{1-(\varepsilon+\eta)}{\beta}} D_\rho F(t,0) 
             =: c_\eqref{eqn:lower_bound_oscillation}> 0 
\end{equation}
by \cref{statement:maximal_oscillation_application}\eqref{item:3:statement:maximal_oscillation_application} and
\cref{statement:lower_bounds_levy_new}\eqref{item:1:statement:lower_bounds_levy_new}.
By \cref{statement:maximal_oscillation_application}\eqref{item:4:statement:maximal_oscillation_application} the martingale 
$\varphi$ is of maximal oscillation with constant $2$ so that
we get, for $t\in (0,T)$,  
\[ \losc_t(\varphi) \geqslant \frac{1}{2} \| \varphi_t - \varphi_0\|_{L_\infty}
                    \geqslant \frac{1}{2} \| \varphi_t \|_{L_\infty} - \frac{1}{2} \| \varphi_0\|_{L_\infty}
                    \geqslant \frac{1}{2 c_\eqref{eqn:lower_bound_oscillation}} (T-t)^{-\frac{1-(\varepsilon+\eta)}{\beta}}   - \frac{1}{2} \| \varphi_0\|_{L_\infty}.\]
It remains to show that $\inf_{t\in (0,T)} \losc_t(\varphi)>0$ which follows from 
\cref{statement:maximal_oscillation_application}\eqref{item:5:statement:maximal_oscillation_application} because 
$\lim_{t\uparrow T} \| \varphi_t - \varphi_0\|_{L_\infty}=\infty$.
\smallskip

\eqref{item:3:statement:lower_bound_complete_new} and \eqref{item:2:statement:lower_bound_complete} 
\cref{statement:lower_bounds_levy_new}\eqref{item:2:statement:lower_bounds_levy_new}
gives $X\in \cL(2)$ and therefore also $0\in \supp(X_s)$, $s\in [0,T]$.
Now the proof of \eqref{item:2:statement:lower_bound_complete}  is analogous to \eqref{item:1:statement:lower_bound_complete},
where we use \cref{statement:consistence-upper_bound_sigma>0},
\cref{statement:maximal_oscillation_application} with $\rho=\delta_0$,
and \cref{statement:lower_bounds_levy_new}\eqref{item:2:statement:lower_bounds_levy_new}.
\end{proof}
\bigskip

We conclude with the natural example of $\beta$-stable like processes where one part of the example
is known from the literature and the other part is close to existing results:

\begin{exam}
\label{statement:U-L-for-beta-like}
Let $\beta \in (0,2)$ and $X=(X_t)_{t\in [0,T]}$ be a symmetric L\'evy process with 
$\sigma=0$ and  L\'evy measure $\nu(\od z) = p_\nu(z)\od z$, where $p_\nu$ is symmetric and satisfies
\begin{align}\label{eq:exam-assumption-levy-density}
      0 < \liminf_{|z| \to 0}|z|^{1+ \beta} p_\nu(z) \leqslant \limsup_{|z| \to 0}|z|^{1+ \beta} p_\nu(z) <\infty.
\end{align}
Then one has  
$X \in \left ( \bigcup_{c>0} \cU(\beta,\R;c) \right ) \cap \cL(\beta)$.
\end{exam}

\begin{proof}
The fact $X \in \bigcup_{c>0} \cU(\beta,\R;c)$ has been proven in
\cite[Example 1.3]{Boetcher:Schilling:Wang:11}. As we work with a $\liminf$-condition instead
a global lower bound and as we need the same preparations to prove 
$X \in \cL(\beta)$ we include the proof for the upper bound for the convenience of the reader.
Let $\psi$ be the characteristic exponent of $X$, i.e. $\E \e^{\im u X_s} = \e^{-s \psi(u)}$
(see \cite[Theorem 8.1]{Sa13}) for $s\in [0,T]$. 
By \cref{eq:exam-assumption-levy-density} we obtain
\begin{align}\label{eq:exam:asymptotic-chracteristic-exponent}
0 < \liminf_{|u| \to \infty} \frac{\mathrm{Re}\psi(u)}{|u|^{\beta}} 
    \leqslant \limsup_{|u| \to \infty} \frac{\mathrm{Re}\psi(u)}{|u|^{\beta}} <\infty.
\end{align}
If $s\in (0,T]$, then $X_s$ has a symmetric density $p_s \in C^\infty(\R)$ with 
$\lim_{|x|\to \infty} (\pd^m p_s/ \pd x^m) (x)=0$ for $m\in \bN_0$
by \cite{Orey:68} (see \cite[Proposition 28.3]{Sa13}).
We combine \cref{eq:exam:asymptotic-chracteristic-exponent} with \cite[Theorem 1.3]{SSW12} 
and \cite[Lemma 4.1]{KS18} and obtain $s_0 \in (0,T]$ and $c_0>0$ such that
$\left\|\pd p_s/\pd x\right\|_{L_1(\R)} \leqslant c_0 s^{-\frac{1}{\beta}}$
for $s \in (0, s_0]$.
If $s\in (s_0, T]$, then
$\left\|\pd p_s/\pd x\right\|_{L_1(\R)} = \left\|(\pd p_{s_0}/\pd x) * p_{s-s_0}\right\|_{L_1(\R)} \leqslant c_0 
 s_0^{-\frac{1}{\beta}}$, so that
\begin{equation}\label{exam:upper-gradient-density_new}
           \left\|\frac{\pd p_s}{\pd x}\right\|_{L_1(\R)} 
\leqslant c_{\cref{exam:upper-gradient-density_new}} s^{-\frac{1}{\beta}} 
\sptext{1}{for}{1} s\in (0,T]
\end{equation}
with $c_{\cref{exam:upper-gradient-density_new}}: = c_0 (\frac{T}{s_0})^{\frac{1}{\beta}}$.
On the other hand, by \cref{eq:exam:asymptotic-chracteristic-exponent} there is a 
$c_{\cref{eq:exam:chracteristic-exponent-2}}=c(\beta, p_\nu)>0$ such that, for $s\in (0,T]$,
\begin{align}\label{eq:exam:chracteristic-exponent-2}
\frac{1}{c_{\cref{eq:exam:chracteristic-exponent-2}} } s^{-\frac{1}{\beta}} \leqslant \int_{\R} \e^{-s \mathrm{Re} \psi(u)}\od u  
\sptext{1}{and}{1}
\int_{\R} \e^{-s \mathrm{Re} \psi(u)}|u|\od u \leqslant c_{\cref{eq:exam:chracteristic-exponent-2}} 
   s^{-\frac{2}{\beta}}.
\end{align}
Combining \cref{eq:exam:chracteristic-exponent-2} with  the proof of \cite[Lemma 7]{KS15} yields
$c_{\cref{eq:exam:lower-bound}},\tilde c_{\cref{eq:exam:lower-bound}}>0$, not depending on $(s,x)$, such that
\begin{align}\label{eq:exam:lower-bound}
  p_s(x) \geqslant c_{\cref{eq:exam:lower-bound}} s^{-\frac{1}{\beta}}
  \sptext{1}{for}{1} |x| < \tilde c_{\cref{eq:exam:lower-bound}} s^{\frac{1}{\beta}} 
  \sptext{1}{and}{1} s\in (0, T]. 
\end{align}
To verify our claim we first deduce from \eqref{eqn:TV_vs_gradient_density} that
\[ \| \p_{z+X_s} - \p_{X_s}\|_{\rm TV}
   \leqslant |z| \left \| \frac{\partial p_s}{\partial x}\right \|_{L_1}
   \leqslant |z| c_{\cref{exam:upper-gradient-density_new}} s^{-\frac{1}{\beta}} \]
so that $X\in \cU(\beta,\R;c_{\cref{exam:upper-gradient-density_new}})$.
On the other side we have 
\[ \p\left ( s^{-\frac{1}{\beta}} X_s \in (a,b) \right )
   = \int_{s^\frac{1}{\beta} a}^{s^\frac{1}{\beta}b} p_s(x) \od x
  \geqslant s^{\frac{1}{\beta}}(b-a) c_{\cref{eq:exam:lower-bound}} s^{-\frac{1}{\beta}}
   = c_{\cref{eq:exam:lower-bound}} (b-a) \]
for $-\tilde c_{\cref{eq:exam:lower-bound}} < a < b < \tilde c_{\cref{eq:exam:lower-bound}}$ 
which verifies $X\in \cL(\beta)$.
\end{proof}

\bigskip

\begin{proof}[Proof of \cref{statement:intro:gradient_LI_space_upper_boud}]
Item \eqref{eqn:1:statement:intro:gradient_LI_space_upper_boud}
follows from \cref{statement:upper_bounds_Drho}\eqref{item:2:statement:upper_bounds_Drho}.
For \eqref{eqn:2:statement:intro:gradient_LI_space_upper_boud} we first apply \eqref{eqn:1:statement:intro:gradient_LI_space_upper_boud}
for $q=2$ and observe that 
\begin{equation}\label{eqn:proof:statement:intro:gradient_LI_space_upper_boud}
           (1/2) \| (\varphi_t(f,\rho)-\varphi_0(f,\rho))_{t\in [0,T)} \|_{\B_{\infty,2}^\alpha} 
 \leqslant \| \varphi(f,\rho) \|_{\B_{\infty,2}^\alpha}
 \leqslant \| D_\rho F         \|_{\B_{b,2}^\alpha}
 \leqslant c_\eqref{eqn:1:statement:intro:gradient_LI_space_upper_boud} \|f\|_{\hoelO{\eta,2}}.
\end{equation}

To estimate the first term in \eqref{eqn:2:statement:intro:gradient_LI_space_upper_boud}, we apply 
\eqref{eqn:relations_between_Bpqalpha} to
$\| \varphi(f,\rho) \|_{\B_{\infty,2}^\alpha}$ and finish with 
\cref{statement:properties_Besov-spaces}\eqref{item:3:statement:properties_Besov-spaces}.
For the second term in \eqref{eqn:2:statement:intro:gradient_LI_space_upper_boud}
we use \cref{statement:properties_Besov-spaces}\eqref{item:2:statement:properties_Besov-spaces}
and again \eqref{eqn:proof:statement:intro:gradient_LI_space_upper_boud}.
\end{proof}


\section[On representations of H\"older functionals]{On representations of H\"older functionals on the  L\'evy-It\^o space}
\label{sec:approximation_random_measure}


\subsection{Galtchouk-Kunita-Watanabe projection}
\label{sec:GKW_projection_general}
For \cref{sec:GKW_projection_general} we assume 
the setting from \cref{sec:setting_levy_case}.
Assume $D\in L_2(\R,\mu)$ and 
$\xi \in L_2$ with a chaos decomposition 
\begin{equation}\label{eqn:special_chaos}
\xi = \E \xi + \sum_{n=1}^\infty I_n\left (\1_{(0,T]}^{\otimes n} f_n\right )
\end{equation}
with {\it symmetric} $f_n\in L_2(\R^n,\cB(\R^n),\mu^{\otimes n})$.
For $n\in \bN$, $x_1,\ldots,x_{n-1}\in\R$, and $0<t_1<\cdots<t_n<T$
we define
\begin{align*}
 h_{n-1}^D (x_1,\ldots,x_{n-1})&:= \!\!
  \begin{cases}
  \int_\R f_n(x_1,\ldots,x_{n-1},z) D(z) \mu(\od z)\hspace*{-.75em} & : \int_\R |f_n(x_1,\ldots,x_{n-1},z)|^2 \mu(\od z) < \infty \\
  0 & : \mbox{ else }
  \end{cases}\!, \\
f_n^D ((t_1,x_1),\ldots,(t_n,x_n)) &:= h_{n-1}^D (x_1,\ldots,x_{n-1}) D(x_n), 
\end{align*}
and extend $f_n^D$ symmetrically to $((0,T]\times \R)^n$ where the function is defined to be zero if
there are $k\not = l$ with $t_k=t_l$.
If $\| D \|_{L_2(\mu)}>0$, then we define
\[ \xi^D :=  \| D\|_{L_2(\mu)}^{-2} \sum_{n=1}^\infty  I_n(f_n^D). \]
The random variable $\xi^D$ can be written as stochastic integral
\begin{equation}\label{eqn:representation_stochastic_integral_wrt_XD}
\xi^D = \| D\|_{L_2(\mu)}^{-2}\int_{(0,T)} \psi_{t-}(\xi,D) \od X_t^D \mbox{ a.s.}
\end{equation}
where the c\`adl\`ag $L_2$-martingales $(\psi_t(\xi,D))_{t\in [0,T)}$ and 
$(X_t^D)_{t\in [0,T]}$ are determined by
\begin{align*}
\psi_t(\xi,D) = \sum_{n=0}^\infty (n+1) I_n\left (\1_{(0, t]}^{\otimes n} h_n^D \right ) \mbox{ a.s.}
\sptext{1}{and}{1}
X_t^D = I_1 \left (\1_{(0,t]}\otimes D \right ) \mbox{ a.s.}
\end{align*}
The proof of \eqref{eqn:representation_stochastic_integral_wrt_XD} is standard by 
the multiplication formula for the product $I_n(\1_{(0,s]}^{\otimes n} h_n^D) I_1(\1_{(s,t]} D)$ 
for $0\leqslant s < t \leqslant T$ (cf. also \cite{lee:shih:04} for general results).
It is known that $L_2(\R,\mu)$ is separable. 
If $(D_j)_{j\in J}$ is an orthonormal basis, where $J=\bN$ or $J$ is finite, then 
$X^{D_i}$ and $X^{D_j}$ are strongly orthogonal for $i\not = j$ (i.e. 
$(X_t^{D_i}X_t^{D_j})_{t\in [0,T]}$ is a martingale) and we obtain in $L_2(\Omega,\F,\p)$
the orthogonal decomposition 
\begin{equation}\label{eqn:orthogonal_decomposition_LI_space}
\xi = \E\xi+\sum_{j\in J} \int_{(0,T)} \psi_{t-}(\xi,D_j) \od X_t^{D_j}.
\end{equation}
Now we link $(\psi_t(\xi,D))_{t\in [0,T)}$ to the approach from \cref{sec:application_levy_case}:
\smallskip

\begin{prop}[Gradient of GKW-projection]
\label{statement:expression_for_integrand}
Assume that $f\in \defX$ with $\E |f(X_T)|^2<\infty$, 
$D\in L_1(\R,\mu)\cap L_2(\R,\mu)$ with $D\ge 0$ and $\mu(D>0)>0$, that $F$ 
is given by \eqref{eqn:def_F} and $\od \rho := \frac{D\od \mu}{\int_\R D \od\mu}$,
and that $t\in (0, T)$. Then there is a null-set $N_t\in \F$ such that for $\omega \not \in N_t$ one has
\begin{equation}\label{eq:expression-integrand-1}
  \frac{\psi_t(f(X_T),D)(\omega)}{\| D \|_{L_1(\mu)}} 
= \rho(\{0\}) \frac{\partial F}{\partial x}(t,X_t(\omega)) + \int_{\RO}
  \frac{F(t, X_t(\omega)+z) - F(t,X_t(\omega))}{z} \rho(\od z)
\footnote{We have
$\int_{\RO}\left |\frac{F(t, X_t(\omega)+z) - F(t,X_t(\omega))}{z} \right |\rho_D(\od z) <\infty$
for $\omega \not \in N_t$
and we omit $\rho(\{0\}) (\partial F/\partial x)(t,X_t(\omega))$ if $\rho(\{0\})=0$.}.
\end{equation}  
\end{prop}
\bigskip

We prove \cref{statement:expression_for_integrand} in \cref{sec:proof_of_expression_for_integrand}. 
Results related to \cref{statement:expression_for_integrand} are provided in  
\cite[Theorem 2.4]{JMP00},
\cite[Theorems 2.1, 3.11, 4.1]{BNLOP03}, and
\cite[Proposition 2]{CTV05}. 
Other techniques use the Fourier transform (see, e.g., \cite{BT11}).

\subsection{Upper and lower bounds for the Galtchouk-Kunita-Watanabe projection}
\label{sec:GKW_projection_upper_lower_bounds}
For Sections \ref{sec:GKW_projection_upper_lower_bounds}-\ref{sec:orthogonal_decomposition_Levy-Ito} we assume 
in addition to \cref{sec:setting_levy_case} the following:
\begin{enumerate}
\item $X=(X_t)_{t\in [0,T]} \subseteq L_2$ is symmetric with $\sigma=0$ and non-degenerate (i.e $\mu(\RO)\in (0,\infty)$).
\item $\beta \in (0,2)$ and the L\'evy measure is of form $\nu(\od z) = p_\nu(z)\od z$, where $p_\nu\geqslant 0$ is symmetric and
      \begin{align}\label{eq:exam-assumption-levy-density-II}
      0 < \liminf_{|z| \to 0}|z|^{1+ \beta} p_\nu(z) \leqslant \limsup_{|z| \to 0}|z|^{1+ \beta} p_\nu(z) <\infty.
      \end{align}
\end{enumerate}
The condition $\mu(\R)<\infty$ guarantees in \cref{statement:space_time_approximation_random_measure}
uniform bounds that do not depend on the choice of the orthonormal system used to decompose the L\'evy-It\^o space.
From assumption \cref{eq:exam-assumption-levy-density-II} we
get a constant $ c_\eqref{eqn:small_ball_mu_beta_like}>0$ such that 
\begin{equation}\label{eqn:small_ball_mu_beta_like}
	\mu([-d,d]) \leqslant c_\eqref{eqn:small_ball_mu_beta_like}  d^{2-\beta} 
	\sptext{1}{for}{1}
	d \in (0,1].
\end{equation}
The results of this section (including their proofs) hold in a
more general setting using the conditions $\cU(\beta,\R;c')$, $\cL(\beta)$, and
\eqref{eqn:small_ball_mu_beta_like}. The symmetry assumptions are due to
\cref{statement:U-L-for-beta-like} to obtain lower bounds. To simplify
the presentation we decided to restrict ourselves to the $\beta$-stable like 
case.
\smallskip
For a time-net $\tau= \{t_i\}_{i=0}^n \in \mathcal T$, $D\in L_2(\mu)$, 
a Borel function $f:\R\to \R$ with $\E |f(X_T)|^2<\infty$, and 
$\xi:= f(X_T)$ we let
\begin{align*}
E_t(f;\tau,D) := \int_{(0,t]} \psi_{s-}(\xi,D) \od X_s^D - \sum_{i=1}^n \psi_{t_{i-1}-}(\xi,D)(X^D_{t_i \wedge t} - X^D_{t_{i-1} \wedge t}),
\quad t\in [0,T),
\end{align*}
where we exploit \cref{chaos-decom} and assume kernels $f_n$ for the representation of $\xi=f(X_T)$ as in  \eqref{eqn:special_chaos}
and define $(\psi_t(f(X_T),D))_{t\in [0,T)} = (\psi_t(\xi,D))_{t\in [0,T)}$
accordingly.

\medskip
For $0 \leqslant a \leqslant t < T$  we use It\^o's isometry and
and choose $\od \langle X^D \rangle_u = \left (\int_\R D^2 \od \mu \right ) \od u$
to get, a.s.,  
\begin{align*}
      \ce{\F_a}{|E_t(f; \tau,D) - E_a(f; \tau,D)|^2} 
& = \ce{\F_a}{\int_a^t \left| \psi_{u-}(\xi,D) - \sum_{i=1}^n \varphi_{t_{i-1}-}(\xi,D) \1_{(t_{i-1}, t_i]}(u)\right|^2 \od \langle X^D \rangle_u} \\
& = \left (\int_\R D^2 \od \mu \right )  \ce{\F_a}{\intsq{\psi(\xi,D)}{\tau}_t - \intsq{\psi(\xi,D)}{\tau}_a}
\end{align*}
where we use $\psi_u(\xi,D) = \psi_{u-}(\xi,D)$ a.s., $u\in [0,T)$, which follows from the
chaos expansion.
Hence
\begin{equation}\label{eqn:E_t_vs_squarefunction}
   \frac{1}{\left (\int_\R D^2 \od \mu \right )} \left \| E(f; \tau,D)        \right \|_{\bmo_2([0,T))}^2 
   = \left \| \intsq{\psi(\xi,D)}{\tau} \right \|_{\bmo_1([0,T))} 
   = \left \| \intsq{\psi(\xi,D)}{\tau} \right \|_{\BMO_1([0,T))}.
\end{equation}

Now we formulate in \cref{statement:upper_bounds_alpha_like} and \cref{statement:lower_bounds_alpha_like} the upper and lower bounds for the gradients
$(\psi_t(f(X_T),D))_{t\in [0,T)}$
and the error process $(E_t(f;\tau,D))_{t\in [0,T)}$:
\bigskip

\begin{theo}[{\bf upper bounds}]
\label{statement:upper_bounds_alpha_like}
For $\beta \in (0,2)$, $\eta\in [0,1]$, $r\in [1,2]$, $1=\frac{1}{r}+\frac{1}{r'}$, 
$D \in L_2(\R,\mu) \cap L_{r'}([-1,1],\mu)$,
$\varepsilon :=  \frac{2-\beta}{r}$,
$\alpha:=\frac{1-(\varepsilon+\eta)}{\beta}$, and
\[ \vvvert D \vvvert_{(1,r')}:=\| D \|_{L_1(\R,\mu)}  \vee \| D \|_{L_{r'}([-1,1],\mu)} \]
one has:
\medskip
\begin{enumerate}[{\rm\bf (1)}]
		
\item \label{item:0:statement:upper_bounds_alpha_like}	
      There is a $c>0$ independent from $D$ such that, for all $f \in \hoel{\eta}$,
      \[ \|\psi_t(f(X_T),D)\|_{L_\infty}
         \leqslant  c \,\, |f|_\eta \,\, \vvvert D \vvvert_{(1,r')}
        \begin{cases}
			(T-t)^{-\alpha}            &:  \eta \in [0,1-\varepsilon) \\
			\ln \left ( 1+(T-t)^{-\frac{1-\eta}{\beta}} \right ) &:  \eta = 1 - \varepsilon < 1 \\
			1                                                      &:  \eta \in (1-\varepsilon,1) \cup \{ 1 \}
       \end{cases}. \]
      \smallskip
	  
\item \label{item:00:statement:upper_bounds_alpha_like}
      Assume additionally, $\beta \in (2-r,2)$, $\eta \in (0,1-\varepsilon)$,
	  $\theta:=1-2\alpha$, and  $q\in [1,\infty)$.
	  \bigskip
      \begin{enumerate}[{\rm\bf (a)}]
      \item \label{item:1:statement:upper_bounds_alpha_like}
      {\sc Singularity of the gradient:} There is a $c>0$ independent from $D$ such that, for $f \in \hoel{\eta,q}$, one has
      \[ \|   \psi(f(X_T),D) \|_{\B_{\infty,q}^\alpha} 
		 \leqslant  c \,  \vvvert D \vvvert_{(1,r')} \, \| f\|_{\hoelO{\eta,q}}. \]
      \vspace*{-.4em}

      \item \label{item:2:statement:upper_bounds_alpha_like}
      {\sc Approximation:}  
	  There is a $c>0$ independent from $D$ such that, for all $f \in \hoel{\eta,2}$ and $\tau\in \cT$, one has
		\[
			\left \| E(f; \tau,D)        \right \|_{\bmo_2([0,T))} 
			\leqslant c \,
			\| D \|_{L_2(\R,\mu)} \, \vvvert D \vvvert_{(1,r')}
			\sqrt{\| \tau \|_{\theta}} \,\, \| f \|_{\hoelO {\eta,2}}.
		\]\vspace*{-.4em}
		
		\item \label{item:3:statement:upper_bounds_alpha_like}
		 {\sc Tail behaviour of the approximation:} 
		If  $f \in \hoelO {\eta,2}$ and $p\in (2,\infty)$, then there is a $c>0$ such that for
		$0\leqslant a < t < T$, $\Phi\in \CL^+([0,t])$ with $1\vee |\Delta X_s^D| \leqslant \Phi_s$ on $[0,t]$, $\Phi_t^*\in L_p$,
		and $\lambda >0$ one has
		\[
			\p_{\F_a}\left (|E_t(f;\tau^\theta_n,D)- E_a(f;\tau^\theta_n,D)|> \lambda \right )
			\leqslant c \min \left \{ \frac{1}{n\lambda^2} , \frac{\ce{\F_a}{\sup_{u\in [a,t]}\Phi_u^p}}{(T-t)^{p\alpha} \lambda^p} \right \}
			\,\, \mbox{a.s.}
		\]
	\end{enumerate}
	\end{enumerate}
\end{theo}
\bigskip

The quantity $\vvvert D \vvvert_{(1,r')}$ describes how well $D$ behaves around the origin. The 
bigger $r'$ is, the more regular $D$ is, which yields to  
a weaker singularity of $(\psi_t(f(X_T),D))_{t\in [0,T)}$ as $t\uparrow T$
because of $\alpha=(1-(\frac{2-\beta}{r}+\eta))/\beta$.
Item \eqref{item:3:statement:upper_bounds_alpha_like} looks at first glance technical, but it gives some freedom with respect 
to the choice of the weight $\Phi$ and its integrability. 
In particular, the term $(T-t)^{-p\alpha}\lambda^{-p}$ behaves better than $n^{-1} \lambda^{-2}$ for large $\lambda$.
Therefore we can take an advantage of the John-Nirenberg theorem, behind the estimate 
$(T-t)^{-p\alpha}\lambda^{-p}$, even if we cannot exploit the factor $n^{-1}$.
The case  $D \in L_2([-1,1],\mu)$ is the typical case that appears later when $D$ is an element of our orthonormal basis,
the case $D \in L_\infty([-1,1],\mu)$ appears in the classical  Galtchouk-Kunita-Watanabe projection when $D\equiv 1$. The following table 
gives an overview about the parameters in  \eqref{item:2:statement:upper_bounds_alpha_like} in the limit cases $r=1$ ($r'=\infty$) and $r=r'=2$:
\bigskip
\begin{center}
\begin{tabular}{|c|c|c|c|c|c|c|}\hline
$r$ & $D$                & $\beta$ & $\eta$                & $\alpha$ & $\theta$ & range for $(\alpha,\theta)$ in dep. on $\eta$ \\ \hline
1   & $L_\infty([-1,1],\mu)$ & $(1,2)$ & $(0,\beta-1)$         &   $1-\frac{\eta+1}{\beta}$ & $\frac{2}{\beta}(\eta+1)-1 $  &
$(0,1-\frac{1}{\beta})\times \left ( \frac{2}{\beta} -1,1 \right )$ \\
2   & $L_2([-1,1],\mu)$      & $(0,2)$ & $(0,\frac{\beta}{2})$ & $\frac{1}{2}-\frac{\eta}{\beta}$ & $\frac{2}{\beta}\eta$ & $(0,\frac{1}{2})\times(0,1)$ \\ \hline
\end{tabular}
\end{center}
\bigskip
One has to read the table in the following way: The range 
for $(\alpha,\theta)$ is the range - in dependence on $\eta$ - where we have a singularity in the 
gradient that has to be compensated. If $r=1$, which means we have 
a 'good' $D\in L_\infty([-1,1],\mu)$, then this range is smaller
than for $r=2$.
\bigskip

\begin{theo}[\bf lower bound]
\label{statement:lower_bounds_alpha_like}
Let $r\in [1,2)$, $1=\frac{1}{r}+\frac{1}{r'}$, $\beta\in (2-r,2)$,
$\varepsilon:= \frac{2-\beta}{r}$, $\eta\in (0,1-\varepsilon)$,
$\alpha := \frac{1-(\varepsilon+\eta)}{\beta}$, and
$\theta := 1 - 2 \alpha$.
Then there is a non-negative Borel function $D:\R\to \R$
supported on $[-1,1]$, with $\mu(D>0)>0$,
\begin{enumerate}[{\rm (a)}]
\item $D\in L_\infty([-1,1],\mu)$ \, if $r=1$,
\item $\sup_{\lambda>0} \lambda^{r'} \mu \left ( \{ z\in [-1,1] : D(z) \geqslant \lambda\} \right ) < \infty$ \, if $r\in (1,2)$, 
\end{enumerate}
such that for all bounded 
$f\in \lbvlocal \eta \cap \hoel \eta$ 
and $\varphi=(\psi_t(f(X_T),D))_{t\in [0,T)}$ one has:
\begin{enumerate}[{\rm (1)}]
\item \label{item:1:statement:lower_bounds_alpha_like} 
      $\inf_{t\in (0,T)} (T-t)^\alpha \losc_t(\varphi)>0$.
\item \label{item:2:statement:lower_bounds_alpha_like} 
      There is a constant $c>0$ such that
      \[            \left \| E(f; \tau,D)        \right \|_{\bmo_2([0,T))} 
         \geqslant  c\, \sqrt{\| \tau \|_{\theta}}\sptext{1}{for all}{1} \tau\in\cT. \]
\end{enumerate}
\end{theo}
From  the conditions on $(r,\beta)$ it automatically follows that
$\varepsilon\in(0,1)$, $\alpha\in (0,1/2)$, and $\theta \in (0,1)$.

\begin{rema}
In \cref{statement:lower_bounds_alpha_like} the lower bound is formulated for  Riemann approximations used in $ E(f; \tau,D)$. However, this lower bound is also true for {\it any approximation} having the correct measurability. In fact, we show in the proof of  \cref{statement:lower_bounds_alpha_like} that one has:
\medskip

There is a $c_{\eqref{eqn:item:1:statement:lower_bounds_alpha_like}} >0$ such that  for all 
$\tau=\{t_i\}_{i=0}^n \in \mathcal T$ with $\|\tau\|_\theta = \frac{t_k-t_{k-1}}{(T-t_{k-1})^{1-\theta}}$ one has 
\begin{align}\label{eqn:item:1:statement:lower_bounds_alpha_like}
\inf_{\vartheta_{i-1}\in L_0(\F_{t_{i-1}})} \sup_{a\in [t_{k-1},t_k)}
\left \|\ce{\F_a}{\int_a^T \left| \varphi_u - \sum_{i=1}^n \vartheta_{i-1} \1_{(t_{i-1},t_i]}(u)  \right|^2 \od u} \right \|_\infty
\geqslant c_{\eqref{eqn:item:1:statement:lower_bounds_alpha_like}}^2  \| \tau  
          \|_{\theta}.
\end{align}
\end{rema}
\bigskip

\subsection{Proofs of the upper and lower bounds}
First we show the impact of the moments to the functional $D$ 
in our context:
\smallskip

\begin{lemm}
	\label{statement:D_to_rho}
	For $r\in [1,2]$, $1=\frac{1}{r}+\frac{1}{r'}$, 
	$\vvvert D\vvvert_{(1,r')}<\infty$ with $D\ge 0$ and $\mu(D>0)>0$ we define on $(\R,\cB(\R))$ 
	the probability measure 
	\begin{equation}\label{eqn:rho_D}
		\od \rho := \frac{D\od \mu}{\int_\R D \od\mu}.
	\end{equation}
	Then, for $d\in (0,1]$ and $c:= \sqrt[r]{c_\eqref{eqn:small_ball_mu_beta_like}}
	\frac{\| D \|_{L_{r'}([-d,d],\mu)}}
	{\| D \|_{L_1(\R,\mu)}}$, one has
	\[ \rho([-d,d])
	\leqslant \frac{\| D \|_{L_{r'}([-d,d],\mu)}}{\int_\R D \od\mu}  \mu([-d,d])^\frac{1}{r}
	\sptext{1}{so that}{1} 
	\rho \in U\left (\frac{2-\beta}{r};c
	\right ). \]
\end{lemm}

\begin{proof}
For $d\in (0,1]$ one has 
\[          \rho([-d,d])
      =     \frac{1}{\int_\R D \od\mu} \int_{[-d,d]} D(z) \mu(\od z) \\
  \leqslant \frac{1}{\int_\R D \od\mu} \| D \|_{L_{r'}([-d,d],\mu)} \mu([-d,d])^\frac{1}{r}. \]
So we conclude with inequality \eqref{eqn:small_ball_mu_beta_like}.
\end{proof}
\smallskip

For the proof of \cref{statement:upper_bounds_alpha_like} we use the following lemma:
\smallskip

\begin{lemm}\label{statement:JN_bmo_new}
For $0\leqslant a \leqslant t \leqslant T$, $r,p\in (0,\infty)$, $Y\in \CL_0([0,t])$,
$\Phi\in \CL^+([0,t])$ with $\Phi_t^* \in L_p$,
and for $\lambda>0$ one has, a.s., 
\begin{multline}\label{eqn:statement:JN_bmo}
	\p_{\F_a}(|Y_t-Y_a| > \lambda) \\ \leqslant 
	c_\eqref{eqn:statement:JN_bmo}   \min \left \{ 
	\frac{\Phi_a^r}{\lambda^r} \| Y\|_{\bmo_r^\Phi([0,t])}^r ,
	\frac{\ce{\F_a}{\sup_{u\in [a,t]} \Phi_u^p}}{\lambda^p}  
	\left [\|Y\|_{\bmo_r^{\Phi}([0,t])}^p  + |\Delta Y|_{\Phi,[0,t]}^p\right ] \right \},
\end{multline}
where 
$|\Delta Y|_{\Phi,[0,t]} := \inf\{c>0 : |\Delta Y_u| \leqslant c \Phi_u \mbox{ for } u\in [0,t] \mbox{ a.s.}\}$
and $c_\eqref{eqn:statement:JN_bmo}>0$ depends at most on $(r,p)$.
\end{lemm}

\begin{proof}
	First we observe that
	\begin{equation}\label{eqn:tail_bmo}
		\p_{\F_a}(|Y_t-Y_a| > \lambda) \leqslant \frac{\Phi_a^r}{\lambda^r} \| Y\|_{\bmo_r^\Phi([0,t])}^r
		\mbox{ a.s.}
	\end{equation}
	Moreover, from \cref{statement:relation_BMO_bmo}\eqref{item:1:relation-bmo} we know that
	\begin{equation}\label{eqn:upper_bound_BMO}
		\| Y\|_{\BMO_r^{\Phi}([0,t])}\leqslant 2^{\left ( \frac{1}{r} - 1 \right )^+} 
		\left [ \|Y\|_{\bmo_r^{\Phi}([0,t])}  + |\Delta Y|_{\Phi,[0,t]}\right ]. 
	\end{equation}
	Using 
        \eqref{eqn:2:statement:Lp-BMO_new} of \cref{statement:Lp-BMO_new} this yields to
	\begin{equation*}
		(\ce{\F_a}{|Y_t-Y_a|^p})^\frac{1}{p}
		\leqslant c_{\eqref{eqn:2:statement:Lp-BMO_new}} \|Y\|_{\BMO_r^{\Phi}([0,t])} \left ( \ce{\F_a}{\sup_{u\in [a,t]} \Phi_u^p}\right )^\frac{1}{p}
		\mbox{ a.s.}
	\end{equation*}
	which implies 
	\begin{equation}\label{eqn:tail_BMO}
		\p_{\F_a}(|Y_t-Y_a| > \lambda) \leqslant \frac{c_{\eqref{eqn:2:statement:Lp-BMO_new}}^p}{\lambda^p}  \|Y\|_{\BMO_r^{\Phi}([0,t])}^p 
		\ce{\F_a}{\sup_{u\in [a,t]} \Phi_u^p} \mbox{ a.s.}
	\end{equation}
	Combining \eqref{eqn:tail_bmo}, \eqref{eqn:tail_BMO}, and \eqref{eqn:upper_bound_BMO}
	implies our statement.
\end{proof}
\bigskip

\begin{proof}[Proof of \cref{statement:upper_bounds_alpha_like}]
Since $\mu$ is a finite measure, $D\in  L_2(\R,\mu)$ implies
$D\in  L_1(\R,\mu)$, so that our assumptions guarantee that 
$\vvvert D \vvvert_{(1,r')}<\infty$. Moreover,
we may assume that $\vvvert D \vvvert_{(1,r')}>0$, otherwise there is nothing to prove
because $h_{n-1}^D\equiv 0$ for $n\in \bN$.
We decompose $D=D^+-D^-$ with $D^+ := D \vee 0$ and $D^- := (-D) \vee 0$
and define 
\[  \od \rho^+ := \frac{D^+ \od \mu}{\int_\R D^+ \od \mu}
    \sptext{1}{and}{1}
    \od \rho^- := \frac{D^- \od \mu}{\int_\R D^- \od \mu} \]
if $\int_\R D^+ \od \mu >0$ or (and) $\int_\R D^- \od \mu >0$,
where we leave out here and in the following the corresponding terms
$(D^+,\rho^+)$ or $(D^-,\rho^-)$ if $\rho^+$ or $\rho^-$ is not defined.
Directly from the definition we get that
\begin{equation}\label{eqn:representation_psi_t}
    \psi_t(f(X_T),D) 
  = \psi_t(f(X_T),D^+) - \psi_t(f(X_T),D^-) \sptext{.7}{for}{.7} t\in [0,T)  \mbox{ a.s.},
\end{equation}
\cref{statement:D_to_rho} implies that
\[ \rho^\pm \in U(\varepsilon; c^\pm )
\sptext{1}{with}{1}
c^{\pm}:= \sqrt[r]{c_\eqref{eqn:small_ball_mu_beta_like}} 
\frac{\| D \|_{L_{r'}([-1,1],\mu)}}{\| D^\pm \|_{L_1(\R,\mu)}}.
\]
Moreover, from \cref{statement:U-L-for-beta-like} we know that there is a constant $c'>0$ such that 
\[ X \in \cU(\beta,\R;c').\]
\eqref{item:0:statement:space_time_approximation_random_measure}
As our assumptions allow us to apply \cref{statement:upper_bounds_Drho},
we obtain $\hoel \eta \subseteq \dom(\Gamma_\rho^0)$.
If we  define
\begin{equation}\label{eqn:definition_G}
G(t,x):=    \| D^+ \|_{L_1(\R,\mu)} D_{\rho^+}F(t,x) - \| D^- \|_{L_1(\R,\mu)} D_{\rho^-}F(t,x),
\end{equation}
then
\[ G(t,X_t) = \psi_t(f(X_T),D) \mbox{ a.s.}\]
by \cref{statement:expression_for_integrand} and \eqref{eqn:representation_psi_t}.
Moreover, we have 
\begin{equation}\label{eqn:Dpm_cpm_vs_D1r}
	 \| D^\pm \|_{L_1(\R,\mu)} \, (1\vee c^\pm)
 =  \| D^\pm \|_{L_1(\R,\mu)} \, \left (1\vee \left (     
     \sqrt[r]{c_\eqref{eqn:small_ball_mu_beta_like}} 
     \frac{\| D \|_{L_{r'}([-1,1],\mu)}}{\| D^\pm \|_{L_1(\R,\mu)}} \right ) \right ) 
\leqslant  \left( 1\vee \sqrt[r]{c_\eqref{eqn:small_ball_mu_beta_like}}\right )
         \vvvert D \vvvert_{(1,r')}
\end{equation}
so that \eqref{item:0:statement:space_time_approximation_random_measure}
follows by 
\begin{align*}
           \| G(t,\cdot)\|_{B_b(\R)}
&\leqslant \| D^+ \|_{L_1(\R,\mu)} \| D_{\rho^+}F(t,\cdot)\|_{B_b(\R)} + \| D^- \|_{L_1(\R,\mu)} \| D_{\rho^-}F(t,\cdot)\|_{B_b(\R)} \\
&\leqslant |f|_\eta c_{\eqref{eqn:item:1:statement:upper_bounds_Drho}} 
           \left [   \| D^+ \|_{L_1(\R,\mu)} (1\vee c^+)  +   \| D^- \|_{L_1(\R,\mu)} (1\vee c^-)  \right ] A(t) \\
&\leqslant |f|_\eta c_{\eqref{eqn:item:1:statement:upper_bounds_Drho}}\, \Big [ 2 \left( 1\vee \sqrt[r]{c_\eqref{eqn:small_ball_mu_beta_like}}\right ) 
         \vvvert D \vvvert_{(1,r')} \Big ] \, A(t)
\end{align*}
where $A(t)$ is the factor in $t$ on the right-hand side of inequality \eqref{eqn:item:1:statement:upper_bounds_Drho},
and where we use  
\cref{statement:upper_bounds_Drho}\eqref{item:1:statement:upper_bounds_Drho},
\cref{statement:Bb_vs_computation_in_zero}, and
\cref{def:Gamma_functional_semi-norm}.
So, if
$   c_{\eqref{eq:stronger_item:0:statement:upper_bounds_alpha_like}}
 := 2 c_{\eqref{eqn:item:1:statement:upper_bounds_Drho}} \left( 1\vee \sqrt[r]{c_\eqref{eqn:small_ball_mu_beta_like}}\right )$,
then 
\begin{equation}\label{eq:stronger_item:0:statement:upper_bounds_alpha_like}
\| G(t,\cdot)\|_{B_b(\R)}
   \leqslant c_{\eqref{eq:stronger_item:0:statement:upper_bounds_alpha_like}}
          |f|_\eta  \vvvert D \vvvert_{(1,r')} \, A(t).
\end{equation}
\medskip
\eqref{item:1:statement:upper_bounds_alpha_like}
We use \cref{statement:upper_bounds_Drho}\eqref{item:2:statement:upper_bounds_Drho}
(observe that $\beta\in (2-r,2)$ implies that $\varepsilon \in (0,1)$)
and \eqref{eqn:Dpm_cpm_vs_D1r} to get
 \begin{align*}
  \| D^\pm \|_{L_1(\R,\mu)} \| D_{\rho^\pm} F\|_{\B_{b,q}^\alpha}
& \leqslant  c_\eqref{eqn:item:2:statement:upper_bounds_Drho} \| D^\pm \|_{L_1(\R,\mu)} 
            \, (1\vee c^\pm) \, \| f\|_{\hoelO{\eta,q}} \\
&\leqslant c_\eqref{eqn:item:2:statement:upper_bounds_Drho}
           \left ( 1\vee \sqrt[r]{c_\eqref{eqn:small_ball_mu_beta_like}}\right ) 
           \,  \vvvert D \vvvert_{(1,r')}
           \, \| f\|_{\hoelO{\eta,q}}
\end{align*}
which implies, with 
$   c_\eqref{eqn:Bbqalpha_Hoelder}
 := 2 c_\eqref{eqn:item:2:statement:upper_bounds_Drho}
    \left ( 1\vee \sqrt[r]{c_\eqref{eqn:small_ball_mu_beta_like}}\right )$,
\begin{align}
	       \| \psi(f(X_T),D)\|_{\B_{\infty,q}^\alpha} 
&\leqslant \| D^+ \|_{L_1(\R,\mu)}  \| D_{\rho^+}F(t,\cdot)\|_{\B_{b,q}^\alpha} + \| D^- \|_{L_1(\R,\mu)} \| D_{\rho^-}F(t,\cdot)\|_{\B_{b,q}^\alpha}
           \label{eqn:sum_for_psi} \\ 
&\leqslant  c_\eqref{eqn:Bbqalpha_Hoelder} \vvvert D \vvvert_{(1,r')} \, \| f\|_{\hoelO{\eta,q}}.\label{eqn:Bbqalpha_Hoelder}
\end{align}
\eqref{item:2:statement:upper_bounds_alpha_like}
We continue from \eqref{item:1:statement:upper_bounds_alpha_like}, 
but fix $q=2$ in the following. From \cref{eqn:E_t_vs_squarefunction} and \cref{statement:intsq_vs_Binfty2} we obtain
\begin{align*}
    \left \| E(f;\tau,D) \right \|_{\bmo_2([0,T))} 
& =  \| D \|_{L_2(\R,\mu)} 
	\sqrt{\left \| \intsq{\psi(f(X_T),D)}{\tau}\right \|_{\BMO_1([0,T))}}\\
&\leqslant \| D \|_{L_2(\R,\mu)} 	\sqrt{c_\eqref{eqn:statement:intsq_vs_Binfty2}}
	\sqrt{\| \tau\|_\theta} \,\, \| \psi(f(X_T),D)- \psi_0(f(X_T),D) \|_{\B_{\infty,2}^\alpha} \\
&\leqslant 2 \| D \|_{L_2(\R,\mu)} 	\sqrt{c_\eqref{eqn:statement:intsq_vs_Binfty2}}
\sqrt{\| \tau\|_\theta} \,\, \| \psi(f(X_T),D) \|_{\B_{\infty,2}^\alpha}.
\end{align*}
Combining this with \eqref{eqn:Bbqalpha_Hoelder} yields 
\begin{equation}\label{eqn:error_vs_Hoelder}
	            \left \| E(f;\tau,D) \right \|_{\bmo_2([0,T))}
   \leqslant c_\eqref{eqn:error_vs_Hoelder}  \| D \|_{L_2(\R,\mu)} \,
             \vvvert D \vvvert_{(1,r')}
             \sqrt{\| \tau\|_\theta} \,\,   \, \| f\|_{\hoelO{\eta,2}}
\end{equation}
with
$c_\eqref{eqn:error_vs_Hoelder}:= 2 \sqrt{c_\eqref{eqn:statement:intsq_vs_Binfty2}} 
 c_\eqref{eqn:Bbqalpha_Hoelder}$.

\eqref{item:3:statement:upper_bounds_alpha_like}
Using $\| \tau_n^\theta \|_{\theta} \leqslant \frac{T^\theta}{\theta n}$
from \eqref{eqn:upper_bound_adapted_nets}, part 
\eqref{item:2:statement:upper_bounds_alpha_like} implies
\begin{equation}\label{eqn:upper_bound_bmo_error_process}
   \left \| E(f; \tau_n^\theta,D) \right \|_{\bmo_2([0,T))} 
   \leqslant  c_\eqref{eqn:error_vs_Hoelder}  \| D \|_{L_2(\R,\mu)} \, 
              \vvvert D \vvvert_{(1,r')} \sqrt{\frac{T^\theta}{\theta n}} \,\, \| f \|_{\hoelO {\eta,2}}.
\end{equation}
Moreover, for $s\in [0,T)$ and $G$ defined in \eqref{eqn:definition_G} we have 
$\psi_t(f(X_T),D)=G(t,X_t)$ a.s. and
\begin{align*}
            \| G (t,\cdot)\|_{B_b(\R)}
&\leqslant \| D^+ \|_{L_1(\R,\mu)}  \| D_{\rho^+}F(t,\cdot)\|_{B_b(\R)} + \| D^- \|_{L_1(\R,\mu)} \| D_{\rho^-}F(t,\cdot)\|_{B_b(\R)}\\ 
&\leqslant \sqrt{2\alpha} c_\eqref{eqn:Bbqalpha_Hoelder} (T-t)^{-\alpha}
           \vvvert D \vvvert_{(1,r')} \, \| f\|_{\hoelO{\eta,2}}
\end{align*}
where we first use \eqref{eqn:relations_between_Bpqalpha} for $\B_{b,2}^\alpha$ 
on the right-hand side of \eqref{eqn:sum_for_psi} and then apply \eqref{eqn:Bbqalpha_Hoelder}.
We derive
\begin{align*}
           |\Delta E_s(f;\tau_n^\theta,D)| 
&\leqslant 2 \left (\sqrt{2\alpha} c_\eqref{eqn:Bbqalpha_Hoelder} (T-s)^{-\alpha} \vvvert D \vvvert_{(1,r')}
           \| f \|_{\hoelO{\eta,2}}
           \right ) |\Delta X_s^D| \\
&\leqslant 2 \left (\sqrt{2\alpha} c_\eqref{eqn:Bbqalpha_Hoelder} (T-s)^{-\alpha}  \vvvert D \vvvert_{(1,r')} 
           \| f \|_{\hoelO{\eta,2}} \right ) \Phi_s 
\end{align*}           
for $s\in [0,t]$ a.s. by \cite[Theorem IX.9.3]{He:Wang:Yan:92} and our assumption on $\Phi_s$.
Finally we conclude by \cref{statement:JN_bmo_new} with $r=2$, where we use $\Phi'_s\equiv 1$ for the
first term in the minimum and $\Phi$ as defined in item 
\eqref{item:2:statement:upper_bounds_alpha_like} for the second term in the minimum.
\end{proof}

\begin{proof}[Proof of \cref{statement:lower_bounds_alpha_like}]
Define
\[ D(z) := |z|^{-\frac{2-\beta}{r'}} \1_{\{ 0<|z|\leqslant 1 \}}. \]
For $r=1$ we have $D\in L_\infty([-1,1],\mu)$ by definition, for
$r\in (1,2)$ we obtain for $\lambda \geqslant 1$ that 
\begin{align*}
     \mu \left ( \left \{ z \in [-1,1] : D(z) \geqslant \lambda \right \} \right )
& =  \mu \left ( \left \{ z \in [-1,1]\setminus \{0\} : |z|^{-\frac{2-\beta}{r'}} \geqslant \lambda \right \} \right ) \\
& =  \mu \left ( \left [-\lambda^{-\frac{r'}{2-\beta}},\lambda^{-\frac{r'}{2-\beta}} \right ] \right ) \\
&\leqslant c_\eqref{eqn:small_ball_mu_beta_like}  \left ( \lambda^{-\frac{r'}{2-\beta}} \right )^{2-\beta} 
  =  c_\eqref{eqn:small_ball_mu_beta_like}  \lambda^{-r'}
\end{align*}
where we used the inequality \eqref{eqn:small_ball_mu_beta_like}.
As $\mu$ is a finite measure,  by replacing $c_\eqref{eqn:small_ball_mu_beta_like}$ by another constant  
we get the same inequality for $\lambda \in (0,1)$.
In both cases we have, for $d \in (0,1]$, that
\[   \int_{-d}^d D(z) |z|^2 |z|^{-1-\beta} \od z
   = 2  \int_0^d z^{\varepsilon -1} \od z 
   = \frac{2}{\varepsilon} d^{\varepsilon}. \]
Therefore we get $\rho\in \left (\bigcup_{c>0} U(\varepsilon;c)\right )\cap L(\varepsilon)$
with $\od \rho := \frac{D\od \mu}{\|D\|_{L_1(\R,\mu)}}$. Moreover, it holds 
$X \in \left ( \bigcup_{c>0} \cU(\beta,\R;c) \right ) \cap \cL(\beta)$ by
\cref{statement:U-L-for-beta-like} and $\| X_s\|_{{\rm TV}(\rho,\eta)}< \infty$ for
$s\in (0,T]$ by  \cref{statement:X-rho_to_TV}\eqref{item:1:statement:X-rho_to_TV}.
Now \cref{statement:lower_bound_complete}\eqref{item:1:statement:lower_bound_complete}
yields
\[ \inf_{t\in (0,T)} (T-t)^{\alpha} \losc_t(\varphi)>0 \]
which proves part \eqref{item:1:statement:lower_bounds_alpha_like}.
Because of $\alpha=\frac{1-\theta}{2}$ \cref{thm:general_lower_bound} implies \eqref{eqn:item:1:statement:lower_bounds_alpha_like}, 
and part \eqref{item:2:statement:lower_bounds_alpha_like} follows from \eqref{eqn:item:1:statement:lower_bounds_alpha_like}.
\end{proof}


\subsection{Orthogonal decomposition of H\"older functionals on the L\'evy-It\^o space}
\label{sec:orthogonal_decomposition_Levy-Ito}

We close with an orthogonal decomposition of the L\'evy-It\^o-space into 
stochastic integrals, where we have the necessary control on the integrands
in order to apply our approximation techniques.
This might be not only of interest in connection to approximation problems, but of general interest, as we obtain 
significantly better bounds as for the corresponding integrands in Clark-Ocone 
type formulas. To illustrate this, assume $\sigma=0$. In this case, under moment 
conditions, a predictable version of 
\[ \frac{F(t,X_t+z)-F(t,X_t)}{z} \]
is used as integrand against the random measure $z \widetilde{N}(\od t,\od z)$
(for recent results and corresponding references the reader is referred to 
\cite{Nguyen:20b}).
Starting with $f(X_T)$, where $f$ is H\"older continues with $\eta\in (0,1)$,
the $L_\infty$-singularity of these integrands is arbitrary close to the order $(T-t)^{-\frac{1}{\beta}}$ according to 
\cref{statement:lower_bound_complete}\eqref{item:1:statement:lower_bound_complete}.
As $-\frac{1}{\beta}<-\frac{1}{2}$ because of $\beta\in (0,2)$, we might
get singularities that are stronger than $\frac{1}{2}$.

The order of integrands, we obtain in \cref{statement:space_time_approximation_random_measure} below, is
$(T-t)^{-\alpha}$ with  $\alpha= \frac{1}{2}-\frac{\eta}{\beta}<1/2$ whenever $\eta>0$.
For example, we get 
$\alpha:= \frac{1}{2}-\frac{\eta}{\beta} \in \left (0,\frac{1}{2} \right )$
if $\eta \in (0,\frac{\beta}{2})$. 
The reason for this is the averaging procedure 
by the operators $D_\rho F$ from 
\cref{sec:application_levy_case} and the crucial
estimate in \cref{statement:D_to_rho} which restricts the measures 
$\rho$ in contrast to \cref{sec:application_levy_case}.

\begin{theo}
\label{statement:space_time_approximation_random_measure}
Let 
$\beta \in (0,2)$, 
$\eta\in [0,1]$,
$\alpha:= \frac{1}{2}-\frac{\eta}{\beta}$,
and let $(D_j)_{j\in J}\subseteq L_2(\R,\mu)$ be an orthonormal basis.
Then, for $f\in \hoel \eta$, one has the orthogonal decomposition 
\[ f(X_T) = \E f(X_T) + \sum_{j\in J} \int_{(0,T)} \psi_{t-}(f(X_T),D_j) \od X^{D_j}_t \quad\mbox{a.s.}\]
and the decomposition satisfies:
\begin{enumerate}[{\rm\bf (1)}]
	
\item \label{item:0:statement:space_time_approximation_random_measure}	
      There is a $c>0$ such that, for all $j\in J$ and $f \in \hoel{\eta}$,
      \[ \|\psi_t(f(X_T),D_j)\|_{L_\infty}
	 \leqslant  c \,\, |f|_\eta \,\, 
		\begin{cases}
			(T-t)^{-\alpha}            &:  \eta \in [0,\frac{\beta}{2}) \\
			\ln \left ( 1+(T-t)^{-\frac{1-\eta}{\beta}} \right ) &:  \eta = \frac{\beta}{2}  \\
			1                                                      &:  \eta \in (\frac{\beta}{2},1]
		\end{cases}. \]
		\bigskip
		
		\item Assume additionally
		 $\eta\in (0,\frac{\beta}{2})$, 
	     $\theta:=\frac{2}{\beta} \eta$, and
	     $q\in [1,\infty)$.
		\bigskip
		\begin{enumerate}[{\rm\bf (a)}]
			\item \label{item:1:statement:space_time_approximation_random_measure}
			{\sc Singularity of the gradient:} There is a $c>0$ such that, for 
			all $j\in J$ and $f \in \hoel{\eta,q}$,
	\[ \|   \psi(f(X_T),D_j) \|_{\B_{\infty,q}^\alpha} 
	\leqslant  c \, \| f\|_{\hoelO{\eta,q}}. \]
	\vspace*{-.4em}
		
	\item  \label{item:2:statement:space_time_approximation_random_measure}
	{\sc Approximation:}  
	There is a $c>0$ such that,
	for all $j\in J$, $f \in \hoel{\eta,2}$ and $\tau\in \cT$,
	\[
	\left \| E(f; \tau,D_j)        \right \|_{\bmo_2([0,T))} 
	\leqslant c \,
	\sqrt{\| \tau \|_{\theta}} \,\, \| f \|_{\hoelO {\eta,2}}.
	\]\
				
	\end{enumerate}
	\end{enumerate}
The constants $c>0$ in 
\eqref{item:0:statement:space_time_approximation_random_measure},
\eqref{item:1:statement:space_time_approximation_random_measure}, and
\eqref{item:2:statement:space_time_approximation_random_measure}
do not depend on the choice of the orthogonal basis $(D_j)_{j\in J}$.

\end{theo}

\begin{proof}
The result follows from \eqref{eqn:orthogonal_decomposition_LI_space},
\cref{statement:upper_bounds_alpha_like} for $r=2$, and
\[  \vvvert D_j \vvvert_{(1,2)} 
    \leqslant \left ( \mu(\R)^\frac{1}{2} \vee 1 \right ) \| D_j \|_{L_2(\R,\mu)}
       = \left ( \mu(\R)^\frac{1}{2} \vee 1 \right ). \qedhere \]
\end{proof}


\appendix
\addtocontents{toc}{\protect\setcounter{tocdepth}{0}}


\section{The class $\cSM_p(\timed)$ and BMO-spaces}
\label{sec:general_properties_BMO}
We summarize some basic facts about the class $\cSM_p(\timed)$ and BMO-spaces that are used in the article.
Regarding the classical non-weighted BMO-spaces, i.e. $\Phi_t\equiv 1$ for $t\in \timed$,
the reader is also referred to \cite[Section X.1]{He:Wang:Yan:92}.
For this we assume a stochastic basis 
$(\Omega,\F,\p,(\F_t)_{t\in [0,T]})$ with $T\in (0,\infty)$ such that 
$(\Omega,\F,\p)$ is complete, $\F_0$ contains all null-sets, and such that $\F_t=\bigcap_{s\in (t,T]} \F_s$ for 
all $t\in [0,T)$. We do not assume that $\F_0$ is generated by the null-sets only.
In the computations below we exploit the following fact:
given stopping times $\sigma,\tau:\Omega\to \timed$ and an integrable random variable $Z:\Omega\to\R$,
we have $\{\sigma=\tau\}\in \F_{\sigma\wedge \tau}$ and
\[   \ce{\F_\sigma}{\1_{\{\sigma=\tau\}} Z} 
   = \ce{\F_{\sigma\wedge \tau}}{\1_{\{\sigma=\tau\}}Z} 
   = \ce{\F_\tau}{\1_{\{\sigma=\tau\}} Z} \mbox{ a.s.} \]
Moreover, we again use $\inf \emptyset := \infty$.

\subsection{Properties of the class $\cSM_p$}
We start by a convenient reduction.
Since $\F_0$ does not need to be trivial we add the assumption 
$\Phi_0\in L_p$ to the definition of
$\cSM_p(\timed)$ in \cref{def:SM_p}.

\begin{prop}
\label{statement:smp_determinstic}
For $p\in (0,\infty)$ and $\Phi\in \CL^+(\timed)$ with $\Phi_0\in L_p$ one has 
$|\Phi|_{\cSM_p(\timed)} = \|\Phi\|_{\cSM_p(\timed)}$,
where 
$| \Phi|_{\cSM_p(\timed)} := \inf c$ is the infimum over 
$c \in [1, \infty)$ such that for all $a\in \timed$ one has
\[ \ce{\F_a}{\sup_{a \leqslant t \in \timed} \Phi_t^p} \leqslant c^p \Phi_a^p \quad \mbox{a.s.} \]
\end{prop}

\begin{proof}
It is clear that $|\Phi|_{\cSM_p(\timed)} \leqslant \|\Phi\|_{\cSM_p(\timed)}$, so that we assume that 
$c:=|\Phi|_{\cSM_p(\timed)} <\infty$. Let $\rho :\Omega \to \timed$ be a stopping time,
$h:[0,T)\to [0,\infty)$ be given by $h(t):= \frac{1}{T-t} - \frac{1}{T}$. For $k,N\in \bN_0$ set
\[ [a_k^N,b_k^N) := h^{-1} \left ( \left [ \frac{k}{2^N},\frac{k+1}{2^N} \right ) \right ) \subseteq [0,T)
   \sptext{1}{and let}{1}
   H^N(t) := \1_{\{T\}}(t) T + \sum_{k=0}^\infty \1_{[a_k^N,b_k^N)}(t) b_k^N. \]
Then $H^N(t) \downarrow t$ for all $t\in [0,T]$ and 
$\rho^N:= H^N(\rho):\Omega \to \timed$ is a stopping time as well.
Then, a.s.,
\begin{align*}
\ce{\F_{\rho^N}}{\sup_{\rho^N \leqslant t \in \timed} \Phi_t^p}
& = \ce{\F_{\rho^N}}{\1_{\{\rho^N=T\}} \Phi_T^p} + \sum_{k=0}^\infty \ce{\F_{\rho^N}}{\1_{\{ \rho^N = b_k^N\}} \sup_{b_k^N \leqslant t \in \timed} \Phi_t^p}\\
& = \1_{\{\rho^N=T\}} \Phi_T^p + \sum_{k=0}^\infty \1_{\{ \rho^N = b_k^N\}}  \ce{\F_{b_k^N}}{\1_{\{ \rho^N = b_k^N\}} \sup_{b_k^N \leqslant t \in \timed} \Phi_t^p}\\
&\leqslant  \1_{\{\rho^N=T\}} \Phi_T^p + \sum_{k=0}^\infty \1_{\{ \rho^N = b_k^N\}}  c^p \Phi_{b_k^N}^p \\
&\leqslant  c^p \Phi_{\rho^N}^p
\end{align*}
where we omit $\1_{\{\rho^N=T\}} \Phi_T^p$ if $\timed=[0,T)$.
This implies that 
$\ce{\F_{\rho}}{\sup_{\rho^N \leqslant t \in \timed} \Phi_t^p} \leqslant c^p \ce{\F_\rho}{\Phi_{\rho^N}^p}$ a.s.
By $N\to \infty$, monotone convergence on the left-hand side and because $\Phi$ is c\`adl\`ag, and dominated convergence on the right-hand side
($\Phi$ is c\`adl\`ag and $\E \sup_{t\in \timed} \Phi_t^p<\infty$) we obtain the assertion.
\end{proof}

We continue with structural properties of the class $\cSM_p$:

\begin{prop}
\label{statement:SM-properties}
For $0<p, p_0, p_1<\infty$ with $\frac{1}{p}=\frac{1}{p_0} + \frac{1}{p_1}$ the following holds:
\begin{enumerate}[\quad \rm(1)]
	
\item\label{item:2:statement:SM-properties}
     $\cSM_q(\timed) \subseteq \cSM_p(\timed)$ and $\|\Phi\|_{\cSM_p(\timed)} \leqslant \|\Phi\|_{\cSM_q(\timed)}$ whenever $0 <  p <q < \infty$.
			
\item\label{item:3:statement:SM-properties}
     If $\Phi \in \cSM_p(\timed)$, then $\Phi^* \in \cSM_p(\timed)$ and 
     $\|\Phi^*\|_{\cSM_p(\timed)} \leqslant \sqrt[p]{1 + \|\Phi\|_{\cSM_p(\timed)}^p}$.

\item\label{item:4:statement:SM-properties}
      For $\Phi^i\in \cSM_{p_i}(\timed)$, $i=0,1$,  and $\Phi=(\Phi_a)_{a\in [0,T)}$ with 
      $\Phi_a := \Phi_a^0 \Phi_a^1$, one has
	  \[ \| \Phi \|_{\cSM_p(\timed)} \leqslant  \| \Phi^0 \|_{\cSM_{p_0}(\timed)} \| \Phi^1 \|_{\cSM_{p_1}(\timed)}.\]
\end{enumerate}
\end{prop}
\smallskip

\begin{proof}  
\eqref{item:2:statement:SM-properties} follows from the definition. Now let $a\in \timed$.
To check \eqref{item:3:statement:SM-properties} we observe 
$\Phi_0^*=\Phi_0\in L_p$ and  
\[ \ce{\F_a}{\sup_{a \leqslant t \in \timed} |\Phi^*_t|^p}  
   = \ce{\F_a}{\sup_{ t \in \timed} \Phi_t^p}
  \leqslant |\Phi^*_a|^p + \|\Phi\|_{\cSM_p(\timed)}^p \Phi_a^p 
  \leqslant (1 + \|\Phi\|_{\cSM_p(\timed)}^p )|\Phi^*_a|^p \mbox{ a.s.} \]
\eqref{item:4:statement:SM-properties} 
We get $\Phi_0^0 \Phi_0^1 \in L_p$ and by the conditional H\"older inequality that, a.s.,  
	\begin{align*}
	\sqrt[p]{\ce{\F_a}{\sup_{a \leqslant t \in \timed} \Phi_t^p}} 
	& = \sqrt[p]{\ce{\F_a}{\sup_{a \leqslant t \in \timed} (\Phi_t^0\Phi_t^1)^p}} \leqslant \sqrt[p]{\ce{\F_a}{ \sup_{a \leqslant t \in \timed} (\Phi_t^0)^p 
            \sup_{a \leqslant t \in \timed} (\Phi_t^1)^p}} \\
	&\leqslant \sqrt[p_0]{\ce{\F_a}{\sup_{a \leqslant t \in \timed} (\Phi_t^0)^{p_0}}} 
	\sqrt[p_1]{\ce{\F_a}{\sup_{a \leqslant t \in \timed} (\Phi_t^1)^{p_1}}} \\
	&\leqslant \| \Phi^0 \|_{\cSM_{p_0}(\timed)} \| \Phi^1 \|_{\cSM_{p_1}(\timed)} \Phi_a^0 \Phi_a^1 \\
	& =        \| \Phi^0 \|_{\cSM_{p_0}(\timed)} \| \Phi^1 \|_{\cSM_{p_1}(\timed)} \Phi_a. \qedhere
	\end{align*}
\end{proof}
\medskip

\subsection{Simplifications in the definitions of BMO-spaces}
The first simplification concerns the case $\timed=[0,T]$:
\medskip
\begin{prop}
\label{statement:only_T_in_bmo_definitions}
For $p\in(0,\infty)$, $Y\in \CL_0([0,T])$, and $\Phi \in \CL^+([0,T])$ define
$|Y|_{\BMO_p^{\Phi}([0,T])}:=\inf c$
and 
$|Y|_{\bmo_p^{\Phi}([0,T])}:=\inf c$, respectively, to be the infimum over all $c\in [0,\infty)$ such that,
for all $\rho\in  \mathcal S_T$,
\[ \ce{\F_\rho}{|Y_T-Y_{\rho-}|^p} \leqslant c^p \Phi_{\rho}^p \mbox{ a.s.} 
   \sptext{1}{and}{1}
   \ce{\F_\rho}{|Y_T-Y_{\rho}|^p} \leqslant c^p \Phi_{\rho}^p \quad \mbox{a.s.}, \]
respectively. Then one has 
\begin{align*}
   |Y|_{\BMO_p^{\Phi}([0,T])}
   &\leqslant \|Y\|_{\BMO_p^{\Phi}([0,T])}
   \leqslant 2^{(\frac{1}{p} -1)^+} [1+\| \Phi\|_{\cSM_p([0,T])}] |Y|_{\BMO_p^{\Phi}([0,T])}, \\
|Y|_{\bmo_p^{\Phi}([0,T])}
&   \leqslant \|Y\|_{\bmo_p^{\Phi}([0,T])}
   \leqslant 2^{(\frac{1}{p} -1)^+} [1+ \| \Phi\|_{\cSM_p([0,T])}] |Y|_{\bmo_p^{\Phi}([0,T])},
\end{align*}
where we additionally assume for the right-hand side inequalities that $\Phi\in \cSM_p([0,T])$.
\end{prop}

\begin{proof}
The inequalities on the left are obvious. To check the inequalities on the right we may assume that 
$c:=|Y|_{\BMO_p^{\Phi}([0,T])}$ or $c:=|Y|_{\bmo_p^{\Phi}([0,T])}$ are finite. To treat both cases 
simultaneously, we let $t\in [0,T]$, $\rho\in\cS_t$, and 
$A=Y_{\rho-}$ or $A=Y_{\rho}$, respectively. Then, a.s.,
\begin{align*}
            \left ( \ce{\F_\rho}{|Y_t-A|^p} \right )^\frac{1}{p} 
& \leqslant 2^{(\frac{1}{p} -1)^+} \left [ \left ( \ce{\F_\rho}{|Y_T-A|^p} \right )^\frac{1}{p} + \left ( \ce{\F_\rho}{|Y_T-Y_t|^p} \right )^\frac{1}{p} 
            \right ] \\
& \leqslant 2^{(\frac{1}{p} -1)^+} \left [ c \Phi_{\rho} + \left ( \ce{\F_\rho}{|Y_T-Y_t|^p} \right )^\frac{1}{p} 
            \right ].
\end{align*}
To estimate the second term we may assume $t\in [0,T)$. In case of $\bmo$-spaces this term can be estimated by
\[     \left ( \ce{\F_\rho}{|Y_T-Y_t|^p} \right )^\frac{1}{p} 
     = \left ( \ce{\F_\rho}{ \ce{\F_t}{|Y_T-Y_t|^p}} \right )^\frac{1}{p} 
   \leqslant  c \left ( \ce{\F_\rho}{ \Phi_t^p} \right )^\frac{1}{p} 
   \leqslant  c \| \Phi \|_{\cSM_p([0,T])} \Phi_\rho \quad\mbox{a.s.}
\]
In case of $\BMO$-spaces we find a sequence $t_n\in (t,T]$ with 
$t_n\downarrow t$. Using Fatou's Lemma for conditional expectations we get, a.s., 
\begin{align*}
    \left ( \ce{\F_\rho}{|Y_T-Y_t|^p} \right )^\frac{1}{p} 
\leqslant \liminf_n \left ( \ce{\F_\rho}{|Y_T-Y_{t_n-}|^p} \right )^\frac{1}{p}  
\leqslant \liminf_n c \left ( \ce{\F_\rho}{\Phi_{t_n}^p} \right )^\frac{1}{p}  
\leqslant  c \| \Phi \|_{\cSM_p([0,T])} \Phi_\rho.
\end{align*}
\end{proof}
\medskip

The second simplification concerns the  $\bmo$-spaces.
For $p\in(0,\infty)$, $Y\in \CL_0(\timed)$, and $\Phi \in \CL^+(\timed)$ we let
$|Y|_{\bmo_p^{\Phi}(\timed)}^{\rm det}:= \inf c$ be the infimum over all 
$c\in [0,\infty)$ such that
\[ \ce{\F_{a}}{|Y_t-Y_a|^p} \leqslant c^p \Phi_a^p \mbox{ a.s.} 
      \sptext{1}{for all}{1} t\in \timed \sptext{.5}{and}{.5} a\in  [0,t]  . \]
With this definition we obtain:

\medskip
\begin{prop}
\label{statement:bmo_determinstic} 
One has $|\cdot|_{\bmo_p^{\Phi}(\timed)}^{\rm det} =\|\cdot\|_{\bmo_p^{\Phi}(\timed)}$ for all $p\in (0,\infty)$.
\end{prop}

\begin{proof}
It is obvious that $|Y|_{\bmo_p^{\Phi}(\timed)}^{\rm det} \leqslant \|Y\|_{\bmo_p^{\Phi}(\timed)}$.
To show $\|Y\|_{\bmo_p^{\Phi}(\timed)} \leqslant |Y|_{\bmo_p^{\Phi}(\timed)}^{\rm det}$ we assume that
$c:= |Y|_{\bmo_p^{\Phi}(\timed)}^{\rm det} <\infty$, otherwise there is nothing to prove.
For $t\in \timed$, $\rho\in \mathcal S_t$, and 
$L\in \bN_0$ we define the new stopping times $\rho_L(\omega):= \psi_L(\rho(\omega))$ where $\psi_L(0):=0$ and
$\psi_{L}(s)= s_\ell^L:=\ell 2^{-L} t$ when $s\in  \left ( s_{\ell-1}^L,s_\ell^L \right ]$ for
$\ell\in \{1,\ldots,2^L\}$. By definition, $\rho_L (\omega) \downarrow \rho(\omega)$ for all $\omega \in \Omega$
as $L\to \infty$.
Then
\[ \ce{\F_{s_\ell^L}}{|Y_t-Y_{s_\ell^L}|^p} \leqslant c^p \Phi_{s_\ell^L}^p \mbox{ a.s.} \]
for $\ell=0,\ldots,2^L$.
Multiplying both sides with $\1_{\{\rho_L = s_\ell^L\}}$ and summing 
over $\ell=0,\ldots,2^L$, we get that
\[ \ce{\F_{\rho_L}}{ | Y_t- Y_{\rho_L}|^p} \leqslant c^p \Phi_{\rho_L}^p \mbox{ a.s.} \]
For any $M>0$ this implies 
\[ \ce{\F_{\rho_L}}{| Y_t-Y_{\rho_L}|^p\wedge M} \leqslant (c^p \Phi_{\rho_L}^p)\wedge M \mbox{ a.s.} \]
and
\[ \ce{\F_{\rho}}{| Y_t- Y_{\rho_L}|^p\wedge M} \leqslant  \ce{\F_\rho}{(c^p \Phi_{\rho_L}^p)\wedge M} \mbox{ a.s.} \]
The c\`adl\`ag properties of $Y$ and $\Phi$ imply 
\[ \ce{\F_{\rho}}{| Y_t- Y_{\rho}|^p\wedge M} \leqslant  \ce{\F_\rho}{(c^p \Phi_{\rho}^p)\wedge M} \mbox{ a.s.} \]
By $M\uparrow \infty$ it follows that $\|Y\|_{\bmo_p^{\Phi}(\timed)}\leqslant c$ as desired.
\end{proof}

\subsection{The relation between $\BMO_p^\Phi$ and $\bmo_p^\Phi$}
The BMO- and bmo-spaces are related to each other as follows:
\smallskip

\begin{prop}
\label{statement:relation_BMO_bmo}
For $\Phi\in \CL^+(\timed)$, $Y\in \CL_0(\timed)$,
\[ |\Delta Y|_{\Phi,\timed} := \inf\{c>0 : |\Delta Y_t| \leqslant c \Phi_t \mbox{ for all } t\in \timed \mbox{ a.s.}\}, \]
and $p\in (0,\infty)$ the following assertions are true:
\begin{enumerate}[{\rm (1)}]
\item \label{item:1:relation-bmo}
      $\|Y\|_{\BMO_p^\Phi(\timed)} \leqslant 2^{(\frac{1}{p}-1)^+} \left [ \|Y\|_{\bmo_p^\Phi(\timed)} + \left|\Delta Y\right|_{\Phi,\timed}\right ]$.
\item \label{item:2:relation-bmo}
      If $\E |\Phi^*_t|^p<\infty$ for all $t\in\timed$, then 
      $\|Y\|_{\bmo_p^\Phi(\timed)}  \leqslant \|Y\|_{\BMO_p^\Phi(\timed)}$
      and
      $|\Delta Y|_{\Phi,\timed}     \leqslant 2^{{\frac{1}{p}\vee 1}}  \|Y\|_{\BMO_p^\Phi(\timed)}$.
\end{enumerate}
\end{prop}

\begin{proof} 
For the proof we set $c_p:= 2^{(\frac{1}{p}-1)^+}$.
\eqref{item:1:relation-bmo}
For $t\in \timed$ and $\rho\in \cS_t$ we have, a.s.,
\[
  \left|\ce{\F_{\rho}}{|Y_t - Y_{\rho-}|^p}\right|^{\frac{1}{p}} 
  \leqslant c_p \left [ \left|\ce{\F_{\rho}}{|Y_t - Y_{\rho}|^p}\right|^{\frac{1}{p}} 
          + |\Delta Y_\rho| \right ]
  \leqslant c_p \Phi_\rho \left [ \|Y\|_{\bmo_p^\Phi(\timed)} + \left|\Delta Y\right|_{\Phi,\timed}
 \right ]
\]
so that $\|Y\|_{\BMO_p^\Phi(\timed)} \leqslant c_p \left [ \|Y\|_{\bmo_p^\Phi(\timed)} + \left|\Delta Y\right|_{\Phi,\timed}\right ]$.
\eqref{item:2:relation-bmo} For $t\in \timed$ and $\rho\in \cS_t$ we have, a.s.,
\begin{align*}
\left|\ce{\F_{\rho}}{|Y_t - Y_{\rho}|^p}\right|^{\frac{1}{p}} 
& = \left|\ce{\F_{\rho}}{\1_{\{\rho <t\}} \lim_n |Y_t - Y_{((\rho+\frac{1}{n})\wedge t)-}|^p}\right|^{\frac{1}{p}} \\
&\leqslant  \liminf_n \left|\ce{\F_{\rho}}{\1_{\{\rho <t\}} |Y_t - Y_{((\rho+\frac{1}{n})\wedge t)-}|^p}\right|^{\frac{1}{p}} \\
& = \liminf_n \left|\ce{\F_{\rho}}{\ce{\F_{(\rho+\frac{1}{n})\wedge t}}{\1_{\{\rho <t\}} |Y_t - Y_{((\rho+\frac{1}{n})\wedge t)-}|^p}}\right|^{\frac{1}{p}} \\
&\leqslant  \liminf_n \| Y \|_{\BMO_p^\Phi(\timed)} \left|\ce{\F_{\rho}}{\Phi_{(\rho+\frac{1}{n})\wedge t}^p}\right|^{\frac{1}{p}} \\
&\leqslant  \| Y \|_{\BMO_p^\Phi(\timed)} \Phi_\rho
\end{align*}
where we used $\E |\Phi_t^*|^p<\infty$. Hence 
$\|Y\|_{\bmo_p^\Phi(\timed)} \leqslant \|Y\|_{\BMO_p^\Phi(\timed)}$.
Moreover, for $\rho\in \cS_t$ with $t\in\timed$ we get that, a.s., 
\begin{align*}
|\Delta Y_{\rho}| 
\leqslant c_p \left [ 
           \left|\ce{\F_{\rho}}{|Y_t - Y_{\rho-}|^p} \right|^\frac{1}{p}
        +  \left|\ce{\F_{\rho}}{|Y_t - Y_{\rho}|^p} \right|^\frac{1}{p} \right ] &
\leqslant 2 c_p \|Y\|_{\BMO_p^\Phi(\timed)} \Phi_\rho.
\end{align*}
Now we show that this implies  
\begin{equation}\label{eqn:uniform_bound_Delta_Y}
 |\Delta Y_s| \leqslant [2 c_p \|Y\|_{\BMO_p^\Phi(\timed)}] \Phi_s \mbox{ for all } s\in \timed \mbox{ a.s.}
\end{equation}
which yields to $|\Delta Y|_{\Phi,\timed}     \leqslant 2 c_p \|Y\|_{\BMO_p^\Phi(\timed)}$.
It is sufficient to check \eqref{eqn:uniform_bound_Delta_Y} for $s\in [0,t]$ for $0<t\in \timed$.
So we define for $k\in \bN$ that
	\begin{align*}
	\rho_{1}^k&:=\inf\left\{s\in (0,t] :  |\Delta Y_s|>\tfrac{1}{k}\right\} \wedge t,\\
	\rho_{n}^{k}&:=\inf\left\{s \in (\rho_{n-1}^{k},t] :  |\Delta Y_s|>\tfrac{1}{k}\right\} \wedge t,\quad n\geqslant 2.
	\end{align*} 
Since the stochastic basis satisfies the usual conditions and $Y$ is adapted and c\`adl\`ag, 
each $\rho_{n}^k \colon \Omega \to [0, t]$ is a stopping time 
(this is known and can be checked            with \cite[Lemma 1, Chapter 3]{Billingsley:99}). 
Hence 
\[ |\Delta Y_{\rho_{n}^k}|\leqslant 2 c_p \|Y\|_{\BMO_p^\Phi(\timed)} \Phi_{\rho_{n}^k} \quad \mbox{a.s.,} \]
and we denote by $\Omega_{n}^k$ the set in which the above inequality holds. 
Set $\Omega^*=\cap_{k =1}^\infty \cap_{n=1}^\infty \Omega_{n}^k$, then $\p(\Omega^*)=1$ and
\[ |\Delta Y_s(\omega)| \leqslant  2 c_p \|Y\|_{\BMO_p^\Phi(\timed)} \Phi_t(\omega) \quad\mbox{for all } (\omega, s) \in \Omega^* \times [0, t],\]
which gives the desired statement.
\end{proof}

\subsection{Distributional estimates}
The BMO-spaces allow for John-Nirenberg theorems.  
One consequence of the following equivalence of moments:

\begin{prop}
\label{statement:Lp-BMO_new}
Let $0<p\leqslant q<\infty$, $r\in (0,\infty)$,  and $\Phi\in \CL^+(\timed)$.
\begin{enumerate}[\rm (1)]
\item \label{item:1:statement:Lp-BMO}
      If $\Phi\in \cSM_q(\timed)$ with $\|\Phi\|_{\cSM_q(\timed)} \leqslant d < \infty$,
      then there is a $c=c(p,q,d)\geqslant 1$ such that
      \[ \|\cdot\|_{\BMO_p^\Phi(\timed)}\sim_{c} \|\cdot\|_{\BMO_q^\Phi(\timed)}. \]

\item \label{item:2:statement:Lp-BMO_new}
      There are $b=b(r)>0$, $\beta=\beta(r) >0$, and $c_{\eqref{eqn:2:statement:Lp-BMO_new}}=c_{\eqref{eqn:2:statement:Lp-BMO_new}}(r,q) >0$ such that for $Y\in \CL_0(\timed)$,
      $0 \leqslant a \leqslant t \in \timed$, $D:= \| (Y_u-Y_a)_{u\in [a,t]}\|_{\BMO_r^\Phi([a,t])}<\infty$, and $\mu,\nu>0$ one has
      \begin{align}
                   \p_{\F_a}\left (\sup_{u\in [a,t]} | Y_u - Y_a | > b D \mu \nu \right ) 
        &\leqslant e^{1-\mu} + \beta \p_{\F_a} \left (\sup_{u\in [a,t]} \Phi_u >  \nu \right ) \mbox{ a.s.},
        \label{eqn:1:statement:Lp-BMO_new}\\
      \ce{\F_a}{\sup_{u\in [a,t]}|Y_u-Y_a|^q} & \leqslant c_{\eqref{eqn:2:statement:Lp-BMO_new}}^q D^q
      \ce{\F_a}{\sup_{u\in [a,t]} \Phi_u^q} \mbox{a.s.} 
      \sptext{1}{if}{1} \sup_{u\in [a,t]} \Phi_u \in L_q.
      \label{eqn:2:statement:Lp-BMO_new}
      \end{align}
\end{enumerate}
\end{prop}

\begin{proof}
(\ref{item:1:statement:Lp-BMO}a)
For $\timed=[0,T]$ and $\Phi>0$ on $[0,T]\times \Omega$ the result follows from \cite[Corollary 1(i)]{Geiss:05},
where we use \cref{statement:only_T_in_bmo_definitions} to relate the formally different BMO-definitions to each other and 
\cref{statement:SM-properties}\eqref{item:2:statement:SM-properties}.

(\ref{item:1:statement:Lp-BMO}b) For $\timed=[0,T)$ and $\Phi>0$ on $[0,T)\times \Omega$ this follows from (1a) by considering
the restrictions of the processes to $[0,t]$ for $t\in [0,T)$.

(\ref{item:1:statement:Lp-BMO}c) For $\timed=[0,T]$ or $\timed=[0,T)$,
and $\Phi\geqslant 0$ on $\timed \times \Omega$ we proceed as follows: For $\varepsilon>0$ we consider $\Phi_t^\varepsilon:=\Phi_t+ \varepsilon$
and observe that $\|\Phi^\varepsilon\|_{\cSM_p(\timed)} \leqslant c_p \|\Phi\|_{\cSM_p(\timed)}$ and 
$\sup_{\varepsilon>0} \|\cdot\|_{\BMO_p^{\Phi^\varepsilon}(\timed)} =  \|\cdot\|_{\BMO_p^{\Phi}(\timed)}$.

\eqref{item:2:statement:Lp-BMO_new}  
We restrict the stochastic basis to $(A,\F_a\cap A,\p_A,(\F_u \cap A)_{u\in [a,t]})$ with $A\in \F_a$ and $\p(A)>0$,
where $\p_A$ is the normalized restriction of $\p$ to $A$ and $\F_u \cap A$ denotes the trace $\sigma$-algebra.
So we can assume that $a=0$ and can replace $\p_{\F_a}$ by $\p$ and $\ce{\F_a}{}$ by $\E$.
Moreover, by replacing $\Phi_u$ by $\Phi_u^\varepsilon$ as above, proving the statement for the new weight, and letting 
$\varepsilon\downarrow 0$, we may assume that $\Phi>0$ on $[0,t]\times \Omega$
($\varepsilon\downarrow 0$ gives $\sup_{u\in [a,t]} \Phi_u \geqslant \nu$ in \eqref{eqn:1:statement:Lp-BMO_new},
by adjusting $b$ it can be changed into $>$).
Now \eqref{eqn:1:statement:Lp-BMO_new} and \eqref{eqn:2:statement:Lp-BMO_new} follow from 
\cite[inequalities (5,6) and step (a) of the proof of Corollary 1]{Geiss:05}.
\end{proof}


\section{Transition density}

\begin{theo}[{\cite[p. 263, p. 44]{Friedman:64}}]
\label{statement:friedman}
For $\hat{b},\hat{\sigma}\in C_b^\infty$ with $\widehat{\sigma}\ge \varepsilon_0>0$ there is a jointly continuous 
transition density $\Gamma_X: (0,T]\times\R\times\R\to (0,\infty)$ such that
$\p(X_t^x\in B) = \int_B \Gamma_X(t,x,\xi)\od\xi$ for $t\in (0,T]$ and
$B\in \mathcal{B}(\R)$, where $(X^x_t)_{t\in [0,T]}$ is the
solution to the SDE (\ref{eqn:sde:X}) starting in $x\in \R$,
such that the following is satisfied:
\begin{enumerate}[{\rm (1)}]
\item One has $\Gamma_X(s,\cdot,\xi)\in C^\infty(\R)$ for $(s,\xi)\in (0,T]\times \R$.
\item For $k\in \bN_0$ there is a $c=c(k)>0$ such that for
       $(s,x,\xi)\in(0,T]\times\R\times\R$ one has that
       \begin{equation}\label{eq:statement:friedman}
      \left | \frac{\partial^k \Gamma_X}{\partial x^k} 
              (s,x,\xi)\right |
          \leqslant c_{\eqref{eq:statement:friedman}}  s^{-\frac{k}{2}} \gamma_{c_{\eqref{eq:statement:friedman}}s}(x-\xi)
          \sptext{.6}{where}{.6}
             \gamma_t(\eta) 
          := \frac{1}{\sqrt{2\pi t}}e^{-\frac{\eta^2}{2t}}.
         \end{equation}
\item For $k\in \bN$ and $f\in C_X$  (the set $C_Y$ from \cref{sec:application_brownian_case} in the case (C1))
      one has
      \[     \frac{\partial^k}
                  {\partial x^k} \int_\R \Gamma_{X}(s,x,\xi)f(\xi)\od\xi
         =  \int_\R \frac{\partial^k \Gamma_X}{\partial x^k} 
                     (s,x,\xi)f(\xi)\od \xi
      \sptext{1}{for}{1}
      (s,x)\in(0,T]\times\R.\]
\end{enumerate}
\end{theo}


\section{A technical lemma}

\begin{lemm}\label{lemm:upper-bound-intergrand-equivalence}
For $\theta \in [0,1] $, a function $\varphi:[0,T)\to \R$, and 
a non-decreasing function $\Psi:[0,T)\to [0,\infty)$ the following assertions are equivalent:
\begin{enumerate}[{\rm (1)}]
\item \label{item:1-lemm:upper-bound-intergrand-equivalence} 
      There is a $c_{\cref{eq:item:1-lemm:upper-bound-intergrand-equivalence}}>0$ such that for any $0 \leqslant s \leqslant a <T$ one has 
      \begin{align} \label{eq:item:1-lemm:upper-bound-intergrand-equivalence}
     |\varphi_a - \varphi_s| \leqslant c_{\cref{eq:item:1-lemm:upper-bound-intergrand-equivalence}} 
     \frac{(T-s)^{\frac{\theta}{2}}}{(T-a)^{\frac{1}{2}}} \Psi_a.
     \end{align}
     \item \label{eq:item:2-lemm:upper-bound-intergrand-equivalence}
     \begin{enumerate}[\rm (a)]
			\item $\theta \in [0,1)$: There is a $c_{\cref{eq:item:2-lemm:upper-bound-intergrand-equivalence-theta<1}}>0$ such that 
for $a\in [0,T)$ one has
	\begin{align} \label{eq:item:2-lemm:upper-bound-intergrand-equivalence-theta<1}
	|\varphi_a - \varphi_0| \leqslant c_{\cref{eq:item:2-lemm:upper-bound-intergrand-equivalence-theta<1}}  
         (T-a)^{\frac{\theta-1}{2}} \Psi_a.
			\end{align}
			
			\item $\theta =1$: There is a  $c_{\cref{eq:item:2-lemm:upper-bound-intergrand-equivalence-theta=1}} >0$ such that 
for $0 \leqslant s \leqslant a <T$ one has
			\begin{align}\label{eq:item:2-lemm:upper-bound-intergrand-equivalence-theta=1}
			|\varphi_a - \varphi_s| \leqslant c_{\cref{eq:item:2-lemm:upper-bound-intergrand-equivalence-theta=1}} 
            \left( 1 + \ln \frac{T-s}{T-a}\right)\Psi_a.
			\end{align}
		\end{enumerate}
		
	\end{enumerate}
\end{lemm}

\begin{proof} 
$\eqref{item:1-lemm:upper-bound-intergrand-equivalence} \Rightarrow \eqref{eq:item:2-lemm:upper-bound-intergrand-equivalence}$
We let $t_n :=T-\frac{T}{2^n}$ for $n\geqslant 0$. 
If $s,a \in [t_{n-1}, t_n]$, $n\geqslant 1$, then \cref{eq:item:1-lemm:upper-bound-intergrand-equivalence} implies
	\begin{align*}
	|\varphi_a - \varphi_s| 
	\leqslant c_{\cref{eq:item:1-lemm:upper-bound-intergrand-equivalence}} \Psi_a T^{\frac{\theta}{2}-\frac{1}{2}} \frac{[1-(1-\frac{1}{2^{n-1}})]^{\frac{\theta}{2}}}{[1-(1-\frac{1}{2^{n}})]^{\frac{1}{2}}}\leqslant c_{\cref{eq:item:1-lemm:upper-bound-intergrand-equivalence}} \Psi_a T^{\frac{\theta -1}{2}}  (\sqrt{2})^{1 + (1-\theta)n}.
	\end{align*}
	We now let $s\in [t_{n-1}, t_n)$ and $a\in [t_{n+m-1}, t_{n+m})$ for  $n\geqslant 1$, $m\geqslant 0$ arbitrarily.
	If $\theta \in [0,1)$, then the triangle inequality and the monotonicity of $\Psi$ give 
\[ 	|\varphi_a - \varphi_0| 
	\leqslant  c_{\cref{eq:item:1-lemm:upper-bound-intergrand-equivalence}}  \Psi_a   T^{\frac{\theta -1}{2}} \sum_{k=1}^{n+m} (\sqrt{2})^{1 + (1-\theta)k} 
	\leqslant  c_{\cref{eq:item:1-lemm:upper-bound-intergrand-equivalence}} c_\theta  \Psi_a T^{\frac{\theta -1}{2}} (\sqrt{2})^{(1-\theta)(n+m-1)}\\
	\leqslant    \frac{c_{\cref{eq:item:1-lemm:upper-bound-intergrand-equivalence}} c_\theta \Psi_a}{(T-a)^{\frac{1-\theta}{2}}} 
\]
for some $c_\theta>0$ depending on $\theta$ only.	
	When $\theta =1$, similarly as above we get
	\begin{align*}
	|\varphi_a - \varphi_s| \leqslant c_{\cref{eq:item:1-lemm:upper-bound-intergrand-equivalence}} \Psi_a  \sqrt{2} (1 + m) \leqslant 2\sqrt{2} c_{\cref{eq:item:1-lemm:upper-bound-intergrand-equivalence}} \Psi_a   \left(1 + \ln \frac{T-s}{T-a}\right).
	\end{align*}
$\eqref{eq:item:2-lemm:upper-bound-intergrand-equivalence} \Rightarrow \eqref{item:1-lemm:upper-bound-intergrand-equivalence}$ 
    If $\theta \in [0,1) $, then  \cref{eq:item:2-lemm:upper-bound-intergrand-equivalence-theta<1} implies for any $0\leqslant s \leqslant a < T$ that 
	\begin{align*}
	|\varphi_a - \varphi_s|
	&\leqslant |\varphi_a -\varphi_0| + |\varphi_s - \varphi_0|
	\leqslant  c_{\cref{eq:item:2-lemm:upper-bound-intergrand-equivalence-theta<1}}\left[\Psi_a(T-a)^{\frac{\theta-1}{2}} + \Psi_s(T-s)^{\frac{\theta-1}{2}}\right]\\
	& \leqslant  c_{\cref{eq:item:2-lemm:upper-bound-intergrand-equivalence-theta<1}} \Psi_a  \left[\(\frac{T-a}{T-s}\)^{\frac{\theta}{2}} + \(\frac{T-a}{T-s}\)^{\frac{1}{2}}\right] 
	\frac{(T-s)^{\frac{\theta}{2}}}{(T-a)^{\frac{1}{2}}} 
	\leqslant 2 c_{\cref{eq:item:2-lemm:upper-bound-intergrand-equivalence-theta<1}} \Psi_a 
	\frac{(T-s)^{\frac{\theta}{2}}}{(T-a)^{\frac{1}{2}}}.
	\end{align*}
	The case $\theta=1$ is derived from the inequality $1 + \ln x \leqslant 2 \sqrt{x}$, $x \geqslant 1$.
\end{proof}


\section{Malliavin Calculus}
\label{sec:malliavin_calculus}
\subsection{It\^o's chaos decomposition and the Malliavin derivative}
\label{sec:malliavin_calculus:basics}

We assume the setting from Section \ref{sec:setting_levy_case}.
The compensated random measure $\wt N$ of $N$ is given by 
$\wt N:= N-\lambda \otimes \nu$ on the ring of $E\in \cB((0,T]\times \R)$ such that 
$(\lambda\otimes \mu)(E)<\infty$. By means of $\wt N$ the random measure $M$ is defined for sets $E \in\mathcal B((0,T]\times \R)$ with 
$(\lambda \otimes \mu)(E) <\infty$ by
\begin{align*}
M(E)   := \sigma \int_{(t,0)\in E} \od W_t 
          + \lim_{n \to \infty}\int_{E\cap ((0,T] \times \{\frac{1}{n} <|x| <n\})}x \widetilde{N}(\od t, \od x),
\end{align*}
where the limit is taken in $L_2$. For $n\geqslant 1$, set 
\begin{align*}
L_2^n :=L_2\(((0,T]\times\R)^n, \mathcal B(((0,T]\times \R)^n), (\lambda\otimes \mu)^{\otimes n}\).
\end{align*}
Let $I_n(f_n)$ denote the multiple integral of an $f_n\in L_2^n$ with respect to the random measure $M$ in the sense of 
\cite{It56} and let $I_n(L_2^n):=\{I_n(f_n) : f_n\in L_2^n\}$. 
If $f^s_n(z_1, \ldots, z_n)=\frac{1}{n!}\sum_{\pi}f_n(z_{\pi(1)}, \ldots, z_{\pi(n)})$ for $z_i=(t_i, x_i)\in (0,T]\times \R$
is the symmetrization of $f_n$, where the sum is taken over all permutations of $\{1,\ldots, n\}$, then $I_n(f_n)=I_n(f^s_n)$ a.s. 
For $n=0$ we agree about $L_2^0=\R$ and that $I_0:\R\to\R$ is the identity, so that $I_0(L_2^0)=\R$.
We also use $f_i^s=f_i$ for $f_i\in L_2^i$, $i=0,1$. 
The orthogonal chaos expansion 
$L_2 =\bigoplus_{n=0}^\infty I_n(L_2^n)$
is due to It\^o \cite{It56}:
given $\xi\in L_2$ there are $f_n\in L_2^n$  such that 
\begin{equation*}
\xi=\sum_{n=0}^\infty I_n(f_n) \mbox{ a.s.}, 
\end{equation*}
so that $I_0(f_0) = \E \xi$. By orthogonality one has 
$\|\xi\|_{L_2}^2=\sum_{n=0}^\infty \|I_n(f_n)\|_{L_2}^2=\sum_{n=0}^\infty n!\|f^s_n\|^2_{L_2^n}$. 
The Malliavin-Sobolev space $\mathbb D_{1,2}$ consists of all $\xi=\sum_{n=0}^\infty I_n(f_n) \in L_2$ such that 
\begin{align*}
\|\xi\|_{\D_{1, 2}}^2:=\sum_{n=0}^\infty (n+1)\|I_n(f_n)\|_{L_2}^2<\infty.
\end{align*}
Given $\xi\in \D_{1,2}$, the Malliavin derivative $D_\cdot \xi:(0,T]\times \R\times \Omega\to \R \in 
L_2(\lambda\otimes\mu\otimes\p)$ 
is formally given by $D_{s,z}\xi = \sum_{n=1}^\infty n I_{n-1}(f_n^s(\cdot,(s,z))$ 
(cf. \cite[Section 2.2]{SUV:07a}) and satisfies
\begin{multline}\label{eq:characterization_Malliavin_derivative}
    \int_\R \int_0^T \E \Big ( (D_{s,z} \xi) I_m(g_m) h(s,z) \Big ) \od s \mu(\od z) \\
 = (m+1)! \int_\R \int_0^T \cdots \int_\R \int_0^T \Big ( f^s_{m+1}((t_1,x_1),\ldots,(t_m,x_m),(s,z)) g_m((t_1,x_1),\ldots,(t_m,x_m))h(s,z) \Big )\\
         \od t_1 \mu(\od x_1) \cdots \od t_m \mu(\od x_m) \od s \mu(\od z) 
\end{multline}
for $h\in L_2^1$, $m\in \bN_0$, and symmetric $g_m \in L_2^m$.
\medskip

We summarize results from the literature we need:

\begin{lemm}\label{chaos-decom}
If a Borel function $f\colon \R\to \R$ satisfies 
$\E|f(X_T)|^2<\infty$, then there exist 
symmetric $f_n^s \in L_2(\mu^{\otimes n}):= L_2(\R^n, \cB(\R^n), \mu^{\otimes n})$ such that the following holds:
\begin{enumerate}[\rm (1)]
\item \label{item:1:chaos-decom} 
      One has $f(X_T)=\E f(X_T) + \sum_{n=1}^\infty I_n(f_n^s \1_{(0,T]}^{\otimes n})$ a.s.
\item \label{item:2:chaos-decom} 
      For $t\in [0, T)$ one has 
      $\ce{\F_t}{f(X_T)} = \E f(X_T) + \sum_{n=1}^\infty I_n(f_n^s \1_{(0,t]}^{\otimes n})$ a.s. 
      and $\ce{\F_t}{f(X_T)} \in \D_{1, 2}$.
\item \label{item:3:chaos-decom} 
      Assume additionally $f\in \defX$, let $F(s,x):= \E f(x+X_{T-s})$ for $(s,x)\in [0,T]\times \R$, and fix 
      $t\in (0,T)$. Then one has that 
      $\int_0^T \int_\Omega \int_\R |D_{s, z} F(t, X_t(\omega))|^2 \mu(\od z) \od \p(\omega) \od s < \infty$ and
      \begin{equation}\label{eqn:Eija_representation} 
           D_{s, z} F(t, X_t) 
         = \frac{\partial F}{\partial x}(t,X_t) \1_{(0, t] \times \{0\}}(s, z) + \frac{F(t, X_t + z) - F(t, X_t)}{z}\1_{(0, t] \times (\RO)}(s, z)
      \end{equation}
      for $\lambda \otimes \mu \otimes \p$-a.a. $(s, z, \omega) \in (0, T] \times \R \times \Omega$,
      where for $\sigma=0$ the first term on the right-hand side is omitted and for $\sigma>0$ one has that
      $(\partial F/\partial x)(t,\cdot)\in C^\infty(\R)$.
\end{enumerate}
\end{lemm}
\smallskip

\begin{proof} \eqref{item:1:chaos-decom} follows from \cite[Theorem 4]{BG17}.
\eqref{item:2:chaos-decom} The first claim is known. For the second one we use
	\begin{align*}
	\sum_{n=1}^\infty (n+1) \|I_n(f_n^s \1_{(0,t]}^{\otimes n})\|_{L_2}^2 & 
 =  \sum_{n=1}^\infty (n+1)! t^n \|f_n^s\|_{L_2(\mu^{\otimes n})}^2 
 =  \sum_{n=1}^\infty (n+1) \frac{t^n}{T^n} \|I_n(f_n^s \1_{(0,T]}^{\otimes n})\|_{L_2}^2 <\infty,
	\end{align*}
	which verifies $\ce{\F_t}{f(X_T)} \in \D_{1, 2}$ for $t\in [0, T)$. 
\eqref{item:3:chaos-decom}
Because $F(t, X_t)=\ce{\F_t}{f(X_T)}$ a.s. item \eqref{item:2:chaos-decom} implies 
$F(t, X_t) \in \D_{1, 2}$ so that \eqref{eqn:Eija_representation} follows by 
by  \cite[Corollary 3.1 of the second article]{La13}
(see also \cite{SUV:07a,Alos:Leon:Vives:08,Steinicke:14,CGeiss:Steinicke:16}).
If $\sigma>0$, then $F(t, \cdot) \in C^\infty(\R)$ by \cref{statement:consistence-upper_bound_sigma>0} for $q=2$.
\end{proof}

\subsection{Proof of \cref{statement:expression_for_integrand}}
\label{sec:proof_of_expression_for_integrand} 
We fix $t\in (0,T)$.
As  both sides in \cref{eqn:Eija_representation} are square-integrable in $(s,z,\omega)$ with respect to 
$\lambda \otimes \mu \otimes \p$,
$\rho \ll \mu$, and because $D\in L_2(\R,\mu)$ implies that both sides of \cref{eqn:Eija_representation} are integrable 
with respect to $\lambda \otimes \rho \otimes \p$,
we apply Fubini's theorem to get 
\begin{align}
&   \frac{1}{t}\int_0^t \int_\R  (D_{s, z} F(t, X_t))(\omega) \rho(\od z) \od s \notag \\
& = \frac{\partial F}{\partial x}(t,X_t(\omega)) \rho(\{0\}) + \int_{\RO} \frac{F(t, X_t(\omega) + z) - F(t, X_t(\omega))}{z} \rho(\od z) 
  = \opD_\rho F(t,X_t(\omega)) \label{eqn:integrated:eija_equation}
\end{align}
for $\omega\in \Omega\setminus N_t$ for some null-set $N_t$, where the integrals on the left-hand side and on the right-hand side (with respect to
$\rho(\od z) \od s$ and  $\rho(\od z)$, respectively) exist for $\omega\not\in N_t$.
Then, for $m\in \bN_0$ and a symmetric $g_m\in L_2^m$ we obtain from 
\eqref{eq:characterization_Malliavin_derivative} with $h(s,z):=\1_{(0,t]}(s) D(z)/\int_\R D\od \mu$ that 
\begin{align*}
&   \E  \left ( \frac{1}{t} \int_0^t \int_\R  D_{s,z} F(t, X_t) \rho(\od z) \od s \right ) I_m(g_m) \\
& = \frac{(m+1)!}{t} \int_0^t \int_\R \int_0^T \int_\R \cdots \int_0^T \int_\R f_{m+1}^s((t_1,x_1),\ldots,(t_n,x_n),(s,z)) \\
&   \hspace*{17em}                     g_m((t_1,x_1),\ldots,(t_n,x_n)) \mu(\od x_1) \od t_1 \cdots \mu(\od x_m) \od t_m \rho(\od z) \od s \\
& = (m+1)! \int_\R \int_0^t \int_\R \cdots \int_0^t \int_\R f_{m+1}^s(x_1,\ldots,x_n,z) \\
&   \hspace*{18em}                     g_m((t_1,x_1),\ldots,(t_n,x_n)) \mu(\od x_1) \od t_1 \cdots \mu(\od x_m) \od t_m \rho(\od z)  \\
& = \frac{m!}{\int_\R D \od \mu} \int_0^T \int_\R \cdots \int_0^T \int_\R \left [ (m+1) h_m^{D}(x_1,\ldots,x_n) \1_{(0,t]}^{\otimes n}(t_1,\ldots,t_n) \right ]
       g_m((t_1,x_1),\ldots,(t_n,x_n)) \\
& \hspace*{31em}  \mu(\od x_1) \od t_1 \cdots  \mu(\od x_m) \od t_m \\
& = \E \frac{\psi_t(f(X_T),D)}{\int_\R D \od \mu} I_m(g_m).
\end{align*}
So \eqref{eqn:integrated:eija_equation} gives
$  \frac{\psi_t(f(X_T),D)}{\int_\R D \od \mu}
 = \frac{1}{t} \int_0^t \int_\R  D_{s,z} F(t, X_t) \rho(\od z) \od s  
 = \opD_\rho F(t,X_t)$ 
outside a null-set $\widetilde{N}_t \supseteq N_t$.
\qed


\section{Proof of $h_{\theta,a}\in \hoelO{\theta,q}$}
\label{sec:proof_example_Hoelder}
The case $a=0$ is obvious as $\hoelO{\theta,\infty}$ are the $\theta$-H\"older continuous functions vanishing 
at zero. Let $a>0$. As 
$K(v,h_{\theta,a};C_b^0(\R),\hoell{1}{0}) \leqslant h_{\theta,a}(1)$
for $v\in [1,\infty)$, we only need to check
that $\left \| v^{-\theta} K(v,h_{\theta,a};C_b^0(\R),\hoell{1}{0}) \right \|_{L_q\left ( (0,1];\frac{\od v}{v} \right )}< \infty$
for $a>1/q$. This follows from 
\[ v^{-\theta} K(v,h_{\theta,a};C_b^0(\R),\hoell{1}{0}) \leqslant (1+\theta) \left (\frac{A}{A-\log v}\right )^a 
  \sptext{1}{for}{1} v\in (0,1]. \]
To verify this fix $v\in (0,1]$, and let
$f:=h_{\theta,a}$ and
$K(y) :=  \1_{\{ 0< y \leqslant 1\}} \theta y^{\theta -1} \left ( \frac{A}{A-\log y} \right )^a$.
For $x \geqslant 0$ we decompose 
$f(x) = f_1^{(v)}(x) + f_b^{(v)}(x)$ with 
$ f_1^{(v)}(x):= \int_0^{x\wedge 1} (K(y) \wedge K(v)) \od y$ and $f_b^{(v)}(x):= f(x) -f_1^{(v)}(x)$
(for $x<0$ the decomposing functions are defined to be zero). 
By definition we have $\| f_1^{(v)}\|_{\hoell{1}{0}}\leqslant K(v)$.
Exploiting the monotonicity of $K$, where we use $a<(1-\theta)A$, we also have
$\|f_b^{(v)}\|_{C_b^0(\R)}\leqslant \int_0^v K(y) \od y$. Finally, a computation yields
$\int_0^v K(y) \od y \leqslant \frac{v}{\theta} K(v)$ so that
\[ K(v,h_{\theta,a};C_b^0(\R),\hoell{1}{0}) 
   \leqslant \frac{v}{\theta} K(v) + v K(v) 
      =  \left ( \frac{1}{\theta}+1\right ) v K(v)
      =  \left ( 1 + \theta \right ) v^\theta \left ( \frac{A}{A-\log v} \right )^a. \]
\qed


\bibliographystyle{amsplain}

\end{document}